\documentclass[b5paper,11pt,openany]{book}

\usepackage{amsmath}
\usepackage{amsfonts}
\usepackage{graphicx}
\usepackage{setspace}
\usepackage{amsmath}
\usepackage{amssymb}
\usepackage{latexsym}
\usepackage{amsmath, amsfonts,amssymb, amsthm, euscript,makeidx,color,mathrsfs}

\usepackage{geometry}
\usepackage[numbers,sort&compress]{natbib}

\usepackage{fancyhdr}
\fancypagestyle{mystyle}{
    \fancyhf{} 
    \fancyhead[RE,LO]{\thepage} 
    \fancyhead[LE]{\footnotesize{\leftmark}} 
    \fancyhead[RO]{\footnotesize{\leftmark}} 
}

\usepackage[linktocpage]{hyperref}
\hypersetup{
    colorlinks=true,
    linkcolor=blue,
    urlcolor=cyan,
    citecolor=cyan
}
\usepackage{caption}

\RequirePackage[capitalize,nameinlink]{cleveref}

\crefname{section}{section}{sections}
\crefname{subsection}{subsection}{subsections}
\Crefname{section}{Section}{Sections}
\Crefname{subsection}{Subsection}{Subsections}

\crefname{condition}{Condition}{Conditions}

\Crefname{figure}{Figure}{Figures}

\crefformat{equation}{\textup{#2(#1)#3}}
\crefrangeformat{equation}{\textup{#3(#1)#4--#5(#2)#6}}
\crefmultiformat{equation}{\textup{#2(#1)#3}}{ and \textup{#2(#1)#3}}
{, \textup{#2(#1)#3}}{ and \textup{#2(#1)#3}}
\crefrangemultiformat{equation}{\textup{#3(#1)#4--#5(#2)#6}}%
{ and \textup{#3(#1)#4--#5(#2)#6}}{, \textup{#3(#1)#4--#5(#2)#6}}{ and \textup{#3(#1)#4--#5(#2)#6}}

\Crefformat{equation}{#2Equation~\textup{(#1)}#3}
\Crefrangeformat{equation}{Equations~\textup{#3(#1)#4--#5(#2)#6}}
\Crefmultiformat{equation}{Equations~\textup{#2(#1)#3}}{ and \textup{#2(#1)#3}}
{, \textup{#2(#1)#3}}{ and \textup{#2(#1)#3}}
\Crefrangemultiformat{equation}{Equations~\textup{#3(#1)#4--#5(#2)#6}}%
{ and \textup{#3(#1)#4--#5(#2)#6}}{, \textup{#3(#1)#4--#5(#2)#6}}{ and \textup{#3(#1)#4--#5(#2)#6}}

\crefdefaultlabelformat{#2\textup{#1}#3}

\newtheorem {theorem}{Theorem}[section]
\newtheorem {lemma}[theorem]{{\bf Lemma}}
\newtheorem {corollary}[theorem]{{\bf Corollary}}
\newtheorem {proposition}[theorem]{{\bf Proposition}}
\theoremstyle{remark}
\newtheorem {remark}{{\bf Remark}}[section]

\newtheorem {condition}{{\bf Condition}}[section]
\theoremstyle{definition}
\newtheorem {definition}{{\bf Definition}}[section]

\theoremstyle{plain} \numberwithin {equation}{section}

\newtheorem{problem}{Problem}
\numberwithin{assumption}{section}

\numberwithin{problem}{chapter}

\usepackage{enumerate}

\def\deq{\mathop{\buildrel\Delta\over=}}

\allowdisplaybreaks
\usepackage{lipsum}

\begin{document}
\title{Inverse problems for stochastic partial differential equations}
\author{Qi L\"{u} and Yu Wang}

\maketitle

\tableofcontents

\newpage

\pagestyle{mystyle}

\chapter{Introduction}

In this chapter, we discuss the importance and challenges of studying inverse problems for stochastic partial differential equations (SPDEs for short). 
We also present the formulation of the main problems in this book.

\section{Why should we study  inverse problems for SPDEs?}

Physical theories enable us to make predictions: with a complete description of a physical
system, we can predict the outcome of  certain  measurements. This issue of predicting measurement results is commonly referred to as the ``direct problem".  The concept of an ``inverse problem" is defined in contrast to a ``direct problem", involving using the actual  
measurements to deduce the values of the parameters of the system. 
One can give a rough description for inverse problem as follows: 
Assume $ U $ and $ V $ are two topological spaces, and $ F: U \rightarrow V $ is a given mapping.
For a given input $ u \in U $, finding the value of $ F(u) $ is called the direct problem. Conversely, given a $ v \in V $, finding the unknown input $ u \in U $ such that $ F(u) = v $ is called the inverse problem.

The inverse problems typically arise in scenarios where the parameters of a ``physical device'' or system cannot be directly measured. 
Instead, these parameters need to be determined from measurements of the effects they produce. 
A typical example is that in order to determine the resistance of a material, people measure the voltage applied to the material and the flux on its surface. Then, the required resistance can be determined through Ohm's Law.

In recent decades, the theory of inverse problems has evolved significantly, becoming a prominent branch of applied mathematics
and has had stimulating effects on partial differential equations, functional analysis, numerical analysis, 
and other fields.

Recall the definition of the well-posedness by Hadamard:
\begin{definition}
\label{defWellPoseHadmard}
Consider the equation
\begin{equation}\label{defWellPoseHadmard-eq1}
F(u) = v, \qquad u\in U,\; v \in V.
\end{equation}
The equation \eqref{defWellPoseHadmard-eq1} is called well-posed in the sense of Hadamard if the following three conditions are satisfied:
\begin{enumerate}[(i)]
    \item For any $ v \in V $, there exists a solution $ u \in U$ such that $ F(u) = v $.
    \item This solution is unique.
    \item The solution $ v = F^{-1}(u) $ depends continuously on $ u $ in $U$.
\end{enumerate}
If at least one of these conditions is failed, the problem is called ill-posed.
\end{definition}
\begin{remark}
If (i) and (ii) are satisfied, we know there is a unique solution to the equation \eqref{defWellPoseHadmard-eq1}. However, for practical purposes, the necessity of condition (iii) becomes apparent. Indeed, in real-world scenarios, the exact value of $v$ is often unattainable, and we must rely on an approximate value $\tilde{v}$ instead. Failure to meet condition (iii) can lead to a significant deviation between the solution $\tilde{u}$
of the equation $F(\tilde{u}) = \tilde{v}$ and the desired solution $u$. Therefore, the continuity of the inverse mapping $F^{-1}$ plays a critical role in practical applications.
\end{remark}

Many classical initial boundary value problems for partial differential equations (PDEs for short) satisfy the three conditions of \cref{defWellPoseHadmard}, that, there is a unique solution to the PDEs with given initial/boundary value, coefficients and nonhomogeneous terms and the solution depends on these data continuously.  
On the other hand, by virtue of its significant applications, lots of inverse problem for PDEs, i.e., determine the initial/boundary value, or/and coefficients or/and nonhomogeneous terms from suitable measurements of the solution, are introduced (e.g.,\cite{HasanovHasanoglu2021,Lavrentiev2003, Lesnic2022,   Prilepko2000}). Unfortunately, many of these inverse problems  do not satisfy   Conditions (i)--(iii) in Definition \ref{defWellPoseHadmard}.  
This means that most inverse problems are generally ill-posed.


Lots of methods and techniques have been developed to solve inverse problems for deterministic PDEs. 
Of particular note is the Carleman estimate, which stands out as one of the most useful  tools. 
We will not  exhaustively list all literatures on this topic since these are too many. 
However, for the application of Carleman estimates to various inverse problems for PDEs, we refer readers to \cite{Beilina2012,Bellassoued2017,Choulli2016,fu2019carleman,Klibanov2021} and the rich references therein.

Due to the complexity of the real world, many phenomena in fields such as chemistry, biology, microelectronics industry, pharmaceuticals, communication, and transmission require the use of stochastic partial differential equations (SPDEs for short) for better characterization (e.g. \cite{Carmona1999,DaPrato2014,Holden2010,Kotelenez2008,Majda2001}). 
Hence, the study of inverse problems for SPDEs satisfies a natural requirement (e.g., \cite[Section 3.6]{Murray2003}). 
Furthermore, the inverse problems of SPDEs are also related to other important issues, such as data assimilation and Kalman filtering for SPDEs (e.g., \cite{Knopov2013,Law2015}).

Mathematically, inverse problems for SPDEs are relatively new
and quite challenging. 
Compared to the extensive literature on deterministic inverse problems for PDEs, there is relatively less research on  inverse problems for SPDEs.
In our view, the main reasons are as follows:

\begin{itemize}
\item Compared to deterministic PDEs, much less is known about SPDEs, although many progresses have been made there in recent years.
\item The formulation of inverse problems for SPDEs may differ fundamentally from that of PDEs, with some problems being genuinely stochastic and cannot be reduced to any deterministic ones. 
Some inverse problems may even be meaningful only within a stochastic setting (e.g. \cite{Lue2015a} or \cref{rkStoWave}).
\item Many powerful methods and tools applicable to solving inverse problems for deterministic PDEs cannot be directly applied to the study of inverse problems for SPDEs.
For instance, due to the almost surely non-differentiability of sample paths of Brownian motion with respect to the time variable, solutions to SPDEs may not be continuously differentiable with respect to the time variable. 
Furthermore, many regularity estimates and extension theorems valid for deterministic PDEs no longer apply to SPDEs.
Therefore, even for very simple SPDEs, it is required to present new mathematical techniques or even fundamentally new approaches.
\end{itemize}

\section{Formulation of inverse problems for SPDEs}\label{Ch1-sec2}

Similar to the formulation of inverse problems for deterministic PDEs, for inverse problems of SPDEs, we  aim to determine unknown components within the system through some measurements derived from the solution of the equations. 
These inverse problems of SPDEs can be broadly categorized into the following four types based on the specific unknown components to be determined:
\begin{enumerate}[(i)]
\item Inverse coefficient problems, i.e., to determine an unknown coefficient (usually a function) in an SPDE in such a way that the solution of the equation satisfies specific conditions or data. In other words, the aim is to infer the parameters or coefficients that lead to the observed phenomena. These problems have wide applications in science and engineering, such as in medical imaging for recovering tissue properties, in seismology for inferring subsurface structures, in materials science for determining material properties, and more.
\item Inverse source problems, which refer to the task of determining an unknown source or input function within an SPDE based on observed effects or responses.  These types of problems are common in various scientific and engineering fields where understanding the hidden sources or parameters is crucial. Applications include medical imaging, geophysics, environmental monitoring, and many others. 
\item Inverse initial value/state problems, which refer to the mathematical problems where the goal is to determine unknown initial value of the SPDEs   based on observed data at a later time or in a time interval.  These problems are important in various fields of industry, where the whole solution or the initial value of an SPDE  cannot be measured directly.
\item Inverse geometric shape of the domain problems, which refer to the task of determining the shape or geometry of a domain where the SPDE is involved from measured data or observations.  The importance of solving inverse geometric shape of the domain problems lies in their practical applications, such as in medical imaging for diagnosing diseases, in geophysics for mapping subsurface structures, and in material science for characterizing material properties.  
\end{enumerate}

Below, we further explain the aforementioned four inverse problems through a simple stochastic parabolic equation. In order to succinctly clarify the key points of the issue, we will only consider a very basic  equation. The formulation of these problems is similar for more complex equations. 

Let $T > 0$, $G \subset \mathbb{R}^{n}$ ($n \in \mathbb{N}$) be a given bounded domain with a $C^{2}$ boundary $\Gamma$ . 
Denote $ Q = (0,T) \times G $ and $ \Sigma = (0,T) \times \Gamma $.

Consider the following stochastic parabolic equation:
\begin{align}\label{eqSP}
\left\{
\begin{aligned}
    & dy  + \Delta y d t= (a y + f ) dt +(by+g) dW(t)&\mbox{ in }Q,\\
  & \gamma_{1} \frac{\partial y}{\partial \nu}  + \gamma_{2}  y = \varphi  &\mbox{ on }\Sigma,\\
  & y(0)=y_0 &\mbox{ in }G,
\end{aligned}
\right.
\end{align}
where  $a$, $b$, $f$, $g$, $\gamma_1$ and $\gamma_2$ are suitable functions/stochastic processes and
$ W(t) $ denotes the standard  Brownian motion (see \cref{defBrowian}).

The direct problem is to determine the solution $ y  $ of system \cref{eqSP} by specifying  $ (a,b,f, g, \gamma_1, \gamma_1, \varphi, y_{0}) $ along with $Q$. This can be achieved using classical well-posedness results of stochastic evolution equations (\cref{eqWellposed}).  
However, in many practical situation it is not possible to measure experimentally all these data. Instead, it is possible to measure   certain information of the solution and use this information together with other data to recover
the missing, unmeasurable data.  More precisely, the inverse problem for \cref{eqSP} involves certain unknown components in  $ (a,b,f, g, \gamma_1, \gamma_1, \varphi, y_{0}) $ and $Q$, which need to be determined through measuring some information about the solution $y$:
\begin{align*}
\mathcal{M}(y)(t,x) = h(t,x), \quad \quad (t,x) \in \mathcal{O} \subset [0,T] \times \overline{G}.
\end{align*}
Here, $\mathcal{M}(\cdot)$ is a known operator, representing the method of obtaining measurement data, and $\mathcal{O}$ denotes the position of the measurement.

In general, there are three distinct types of measurements:
\begin{itemize}
	\item Distributed measurement involves measuring the solution value on a subset of the domain.
	\item Terminal measurement entails measuring the solution value at the terminal state.
	\item Boundary measurement refers to observing Dirichlet or Neumann data on a portion of the boundary.
\end{itemize}

Determining unknown functions $a$ and/or $b$ are known as inverse coefficient problems.
Reconstructing unknown components in $f$ and/or $g$  is termed as inverse source problems.
When the goal is to reconstruct the initial condition $y_{0}$, it is known as an inverse initial value problem. 
Furthermore, in cases where a section of the boundary $\Gamma$ is both unknown or changing with time, deducing the shape or geometry of the domain from this boundary information can be seen as a type of inverse geometric shape of the domain problem.

For a given inverse problem, generally speaking, three main concerns are the following:
\begin{enumerate}[(i)]
\item Can measurement data determine unknown components, or what information can be derived from it?
\item Is it feasible to reconstruct unknown components from measurement data?
\item Are there efficient algorithms available to construct an approximate solution?
\end{enumerate}

While the importance of SPDEs is increasingly,  inverse problems
governed by these equations arise naturally. Various types of inverse problems for SPDEs have attracted the interest of more and more researchers in many fields. 
In the past decade, some interesting progress has been made (e.g., \cite{Dou2022,Feng2022,Fu2017,Garnier2009,Gong2021,Helin2014,Helin2018,Li2013,Li2022,Li2021,Lue2012,Lue2013a,Lue2013c,Lue2015b,Lue2015a,Ma2021,Niu2020,Wu2020,Wu2022} and the references therein).  This book aims to provide a brief overview of recent advancements in the theory of inverse problems for stochastic partial differential equations. In order to keep the content concise, we will only discuss the inverse problems of two typical classes of stochastic partial differential equations: second-order stochastic parabolic equations and second-order stochastic hyperbolic equations. 

The main tool for studying these inverse problem is Carleman estimate. We do not intend to pursue any general treatment of the Carleman estimates
themselves and choose
direct arguments based on basic stochastic calculus, rather than more general sophisticated
methods.  

As this field is still developing and there are many challenging issues to be addressed, the purpose of this book is not to serve as a comprehensive summary, but rather to spark interest and encourage further exploration in this area among readers. We prefer to present results that, from our perspective, include fresh and promising ideas. In cases where a complete mathematical theory is lacking, we only provide the available results.
We do not intend for the current book to be
encyclopedic in any sense, and the references are limited. 

We aim  to organize the material
in such a way that no prerequisites except a moderate knowledge of partial differential equations
and functional analysis are expected. During the preparation of
this book we found that this goal was too ambitious. For example, it will cost lots of space for introducing the basic knowledge for SPDEs. Therefore, in   \Cref{chApp} we collect  some definitions and theorems concerning SPDEs needed in this book and provide references for further study. 

From an expert's perspective, certain content in various chapters or sections may have slight overlaps, which are intentionally retained to aid the convenience of readers, especially beginners. 

The rest of this book is organized as follows: 
Chapters \ref{ch2} and \ref{ch3} are devoted to inverse problems for stochastic parabolic equations and stochastic hyperbolic equations, respectively. In  \Cref{chApp}, we  recall the mathematical background material for SPDEs needed in this book.

To present the idea and methods, we do not aim for the generality of equations, nor do we intend to present the most general results. Readers interested in these aspects can find suitable references in Sections \ref{secFurPro} and \ref{secFChyper}.

A number of our colleagues, who are listed below in the alphabetical order, have
helped us in our work on this book. Prof. Fangfang Dou and Dr. Peimin L\"u have collaborated with the second author on the topic of reconstruction problem for stochastic parabolic equations (\cite{Dou2024}).  Prof. Zhongqi Yin
has collaborated with the first author on the subject of inverse state problem for stochastic hyperbolic equations (\cite{Lue2020a}). He also read the draft of this book carefully and provided numerous useful suggestions. Prof. Xu Zhang has collaborated
with the first author on the development of the idea of using Carleman estimate to study inverse source problem for stochastic hyperbolic equations (\cite{Lue2015a}). We sincerely appreciate a great help of all these individuals.

\chapter[Inverse problems for stochastic parabolic equations]{Inverse problems for stochastic parabolic equations}
\label{ch2}


This chapter is devoted to   inverse problems for stochastic parabolic equations.

Throughout this chapter, let $T>0$ and $(\Omega, \mathcal{F}, \mathbf{F}, \mathbb{P})$ with $\mathbf F=\{\mathcal{F}_{t}\}_{t \geq 0}$ be a complete filtered probability space on which a one-dimensional standard Brownian motion $\{W(t)\}_{t \geq 0}$ is defined and $\mathbf{F}$ is the natural filtration generated by $W(\cdot)$.
Write $ \mathbb{F} $ for the progressive $\sigma$-field with respect to $\mathbf{F}$. Readers who are not familiar with the notions aforementioned are referred to \cref{chApp}.

Let $T > 0$, $G \subset \mathbb{R}^{n}$ ($n \in
\mathbb{N}$) be a given bounded domain with a
$C^{2}$ boundary $\Gamma$ and $G_0\subset G$ be a nonempty open subset. Denote by $ \mathbf{L}(G) $ all Lebesgue measurable subsets of $G$.  Put
\begin{align*}
Q = (0,T) \times G, \quad \Sigma = (0,T) \times \Gamma, \quad  Q_0=(0,T) \times G_0.
\end{align*}
Also, unless other stated, we denote
by $\nu(x) = (\nu^1(x), \nu^2(x), \cdots, \nu^n(x))$ the unit
outward normal vector of $\Gamma$ at $x\in \Gamma$. Let $(b^{jk})_{1\leq j,k\leq n} \in C^1(G;{\mathbb{R}}^{n\times n})$ satisfying that $b^{jk} = b^{kj}$ ($j,k = 1,2,\cdots, n$), and for some constant $s_0
> 0$,
\begin{align} \label{eq.bijGeq}
\sum_{j,k=1}^nb^{jk}\xi_{j}\xi_{k} \geq s_0 |\xi|^2,
\quad  \forall\, (x,\xi)\buildrel \triangle \over = (x,\xi_{1}, \cdots,
\xi_{n}) \in G \times \mathbb{R}^{n}.
\end{align}
In what follows, we use ${\cal C}$ to denote a generic positive constant, which may vary from
line to line.

\section{Formulation of the problems}

\subsection{Inverse state problem with an internal measurement. I}

Let $F: [0,T] \times \Omega \times G \times \mathbb{R}^n \times \mathbb{R} \to \mathbb{R}$ and $K: [0,T] \times \Omega \times G \times \mathbb{R} \to \mathbb{R}$ be two functions such that, for each $(\varrho,\zeta) \in \mathbb{R}^n \times \mathbb{R}$, the functions $F(\cdot,\cdot,\cdot,\varrho,\zeta): [0,T] \times \Omega \times G \to \mathbb{R}$ and $K(\cdot,\cdot,\cdot,\zeta): [0,T] \times \Omega \times G \to \mathbb{R}$ are $\mathbb{F} \times \mathbf{L}(G)$-measurable, where $ \mathbf{L}(G) $ denotes all Lebesgue measurable sets; and for a.e. $(t,\omega,x)\in [0,T]\times\Omega\times G$ and any $(\varrho_{i},\zeta_{i})\in \mathbb{R}^n \times\mathbb{R}$ ($i=1,2$),
\begin{align}\label{10.5-eq1-1}
\left\{
\begin{aligned}
&|F(t,x,\varrho_1,\zeta_1)-F(t,x,\varrho_2,\zeta_2)|\le
L(|\varrho_1-\varrho_2|_{\mathbb{R}^n}+|\zeta_1-\zeta_2|), \\
&|K(t,x,\zeta_1)-K(t,x,\zeta_2)| \leq
L|\zeta_1-\zeta_2|,\\
&|F(\cdot,\cdot,\cdot,0,0)|\in L^2_{\mathbb{F}}(0,T;L^2(G)), \quad
|K(\cdot,\cdot,\cdot,0)|\in L^2_{\mathbb{F}}(0,T;H^1_0(G))
\end{aligned}
\right.
\end{align}
for some constant $L>0$.

Let us consider the following semilinear stochastic parabolic equation:
\begin{align}\label{sp-eq1}
\left\{
\begin{aligned}
& dy - \sum_{j,k=1}^n(b^{jk}y_{x_j})_{x_k}dt=F(\nabla y, y)dt + K(y) dW(t)&\mbox{ in }Q,\\
& y=0&\mbox{ on }\Sigma,\\
& y(0)=y_0 &\mbox{ in }G,
\end{aligned}
\right.
\end{align}
where $y_0\in L^2_{{\mathcal F}_0}(\Omega; L^2(G))$ is the initial data.

By the classical well-posedness result for stochastic evolution equations (see \cref{eqWellposed}), we know that \cref{sp-eq1}  admits a unique weak solution $y\in
L^2_\mathbb{F}(\Omega;$  $C([0,T];L^2(G)))\times L^2_\mathbb{F}(0,T;H^2(G)\cap H_0^1(G))$.

In order to simplify notation, we will suppress the time variable $t$ ($\in [0,T]$), the sample point $\omega$ ($\in \Omega$), and/or the space variable $x$ ($\in G$) in the functions if there is no risk of confusion.

Define a map  ${\mathcal M}_1: L^2_{{\mathcal F}_0}(\Omega; L^2(G))\to L^2_{{\mathbb{F}}}(0,T;L^2(G_0))$ by
\begin{align*}
\mathcal M_1(y_0)=  y \big |_{(0,T)\times G_0}
\end{align*}
where $y$ solves the equation \cref{sp-eq1}.
The direct problem means that once $y_0$ is given, one can solve the equation \cref{sp-eq1} to obtain ${\mathcal M}_1(y_0)$.
The first inverse problem considered in this chapter  involves determining $y_0$ from the value ${\mathcal M}_1(y_0)$.
More precisely, we consider the following question:
\begin{problem}\label{prob.2.0}
Can we determine $y_0$ from the measurement ${\mathcal M}_1(y_0)$?
\end{problem}
%

The problem described in   \Cref{prob.2.0} involves determining the unknown initial value of the equation \eqref{sp-eq1}. As discussed in Chapter 1, in practical applications, our goal is not only to find $y_0$ from ${\mathcal M}_1(y_0)$, but also to ensure that $y_0$ varies continuously with ${\mathcal M}_1(y_0)$. Specifically, in the presence of measurement errors in ${\mathcal M}_1(y_0)$, we anticipate that the discrepancy between the function derived from this noisy measurement and the true initial value $y_0$ remains bounded by the measurement error in a suitable manner. However, due to the regularization effect of equation \eqref{sp-eq1}, the inverse of the mapping ${\mathcal M}_1$
is not a bounded linear operator. Therefore, it becomes necessary to either relax the constraints or incorporate additional prior information about the initial value to achieve the desired error estimation. We will  address the first approach  and study the following problem (The second one will be introduced to study \cref{prob.2.2} below).

\begin{problem}\label{prob.2.1}
For each $t\in (0,T]$, does there exist a constant
${\mathcal C}(t) > 0$ such that for any initial data $y_0, \hat y_0\in L^2_{{\mathcal F}_0}(\Omega; L^2(G))$,
\begin{equation}\label{sp-eq2}
|y(t) -\hat y(t)|_{L^2_{{\mathcal F}_t}(\Omega;L^2(G))} \leq {\mathcal C}(t)|{\mathcal M}_1(y_0)
-{\mathcal M}_1(\hat y_0)|_{L^2_{{\mathbb{F}}}(0,T;L^2(G_0))}?
\end{equation}
Here $y(\cdot)$ and $\hat y(\cdot)$ are solutions to \cref{sp-eq1} with the
initial data $y_0 $ and $\hat y_0 $, respectively.
\end{problem}
\begin{remark}
\cref{prob.2.1} asks whether the solution $y(\cdot)$ of  system \cref{sp-eq1} at $t\in (0,T]$ depends on the observation $ y|_{(0,T)\times     G_0}$   continuously within the space $L^2_{{\mathcal F}_t}(\Omega;L^2(G))$. As a consequence of the inequality \eqref{sp-eq2}, we can deduce that for any $t\in (0,T]$, $y(t;y_0)$ can be uniquely determined by the measurement ${\mathcal M}_1(y_0)$. This, combined with the continuity of $y(\cdot;y_0)$ with respect to the time variable, implies that ${\mathcal M}_1(y_0)$ can uniquely determine $y_0$. Therefore, if the answer to   \cref{prob.2.1} is affirmative, we promptly obtain a positive solution to  \cref{prob.2.0}. On the contrary,  in the inequality \eqref{sp-eq2}, the constant ${\cal C}(t)$ is time-dependent, and it diverges as $t\to\infty$. Consequently, we lose the continuous dependence of $y_0$ on ${\mathcal M}_1(y_0)$ within the space $L^2_{{\mathcal F}_0}(\Omega;L^2(G))$.
\end{remark}

We will provide a positive answer to this problem in \Cref{secDeterminationProblemdistributing1}.

\subsection{Inverse state problem with an internal measurement. II}

In the above inverse problem, we need some boundary conditions.
Next, we introduce an inverse problem where the boundary condition is unknown.

Consider the following stochastic parabolic equation:
\begin{equation}
\label{system-sp1}
d \tilde y -\sum_{j,k=1}^n (b^{jk}\tilde y_{x_j})_{x_k} dt =
\tilde a_1\cdot \nabla \tilde y dt + \tilde a_2 \tilde y dt + \tilde a_3\tilde y dW(t)\quad \text{in } Q.
\end{equation}
Here, $\tilde a_1 \in L^{\infty}_{{\mathbb{F}}}(0, T; L^{\infty}_{\rm loc}(G; {\mathbb{R}}^n))$, $\tilde a_2\in L^{\infty}_{{\mathbb{F}}}(0, T; L^{\infty}_{\rm loc}(G))$ and $\tilde a_3 \in  L^{\infty}_{{\mathbb{F}}}(0, T$;  $ H^{1}_{\rm loc}(G))$.

We first recall the definition of the solution to the equation \cref{system-sp1}.
\begin{definition}
We call  a stochastic process
\begin{align*}
\tilde y\in L_{{\mathbb{F}}}^2(\Omega;C([0,T];L^2_{\rm loc}(G)))\cap L_{{\mathbb{F}}}^2(0,T;H^1_{\rm loc}(G))
\end{align*}
a solution to the equation \cref{system-sp1} if for any
$t\in[0,T]$, $\widehat G\subset\!\subset G$ and $\rho\in H_0^1(\widehat G)$, it holds that
\begin{align}\label{def eq2} \notag
& \int_{\widehat G}\tilde y(t,x)\rho(x)dx-\int_{\widehat G}\tilde y(0,x)\rho(x)dx
\\ \notag
& = \int_0^t\int_{\widehat G} \!\Big[\!-\!\sum_{j,k=1}^n b^{jk}
\tilde y_{x_j}(s,x)\rho_{x_k}(x) \!+\!\tilde a_1\cdot \nabla \tilde y(s,x)\rho(x)
\!+\! \tilde a_2     \tilde y(s,x)\rho(x)     \Big]dx ds
\\
& \quad +\int_0^t\int_{\widehat G} \tilde a_3 \tilde y(s,x)\rho (x) dx dW(s), \quad ~ \text{${\mathbb{P}}$-a.s.}
\end{align}
\end{definition}

We consider the following problem
\begin{problem}\label{probP8}
Can we determine $\tilde y$ in $(0,T)\times G$ uniquely from
$\tilde y \big|_{(0,T)\times G_0}$?
\end{problem}

\begin{remark}
Please note that the problem described in \cref{probP8} remains an inverse state problem. It arises when  the physical process is governed by a stochastic parabolic equation, but boundary value of the solution is unknown. We can determine $\tilde y(\cdot)$  in $Q$  uniquely. However, unlike the inverse state problem for the equation \cref{sp-eq1}, we will lose the continuous dependence of $\tilde y(t)$ ($t\in (0,T]$) with respect to $\tilde y \big|_{(0,T)\times G_0}$ in $L^2_{\mathcal{F}_t}(\Omega;L^2(G))$. In such case,
instead of whole domain $ G $, we will determine solutions over a subset of whole domain $ G $  (see \cref{main result}).
\end{remark}

\subsection{Inverse state problem with a terminal measurement}

Another important class of inverse state problems is  known as the backward problem, which means determining the solution $ y $ to the equation \cref{sp-eq1} when provided with data $ y(T) $ at time $ t = T $. Let us present it below.

%


\begin{problem}\label{prob.2.2}
Can we  determine $y(t_0)$ in $G$, ${\mathbb{P}}$-a.s. for $t_0\in [0,T)$ from $y(T)$?
\end{problem}

Stochastic parabolic equations are ill-posed backward in time.
Minor errors in measuring the terminal data may cause huge deviations in the final results, i.e., there is no stability for \cref{prob.2.2}.
Nevertheless, if we assume an \emph{a priori} bound for $y(0)$ (an assumption aligned with practicality), then stability can be reestablished in some sense.
Conditional stability is the term used to describe such a kind of stability.

In a general framework, the conditional stability problem (for the solution $y$ to the equation \cref{sp-eq1}) can be formulated as follows:

{\it For given $t_0\in (0,T]$, $\alpha_1\geq 0$, $\alpha_2\geq 0$  and $M>0$, find (if possible) a nonnegative function $\beta\in C[0,+\infty)$ such that $\lim\limits_{\eta\to 0}\beta(\eta)=0$ and that
\begin{align}
\label{eqStableTotal}
|y(t_0)|_{L^2_{\mathcal{F}_{t_0}}(\Omega;H^{\alpha_1}(G))} \leq
\beta\big(|y(T)|_{L^2_{\mathcal{F}_{T}}(\Omega;H^{\alpha_2}(G))}\big),\quad\forall\;y_0\in U_{M,\alpha_1},
\end{align}
where
\begin{align} \label{eqStableM}
U_{M,\alpha_1} \buildrel \triangle \over = \big\{f \in L^2_{\mathcal{F}_0}(\Omega;H^{\alpha_1}(G))\; \big|\;|f|_{L^2_{\mathcal{F}_0}(\Omega;H^{\alpha_1}(G))} \leq M \big\}.
\end{align}
}

\begin{remark}
Here we expect the existence of $\beta$ with the assumption that $y_0$ belongs to a special set $U_{M, \alpha_1}$, which means that $y_0$ enjoys an \emph{a priori} bound in some sense.
In general, $\beta$ depends on $t_0$, $M$, $\alpha_1$, and $\alpha_2$.
Hence, the stability result inferred from $\beta$ depends on the choice of the initial data, leading to its designation as ``conditional stability.''
\end{remark}

\begin{remark}
The requirement that $\lim\limits_{\eta\to 0}\beta(\eta)=0$ is used to ensure the continuous dependence of $y(t_0)$ with respect to $y(T)$. Once $\beta$ exists, it is not unique.
To illustrate this, consider $\tilde \beta(x)=\beta(x)+x$, which is an alternative function that satisfies the desired properties.
\end{remark}

\subsection{Inverse state problem with a boundary measurement}

In this subsection, we consider the inverse state problem with the boundary measurement.

Let $ \Gamma_{0} $ be a given nonempty open subset of $ \partial G $.
Consider the following stochastic parabolic equation:
\begin{align} \label{eqInversPParabolicPartialBoudaryeq}
\left\{\begin{aligned}
&d y -\sum_{j,k=1}^n (b^{j k}  y_{x_j})_{x_k} dt  =\left(a_1 \cdot \nabla y+a_2 y  \right) d t+ a_{3} y  d W(t) & & \text { in } Q, \\
& y = g_{1} & & \hspace{-1.6cm} \text { on } (0,T) \times \Gamma_{0} , \\
& \frac{\partial y}{ \partial \nu} = g_{2} & & \hspace{-1.6cm} \text { on } (0,T) \times \Gamma_{0} ,
\end{aligned}\right.
\end{align}
where
\begin{align}
\label{eqAcases}
\begin{cases}
    \begin{aligned}
        & a_1 \in L^{\infty}_{{\mathbb{F}}}(0,T;W^{1,\infty}(G;{\mathbb{R}}^n)),
        \quad  a_2\in L^{\infty}_{{\mathbb{F}}}(0,T;L^{\infty}(G)),
        \\
        & a_3 \in L^{\infty}_{{\mathbb{F}}}(0,T;W^{1,\infty}(G))
        .
    \end{aligned}
\end{cases}
\end{align}

Let $ \mathbf{G}_{\Gamma_{0}} \deq \{ G' \subset G  \mid \partial G' \cap \partial G \subset \Gamma_{0} \} $ and
\begin{align*}
 H_{\Gamma_{0}}^2(G) \deq   \bigg\{\eta \in H_{l o c}^2(G) & \mid\left.\eta\right|_{G^{\prime}} \in H^2\left(G^{\prime}\right), ~ \forall G^{\prime} \in \mathbf{G}_{\Gamma_{0}},
     \eta |_{\Gamma_{0}}=g_{1},
~   \frac{\partial\eta }{\partial \nu} \bigg|_{\Gamma_{0}}=g_2 \bigg \}.
\end{align*}

We consider the following problem
\begin{problem}\label{probP9}
Can we determine $y$ in $Q$ uniquely by the boundary measurements $y \big|_{(0,T)\times \Gamma_0}$ and $ \dfrac{\partial y}{\partial \nu}  \bigg|_{(0,T)\times \Gamma_0}$?
\end{problem}

The answer to this problem is provided in \Cref{secIllPosedCauchy}.

\begin{remark}
In \cref{probP9}, we only do the measurement on $(0,T) \times \Gamma_{0}$. Hence, similar to \cref{probP8}, we  can determine $ y(\cdot)$  in $(0,T)\times G$  uniquely but we  will lose the continuous dependence of $y(t)$ ($t\in (0,T]$) with respect to $y \big|_{(0,T)\times \Gamma_0}$ and $ \frac{\partial y}{\partial \nu}  \big|_{(0,T)\times \Gamma_0}$. In such case,
instead of whole domain $Q$, we will determine solutions over a subset of whole domain $Q$  (see \cref{thmProb9}).
\end{remark}
\begin{remark}
If we have the full information of $y$ on $(0,T) \times \Gamma$, then we can get a better stability result for \cref{probP9}, which is similar to the one stated in \cref{prob.2.1}. The proof is also very similar. We leave the proof of that result to interested readers.
\end{remark}

\subsection{Inverse source problem with a terminal measurement}\label{ssec-final mea}

Consider the following stochastic parabolic equation:
\begin{align}
\label{eqISPTerEq}
\left\{\begin{aligned}
    &d y -\sum_{j,k=1}^n (b^{j k}  y_{x_j})_{x_k} dt  =\left(a_1 \cdot \nabla y+a_2 y  \right) d t+ (a_{3} y + g)  d W(t) & & \text { in } Q, \\
& y = 0 & &  \text { on } \Sigma , \\
& y(0) = y_{0} & &  \text { in } G,
\end{aligned}\right.
\end{align}
where  $ a_{1}, a_{2} $ and $ a_{3} $ satisfy \cref{eqAcases}.

Consider the following problem:
\begin{problem}
\label{probISPCS}
For any $ t_{0} \in (0,T) $, does $ y(T) $ uniquely determine $ (y(t_{0}), g) $?
\end{problem}
We will provide the answer to this problem in \Cref{secISPTer}.

\subsection{Inverse source problem with a boundary measurement}\label{ssec-boundary mea}


Denote $x=(x_1,x')\in {\mathbb{R}}^n$, where $x' = (x_2,\cdots,x_n)\in {\mathbb{R}}^{n-1}$. Let  $G=(0,l)\times G'$, where $l>0$ and $G'\subset {\mathbb{R}}^{n-1}$ be a bounded domain with a $C^2$ boundary. Consider the following stochastic parabolic equation:
\begin{align}\label{system2bu}
\left\{
\begin{aligned}
& dy -\sum_{j,k=1}^n (b^{j k}  y_{x_j})_{x_k}d t  \!=  \!\big( a_1\cdot \nabla y  \!+ \! \!a_2 y \! + h(t,x')R(t,x)\big) dt
+ g(t, x)dW(t) \!\!\!\!\! &&  \text{ in } Q,\\
& y = 0 && \text{ on } \Sigma,\\
\end{aligned}
\right.
\end{align}
where
$ a_1 \in L^{\infty}_{{\mathbb{F}}}(0,T;W^{1,\infty}(G;{\mathbb{R}}^n))$,
$ a_2\in L^{\infty}_{{\mathbb{F}}}(0,T;W^{1,\infty}(G))$,
$h R \in L^2_{{\mathbb{F}}}(0,T;$  $L^{2}(G))$. Note that $R$ is a deterministic function and $h$ does not depend on $x_1$.

Let us first assume that $ g = a_{3} y $, where $ a_{3} \in L^{\infty}_{\mathbb{F}}(0,T; W^{2, \infty}(G)) $, and $ y(0) = y_{0} $ in $ G $.
Consider the following problem, whose answer will be provided in \Cref{secInverseSourcePre}:
\begin{problem}\label{prob.p7pre}
For a given $R$, determine the
source function $h(\cdot)$  by means of the measurement $\dfrac{\partial y}{\partial \nu}\Big|_{\Sigma}$.
\end{problem}

\subsection{Inverse source problem with a boundary measurement and a terminal measurement}\label{ssec-final-boundary mea}

In  \Cref{ssec-final mea,ssec-boundary mea}, we consider inverse source problems with the final time measurement and the boundary measurement, respectively. In this subsection, we consider the case that we can measure both the final time state and boundary value.

Assume $g \in L^2_{\mathbb{F}}(0,T; H_0^1(G))$ and the value of $y(0)$ is initially unknown.

Consider the following problem:
\begin{problem}\label{prob.p7}
For a given $R$, determine the
source function $h(\cdot)$ and $ g(\cdot, \cdot) $ by means of the measurement $\dfrac{\partial y}{\partial \nu}\Big|_{\Sigma}$ and $ y(T) $.
\end{problem}
We will provide the answer to this problem in \Cref{secInverseSource}.

\subsection{Reconstruct the state with a terminal measurement}
\label{secReconstructTerFormu}

Refer back to \cref{prob.2.2}. If the answer is positive, i.e.,  the state $y$ can be uniquely determined by a terminal measurement $y(T)$, how to reconstruct the state $ y $?

Consider the following stochastic parabolic equation:
\begin{align}\label{eqDou1}
    \left\{
    \begin{aligned}
        & dy - \sum_{j,k=1}^n(b^{jk}y_{x_j})_{x_k}dt=
        (b_{1} \cdot \nabla y + b_{2} y + f) dt
        + (b_{3}y + g) d W(t)
        &\mbox{ in }Q,\\
        & y=0&\mbox{ on }\Sigma,\\
        & y(0)=y_0 &\mbox{ in }G,
    \end{aligned}
    \right.
\end{align}
where
\begin{align*}
    \left\{
    \begin{aligned}
        & b_{1} \in L_{\mathbb{F}}^{\infty}\left(0, T ; W^{1, \infty}\left(G ; \mathbb{R}^n\right)\right), \quad
        b_{2} \in L_{\mathbb{F}}^{\infty}\left(0, T ; L^{\infty}\left(G\right)\right), \quad
        \\
        &
        b_{3} \in L_{\mathbb{F}}^{\infty}\left(0, T ; W^{1, \infty}\left(G\right)\right), \quad
        \\
        &
        f \in L^{2}_{\mathbb{F}}(0, T; L^{2}(G)), \quad
        g \in L^{2}_{\mathbb{F}}(0, T; H_{0}^{1}(G))
        .
    \end{aligned}
    \right.
\end{align*}

\begin{problem}
    \label{probRI}
    Find a function $ y \in L^{2}_{\mathbb{F}}(\Omega; C([0,T]; L^{2}(G))) \cap L^{2}_{\mathbb{F}}(0,T;H^{1}_{0}(G)) $
    that satisfies \cref{eqDou1} with $ y(T) = y_{T} $ which is previously known.
\end{problem}

\subsection{Reconstruct the state with   boundary measurements}
\label{secReconstructProb}

The reconstruction problem mentioned in \Cref{secReconstructTerFormu} employs the terminal measurement to reconstruct the state.
For boundary measurements, i.e., \cref{probP9}, we can similarly consider the reconstruction problem.

This lead to the following problem:
\begin{problem}\label{probICPDou}
Find the unknown solution
\begin{align*}
y \in L^{2}_{\mathbb{F}}(\Omega; C([0,T]; L^{2}_{loc}(G))) \cap L^{2}_{\mathbb{F}}(0,T;H^{2}_{\Gamma_{0}}(G))
\end{align*}
to the equation \cref{eqInversPParabolicPartialBoudaryeq} via $y \big|_{(0,T)\times \Gamma_0}$ and $ \dfrac{\partial y}{\partial \nu}  \bigg|_{(0,T)\times \Gamma_0}$.
\end{problem}

The rest of this chapter is organized as follows. In \Cref{Sp-identity}, we present  a fundamental pointwise weighted identity for second order stochastic para\-bolic-like operators, which is used to derive Carleman estimates for stochastic parabolic equations.  Sections \ref{secDeterminationProblemdistributing1}--\ref{secReconstruct} are devoted to studying Problems \ref{prob.2.0}--\ref{probICPDou}, respectively.
At last, in Section \ref{secFurPro}, we give some comments for the content in this chapter and some related open problems.

\section{A fundamental weighted identity}\label{Sp-identity}

This section presents a fundamental pointwise weighted identity for second order stochastic parabolic-like operators, which is used to derive Carleman estimates for stochastic parabolic equations.

\begin{theorem}
\label{thm.fIForParablic}
Let
\begin{align} \label{eqfIPara1}
    b^{j k}=b^{k j} \in L_{\mathbb{F}}^2\left(\Omega ; C^1\left([0, T] ; W^{2, \infty}(G)\right)\right), \quad j, k=1,2, \cdots, n,
\end{align}
$\ell \in C^{1,3}((0,T)\times G)$ and $\Psi \in C^{1,2}((0,T)\times G)$.
Let $ u $ be an $H^2(G)$-valued continuous Itô process.
Set  $\theta=e^{\ell}$ and $w=\theta u$. Then, for any $t \in[0, T]$ and a.e. $(x, \omega) \in G \times \Omega$,
\begin{align}\label{eq.fIForParablic}\notag
    & 2 \theta \Big [\!-\!\sum_{j, k=1}^n \big(b^{j k} w_{x_j}\big)_{x_k}+\mathcal{A} w \Big ]\Big[d u\!-\!\sum_{j, k=1}^n\big(b^{j k} u_{x_j}\big)_{x_k} d t\Big]
    \!+\! 2 \sum_{j, k=1}^n\big(b^{j k} w_{x_j} d w\big)_{x_k}
    \\\notag
    &
     \quad+2 \sum_{j, k=1}^n\Big[\sum_{j^{\prime}, k^{\prime}=1}^n\big(2 b^{j k} b^{j^{\prime} k^{\prime}} \ell_{x_{j^{\prime}}} w_{x_j} w_{x_{k^{\prime}}}-b^{j k} b^{j^{\prime} k^{\prime}} \ell_{x_j} w_{x_{j^{\prime}}} w_{x_{k^{\prime}}}\big)
    \\
    &
     \quad \quad \quad \quad \quad  ~ +\Psi b^{j k} w_{x_j} w-b^{j k}\Big(\mathcal{A} \ell_{x_j}+\frac{\Psi_{x_j}}{2}\Big) w^2\Big]_{x_k} d t
    \\\notag
    &
     =2 \sum_{j, k=1}^n c^{j k} w_{x_j} w_{x_k} d t+\mathcal{B} w^2 d t+d\Big(\sum_{j, k=1}^n b^{j k} w_{x_j} w_{x_k}+\mathcal{A} w^2\Big)
     -\theta^2 \mathcal{A}(d u)^2,
    \\ \notag
    &
     +2\Big[-\sum_{j, k=1}^n\big(b^{j k} w_{x_j}\big)_{x_k}+\mathcal{A} w\Big]^2 d t
      -\theta^2 \sum_{j, k=1}^n b^{j k}\big(d u_{x_j}+\ell_{x_j} d u\big)\big(d u_{x_k}+\ell_{x_k} d u\big)
\end{align}
where
\begin{align}
    \label{eq.identyDetail}
    \left\{\begin{aligned}
        & \mathcal{A} \triangleq-\sum_{j, k=1}^n\big(b^{j k} \ell_{x_j} \ell_{x_k}-b_{x_k}^{j k} \ell_{x_j}-b^{j k} \ell_{x_j x_k}\big)-\Psi-\ell_t, \\
        & \mathcal{B} \triangleq 2\Big[\mathcal{A} \Psi-\sum_{j, k=1}^n\big(\mathcal{A} b^{j k} \ell_{x_j}\big)_{x_k}\Big]-\mathcal{A}_t-\sum_{j, k=1}^n\big(b^{j k} \Psi_{x_k}\big)_{x_j}, \\
        & c^{j k} \triangleq \sum_{j^{\prime}, k^{\prime}=1}^n\Big[2 b^{j k^{\prime}}\big(b^{j^{\prime} k} \ell_{x_{j^{\prime}}}\big)_{x_{k^{\prime}}}-\big(b^{j k} b^{j^{\prime} k^{\prime}} \ell_{x_{j^{\prime}}}\big)_{x_{k^{\prime}}}\Big]-\frac{b_t^{j k}}{2}+\Psi b^{j k}.
        \end{aligned}\right.
\end{align}
\end{theorem}

\medskip

\begin{proof}

The proof is divided into four steps.

\emph{Step 1.}  Since $ \theta = e^{\ell} $ and $  w = \theta u $, we have that
\begin{align*}
\theta d u =   d w - \ell_{t} w d t,
\end{align*}
and that
\begin{align} \label{eqIDParabolic1}
& \theta \sum_{j, k =1}^{n} (b^{j k} u_{x_{j}})_{x_{k}}
\\ \notag
& =
\sum_{j, k =1}^{n} (b^{j k} w_{x_{j}})_{x_{k}}
\!-\! 2 \sum_{j, k =1}^{n} b^{j k} \ell_{x_{k}} w_{x_{j}}
\!+\! \sum_{j, k =1}^{n} (
b^{j k} \ell_{x_{j}} \ell_{x_{k}}
- b^{j k}_{x_{k}} \ell_{x_{j}}
- b^{j k} \ell_{x_{j}x_{k}}
) w.
\end{align}
Put
\begin{align}\label{eqIDParabolic2}
    I   = - \sum_{j, k =1}^{n} (b^{j k} w_{x_{j}})_{x_{k}}  + \mathcal{A} w,
 \quad 
I_{2}   = d w + 2 \sum_{j, k =1}^{n} b^{j k} \ell_{x_{k}} w_{x_{j}} d t,
 \quad 
I_{3}   = \Psi w d t.
\end{align}
Combining \cref{eqIDParabolic1,eqIDParabolic2}, we obtain
\begin{align}\label{eqIDParabolic3}
2 \theta  I  \Big(d u - \sum_{j, k =1}^{n} (b^{j k} u_{x_{j}})_{x_{k}} d t \Big) = 2 I^{2} d t + 2 I I_{2} + 2 I I_{3}
.
\end{align}

\emph{Step 2.} In this step, we handle the term $ I I_{2} $. Using It\^o's formula, we have
\begin{align}\label{eqIDParabolic4}\notag
& 2 \Big[  - \sum_{j, k =1}^{n} (b^{j k} w_{x_{j}})_{x_{k}} + \mathcal{A} w \Big] d w
\\\notag
& =
- 2 \sum_{j, k =1}^{n} (b^{j k} w_{x_{j}})_{x_{k}}
+ d \Big( \sum_{j, k =1}^{n} b^{j k} w_{x_{j}} w_{x_{k}} + \mathcal{A} w^{2} \Big)
- \sum_{j, k =1}^{n} b^{j k}_{t} w_{x_{j}} w_{x_{k}} d t
\\
& \quad
- \mathcal{A}_{t} w^{2} d t
- \sum_{j, k =1}^{n} d w_{x_{j}}  d w_{x_{k}}
- \mathcal{A} (d w)^{2}
.
\end{align}

Thanks to \cref{eqfIPara1}, we obtain
\begin{align*}
& 4 \sum_{j, k, j', k'=1}^{n} b^{j k} b^{j' k'} \ell_{x_{k'}} w_{x_{j}} w_{x_{k} x_{j'}}
\\
& = 2 \sum_{j, k, j', k'=1}^{n} \big[
( b^{j k} b^{j' k'} \ell_{x_{k}} w_{x_{j'}} w_{x_{k'}})_{x_{j}}
- ( b^{j k} b^{j' k'} \ell_{x_{k'}} )_{x_{j'}} w_{x_{j}} w_{x_{k}}
\big]
.
\end{align*}
Hence, it holds that
\begin{align} \label{eqIDParabolic5}\notag
& 4 \sum_{j', k'=1}^{n} \Big[
- \sum_{j, k =1}^{n} (b^{j k} w_{x_{j}})_{x_{k}}
+ \mathcal{A} w
\Big]
b^{j' k'} \ell_{x_{k'}} w_{x_{j'}}
\\\notag
& =
\!- 2 \sum_{j, k=1}^{n} \!\Big[
\! \sum_{j', k'=1}^{n} \!(
    2 b^{j k} b^{j' k'} \ell_{x_{j'}} w_{x_{j}} w_{x_{k'}}
    \!-\! b^{j k} b^{j' k'} \ell_{x_{j}} w_{x_{j'}} w_{x_{k'}}
)
\!- \! \mathcal{A} b^{j k} \ell_{x_{j}} w^{2}
\Big]_{x_{k}} \!
\\ 
& \quad
+ 2 \sum_{j, k, j', k'=1}^{n} \big[
2 b^{j k'} (b^{j' k} \ell_{x_{j'}} )_{x_{k'}}
- (b^{j k}  b^{j' k'} \ell_{x_{k'}})_{x_{j}}
\big] w_{x_{j}} w_{x_{k}}
\\\notag
& \quad
- 2 \sum_{j, k=1}^{n} (\mathcal{A} b^{j k} \ell_{x_{j}} )_{x_{k}} w^{2}
.
\end{align}

\emph{Step 3.} In this step, we deal with $ I I_{3} $.
Thanks to \cref{eqIDParabolic2}, we obtain
\begin{align}\label{eqIDParabolic6}\notag
  2 I I_{3}
& =
2 \Big[
- \sum_{j, k=1}^{n} (b^{j k} w_{x_{j}})_{x_{k}}
+ \mathcal{A} w
\Big] \Psi w d t
\\\notag
& =
- 2 \sum_{j, k=1}^{n} \big(
 b^{j k} \Psi w_{x_{j}} w
-  b^{j k} \Psi_{x_{j}} w^{2}
\big)_{x_{k}}  d t
+ 2 \sum_{j, k=1}^{n} \Psi b^{j k} w_{x_{j}} w_{x_{k}} d t
\\
& \quad  \,
+ \Big(
\mathcal{A} \Psi
- \sum_{j, k=1}^{n} (b^{j k} \Psi_{x_{j}})_{x_{k}}
\Big) w^{2} d t
.
\end{align}

\emph{Step 4.}  Combining \cref{eqIDParabolic3,eqIDParabolic4,eqIDParabolic5,eqIDParabolic6},  and noting that
\begin{align*}
& \sum_{j, k =1}^{n} b^{j k} d w_{x_{j}} d w_{x_{k}} + \mathcal{A} (d w)^{2}
\\
& =
\sum_{j, k =1}^{n} \theta^{2} b^{j k} ( d u_{x_{j}}  + \ell_{x_{j}} d u )( d u_{x_{k}}
+ \ell_{x_{k}} d u )
+ \theta^{2} \mathcal{A} (d u)^{2}
,
\end{align*}
we immediately have the desired equality \cref{eq.fIForParablic}.
\end{proof}

\section{Solution to inverse state problem with an internal measurement I}
\label{secDeterminationProblemdistributing1}

In this section, we solve \cref{prob.2.0} and \cref{prob.2.1}.
The following theorem, which is a kind of observability estimate, provides a straightforward extension of \cite[Theorem 2.4]{Tang2009}, presenting a positive answer to \cref{prob.2.0} and \cref{prob.2.1}.

\begin{theorem}
\label{thm.2.1}
For each $t \in (0,T]$, there exists a constant ${\cal C}(t) >0$ such that the inequality
\cref{sp-eq2} holds for any
$y_0, \hat y_0\in L^2_{\mathcal{F}_0}(\Omega;L^2(G))$.
\end{theorem}

\begin{remark}
\Cref{thm.2.1} asserts  that the state $y(t)$ of  system \cref{sp-eq1} (at $t\in (0,T]$)
is   uniquely determined from the observation  $ y|_{(0,T)\times     G_0}$, ${\mathbb{P}}$-a.s., and continuously depends on it.
As a result of the arbitrariness of $t\in (0,T]$ in \cref{sp-eq2}, we deduce that the answer to \cref{prob.2.0} is certainly positive.
\end{remark}

In order to prove \cref{thm.2.1}, we present a global Carleman estimate for the following stochastic parabolic equation:
\begin{equation}\label{h9}
\left\{ \begin{aligned}
& dz-\sum_{j,k=1}^n(b^{jk}z_{x_j})_{x_k}dt=fdt+gdW(t) \quad &\hbox{in }Q,\\
& z=0 &\hbox{on }\Sigma,
\\
& z(0)=z_0 &\hbox{in }G,
\end{aligned} \right.
\end{equation}
where $z_0\in L^2_{\mathcal{F}_0}(\Omega; L^2(G))$, $f\in
L_{{\mathbb{F}}}^2(0,T;L^2(G))$ and $g\in L_{{\mathbb{F}}}^2(0,T; H^1(G))$.

Let us first recall the following well-known result.
\begin{lemma}\label{hl1}\cite[p. 4, Lemma 1.1]{Fursikov1996}
For any nonempty open subset $G_1$ of $G$, there is a  $\psi\in C^4({\overline{G}})$ such that $\psi>0$ in $G$, $\psi=0$ on
$\Gamma$, and $| \nabla\psi(x)|>0$ for all $x\in \overline{G \setminus G_1}$.
\end{lemma}

With the function $\psi$ given in  \Cref{hl1},
we choose the weighted functions in \cref{thm.fIForParablic} as follows:
\begin{equation}\label{alphad}
\theta=e^{\ell },
\quad\ell=\lambda\alpha,\quad\alpha(t,x)=\frac{e^{\mu\psi(x)}-e^{2\mu|\psi|_{C(\overline{G
})}}}{ t(T-t)},\quad \varphi(t,x)=\frac{e^{\mu\psi(x)}}{t(T-t)},
\end{equation}
where $\lambda>1$, $\mu>1$. Further, let
\begin{equation}\label{h5}
\Psi=2\sum_{j,k=1}^nb^{jk}\ell_{x_jx_k}.
\end{equation}

For $j,k=1,2,\cdots,n$, it is straightforward  to verify  that
\begin{equation}\label{heat 1.1 h2}
\ell_t=\lambda\alpha_t,\quad \ell_{x_j}=\lambda\mu\varphi\psi_{x_j},\quad
\ell_{x_jx_k}=\lambda\mu^2\varphi\psi_{x_j}\psi_{x_k}+\lambda\mu\varphi\psi_{x_jx_k}
\end{equation}
and that
\begin{align} \label{h4}
&   \alpha_t
= \varphi^2O\big(e^{2\mu |\psi|_{C(\overline{G})}}\big),
\quad \quad 
\varphi_t
= \varphi^2O\big(e^{\mu |\psi|_{C(\overline{G})}}\big)
.
\end{align}
Here and in what follows, for a positive constant $ r $, we denote $ O(\mu^{r}) $ as a function of order $ \mu^{r} $ for large $ \mu $.
We use $ O\big(e^{\mu |\psi|_{C(\overline{G})}}\big) $ in a similar way.

We state the Carleman estimate for the equation \cref{h9} as follows:

\begin{theorem}
\label{thm.CarlemanForHeat}
There is a constant $\mu_0=\mu_0(G, G_0,(b^{jk})_{n\times n})>0$
such that for all $\mu \ge \mu_0$, one can find two constants
${\cal C}={\cal C}(\mu)>0$ and $\lambda_0=\lambda_0(\mu)>0$ such that for any
$\lambda\ge\lambda_0$,  the
solution $z$ to \cref{h9} satisfies
\begin{align}\label{h7}
& \lambda^3\mu^4\mathbb{E}\int_Q\varphi^3\theta^2z^2dxdt+
\lambda\mu^2\mathbb{E}\int_Q\varphi\theta^2|\nabla
z|^2dxdt
\notag
\\
 & \leq {\cal C}\mathbb{E}\Big[\int_Q \theta^2\big(f^2+|\nabla
g|^2+\lambda^2\mu^2\varphi^2 g^2\big)dxdt+\lambda^3\mu^4\int_{Q_0}\varphi^3
\theta^2z^2dxdt\Big].
\end{align}
\end{theorem}

\medskip

\begin{proof}
The proof is divided  into four steps.

\medskip

\emph{Step 1.} In this step, we estimate the order of the terms in the equality \eqref{eq.fIForParablic} with respect to $\lambda$ and $\mu$.

Combining \cref{h5,heat 1.1 h2,eq.identyDetail,eq.bijGeq}, we have
\begin{align*}
\ell_{x_jx_k}=\lambda\mu^2\varphi   \psi_{x_j}\psi_{x_k}+\lambda\varphi O(\mu)
\end{align*}
and
\begin{align}\label{h6-1}
\notag
& \sum_{j,k=1}^n c^{jk}w_{x_j}w_{x_k}
\\ \notag
& = 2 \lambda \mu^2 \varphi \Big (\sum_{j, k=1}^n b^{j k} \psi_{x_j} w_{x_k} \Big )^2+\lambda \mu^2 \varphi \Big (\sum_{j, k=1}^n b^{j k} \psi_{x_j} \psi_{x_k} \Big )  \Big (\sum_{j, k=1}^n b^{j k} w_{x_j} w_{x_k} \Big )
\\ \notag
& \quad
+ \lambda \varphi O(\mu)  |\nabla w|^{2}
\\
&\ge \big(s_0^2\lambda\mu^2\varphi
|\nabla\psi|^2+\lambda\varphi O(\mu)\big)|\nabla w|^2.
\end{align}

Thanks to \cref{eq.identyDetail,h5,heat 1.1 h2,h4}, we obtain that
\begin{align} \label{h6-2}
\notag
\mathcal{A}
& = - \sum_{j, k=1}^n\big(b^{j k} \ell_{x_j} \ell_{x_k}-b_{x_k}^{j k} \ell_{x_j}+ b^{j k} \ell_{x_j x_k}\big) -\ell_t
\\ \quad
& =-\lambda^2\mu^2\varphi^2\sum_{j,k=1}^nb^{jk}
\psi_{x_j}\psi_{x_k}+\lambda\varphi^2O\big(e^{2\mu |\psi|_{C(\overline{G})}}\big).
\end{align}

Next, we estimate $\mathcal{B}$ term by term (referring to the definition of $\mathcal{B}$ in \cref{eq.identyDetail}).

Thanks to \cref{h5,heat 1.1 h2}, we see that
\begin{align}
\label{eq.Thm2-4-1S1T-1}
\Psi = 2 \lambda \mu^{2} \varphi \sum_{j,k=1}^nb^{jk} \psi_{x_j}\psi_{x_k} + \lambda \varphi O(\mu).
\end{align}
This, together  with \cref{h6-2}, implies  that
\begin{align}
\label{eq.Thm2-4-1S1T0}
\mathcal{A} \Psi =
- 2 \lambda^{3} \mu^{4} \varphi^{3} \Bigl(\sum_{j,k=1}^nb^{jk} \psi_{x_j}\psi_{x_k}\Bigr)^{2}
+ \lambda^{3} \varphi^{3} O(\mu^{3})
+ \lambda^{2}\varphi^3 O\big(\mu^{2} e^{2\mu |\psi|_{C(\overline{G})}}\big).
\end{align}

From \cref{h5}, we obtain that
\begin{align}
\label{eq.Thm2-4-1S1T1}
\sum_{j, k=1}^n\big(\mathcal{A} b^{j k} \ell_{x_j}\big)_{x_{k}}
= \sum_{j, k=1}^n b^{j k} \mathcal{A}_{x_{k}} \ell_{x_{j}}
+ \sum_{j, k=1}^n \mathcal{A}  b^{j k}_{x_{k}}  \ell_{x_{j}}
+ \frac{1}{2} \mathcal{A} \Psi
.
\end{align}
Combining \cref{eq.identyDetail,alphad,h6-2}, we get that
\begin{align*}
\mathcal{A}_{x_{k}} =
-2 \lambda^2\mu^3\varphi^2 \sum_{j',k'=1}^n b^{j'k'}\psi_{x_j'}\psi_{x_k'} \psi_{x_{k}}
+ \lambda^{2} \varphi^{2} O(\mu^{2})
+\lambda \varphi^2 O\big( \mu e^{2\mu |\psi|_{C(\overline{G})}}\big).
\end{align*}
This, together with \cref{heat 1.1 h2}, yields that
\begin{align}
\label{eq.Thm2-4-1S1T2}
\notag
\sum_{j, k=1}^n b^{j k} \mathcal{A}_{x_{k}} \ell_{x_{j}}
& =
-2 \lambda^3\mu^4\varphi^3 \Bigl( \sum_{j,k=1}^n b^{jk}\psi_{x_j}\psi_{x_k} \Bigr)^{2}
+ \lambda^{3} \varphi^{3} O(\mu^{3})
\\
& \quad
+\lambda^{2} \varphi^3 O\big( \mu^{2} e^{2\mu |\psi|_{C(\overline{G})}}\big).
\end{align}
From \cref{h6-2,heat 1.1 h2}, we get that
\begin{align}
\label{eq.Thm2-4-1S1T3}
\sum_{j, k=1}^n \mathcal{A}  b^{j k}_{x_{k}}  \ell_{x_{j}}
=
\lambda^{3} \varphi^{3} O(\mu^{3})
+\lambda^{2} \varphi^3 O\big( \mu e^{2\mu |\psi|_{C(\overline{G})}}\big).
\end{align}

By \cref{h4,h6-2}, we find that
\begin{align}
\label{eq.Thm2-4-1S1T4}
\mathcal{A}_{t} =
\lambda^{2} \varphi^{3} O\big( \mu^{2} e^{2\mu |\psi|_{C(\overline{G})}}\big)
+ \lambda \varphi^{3} O\big( e^{2\mu |\psi|_{C(\overline{G})}}\big).
\end{align}
From \cref{eq.Thm2-4-1S1T-1,alphad}, we obtain
\begin{align}
\label{eq.Thm2-4-1S1T5}
\sum_{j, k=1}^n\big(b^{j k} \Psi_{x_k}\big)_{x_j}
& = 2 \lambda \mu^{4} \varphi \Bigl( \sum_{j,k=1}^n b^{jk}\psi_{x_j}\psi_{x_k} \Bigr)^{2} + \lambda \varphi O (\mu^{3}).
\end{align}

Combining \cref{eq.Thm2-4-1S1T0,eq.Thm2-4-1S1T1,eq.Thm2-4-1S1T2,eq.Thm2-4-1S1T3,eq.Thm2-4-1S1T4,eq.Thm2-4-1S1T5}, we have that
\begin{align}\label{h6-3}
\mathcal{B}
& \ge 2s_0^2\lambda^3\mu^4\varphi ^3|\nabla \psi|^4
+\lambda^3\varphi^3O(\mu^3)+\lambda^2\varphi^3 O\big(\mu^2e^{2\mu
|\psi|_{C(\overline{G})}}\big)
+\lambda\varphi^3O\big(e^{2\mu |\psi|_{C(\overline{G})}}\big).
\end{align}

\emph{Step 2.}   In this step, we apply \cref{thm.fIForParablic} to the equation \cref{h9} and handle the boundary terms.

Let us choose  $u=z$ in \cref{thm.fIForParablic}. Integrating the equality \cref{eq.fIForParablic} over $Q$, taking mathematical expectation in both sides, and noting \cref{h6-1,h6-2,h6-3}, we conclude that there is a constant $c_0>0$ such that
\begin{align}\label{h10}\notag
& 2\mathbb{E}\int_Q\theta\Big[-\sum_{j,k=1}^n (b^{jk}w_{x_j})_{x_k}+ {\cal A}
w\Big]\Big[dz-\sum_{j,k=1}^n (b^{jk}z_{x_j})_{x_k}dt
\Big]dx
\\\notag
&+ 2\mathbb{E}\int_Q\sum_{j,k=1}^n\Big[\sum_{j',k'=1}^n\Big (2b^{jk}
b^{j'k'}\ell_{x_{j'}}w_{x_j}w_{x_{k'}}
-b^{jk}b^{j'k'}\ell_{x_j}w_{x_{j'}}w_{x_{k'}}\Big )
+\Psi b^{jk}w_{x_j}w
\\\notag
&\qquad\qquad\qquad ~  
- b^{jk}\Big ({\cal A}\ell_{x_j}+\frac{\Psi_{x_j}}{2}\Big )w^2\Big]_{x_k}dxdt
+2\mathbb{E}\int_Q \sum_{j,k=1}^n (b^{jk}w_{x_j}dw)_{x_k}dx
\\\notag
& \geq 2c_0\mathbb{E}\int_Q\Big[\varphi \big(\lambda\mu^2|\nabla\psi|^2+\lambda
O(\mu)\big)|\nabla w|^2+\varphi^3\Big(\lambda^3\mu^4|\nabla
\psi|^4+\lambda^3O(\mu^3)\\\notag
&\qquad\qquad\quad +\lambda^2O\big(\mu^2e^{2\mu |\psi|_{C(\overline{G})}}\big)+\lambda
O\big(e^{2\mu
|\psi|_{C(\overline{G})}}\big)\Big)w^2\Big] dxdt\\\notag
&\quad +2\mathbb{E} \int_Q \Big|- \sum_{j,k=1}^n
\big(b^{jk}w_{x_j}\big)_{x_k}+ {\cal A}
w\Big|^2dxdt\\
&\quad - \mathbb{E}\int_Q\theta^2 \sum_{j,k=1}^n b^{jk}
\big(dz_{x_j}+\ell_{x_j}dz\big)\big(dz_{x_k}+\ell_{x_k}dz\big)dx-
\mathbb{E} \int_Q \theta^2{\cal A}(dz)^2dx,
\end{align}
where we use the fact that, for all $ x \in G $, $ \lim\limits_{t \rightarrow 0+} \alpha = \lim\limits_{t \rightarrow T-} \alpha = - \infty $.

By \cref{h9}, we obtain
\begin{align}\label{h11} \notag
& 2\mathbb{E}\int_Q\theta\Big[-\sum_{j,k=1}^n
\big(b^{jk}w_{x_j}\big)_{x_k}+{\cal A} w\Big]
\Big[dz-\sum_{j,k=1}^n \big(b^{jk}z_{x_j}\big)_{x_k}dt\Big]dx
\\
& \leq \mathbb{E}\int_Q\Big|-\sum_{j,k=1}^n
\big(b^{jk}w_{x_j}\big)_{x_k}+{\cal A} w\Big|^2 d t d x+\mathbb{E}\int_Q\theta^2
f^2 d t d x.
\end{align}

Next, we deal with  divergence terms.
Since $z=0$ on $ \Sigma $, we get
\begin{align*}
w|_{\Sigma}=0  \text{ and } \;w_{x_j}|_{\Sigma}=\theta z_{x_j}|_{\Sigma}=\theta
\frac{\partial z}{\partial\nu}\nu_{j}\Big|_{\Sigma}, \quad \text{for} \; j=1,\cdots,n.
\end{align*}
%
Similarly, by \cref{hl1}, we get
\begin{align*}
\ell_{x_j}=\lambda\mu\varphi\psi_{x_j} =\lambda\mu\varphi\frac{\partial\psi}{\partial\nu}\nu_{j}\quad
\mbox{for }j=1,\cdots,n, \hbox{ and }\quad \frac{\partial\psi}{\partial\nu}< 0\quad
\hbox{on }\Sigma.
\end{align*}
Therefore, via integration by parts and noting $w=0$ on $ \Sigma $,
we obtain that
\begin{equation}\label{h12}
2\mathbb{E}\int_Q \sum_{j,k=1}^n
(b^{jk}w_{x_j}dw)_{x_k}dx=2\mathbb{E}\int_{\Sigma}\sum_{j,k=1}^n
b^{jk}w_{x_j}\nu_{k} d w d x=0,
\end{equation}
and that
\begin{align}\label{h13}
& 2\mathbb{E}\int_Q\sum_{j,k=1}^n\Big[\sum_{j',k'=1}^n\Big (2b^{jk}
b^{j'k'}\ell_{x_{j'}}w_{x_j}w_{x_{k'}}
-b^{jk}b^{j'k'}\ell_{x_j}w_{x_{j'}}w_{x_{k'}}\Big )\nonumber
\\
&\qquad\qquad\qquad+\Psi b^{jk}w_{x_j}w- b^{jk}\Big ({\cal A}\ell_{x_j}
+\frac{\Psi_{x_j}}{2}\Big )w^2\Big]_{x_k}d t d x
\\
&=2\lambda\mu\mathbb{E}\int_{\Sigma}\varphi\frac{\partial\psi}{\partial\nu}\Big (\frac{\partial
z}{\partial\nu}\Big )^2 \Big (\sum_{j,k=1}^nb^{jk}\nu_{j}\nu_{k}\Big )^2d\Gamma
dt\le0.\nonumber
\end{align}

By \cref{h9}, we have
\begin{align}\label{h15}\notag
& \mathbb{E}\int_Q\theta^2 \sum_{j,k=1}^n b^{jk}
\big(dz_{x_j}+\ell_{x_j}dz\big)\big(dz_{x_k}+\ell_{x_k}dz\big)dx+
\mathbb{E} \int_Q
\theta^2{\cal A}(dz)^2dx\\
& = \mathbb{E}\int_Q \theta^2\sum_{j,k=1}^n
b^{jk}\big(g_{x_j}+\ell_{x_j}g\big) \big(g_{x_k}+\ell_{x_k}g\big)
dxdt+\mathbb{E}\int_Q \theta^2{\cal A} g^2dxdt.
\end{align}

Combining \cref{h10,h11,h12,h13,h15}, we arrive at
\begin{align}\label{h16}
& 2c_0\mathbb{E}\int_Q\Big[\varphi \big(\lambda\mu^2|\nabla\psi|^2+\lambda
O(\mu)\big)|\nabla w|^2+\varphi^3\Big(\lambda^3\mu^4|\nabla
\psi|^4+\lambda^3O(\mu^3)\nonumber\\
&\qquad\qquad +\lambda^2O\big(\mu^2e^{2\mu |\psi|_{C(\overline{G})}}\big)+\lambda
O\big(e^{2\mu
|\psi|_{C(\overline{G})}}\big)\Big)w^2\Big]dxdt\\
&\leq \mathbb{E}\int_Q\theta^2\Big[f^2+\sum_{j,k=1}^n
b^{jk}\big(g_{x_j}+\ell_{x_j}g\big)
\big(g_{x_k}+\ell_{x_k}g\big)+{\cal A} g^2\Big] dxdt.\nonumber
\end{align}

\emph{Step 3.} By \cref{hl1}, we have $\min_{x\in G\setminus G_1}|\nabla\psi|>0$.
Hence, there exists a $\mu_1>0$ such that for all $\mu\ge \mu_1$, one can find a constant $\lambda_1=\lambda_1(\mu)$ such that for any $\lambda\ge \lambda_1$, it holds that
\begin{align}\label{h18}
& 2 \mathbb{E}\int_Q\Big[\varphi \big(\lambda\mu^2|\nabla\psi|^2+\lambda
O(\mu)\big)|\nabla w|^2+\varphi^3\Big(\lambda^3\mu^4|\nabla
\psi|^4+\lambda^3O(\mu^3)\nonumber\\ \notag
&\qquad\quad +\lambda^2O\big(\mu^2e^{2\mu |\psi|_{C(\overline{G})}}\big)+\lambda
O\big(e^{2\mu
|\psi|_{C(\overline{G})}}\big)\Big)w^2\Big]dxdt
\\
&\geq c_1\lambda\mu^2 \mathbb{E}\int_0^T\int_{G\setminus G_1}\varphi\Big(
|\nabla w|^2+\lambda^2\mu^2\varphi^2w^2\Big)dxdt,
\end{align}
where $ c_1 \deq \min\big(\min\limits_{x\in G\setminus G_1}|\nabla\psi|^2,\min\limits_{x\in G\setminus G_1}|\nabla\psi|^4\big) $.

Note that
$$
\frac{1}{{\cal C}}\theta^2\big(|\nabla z|^2+\lambda^2\mu^2\varphi^2z^2\big)\leq|\nabla
w|^2+\lambda^2\mu^2\varphi^2w^2\leq{\cal C}\theta^2\big(|\nabla
z|^2+\lambda^2\mu^2\varphi^2z^2\big).
$$
Combining this with \cref{h18,h16}, implies that
\begin{align}\label{h20}
& \lambda\mu^2\mathbb{E}\int_Q\varphi\theta^2\big(|\nabla
z|^2+\lambda^2\mu^2\varphi^2z^2\big)dxdt\nonumber
\\ \notag
& 
\leq {\cal C}\Big[\mathbb{E}\int_Q\theta^2\big(f^2+|\nabla
g|^2+\lambda^2\mu^2\varphi^2 g^2\big) dxdt
\\  
& \quad \quad ~
+\lambda\mu^2\mathbb{E}\int_0^T\int_{G_1} \varphi\theta^2\big(|\nabla z|^2 +
\lambda^2\mu^2\varphi^2 z^2\big) dxdt
\Big].
\end{align}

\emph{Step 4.} In this step, we get rid of the  term containing $\nabla z$ in the right hand side of \cref{h20} and complete the proof of \cref{h7}.

Choosing a cut-off function $ \zeta \in C_{0}^{\infty}(G_{0}) $ satisfied $ 0 \leq \zeta \leq 1 $ and $ \zeta =1 $ in $ G_{1} $. Noting that
\begin{align*}
d (\theta^{2} \varphi z^{2} ) = z^2\left(\theta^2 \varphi\right)_t d t+2 \theta^2 \varphi z d z+\theta^2 \varphi(d z)^2,
\end{align*}
and $\lim\limits_{t \rightarrow 0^{+}} \varphi(t, \cdot) = \lim\limits_{t \rightarrow T^{-}} \varphi(t, \cdot) = 0 $, utilizing \cref{h9} and Divergence Theorem, we obtain that
\begin{align*}
0 & =\mathbb{E} \int_{Q_0} \zeta^2\big[z^2\left(\theta^2 \varphi\right)_t d t+2 \theta^2 \varphi z d z+\theta^2 \varphi(d z)^2\big] d x 
\\
& = 
\mathbb{E} \int_{Q_0}  \theta^2\Big[\zeta^2 z^2\left(\varphi_t+2 \lambda \varphi \alpha_t\right)-2 \zeta^2 \varphi \sum_{j, k=1}^n b^{j k} z_{x_j} z_{x_k} 
\\
&\quad \quad \quad \quad \quad ~ 
-2 \mu \zeta^2 \varphi(1+2 \lambda \varphi) z \sum_{j, k=1}^n b^{j k} z_{x_j} \psi_{x_k} \\
& \quad \quad \quad \quad \quad ~ -4 \zeta \varphi z \sum_{j, k=1}^n b^{j k} z_{x_j} \zeta_{x_k}+2 \zeta^2 \varphi f z+\zeta^2 \varphi g^2\Big] d x d t .
\end{align*}
This, along with \cref{h4}, yields that there exists a $\mu_2>0$ such that for all $\mu\ge \mu_2$, one can find a constant $\lambda_2=\lambda_2(\mu)$ so that for any $\lambda\ge \lambda_2$ and  $ \varepsilon \ge 0 $, it holds
\begin{align*}
& 2 \mathbb{E} \int_{Q_0} \theta^2 \zeta^2 \varphi \sum_{j, k=1}^n b^{j k} z_{x_j} z_{x_k} d x d t 
\\
& \leq
 \varepsilon \mathbb{E} \int_{Q_0} \theta^2 \zeta^2 \varphi|\nabla z|^2 d x d t+\frac{\mathcal{C}}{\varepsilon} \mathbb{E} \int_{Q_0} \theta^2\Big(\frac{1}{\lambda^2 \mu^2} f^2+\lambda^2 \mu^2 \varphi^3 z^2\Big) d x d t 
 \\
& \quad
+\mathbb{E} \int_{Q_0} \theta^2 \varphi g^2 d x d t .
\end{align*}
By this and \cref{eq.bijGeq}, choosing $ \varepsilon $ sufficiently small, we conclude that
\begin{align}
\label{eq.h20-1}\notag
& \mathbb{E} \int_0^T \int_{G_1} \theta^2 \varphi|\nabla z|^2 d x d t 
\\
& 
\leq \mathcal{C}\Big[\mathbb{E} \int_Q \theta^2\Big(\frac{1}{\lambda^2 \mu^2} f^2+\varphi g^2\Big) d x d t+\lambda^2 \mu^2 \mathbb{E} \int_{Q_0} \theta^2 \varphi^3 z^2 d x d t\Big] .
\end{align}
Finally, combining \cref{h20,eq.h20-1}, choosing $ \mu_{0} = \max \{ \mu_{1}, \mu_{2} \} $ and $ \lambda_{0} = \max \{ \lambda_{1}, \lambda_{2} \} $, we obtain \cref{h7}.
\end{proof}

\begin{remark}
According to the findings in \cite{Liu2014}, the assumption $g\in
L_{{\mathbb{F}}}^2(0,T;$  $ H^1(G))$ in  \cref{thm.CarlemanForHeat} can be relaxed to  $g\in  L_{{\mathbb{F}}}^2(0,T; L^2(G))$.
\end{remark}


\begin{proof}[Proof of \cref{thm.2.1}]
Let $ z=y-\hat y$.
From \cref{sp-eq1}, we obtain that $z$ satisfies
\begin{align*}
\left\{
\begin{aligned}
    & dz \!- \!\sum_{j,k=1}^n(b^{jk}z_{x_j})_{x_k} dt  \!=\!
      \big[ F(\nabla y,
    y)\!-\!F(\nabla \hat y, \hat y) \big]dt +  \big[
    K(y)-K(\hat y) \big]dW(t) \!\!\!\!\!\! 
    && \text{ in } Q,\\
    &  z=0 &&\text{ on }\Sigma,\\
    & z(0)=y_0-\hat y_0 &&\text{ in } G.
    \end{aligned}
\right .
\end{align*}
From \cref{10.5-eq1}, it follows that
$ F(\nabla y, y)-F(\nabla \hat y, \hat y)\in L^2_{{\mathbb{F}}}(0,T;L^2(G)) $
and
$ K(y)-K(\hat y)\in L^2_{{\mathbb{F}}}(0,T;$ $H^1(G))$.
Hence,  $z$ fulfills the equation \cref{h9} with
\begin{align}
\label{eq.thm3.1-1}
f=F(\nabla y, y)-F(\nabla \hat y, \hat y),\quad g=K(y)-K(\hat y).
\end{align}

Thanks to \cref{thm.CarlemanForHeat}, there exist constants $ \mu_{0}, \lambda_{0} > 0 $, such that for all $ \mu \geq \mu_{0} $ and  $ \lambda \geq \lambda_{0} $, it holds that
\begin{align*}
& \lambda^3\mu^4\mathbb{E}\int_Q\varphi^3\theta^2z^2dxdt+
\lambda\mu^2\mathbb{E}\int_Q\varphi\theta^2|\nabla
z|^2dxdt
\notag
\\
& \leq {\cal C}\mathbb{E}\Big[\int_Q \theta^2\big(f^2+|\nabla
g|^2+\lambda^2\mu^2\varphi^2 g^2\big)dxdt+\lambda^3\mu^4\int_{Q_0}\varphi^3
\theta^2z^2dxdt\Big].
\end{align*}
From \cref{eq.thm3.1-1,10.5-eq1}, we have
\begin{align*}
\mathbb{E}\int_Q \theta^2 |f|^2dxdt
& \leq
L^2\mathbb{E}\int_Q \theta^2 \big(|\nabla z|^2+|z|^2\big)dxdt,
\end{align*}
and
\begin{align*}
\mathbb{E}\int_Q \theta^2 \big(|\nabla g|^2 + \lambda^2\mu^2\varphi^2
|g|^2\big)dxdt \leq L^2\mathbb{E}\int_Q \theta^2 \big(|\nabla z|^2 +
\lambda^2\mu^2\varphi^2 |z|^2\big)dxdt.
\end{align*}
Hence,
\begin{align*}
&\mathbb{E}\int_Q \theta^2\big(\lambda^3\mu^4\varphi^3 |z|^2 +
\lambda\mu^2\varphi
|\nabla z|^2\big) dxdt \\
&\leq  {\cal C}\Big\{L^2 \mathbb{E}\int_Q \theta^2 \big[|\nabla
z|^2 + \big(1 + \lambda^2\mu^2\varphi^2 \big) |z|^2 \big]dxdt + \lambda^3\mu^4\int_{Q_0}\varphi^3
\theta^2z^2dxdt \Big\}.
\end{align*}
Choosing $ \mu_{1} \geq \max\{ \mu_{0}, \mathcal{C} L^{2} \} $, for all $ \lambda \geq \lambda_{0} $ and $ \mu \geq \mu_{1} $, we obtain
\begin{align} \label{ch-6-final0}
\mathbb{E}\int_Q \theta^2\big(\lambda^3\mu^4\varphi^3 |z|^2 +
\lambda\mu\varphi |\nabla z|^2\big) dxdt \leq {\cal C} \lambda^3\mu^4\int_{Q_0}\varphi^3
\theta^2z^2dxdt.
\end{align}

From \cref{alphad}, we find that
\begin{align} \label{ch-6-final1} \notag
& \mathbb{E}\int_Q \theta^2\big(\varphi^3 |z|^2 + \varphi
|\nabla z|^2\big) dxdt\\
& \geq  \displaystyle
\min_{x\in\overline{G}}\Big(\varphi\Big(\frac{T}{2},x\Big)
\theta^2\Big(\frac{T}{4},x\Big)\Big)\mathbb{E}
\int_{\frac{T}{4}}^{\frac{3T}{4}}\int_G\big(|z|^2+|\nabla
z|^2\big)dxdt
\end{align}
and that
\begin{align}  \label{ch-6-final2}
{\mathbb{E}}\int_{Q_0}\varphi^3 \theta^2 z^2dxdt\leq \max_{(x,t)\in
\overline{Q}}\big(\varphi(t,x)\theta^2(t,x)\big) \mathbb{E}
\int_{Q_0}  z^2 dx dt.
\end{align}

Using a standard energy estimate, from \cref{h9,eq.thm3.1-1,10.5-eq1}, we can  obtain that for any $ 0 \leq s \leq t \leq T  $,
\begin{align}\label{Ch-6-Eyt2}
\mathbb{E}| z(t)|^2_{ L^2(G)} \leq e^{{\cal C} L^2}\mathbb{E} | z(s)|^2_{
L^2(G)}.
\end{align}
Combining \cref{ch-6-final0,ch-6-final1,ch-6-final2,Ch-6-Eyt2}, we obtain the desired inequality \cref{sp-eq2}.
\end{proof}

\section{Solution to inverse state problem with an internal measurement. II}
\label{secDeterminationProblemdistributing2}

For the situation when boundary conditions are unknown, i.e., Problem \ref{probP8}, we can obtain the following quantitative unique continuation property for the equation \cref{system-sp1}.

\begin{theorem}\label{main result}
For any $\widehat G\subset\!\subset G$ with $G_0\subset \subset
\widehat G$ and $\kappa \in (0, \frac 12)$ being arbitrarily fixed,
there exists a constant ${\cal C}$ and a subdomain $\widetilde G \subset
\subset G$ with $\widehat G\subset \subset \widetilde G$, such that
for any solution $\tilde y$ to the equation \eqref{system-sp1}, it
holds that
\begin{align}\label{eq1}\notag
    &   {\mathbb{E}} \int_{\frac T 2 - \kappa T}^{\frac T 2 + \kappa T}\int_{\widehat G}
    \big(|\nabla \tilde y|^2 + \tilde y^2\big)dx dt\\
    &   \leq {\cal C}\Big[ \exp\Big (\frac {\cal C}
    \varepsilon\Big ) {\mathbb{E}}\! \int_0^T\int_{G_0}\big( \tilde y^2 \!+ |\nabla \tilde y|^2\big)dx dt +
    \varepsilon{\mathbb{E}}\! \int_0^T  \int_{\widetilde G}(\tilde y^2 \!+\! |\nabla
    \tilde y|^2)dxdt\Big],
\end{align}
for all $ \varepsilon > 0 $.
\end{theorem}
\begin{remark}
From the arbitrariness of $\widehat G\subset\!\subset G$, \cref{main result} provide a positive answer to \cref{probP8}. Moreover, it provides the continuous dependence of the solution in $(\frac T 2 - \kappa T,\frac T 2 + \kappa T)\times \widehat G$ with respect to the measurement on $(0,T)\times G_0$.
\end{remark}

\begin{proof}[Proof of \cref{main result}]
We divide the proof into four steps.

\medskip

{\it Step 1}. In this step, we  introduce a cut-off function to convert the   equation \cref{system-sp1} to an equation with the  boundary condition.

Without loss of  generality, we assume that the boundary $\partial G$ of $G$ is     smooth.
Otherwise, we can choose a domain $G''$ with smooth boundary $\partial G''$ such that $\widehat G \subset\!\subset   G''\subset\!\subset G$ and study
the local Carleman estimate on $G''$.

Let $G_0\subset \subset G$ be an open subset.
Thanks to  \cref{hl1}, there exists a function $\psi\in C^4(G)$ such that
\begin{align*}
\psi > 0    \text{ in } G, \quad 
 \psi = 0   \text{ on } \partial G, \quad 
 |\nabla \psi | > 0   \text{ in } G\backslash G_0.
\end{align*}

Since $\widehat G\subset\!\subset G''\subset\!\subset G$, it is clear that we can fix $N$ large enough such that
\begin{equation}\label{G'}
\widehat G \subset \left\{ x \in G \; \Big | \; \psi(x)> \frac
4N|\psi|_{L^{\infty}(G)}\right\}.
\end{equation}
For $\kappa \in (0, \frac 12)$, take $\iota= \frac 1{\sqrt 2}(\frac 12 - \kappa)T> 0$. Then $\sqrt 2 \iota= (\frac 12 -\kappa)T$.
Choose a constant $c>0$ satisfying
\begin{equation}
\label{bound psi} c\delta^2 < |\psi|_{L^{\infty}(G)} < 2 c\delta^2.
\end{equation}
Let $t_0 \in [\sqrt 2 \delta, T -\sqrt 2\delta]$.
Set
\begin{equation*}
\phi(t,x) \deq \psi(x) - c(t - t_0)^2,\quad \quad  \alpha(t,x) \deq e^{\mu \phi} .
\end{equation*}
Put
\begin{equation*}
\beta_k \deq  \exp\Big (\mu \Big (\frac k N |\psi|_{L^{\infty}(G)} - \frac c
N \delta^2\Big )\Big ), \quad k =1,2,3,4
\end{equation*}
and
\begin{equation} \label{eqQkDef}
Q_k \deq \big\{(t,x)\in [0, T]\times\overline{G} \mid  \alpha(t,x) >
\beta_k\big\}.
\end{equation}
It is obvious that $Q_{k+1}\subset Q_{k}, k=1,2,3$.

We claim that
\begin{equation}\label{eqQ4Q1}
\Big (t_0 - \frac \iota{\sqrt N}, t_0 + \frac \iota{\sqrt N}\Big )\times
\widehat G \subset Q_4 \subset Q_1 \subset (t_0 - \sqrt 2 \delta, t_0 +
\sqrt 2 \delta) \times \overline G.
\end{equation}
In fact, thanks to \cref{G'}, for all $ (t,x) \in  \big (t_0 - \frac \iota{\sqrt N}, t_0 + \frac \iota{\sqrt N}\big ) \times \widehat{G}$, it holds that
\begin{align*}
(t - t_{0})^{2} \leq \frac{\delta^{2}}{N}, \quad \quad \psi(x) > \frac{4}{N} |\psi|_{L^{\infty}(G)},
\end{align*}
which implies that
\begin{align*}
\phi(t,x) = \psi(x) - c (t-t_{0})^{2}
> \frac{4}{N} |\psi|_{L^{\infty}(G)} -  \frac{c \delta^{2}}{N}.
\end{align*}
Consequently, we have $ \alpha(t,x) > \beta_{4} $ and $ (t,x) \in Q_{4} $.
On the other hand, for all $ (t,x) \in Q_{1} $, it holds that
$ \frac{1}{N} |\psi|_{L^{\infty}(G)} -  \frac{c \delta^{2}}{N}
< \psi - c (t-t_{0})^{2} $.
From \cref{bound psi}, we obtain
$ 0 < 2 c \delta^{2} - c (t-t_{0})^{2} $.
Hence, it holds that $ (t-t_{0})^{2} \leq 2 \delta^{2} $, which implies $ (t, x) \in (t_0 - \sqrt 2 \delta, t_0 + \sqrt 2 \delta) \times \overline G $ and \cref{eqQ4Q1} is proved.

Let $\eta\in C_0^{\infty}(Q_2)$ be such that $\eta \in [0, 1]$ and $\eta = 1$  in   $Q_3$. Let $z = \eta \tilde y$. Then
\begin{align}\label{re-sys}
\left\{
\begin{aligned}
& dz - \sum_{j,k=1}^n \big(b^{jk} z_{x_j}\big)_{x_k} dt = \tilde a_1\cdot
\nabla z dt + \tilde a_2z dt + f dt +
\tilde a_3 zdW(t)   && \text{ in } Q_1\\
& z = \dfrac{\partial z}{\partial \nu} =  0 && \text{ on } \partial Q_1,
\end{aligned}
\right.
\end{align}
where
\begin{equation} \label{eqP8eq1a1}
f = \eta_t \tilde y - 2\sum_{j,k=1}^n b^{jk}\eta_{x_j} \tilde y_{x_k}
-\tilde y\sum_{j,k=1}^n (b^{jk}\eta_{x_j})_{x_k} - \tilde y \tilde a_1\cdot\nabla \eta.
\end{equation}
Clearly, $f$ is supported in $Q_2 \setminus Q_3$.

\medskip

{\it Step 2}. In this step, we apply \cref{thm.fIForParablic} to the equation \cref{eqP8eq1a1} and handle the boundary terms.

Let us choose  $u=z$ in \cref{thm.fIForParablic}. Integrate \cref{eq.fIForParablic} over $ G \times (0,T) $. Noting that $ z $ is supported in $ Q_{1} $, choosing $ \theta = e^{\ell} $ and $ \ell = \lambda \alpha $, setting $ w = \theta z $,  after taking mathematical expectation in both sides, we have
\begin{align} \notag \label{eqP8eq1}
& 2 \mathbb{E} \int_{Q_{1}} \theta \Big [\!-\!\sum_{j, k=1}^n \big(b^{j k} w_{x_j}\big)_{x_k}+\mathcal{A} w \Big ]\Big[d z\!-\!\sum_{j, k=1}^n\big(b^{j k} z_{x_j}\big)_{x_k} d t\Big] d x
\\\notag
&\quad
+  2 \mathbb{E} \int_{Q_{1}}  \sum_{j, k=1}^n\Big[\sum_{j^{\prime}, k^{\prime}=1}^n\big(2 b^{j k} b^{j^{\prime} k^{\prime}} \ell_{x_{j^{\prime}}} w_{x_j} w_{x_{k^{\prime}}}-b^{j k} b^{j^{\prime} k^{\prime}} \ell_{x_j} w_{x_{j^{\prime}}} w_{x_{k^{\prime}}}\big)
+\Psi b^{j k} w_{x_j} w
\\\notag
& \quad \quad \quad \quad \quad \quad \quad  ~
\quad
-b^{j k}\Big(\mathcal{A} \ell_{x_j}
+\frac{\Psi_{x_j}}{2}\Big) w^2\Big]_{x_k} d xd t
+  2   \mathbb{E} \int_{Q_{1}}  \sum_{j, k=1}^n\big(b^{j k} w_{x_j} d w\big)_{x_k}  d x
\\\notag
&
=2   \mathbb{E} \int_{Q_{1}} \sum_{j, k=1}^n c^{j k} w_{x_j} w_{x_k} d t d x
+ 2 \mathbb{E} \int_{Q_{1}} \mathcal{B} w^2 d t d x
\\\notag
&
\quad+2  \mathbb{E} \int_{Q_{1}}\Big[-\sum_{j, k=1}^n\big(b^{j k} w_{x_j}\big)_{x_k}+\mathcal{A} w\Big]^2 d t d x
- \mathbb{E} \int_{Q_{1}} \theta^2 \mathcal{A}(d z)^2 d x
\\
& \quad
- \mathbb{E} \int_{Q_{1}} \theta^2 \sum_{j, k=1}^n b^{j k}\big(d z_{x_j}+\ell_{x_j} d z\big)\big(d z_{x_k}+\ell_{x_k} d z\big) d x
.
\end{align}
Next, we estimate terms in \cref{eqP8eq1} one by one.

From \cref{re-sys}, the first term in the left hand side of \cref{eqP8eq1} satisfies that
\begin{align} \label{eqP8eq2} 
&  2{\mathbb{E}} \int_{Q_1}\theta
\Big[-\sum_{j,k=1}^n \big(b^{jk}w_{x_j}\big)_{x_k} + {\cal A} w\Big]\Big[dz
-\sum_{j,k=1}^n \big(b^{jk}z_{x_j}\big)_{x_k}
dt\Big]dx\\ \notag
&\leq {\mathbb{E}} \int_{Q_1} \Big|\!-\!\sum_{j,k=1}^n
\big(b^{jk}w_{x_j}\big)_{x_k}\! \!+\! {\cal A} w\Big|^2 dxdt  +  {\mathbb{E}} \int_{Q_1}
\theta^2 \big(\tilde{a}_1 \cdot\nabla z \!+\! \tilde{a}_2 z\!+\! f\big)^2 dxdt.
\end{align}
Thanks to Divergence Theorem, recalling that  $ z $ is supported in $ Q_{1} $, we have
\begin{align}  \label{eqP8eq3} \notag
& 2 \mathbb{E} \int_{Q_{1}}  \sum_{j, k=1}^n\Big[\sum_{j^{\prime}, k^{\prime}=1}^n\big(2 b^{j k} b^{j^{\prime} k^{\prime}} \ell_{x_{j^{\prime}}} w_{x_j} w_{x_{k^{\prime}}}-b^{j k} b^{j^{\prime} k^{\prime}} \ell_{x_j} w_{x_{j^{\prime}}} w_{x_{k^{\prime}}}\big)
\\
& \quad \quad \quad \quad    ~
\quad+\Psi b^{j k} w_{x_j} w-b^{j k}\Big(\mathcal{A} \ell_{x_j}+\frac{\Psi_{x_j}}{2}\Big) w^2\Big]_{x_k} d xd t = 0
.
\end{align}

Combining \cref{eqP8eq3,eqP8eq2,eqP8eq1}, we have
\begin{align}
\label{pr-2}\notag
&   {\mathbb{E}} \int_{Q_1}\theta^2 \big(a_1
\cdot\nabla z + a_2 z + f\big)^2 d t d  x  \\\notag
&\geq 2{\mathbb{E}}\int_{Q_1}\sum_{j,k=1}^n c^{jk} w_{x_j} w_{x_k}
d t d  x  + {\mathbb{E}}\int_{Q_1} {\cal B} w^2 dx dt- {\mathbb{E}}\int_{Q_1} \theta^2 {\cal A}
(dz)^2 dx
\\
&  \quad - {\mathbb{E}} \int_{Q_1}
\theta^2 \sum_{j,k=1}^n b^{jk} \big[\big(dz_{x_j} + \ell_{x_j}
dz\big)\big(dz_{x_k} + \ell_{x_k} dz\big)\big] dx.
\end{align}
From \cref{eqP8eq1a1}, we have
\begin{align} \label{eqP8eq4} \notag
& \ {\mathbb{E}} \int_{Q_1} \theta^2 \big(\tilde a_1
\cdot\nabla z + \tilde a_2 z + f\big)^2 dxdt\\
& \leq  \mathcal{C}
{\mathbb{E}} \int_{Q_1} \theta^2 |\nabla z|^2 dx dt   +
\mathcal{C} {\mathbb{E}}
\int_{Q_1}\theta^2 z^2 dxdt
+ {\cal C} {\mathbb{E}} \int_{Q_2\backslash Q_3} \theta^2 (|\nabla
\tilde y|^2 + \tilde y^2)dxdt.
\end{align}

{\it Step 3}. In this step,  we handle the terms in the right hand side of inequality \cref{pr-2}.
Choosing $ \Psi $ as \cref{h5}, for $j,k=1,2,\cdots,n$, it is straightforward  to verify  that
\begin{equation}\label{eqP8eq5}
\ell_t= \lambda \alpha O(\mu),\quad \ell_{x_j}=\lambda\mu\alpha\psi_{x_j},\quad
\ell_{x_jx_k}=\lambda\mu^2\alpha\psi_{x_j}\psi_{x_k}+\lambda\mu\alpha\psi_{x_jx_k}
.
\end{equation}
Similarly to the proof of \cref{thm.CarlemanForHeat}, combining \cref{h5,eq.identyDetail,eqP8eq5}, we have that
\begin{align}\label{eqP8eq6}
  \sum_{j,k=1}^n c^{jk}w_{x_j}w_{x_k}
  & = 2 \lambda \mu^2 \alpha \Big (\sum_{j, k=1}^n b^{j k} \psi_{x_j} w_{x_k} \Big )^2
  \\
\notag
& \quad 
  +\lambda \mu^2 \alpha \Big (\sum_{j, k=1}^n b^{j k} \psi_{x_j} \psi_{x_k} \Big )  \Big (\sum_{j, k=1}^n b^{j k} w_{x_j} w_{x_k} \Big )
+ \lambda \alpha O(\mu)  |\nabla w|^{2}
.
\end{align}
Thanks to \cref{eq.identyDetail,eqP8eq5,h5}, we obtain that
\begin{align}\label{eqP8eq7}
{\cal A} &\displaystyle   = -\lambda^2\mu^2 \alpha^2\sum_{j,k=1}^n b^{jk}\psi_{x_j}
\psi_{x_k} + \lambda \alpha O(\mu^{2}),
\end{align}
and that
\begin{align}\label{eqP8eq8}
\mathcal{B}
& =  2\lambda^3 \mu^4     \alpha^3\Big (\sum_{j,k=1}^n b^{jk}\psi_{x_j}\psi_{x_k}\Big )^2
+\lambda^3\alpha^3O(\mu^3)
+\lambda^2\alpha^2O(\mu^4)
.
\end{align}
From \cref{eqP8eq1a1}, we get that
\begin{align}\label{eqP8eq9} \notag
&    {\mathbb{E}} \int_Q \theta^2 \sum_{j,k=1}^n b^{jk}(d z_{x_j} +
\ell_{x_j} dz)(dz_{x_k} + \ell_{x_k}
dz)dx\\
&  \leq {\cal C}
\Big[{\mathbb{E}}\!\int_Q \theta^2(z^2\! +\! |\nabla z|^2)dxdt +
\!\lambda^2\mu^2 {\mathbb{E}}\! \int_Q\! \theta^2\alpha^2 |z|^2 dxdt\Big].
\end{align}
Noting that $ w = \theta z $ and $ z = \theta^{-1} w $, we get that
\begin{align*}
\frac 1 {\cal C} \theta^2 \big(|\nabla z|^2 \!+
\lambda^2\mu^2\alpha^2 z^2\big)\leq |\nabla w|^2\! +
\lambda^2\mu^2\alpha^2 |w|^2 \leq {\cal C} \theta^2 \big(|\nabla z|^2\! +
\lambda^2\mu^2\alpha^2 z^2\big).
\end{align*}
Combining it with \cref{pr-2,eqP8eq4,eqP8eq6,eqP8eq7,eqP8eq8,eqP8eq9}, we obtain
\begin{align*}
& \mathcal{C} \mathbb{E} \int_{Q_{2} \setminus Q_{3}} \theta^{2} ( |\nabla \tilde{y} |^{2} + |\tilde{y}|^{2} ) d x d t
+ \mathcal{C} \mathbb{E} \int_{Q_{1}} \theta^{2} (
|\nabla z |^{2} + (\lambda^{2} \mu^{2} \alpha^{2} + 1) z^{2}
) d x d t
\\
& \quad
+ \mathcal{C} \mathbb{E} \int_{Q_{0}} \theta^{2} (
\lambda \mu^{2} \alpha |\nabla z|^{2}
+ \lambda^{3} \mu^{4} \alpha^{3} |  z|^{2}
) d x d t
\\
& \geq
\mathbb{E} \int_{Q_{1} \cup Q_{0}} \theta^{2} (
\lambda \mu^{2} \alpha |\nabla z|^{2}
+ \lambda^{3} \mu^{4} \alpha^{3} |  z|^{2}
) d x d t
\\
& \quad
+ \mathbb{E} \int_{Q_{1}} \theta^{2} \big[
\lambda \alpha O(\mu) |\nabla z |^{2}
+(  \lambda^{3}  \alpha^{3} O(\mu^{3}) + \lambda^2\alpha^2O(\mu^4)) z^{2}
\big] d x d t
\end{align*}
This implies that  there exists a constant $ \mu_{0} > 1 $, such that for all $ \mu > \mu_{0} $, there exists a constant $ \lambda_{0} = \lambda_{0}(\mu) $, so that for any $ \lambda > \lambda_{0}(\mu) $, it holds that
\begin{align} \label{eqP8eq10} \notag
& \mathcal{C} \mathbb{E} \int_{Q_{2} \setminus Q_{3}} \theta^{2} ( |\nabla \tilde{y} |^{2} + |\tilde{y}|^{2} ) d x d t
+ \mathcal{C} \mathbb{E} \int_{Q_{0}} \theta^{2} (
\lambda \mu^{2} \alpha |\nabla z|^{2}
+ \lambda^{3} \mu^{4} \alpha^{3} |  z|^{2}
) d x d t
\\
& \geq
\mathbb{E} \int_{Q_{1}} \theta^{2} (
\lambda \mu^{2} \alpha |\nabla z|^{2}
+ \lambda^{3} \mu^{4} \alpha^{3} |  z|^{2}
) d x d t
\end{align}

Noting that $ z = \tilde{y} $ in $ Q_{3} $ and $ Q_{3}  \subset Q_{1} $, for $ \mu > \mu_{0} $ and $ \lambda > \lambda_{0} $, we deduce from \cref{eqP8eq10} that
\begin{align} \label{eqP8eq11} \notag
& \mathcal{C} \mathbb{E} \int_{Q_{2} \setminus Q_{3}} \theta^{2} ( |\nabla \tilde{y} |^{2} + |\tilde{y}|^{2} ) d x d t
+ \mathcal{C} \mathbb{E} \int_{Q_{0}} \theta^{2} (
\lambda \mu^{2} \alpha |\nabla \tilde{y}|^{2}
+ \lambda^{3} \mu^{4} \alpha^{3} |  \tilde{y}|^{2}
) d x d t
\\
& \geq
\mathbb{E} \int_{Q_{3}} \theta^{2} (
\lambda \mu^{2} \alpha |\nabla \tilde{y}|^{2}
+ \lambda^{3} \mu^{4} \alpha^{3} |  \tilde{y}|^{2}
) d x d t
.
\end{align}
From the definition of $Q_k, k=1,2,3,4$ (see \cref{eqQkDef}) and $\widehat G$, we have that
\begin{align} \label{eqP8eq12} \notag
&  \lambda^3 \mu^4 {\mathbb{E}} \int_{Q_3}\alpha^3 \theta^2 \tilde y^2 dxdt
+
\lambda\mu^2 {\mathbb{E}}\int_{Q_3}\alpha \theta^2 |\nabla \tilde y|^2 dxdt 
\\
&  
\geq {\cal C} e^{2\lambda \beta_4} {\mathbb{E}} \int_{ t_0 - \frac
\iota{\sqrt N}}^{{ t_0 + \frac \iota{\sqrt
N}}}\int_{G'}\big(\lambda^3 \tilde y^2 + \lambda |\nabla \tilde y|^2\big)dxdt.
\end{align}
From the relationship between $Q_2$ and $Q_3$, we find
\begin{equation}  \label{eqP8eq13}
{\mathbb{E}} \int_{Q_2\backslash Q_3} \theta^2 (|\nabla \tilde y|^2 + \tilde y^2)dxdt \leq
e^{2\lambda \beta_3} {\mathbb{E}} \int_{Q_1}\big(\tilde y^2 + |\nabla
\tilde y|^2\big)dxdt.
\end{equation}

Combining \cref{eqP8eq11,eqP8eq12,eqP8eq13}, for $ \mu > \mu_{0} $ and $ \lambda > \lambda_{0} $, we find that
\begin{align*}
&   e^{2\lambda \beta_4} {\mathbb{E}} \int_{ t_0 - \frac
\iota{\sqrt N}}^{{ t_0 + \frac \iota{\sqrt
N}}}\int_{G'}\big(\lambda^3 \tilde y^2 + \lambda |\nabla\tilde  y|^2\big)dxdt\\
& \leq  {\cal C} \lambda^3\mu^4 {\mathbb{E}} \int_{Q_0}\alpha^3 \theta^2
\tilde y^2 dx dt + {\cal C} \lambda\mu^2 {\mathbb{E}} \int_{Q_0}\alpha \theta^2 |\nabla
\tilde y|^2 dxdt
\\
& \quad
+ {\cal C} e^{2\lambda \beta_3}{\mathbb{E}} \int_{Q_1}\big(\tilde y^2 +
|\nabla \tilde y|^2\big)dxdt,
\end{align*}
which yields that
\begin{align} \label{eqP8eq15} \notag
&   {\mathbb{E}} \int_{ t_0 - \frac \iota{\sqrt N}}^{{ t_0 + \frac
\iota{\sqrt
N}}}\int_{\widehat G}(\tilde y^2 + |\nabla \tilde y|^2)dxdt\\
& \leq  {\cal C} e^{{\cal C}\lambda} {\mathbb{E}} \int_{Q_0}\Big (\tilde y^2 + |\nabla
\tilde y|^2\Big )dx dt + {\cal C}  e^{-2\lambda(\beta_4 - \beta_3)} {\mathbb{E}}
\int_{Q_1}(\tilde y^2 + |\nabla \tilde y|^2)dxdt.
\end{align}
Put $\varepsilon _0 = \exp\big (-2 \lambda_1(\beta_4 - \beta_3)\big )$.
Then one finds that for any $0 < \varepsilon < \varepsilon_0$, it holds that
\begin{align} \label{eqP8eq14} \notag
&   {\mathbb{E}} \int_{ t_0 - \frac \iota{\sqrt N}}^{{ t_0 + \frac
\iota{\sqrt N}}}\int_{\widehat G}(\tilde y^2 + |\nabla \tilde y|^2)dxdt\\
& \leq  {\cal C} \exp \Bigl( \frac{\mathcal{C}}{\varepsilon} \Bigr) {\mathbb{E}} \int_{Q_0}\Big (\tilde y^2 + |\nabla
\tilde y|^2\Big )dx dt + {\cal C} \varepsilon {\mathbb{E}}
\int_{Q_1}(\tilde y^2 + |\nabla \tilde y|^2)dxdt.
\end{align}
Clearly, \cref{eqP8eq14} holds for $ \varepsilon > \varepsilon_{0} $.
Hence, it is true for all $ \varepsilon > 0 $.

\medskip

{\it Step 4}. In this step,  we handle  the left hand side of \cref{eqP8eq14} to end the proof.

Let $m\in{\mathbb{N}}$ such that
\begin{equation*}
\sqrt 2 \;\iota+\frac{m \iota}{\sqrt N} \leq T - \sqrt 2
\iota\leq \sqrt 2 \; \iota+ \frac{(m+1)\iota}{\sqrt N}.
\end{equation*}
Recalling that $t_0$ can be any element in $[\sqrt 2 \iota, T - \sqrt 2 \iota]$, we take
\begin{equation*}
t_0 = \sqrt 2 \iota+ \frac{k \iota}{\sqrt N}, \quad k =0, 1,
2, \cdots, m.
\end{equation*}
Then, it follows from \cref{eqP8eq14} that
\begin{align*}
&   {\mathbb{E}} \int_{ \frac{T}{2} - \kappa T}^{ \frac{T}{2} +  \kappa T}\int_{\widehat G}(\tilde y^2 + |\nabla \tilde y|^2)dxdt\\
& \leq  {\cal C} \exp \Bigl( \frac{\mathcal{C}}{\varepsilon} \Bigr) {\mathbb{E}} \int_{Q_0}\Big (\tilde y^2 + |\nabla
\tilde y|^2\Big )dx dt + {\cal C} \varepsilon {\mathbb{E}}
\int_{(0,T)\times \tilde{G}}(\tilde y^2 + |\nabla \tilde y|^2)dxdt.
\end{align*}
This completes the proof of \cref{main result}.
\end{proof}

\section{Solution to inverse state problem with a terminal measurement}
\label{secBackwardProblem}

In this section, we prove two results  for the conditional stability of \cref{prob.2.2}.
The first one  is
\begin{theorem}\label{th consta}
For any $t_0\in (0,T]$, there exists a positive constant $\mathcal{C}=\mathcal{C}(t_0)$ such that for some generic constant $\gamma\in (0,1)$ (which is independent of $t_0$) and for any $M>0$, for all $ y_0, \hat y_0\in U_{M,0} $, it holds that
\begin{equation}\label{10.1-eq1}
| y(t_0)-\hat y(t_0)|_{L^2_{\mathcal{F}_{t_0}}(\Omega;L^{2}(G))} \leq
\mathcal{C} M^{1-\gamma}| y(T)-\hat y(T)|_{L^2_{\mathcal{F}_{T}}(\Omega;H^{1}(G))}^{\gamma},
\end{equation}
where $y(\cdot)$ (resp. $\hat y(\cdot)$) is the solution to \cref{sp-eq1} with respect to the initial datum $y_0$ (resp. $\hat y_0$).
\end{theorem}
\begin{remark}
If  $G$ is convex, the term
$|y(T)-\hat y(T)|_{L^2_{\mathcal{F}_{T}}(\Omega;H^{1}(G))}$ in the
right hand side of \eqref{10.1-eq1} can be
replaced by
$|y(T)-\hat y(T)|_{L^2_{\mathcal{F}_{T}}(\Omega;H^{1}(G_0))}$ for an
arbitrary nonempty open subset $G_0$ of $G$
(see \cite{Lue2015b}).
\end{remark}

\begin{remark}
From \cref{th consta}, we know the answer to \cref{prob.2.2} is positive. Moreover, the inequality \eqref{10.1-eq1} shows that the function $\beta(\cdot)$ in \eqref{eqStableTotal} can be chosen as $\beta(\eta) = \mathcal{C} M^{1-\gamma}\eta^{\gamma}$ for $\eta\in [0,\infty)$.
\end{remark}

By a standard energy estimate, \cref{th consta} follows from the following result, which is a generalization of \cite[Theorem 1.2]{Lue2012}.

\begin{theorem}\label{inv th1}
There exists a constant $\gamma\in (0,1)$ such that for each $t_0\in [0,T]$, one can find a constant $\mathcal{C}=\mathcal{C}(t_0)>0$ satisfying that
\begin{align}\label{conditional sta1} 
& |y (t_0)-\hat y(t_0)|_{L^2_{\mathcal{F}_{t_0}}(\Omega;L^2(G))}
 \leq
\mathcal{C} |y (\cdot)-\hat y(\cdot) |^{1-\gamma}_{L^2_{{\mathbb{F}}}(0,T;L^2(G))}
|y(T)-\hat y(T)|^{\gamma}_{L^2_{\mathcal{F}_T}(\Omega;H^1(G))},
\end{align}
where $y(\cdot)$ (resp. $\hat y(\cdot)$) solves the equation \cref{sp-eq1} with the initial datum $y_0$ (resp. $\hat y_0$).
\end{theorem}

In \cref{th consta}, the left hand side of the condition stability estimate \cref{10.1-eq1} is
$| y(t_0)-\hat y(t_0)|_{L^2_{\mathcal{F}_{t_0}}(\Omega;L^{2}(G))}$ and the right hand side is $| y(T)-\hat y(T)|_{L^2_{\mathcal{F}_{T}}(\Omega;H^{1}(G))}^{\gamma}$. This means that  the unknown state is continuously dependent on the measurement  in $L^2_{\mathcal{F}_{t_0}}(\Omega;L^{2}(G))$ norm when the measurement is in $L^2_{\mathcal{F}_{T}}(\Omega;H^{1}(G))$. In other words, the norm for the measured data is stronger than the  unknown state. Generally speaking, the stronger of the norm for the measured data, the more accurate measurement should we do. On the other hand, the weaker the norm for the unknown state, the slower the convergence rate of the numerical approximation for the unknown state will be.
Hence, it is natural to expect to relax the norm in the right hand side of \cref{10.1-eq1} while enforcing a stronger norm in the left hand side of \cref{10.1-eq1}. This is done by the following result.

\begin{theorem} \label{thmDou2022}
For any $t_0\in (0,T)$, there exists a positive constant $\mathcal{C}=\mathcal{C}(t_0)$ such that for some generic constant $\gamma = \gamma(t_{0}, T) \in (0,1)$  and for any $M>0$, for all $ y_0, \hat y_0\in U_{M,1} $,
letting
\begin{align}
    \label{thmDou2022beta}
    \beta(\eta) =
    \mathcal{C} \eta
    +
    \mathcal{C} M^{2} \operatorname{exp}  [-  3^{-\gamma} \ln^{\gamma} (\eta^{-1}) ]
\end{align}
it holds that
\begin{equation}\label{eqDou2022thm}
| y(t_0)-\hat y(t_0)|^{2}_{L^2_{\mathcal{F}_{t_0}}(\Omega;H^{1}(G))} \leq \beta\Big( | y(T) - \hat{y}(T) |_{L^{2}_{\mathcal{F}_{T}}(\Omega; L^{2}(G))} \Big),
\end{equation}
where $y(\cdot)$ (resp. $\hat y(\cdot)$) is the solution to \cref{sp-eq1} with respect to the initial datum $y_0$ (resp. $\hat y_0$).
\end{theorem}
%

\subsection{Proof of the conditional stability I}

In this subsection, we will prove the first result of conditional stability, i.e., \cref{inv th1}.

Let $s\in (0,+\infty)$, $t\in (0,+\infty)$, and $\psi \in C^{\infty}({\mathbb{R}})$
with $\psi_t\geq 1$.
Put
\begin{equation}\label{weight1-1} \varphi=e^{\lambda
\psi},\qquad\theta = e^{s\varphi}.
\end{equation}

As a consequence of \cref{thm.fIForParablic}, we have the following result.

\begin{proposition}\label{identity1}
Assume that $u$ is an $H^2(\mathbb{R}^n)$-valued
continuous It\^o process. Put $v=\theta u$ (with $\theta$ given by \cref{weight1-1}). Then,
\begin{align}\label{bu1}
\notag
& \quad-\theta\Big[ \sum_{j,k=1}^n (b^{jk}v_{x_j})_{x_k}+s\lambda\varphi \psi_t v
\Big]\Big[ du - \sum_{j,k=1}^n (b^{jk}u_{x_j})_{x_k}dt \Big]
\\ \notag
& \quad+
\frac{1}{4}\lambda\theta v \Big[du - \sum_{j,k=1}^n
(b^{jk}u_{x_j})_{x_k}dt\Big]
\\ \notag
& = - \!\sum_{j,k=1}^n\! \Big(b^{jk}v_{x_j} dv\! +\!
\frac{1}{4}b^{jk}v_{x_j} vdt\Big)_{x_k}  \!
+ \frac{1}{2}d\Big(
\sum_{j,k=1}^n b^{jk}v_{x_j} v_{x_k}\! -\!
s\lambda\varphi\psi_t v^2\! + \!\frac{1}{8}\lambda v^2 \Big)
\\ \notag
& \quad + \Big(\frac{1}{4}\lambda
\sum_{j,k=1}^n b^{jk}v_{x_j} v_{x_k} dt - \frac{1}{2}\sum_{j,k=1}^n b^{jk}dv_{x_j}
dv_{x_k}
\Big) + \frac{1}{2}s\lambda^2\varphi \psi_t^2 v^2 dt
\\ \notag
& \quad
+ \frac{1}{2}   s\lambda\varphi\psi_{tt} v^2dt
- \frac{1}{4}s\lambda^2\psi_t \varphi v^2 dt
+ \frac{1}{2}s\lambda\varphi\psi_t (dv)^2
- \frac{1}{8}\lambda (dv)^2
\\
&  \quad
+ \Big[ \sum_{j,k=1}^n (b^{jk}v_{x_j})_{x_k}
+        s\lambda\varphi \psi_t v \Big]^2 dt.
\end{align}
\end{proposition}
\begin{proof}
Noting that $ \psi $ is independent of the $ x $-variable, choosing $ \Psi = - \frac{1}{4} \lambda $ in \cref{thm.fIForParablic}, the proof is straightforward.
\end{proof}

We have the following Carleman estimate for \cref{h9}, which is a generalization of Theorem 2.1 in \cite{Yuan2021}:

\begin{theorem}\label{carleman est1}
Let $\delta\in[0,T)$, and let $\varphi$ and $\theta$ be given in
\eqref{weight1-1}. There exist constants $ \mathcal{C} > 0 $ and  $\lambda_1>0$ such that for all
$\lambda\geq\lambda_1$, there exists an $s_0(\lambda)>0$ so that for all $s\geq
s_0(\lambda)$, for any solution $ z $ to the equation \cref{h9}, it holds that
\begin{align}\label{car1}
\notag
&\lambda \mathbb{E}\int_\delta^T\int_G \theta^2|\nabla z|^2 dxdt + s\lambda^2
\mathbb{E}\int_\delta^T\int_G
\varphi \theta^2 z^2 dxdt
+ s \lambda \mathbb{E} \int_{\delta}^{T} \int_{G} \varphi \theta^{2} g^{2} d x d t
\\
\notag
&  \leq \mathcal{C}\,\mathbb{E}\Big[ \theta^2(T)|\nabla z(T)|_{L^2(G)}^2 +
\theta^2(\delta)|\nabla
z(\delta)|_{L^2(G)}^2 + s\lambda\varphi(T)\theta^2(T)|z(T)|_{L^2(G)}^2
\\
&  \qquad \quad
+ s\lambda\varphi(\delta)\theta^2(\delta)
|z(\delta)|_{L^2(G)}^2
+\int_\delta^T\int_G  \theta^2\big( f^2
+ |\nabla g|^2\big)dxdt \Big].
\end{align}
\end{theorem}

\begin{remark}
\label{rkCarleman}
If we replace the assumption of $ \psi_{t} $ with $ |\psi_{t}| \geq 1 $, then one can follow the proof of \cref{car1} step by step to get
\begin{align}\label{car1Lue}
    \notag
    &\lambda \mathbb{E}\int_\delta^T\int_G \theta^2|\nabla z|^2 dxdt + s\lambda^2
    \mathbb{E}\int_\delta^T\int_G
    \varphi \theta^2 z^2 dxdt
    \\
    \notag
    &  \leq \mathcal{C}\,\mathbb{E}\Big[\theta^2(T)|\nabla z(T)|_{L^2(G)}^2 +
    \theta^2(\delta)|\nabla
    z(\delta)|_{L^2(G)}^2 + s\lambda\varphi(T)\theta^2(T)|z(T)|_{L^2(G)}^2
    \\
    &  \qquad \quad
    + s\lambda\varphi(\delta)\theta^2(\delta)
    |z(\delta)|_{L^2(G)}^2
    \!+\!\int_\delta^T\!\!\int_G  \theta^2\big( f^2
    \!+ \!|\nabla g|^2 \!+ \!s \lambda \varphi g^{2} \big)dxdt \Big].
\end{align}
\end{remark}
\begin{proof}


Applying \cref{identity1} to the equation \cref{h9} with $u=z$, integrating the equality \cref{bu1} on $[\delta,T]\times G$, and taking mathematical expectation on  both sides, we obtain that
\begin{align}\label{buil1}
&\quad-\mathbb{E}\int_\delta^T \int_G \theta\Big[ \sum_{j,k=1}^n
        (b^{jk}v_{x_j})_{x_k} + s\lambda\varphi\psi_t v \Big] \Big[  dz  -
        \sum_{j,k=1}^n
        (b^{jk}z_{x_j})_{x_k}dt  \Big]dx \nonumber \\ \notag
& \quad + \frac{1}{4}\lambda\mathbb{E}\int_\delta^T \int_G \theta v  \Big[
        dz -  \sum_{j,k=1}^n
        (b^{jk}z_{x_j})_{x_k}dt \Big]dx \\ \notag
&= \! - \mathbb{E} \int_\delta^T  \int_G\sum_{j,k=1}^n
        \Big( b^{jk}v_{x_j} dv  +  \frac{1}{4}\lambda b^{jk}v_{x_j}
        vdt \Big)_{x_k} dx
\\ \notag
& \quad
+  \frac{1}{2}\mathbb{E}  \int_\delta^T  \int_G  d\Big(
\sum_{j,k=1}^n b^{jk}v_{x_j} v_{x_k}  -
s\lambda\varphi\psi_t v^2  +  \frac{1}{8}\lambda v^2  \Big)dx\nonumber \\
&\quad +\mathbb{E}\int_\delta^T\int_G\Big(\frac{1}{4}\lambda
\sum_{j,k=1}^n b^{jk}v_{x_j} v_{x_k} dt - \frac{1}{2}\sum_{j,k=1}^n
b^{jk}dv_{x_j}
dv_{x_k}   \Big)dx \nonumber\\ \notag
&\quad +
\mathbb{E}\int_\delta^T\int_G\Big[\frac{1}{2}s\lambda^2\varphi\psi_t^2 v^2 dt
+  \frac{1}{2} s\lambda\varphi\psi_{tt} v^2dt -
\frac{1}{4}s\lambda^2\varphi\psi_t v^2 dt
\\ \notag
& \quad \quad \quad \quad \quad \quad ~
+ \frac{1}{2}s\lambda\varphi\psi_t(dv)^2 - \frac{1}{8}\lambda (dv)^2\Big]dx \nonumber\\
&\quad + \mathbb{E}\int_\delta^T\int_G\Big[ \sum_{j,k=1}^n
(b^{jk}v_{x_j})_{x_k} + s\lambda\varphi \psi_t v \Big]^2 dxdt.
\end{align}

Now we estimate the terms in the right  hand side of \cref{buil1}  one by one.
Noting that $v|_\Sigma  = z|_\Sigma = 0$, thanks to Divergence Theorem, we have
\begin{align}\label{buil2} \notag
&\quad- \mathbb{E}\int_\delta^T\int_G\sum_{i,j}^n \Big(b^{jk}v_{x_i} dv +
\frac{1}{4}\lambda
b^{jk}v_{x_i} vdt\Big)_{x_j}dx\\
& = - \mathbb{E}\int_\delta^T\int_{\Gamma}\sum_{j,k=1}^n b^{jk}\Big(
v_{x_i} dv + \frac{1}{4}\lambda v_{x_i} vdt\Big)\nu^j d\,\Gamma = 0.
\end{align}
The second one reads as follows:
\begin{align}\label{buil3}  
& \frac{1}{2}\mathbb{E}\int_\delta^T\int_G d\Big( \sum_{j,k=1}^n
b^{jk}v_{x_j} v_{x_k} - s\lambda\varphi \psi_t v^2 + \frac{1}{8}\lambda v^2
\Big)dx \\ \notag
& \geq -\mathcal{C}\mathbb{E}\Big(|\nabla v(T)|_{L^2(G)}^2 + |\nabla
v(\delta)|_{L^2(G)}^2 + s\lambda\varphi(T)|v(T)|_{L^2(G)}^2
+ s\lambda\varphi(\delta) |v(\delta)|_{L^2(G)}^2\Big).
\end{align}
For the third one, thanks to \cref{eq.bijGeq,h9}, it holds that
\begin{align}\label{buil4}  \notag
& \mathbb{E}\int_\delta^T\int_G\Big(\frac{1}{4}\lambda \sum_{j,k=1}^n
b^{jk}v_{x_j} v_{x_k} dt - \frac{1}{2}\sum_{j,k=1}^n b^{jk}dv_{x_j} dv_{x_k}
\Big)dx \\
& \geq \frac{1}{4}\lambda \mathbb{E}\int_\delta^T\int_G s_{0} |\nabla v|^2 dxdt
- \mathcal{C}\mathbb{E}\int_\delta^T\int_G \theta^2 |\nabla g|^2
dxdt.
\end{align}

Noting that  $\psi_t \geq 1$, 
there exists a $\lambda_0>0$ such that for all $\lambda\geq\lambda_0$, it holds that
\begin{align}\label{buil5.1} \notag
    &\mathbb{E}\!\int_\delta^T \! \! \int_G\Big[\frac{1}{2}s\lambda^2\varphi\psi_t^2 v^2
    dt \!+\!  \frac{1}{2} s\lambda\varphi\psi_{tt} v^2dt \!-\!
    \frac{1}{4}s\lambda^2\varphi\psi_t v^2 dt
    \!+\!
    \frac{1}{2}s\lambda\varphi\psi_t(dv)^2 \!-\! \frac{1}{8}\lambda (dv)^2\Big]dx
\\
&\geq \frac{1}{8}s\lambda^2\mathbb{E}\int_\delta^T\int_G \varphi v^2
dxdt
+ \frac{1}{4}  s \lambda \mathbb{E}\int_\delta^T\int_G \varphi \theta^{2} g ^2 dxdt.
\end{align}

Now, we estimate the terms in the left hand side of \cref{buil1}.
Thanks to \cref{h9}, we obtain
\begin{align}\label{buir1}\notag
& -\mathbb{E}\int_\delta^T\int_G\theta\Big[ \sum_{j,k=1}^n
(b^{jk}v_{x_j})_{x_k}+s\lambda\varphi  \psi_t v \Big]\Big[ dz -
\sum_{j,k=1}^n
(b^{jk}z_{x_j})_{x_k}dt \Big]dx \\
& \leq \mathbb{E}\int_\delta^T\int_G \Big[ \sum_{j,k=1}^n
(b^{jk}v_{x_j})_{x_k}+s\lambda\varphi \psi_t v \Big]^2 dxdt +
\mathbb{E}\int_\delta^T\int_G
\theta^{2} f^2 dxdt,
\end{align}
and
\begin{align}\label{buir2}\notag
& \frac{1}{4}\lambda \mathbb{E}\int_\delta^T\int_G \theta v \Big[ dz -
\sum_{j,k=1}^n (b^{jk}z_{x_j})_{x_k}dt \Big]dx 
\\
&  
\leq \frac{1}{64}\lambda^2 \mathbb{E}\int_\delta^T\int_G v^2 dxdt +
\mathbb{E}\int_\delta^T\int_G \theta^{2}
f^2 dxdt.
\end{align}
From (\ref{buil1})--(\ref{buir2}), we find
\begin{align}\label{bui1} \notag
&\frac{1}{4}\lambda\mathbb{E}\int_\delta^T \int_G   |\nabla v|^2 dxdt
\!+\!  \frac{1}{4}  s \lambda \mathbb{E}\int_\delta^T   \int_G \varphi\theta^{2} g ^2 dxdt
\nonumber
\!-\!
\mathcal{C} \mathbb{E}\int_\delta^T   \int_G
\theta^{2} \big(|\nabla  g|^2  + |f|^2 \big)dxdt
\\
&\quad + \Big(\frac{1}{8}s\lambda^2-
\frac{1}{64}\lambda^2\Big)\mathbb{E}\int_\delta^T\int_G \varphi v^2 dxdt \nonumber\\ \notag
& \leq  \mathcal{C}\mathbb{E}\Big(|\nabla v(T)|_{L^2(G)}^2 + |\nabla
v(\delta)|_{L^2(G)}^2 + s\lambda\varphi(T)|v(T)|_{L^2(G)}^2
\\
& \quad \quad \quad  + s\lambda\varphi(\delta)
|v(\delta)|_{L^2(G)}^2 \Big).
\end{align}
Hence, there exists a $\lambda_1\geq \lambda_0$ such that for all $\lambda \geq\lambda_1$, one can find an $s_0(\lambda)>0$ so that for all $s\geq s_0(\lambda)$, it holds
\begin{align*}  \notag
& \lambda\mathbb{E}\int_\delta^T\! \int_G  \!|\nabla v|^2 dxdt
+   s \lambda \mathbb{E}\int_\delta^T\int_G \varphi \theta^{2} g ^2 dxdt
+  s\lambda^2 \mathbb{E}\int_\delta^T\int_G \varphi v^2 dxdt \nonumber
\\ \notag
& \leq  \mathcal{C}\mathbb{E}\Big[|\nabla v(T)|_{L^2(G)}^2 + |\nabla
v(\delta)|_{L^2(G)}^2 + s\lambda\varphi(T)|v(T)|_{L^2(G)}^2
\\
& \quad \quad \quad  + s\lambda\varphi(\delta)
|v(\delta)|_{L^2(G)}^2 \Big]
+ \mathcal{C} \mathbb{E}\int_\delta^T\! \int_G \! \theta^{2} \big(|\nabla
g|^2  + |f|^2 \big)dxdt,
\end{align*}
which implies inequality \cref{car1} immediately.
\end{proof}

Now, we are in a position to prove  \cref{inv th1}.

\begin{proof}[Proof of \cref{inv th1}]

Choose $t_1$ and $t_2$ such that $0< t_1 < t_2 < t_0$.
Set $\rho\in C^\infty(\mathbb{R})$ such that $0\leq\rho\leq 1$ and
\begin{equation}\label{chi}
\rho = \left\{
\begin{array}{ll}
    1, & t\geq t_2,\\
    0, & t \leq t_1.
\end{array}
\right.
\end{equation}

Let $z=\rho (y -\hat y)$.  Then, $z$ solves
\begin{equation}\label{inv equ1}
\left\{
\begin{aligned}
     &dz - \sum\limits_{j,k=1}^n (b^{jk}z_{x_j})_{x_k} dt= \big[\rho \big(F(\nabla y, y)- F(\nabla \hat y,\hat y)\big) + \rho_t\big(y -\hat y\big) \big]dt\\
    &\qquad\qquad\qquad\qquad\qquad ~ +
    \rho \big(K(y)-K(\hat y)\big) dW(t) &\mbox{ in } Q,\\
     &z=0 &\mbox{ on } \Sigma,\\
     &z(0) = 0 &\mbox{  in  }G.
\end{aligned}
\right.
\end{equation}
Applying \cref{carleman est1} with $\psi=t$ and $\delta=0$ to the
equation \cref{inv equ1}, for  $\lambda\geq\lambda_1$ and $s\geq s_0(\lambda)$,
we have
\begin{align}\label{inv eq1}\notag
& \lambda\mathbb{E}\int_Q \theta^2|\nabla z|^2dxdt + s\lambda^2\mathbb{E}\int_Q
\theta^2\varphi
|z|^2dxdt\\\notag
&\leq \mathcal{C}\mathbb{E}\Big[\theta^2(T)\big|\nabla  z(T)\big|^2_{L^2(G)}
+ s\lambda\varphi(T)\theta^2(T)\big|z(T)\big|^2_{L^2(G)} \\\notag
&\qquad\quad
+ \int_Q \theta^2|\rho_t(y -\hat y)|^2 dxdt
+ \int_Q \theta^2 |\rho \big(F(\nabla y, y)-  F(\nabla \hat y,\hat y)\big)|^2 dxdt\\
&\qquad\quad
+ \int_Q \theta^2  \big|\rho \nabla \big(K(y)-K(\hat  y)\big)\big|^2    dxdt
\Big].
\end{align}
From the choice of $\rho$ and the fact that $ \theta(t) \leq \theta(s) $ for $ t \leq s $, we get that
\begin{align}\label{inv eq2} 
\mathbb{E}\int_Q \theta^2|\rho_t (y -\hat y)|^2 dxdt 
& \leq \mathcal{C}
\theta(t_2)^2|y -\hat y|^2_{L^2_{{\mathbb{F}}}(0,T;L^2(G))}.
\end{align}
From \cref{10.5-eq1}, we have
\begin{align}\label{inv eq2xxx}\notag
& \int_Q \theta^2 \big|\rho \big(F(\nabla y, y)-  F(\nabla \hat y,\hat y)\big)\big|^2 dxdt
+ \int_Q \theta^2  \big|\rho \nabla \big(K(y)-K(\hat y)\big)\big|^2  dxdt\\
&\leq \mathcal{C}
{\mathbb{E}}\int_Q \theta^2 \big(z^2 + |\nabla z|^2\big) dxdt.
\end{align}
This, together with the inequality \cref{inv eq1} and  the fact that $\theta(t)\leq\theta(s)$  for $t\leq s$, implies that there exists a $\lambda_2\geq \lambda_1$ such that for all $ \lambda \geq \lambda_{2} $, we have
\begin{align}\label{inv eq3} \notag
&
 \lambda\theta^2(t_0)\mathbb{E}\int_{t_0}^T\int_G
|\nabla y-\nabla \hat y|^2dxdt +
s\lambda^2\theta^2(t_0)\mathbb{E}\int_{t_0}^T\int_G \varphi
|y- \hat y|^2dxdt \nonumber\\ \notag
& \leq \lambda\mathbb{E}\int_Q \theta^2|\nabla z|^2dxdt +
s\lambda^2\mathbb{E}\int_Q \theta^2\varphi
|z|^2dxdt\\
& \leq \mathcal{C}\theta^2(t_2)|y- \hat y|^2_{L^2_{{\mathbb{F}}}(0,T;L^2(G))}   \\ \notag
&\quad+
\mathcal{C}\mathbb{E}\Big( \theta^2(T)\big|\nabla y(T) - \nabla \hat y(T)\big|^2_{L^2(G)}+
s\lambda\varphi(T)\theta^2(T)\big|y(T) -\hat y(T)\big|^2_{L^2(G)} \Big).
\end{align}


Thanks to  It\^o's formula, \cref{sp-eq1,10.5-eq1-1}, we obtain that
\begin{align}\label{inv eq4} 
    & \mathbb{E}\int_G |y(t_0)-\hat y(t_0)|^2dx
    \\\notag
    & =
    \mathbb{E}\int_G |y(T) \!- \! \hat y(T)|^2dx
    - \mathbb{E} \int_{t_{0}}^{T} \int_{G} \{ 2 (y - \hat{y}) d (y - \hat{y}) + [d(y-\hat{y})]^{2} \} d x
    \\ 
    \notag
    & \leq
    \mathbb{E}\int_G |y(T) \!- \! \hat y(T)|^2dx +
    \mathcal{C}\mathbb{E}\int_{t_0}^T\int_G |\nabla y - \!\nabla \hat y |^2 dxdt
    + \mathcal{C}
    \mathbb{E}\int_{t_0}^T\int_G |y-\hat y|^2 dxdt.
\end{align}
Recalling that $\varphi\geq 1$, combining  the inequalities \cref{inv eq3,inv eq4}, for any   $\lambda\geq  \lambda_2 $ and $s\geq s_0(\lambda)$,
we have
\begin{align}\label{inv eq5} \notag
&\mathbb{E}\int_G |y(t_0)-\hat y(t_0)|^2dx \\
 &\leq \mathcal{C}
\theta^2(t_2)\theta^{-2}(t_0)\big|y-\hat y\big|^2_{L^2_{\mathbb{F}}(0,T;L^2(G))} \\ \notag
 &\quad +
\mathcal{C}\mathbb{E}\left( \theta^2(T)\big|\nabla y(T)\!- \! \nabla \hat y(T)\big|^2_{L^2(G)} +
s\lambda\varphi(T)\theta^2(T)\big|y(T)  -  \hat y(T)\big|^2_{L^2(G)} \right).
\end{align}
Now we fix $\lambda =\lambda_2 $. It follows from the
inequality \cref{inv eq5} that
\begin{align}\label{inv eq6} \notag
&    \mathbb{E}\int_G |y(t_0)-\hat y(t_0)|^2dx \\ &  \leq \mathcal{C}\big(
\theta^2(t_2)\theta^{-2}(t_0)\big|y - \hat y\big|^2_{L^2_{{\mathbb{F}}}(0,T;L^2(G))}
+ \theta^2(T)\mathbb{E}\big|y(T)-\hat y(T)\big|^2_{H^1(G)}\big).
\end{align}
Replacing $\mathcal{C}$ by $\mathcal{C}e^{s e^{\lambda_2 T}}$, from inequality \eqref{inv eq6}, for any $s \geq 0$, it holds that
\begin{align}\label{inv eq7} \notag
\mathbb{E}\int_G |y(t_0)-\hat y(t_0)|^2dx
& \leq \mathcal{C}e^{-2s(e^{\lambda_2 t_2}-e^{\lambda_2
t_0})}\big|y - \hat y\big|^2_{L^2_{{\mathbb{F}}}(0,T;L^2(G))}
\\
& \quad 
 + \mathcal{C}e^{\mathcal{C}s}\mathbb{E}\big|y(T)-\hat y(T)\big|^2_{H^1(G)}.
\end{align}
Choosing $s \geq 0$ which minimize
the right-hand side of inequality \cref{inv eq7}, we obtain that
\begin{equation}\label{inv eq8}
\mathbb{E}\big|y(t_0)-\hat y(t_0)\big|_{L^2(G)} \leq \mathcal{C}
\big|y - \hat y\big|^{1-\gamma}_{L^2_{{\mathbb{F}}}(0,T;L^2(G))}
\mathbb{E}\big|y(T)-\hat y(T)\big|^{\gamma}_{H^1(G)},
\end{equation}
with $ \gamma = \frac{2(e^{\lambda_2 t_0}-e^{\lambda_2 t_2})}{\mathcal{C}+2(e^{\lambda_2 t_0}-e^{\lambda_2 t_2})} $.
This completes the proof of  \cref{inv th1}.
\end{proof}

\subsection{Proof of the conditional stability II}

In this subsection, we will prove the second  conditional stability result, i.e., \cref{thmDou2022}.
To this end, we need to establish the following Carleman weight function, which is different from \cref{weight1-1}.

Let $ \lambda \in (0,+\infty)$ and  $t\in (0,+\infty)$. Put
\begin{equation}\label{eqDouCarleman1}
\psi = (t+1)^{\lambda},\qquad \theta = e^{\psi}.
\end{equation}

We have the following Carleman estimate for \cref{h9}:

\begin{theorem}\label{carleman est1-1}
Let $\delta \in[0,T)$, and  $\psi$ and $\theta$ be given in
\eqref{eqDouCarleman1}. There exist constants $ \mathcal{C} > 0 $ and  $\lambda_0>0$ such that for all
$\lambda \geq \lambda_0$ and  for any solution $ z $ to the equation \cref{h9}, it holds that
\begin{align}\label{car1-1}
\notag
& \mathbb{E}  \int_{\delta}^{T} \int_{G} \lambda^{2} (t+1)^{(\lambda-2)} \theta^{2} z^{2} d x d t
+ \mathbb{E}  \int_{\delta}^{T} \int_{G} \lambda  (t+1)^{-1} \theta^{2} |\nabla z|^{2} d x d t
\\ \notag
&  \leq \mathcal{C}\,\mathbb{E}\Big[
    \lambda (T+1)^{\lambda -1 } \theta^{2}(T) | z (T)|_{L^{2}(G)}^{2}
    + \theta^{2}(\delta) |  \nabla  z (\delta)|_{L^{2}(G)}^{2}
    \\
    & \quad \quad \quad
    + \int_{\delta}^{T} \int_{G} \theta^{2} ( f^{2} + g^{2} + |\nabla g|^{2} ) d x d t
\Big].
\end{align}
\end{theorem}

\begin{proof}
Let $ \ell = \psi $ and $ \Psi = \frac{1}{4} \lambda (t+1)^{-1} $. Apply \cref{thm.fIForParablic} to the equation \cref{h9} with $u=z$.

For $ j =1 , \cdots  n $, it is easy to verify that
$ \ell_{x_{j}} = 0 $ and  $  \ell_{t} = \lambda (t+1)^{\lambda-1} $.
Integrating the equality \cref{eq.fIForParablic} on $[\delta,T]\times G$ and taking mathematical expectation in both sides, we obtain that
\begin{align} \label{eqDouCarlemanPf1} \notag
&- 2 \mathbb{E}\int_\delta^T \int_G \theta\bigg\{ \sum_{j,k=1}^n
        (b^{jk}w_{x_j})_{x_k}
        + \lambda (t+1)^{-1} \Big[\frac{1}{4} +(t+1)^{\lambda}\Big]w
        \bigg\}
        \\
        &\quad \quad \quad \quad \quad \quad \quad
        \times \Big[d z  -  \sum_{j,k=1}^n
        (b^{jk} z_{x_j})_{x_k}dt  \Big]dx   \\ \notag
&= \! -  2  \mathbb{E} \int_\delta^T  \int_G \sum_{j,k=1}^n
        \Big[ b^{jk}w_{x_j} d w  +  \frac{1}{4}\lambda (t+1)^{-1} b^{j k}w_{x_j}
        w d t \Big]_{x_k} dx
\\ \notag
& \quad
+ 2    \mathbb{E} \int_\delta^T  \int_G    \sum_{j,k=1}^n
\Big[\frac{1}{4} \lambda (t+1)^{-1} b^{j k} - \frac{b^{j k}_{t}}{2}  \Big] w_{x_{j}} w_{x_{k}} d x d t
\\ \notag
& \quad
+ 2    \mathbb{E} \int_\delta^T  \int_G   \Big[
    - \frac{1}{8} \lambda^{2} (t+1)^{-2}
    - \frac{1}{2} \lambda^{2} (t+1)^{\lambda - 2}
    - \frac{1}{4} \lambda (t+1)^{-2}
    \\ \notag
    & \quad \quad \quad \quad \quad \quad \quad
    + \lambda (\lambda - 1) (t+1)^{\lambda - 2}
\Big]  w^{2} d x d t
\\ \notag
& \quad
+   \mathbb{E}  \int_\delta^T  \int_G  d\bigg\{
\sum_{j,k=1}^n b^{jk}w_{x_j} w_{x_k}
- \lambda (t+1)^{-1} \Big[\frac{1}{4} +(t+1)^{\lambda}\Big]w^{2} \bigg\}dx\nonumber
\\  \notag
& \quad
+ 2  \mathbb{E} \int_\delta^T  \int_G \bigg\{
    \sum_{j,k=1}^n (b^{j k}w_{x_j})_{x_k}
+ \lambda (t+1)^{-1} \Big[ \frac{1}{4} +(t+1)^{\lambda}\Big]w
\bigg\}^{2} d x d t
\\ \notag
& \quad
+   \mathbb{E} \int_\delta^T  \int_G \theta^{2} \bigg[
    - \sum_{j,k=1}^n b^{jk} d z_{x_j} d z_{x_k}
    + \lambda (t+1)^{-1} \Big[\frac{1}{4} +(t+1)^{\lambda}\Big] (d z)^{2}
\bigg] d x.
\end{align}

Now we handle the terms in \cref{eqDouCarlemanPf1}.
Since $ z = w = 0 $ on $ \Sigma $, by Divergence Theorem, we have
\begin{align}\label{eqDouCarlemanPf2}
-  2  \mathbb{E} \int_\delta^T  \int_G \sum_{j,k=1}^n
\Big[ b^{jk}w_{x_j} d w  +  \frac{1}{4}\lambda (t+1)^{-1} b^{j k}w_{x_j}
w d t \Big]_{x_k} dx = 0.
\end{align}
Noting \cref{eq.bijGeq}, we get that
\begin{align}\label{eqDouCarlemanPf3} \notag
& 2    \mathbb{E} \int_\delta^T  \int_G    \sum_{j,k=1}^n
\Big[ \frac{1}{4} \lambda (t+1)^{-1} b^{j k} - \frac{b^{j k}_{t}}{2}  \Big] w_{x_{j}} w_{x_{k}} d x d t
\\
&
\geq
\frac{s_{0}}{2}  \mathbb{E} \int_\delta^T  \int_G  \lambda (t+1)^{-1} |\nabla w|^{2} d x d t
- \mathcal{C} \mathbb{E} \int_\delta^T  \int_G  |\nabla w|^{2} d x d t
.
\end{align}
Since $ \lambda > 0  $, for  $ t \in (\delta, T] $, we have that $ (t+1)^{-2} \leq (t+1)^{\lambda - 2} $ and that
\begin{align}\label{eqDouCarlemanPf4} \notag
&  2    \mathbb{E} \int_\delta^T  \int_G   \bigg[
    - \frac{1}{8} \lambda^{2} (t+1)^{-2}
    - \frac{1}{2} \lambda^{2} (t+1)^{\lambda - 2}
    - \frac{1}{4} \lambda (t+1)^{-2}
    \\ \notag
    & \quad \quad \quad \quad  \quad 
    + \lambda (\lambda - 1) (t+1)^{\lambda - 2}
\bigg]  w^{2} d x d t
\\
& \geq   \mathbb{E} \int_\delta^T  \int_G \bigg[
    \frac{1}{3} \lambda^{2} (t+1)^{\lambda - 2}
    + O(\lambda)
\bigg] w ^{2} d x d t
.
\end{align}
Thanks to \cref{eq.bijGeq} and $ \lambda > 0 $, we obtain that
\begin{align}\label{eqDouCarlemanPf5} \notag
& \mathbb{E}  \int_\delta^T  \int_G  d\bigg\{
\sum_{j,k=1}^n b^{jk}w_{x_j} w_{x_k}
- \lambda (t+1)^{-1} \Big[\frac{1}{4} +(t+1)^{\lambda}\Big] w^{2} \bigg\} dx
\\
& \geq
s_{0} \mathbb{E} \int_{G} |\nabla w(T)|^{2} d x
- \frac{5}{4}  \mathbb{E} \int_{G}  \lambda (T+1)^{\lambda} |  w(T)|^{2} d x
- \mathcal{C}  \mathbb{E} \int_{G}   |\nabla w(\delta)|^{2} d x
.
\end{align}
Noting \cref{h9}, we have that
\begin{align} \label{eqDouCarlemanPf6}
& \mathbb{E} \int_\delta^T  \int_G \theta^{2} \bigg(
    - \sum_{j,k=1}^n b^{jk} d z_{x_j} d z_{x_k}
\bigg) d x
\geq - \mathcal{C} \mathbb{E} \int_\delta^T  \int_G \theta^{2} |\nabla   g  |^{2} d x d t
.
\end{align}
Choosing $ \lambda_{1} \geq \mathcal{C} (T + 1) $, for $ \lambda \geq \lambda_{1} $, it holds that  $ \mathcal{C} \lambda (t+1)^{\lambda - 1} \leq \lambda^{2} (t+1)^{\lambda-2} $. Hence,
\begin{align} \label{eqDouCarlemanPf7} \notag
& \mathbb{E} \int_\delta^T  \int_G \theta^{2}
\lambda (t+1)^{-1} \Big[\frac{1}{4} +(t+1)^{\lambda}\Big] (d z)^{2}
d x
\\
& \geq
- \frac{1}{4} \mathbb{E} \int_\delta^T  \int_G \theta^{2} \lambda^{2} (t+1)^{\lambda-2} z^{2} d x d t
- \mathcal{C} \mathbb{E} \int_\delta^T  \int_G \theta^{2} \lambda(t+1)^{\lambda-1} g^{2} d x d t
.
\end{align}

As for the left hand side of \cref{eqDouCarlemanPf1}, from \cref{h9}, we have
\begin{align} \label{eqDouCarlemanPf8} \notag
& - 2 \mathbb{E}\int_\delta^T \int_G \theta\bigg\{ \sum_{j,k=1}^n
    (b^{jk}w_{x_j})_{x_k}
    + \lambda (t+1)^{-1} \Big[\frac{1}{4} +(t+1)^{\lambda}\Big]w
    \bigg\}
\\ \notag
&\quad \quad \quad \quad \quad \quad \quad
\times \Big[  d z  -    \sum_{j,k=1}^n
(b^{jk} z_{x_j})_{x_k}dt  \Big]dx
\\ \notag
& \leq \mathbb{E} \int_\delta^T  \int_G \bigg\{
\sum_{j,k=1}^n (b^{j k}w_{x_j})_{x_k}
+ \lambda (t+1)^{-1} \Big[\frac{1}{4} +(t+1)^{\lambda}\Big]w
\bigg\}^{2} d x d t
\\
& \quad
+ \mathcal{C} \mathbb{E} \int_\delta^T  \int_G  \theta^{2}   f^{2} d x d t.
\end{align}

Combining \cref{eqDouCarlemanPf1,eqDouCarlemanPf2,eqDouCarlemanPf3,eqDouCarlemanPf4,eqDouCarlemanPf5,eqDouCarlemanPf6,eqDouCarlemanPf7,eqDouCarlemanPf8}, and noting that $ w = \theta z $ and $ \nabla w = \theta \nabla z $, we arrive that
\begin{align}\label{eqDouCarlemanPf9} \notag
& \mathcal{C} \mathbb{E} \int_\delta^T  \int_G  \theta^{2} \big(f^{2} + |\nabla g|^{2} + \lambda (t+1)^{\lambda-1} |g|^{2} \big) d x d t
\\ \notag
& \geq
\mathbb{E} \int_\delta^T  \int_G  \theta^{2} \big[
 \lambda (t+1)^{-1} |\nabla z|^{2}
 + \lambda^{2} (t+1)^{\lambda-2} z^{2}
\big]  d x d t
\\ \notag
& \quad
- \mathcal{C}   \mathbb{E}\big[
\lambda (T+1)^{\lambda} \theta^{2}(T) |  z(T)|_{L^{2}(G)}^{2}
+    \theta^{2}(\delta)|\nabla z(\delta)|_{L^{2}(G)}^{2}
\big]
\\
& \quad
- \mathcal{C} \mathbb{E} \int_\delta^T  \int_G  \theta^{2} \big(|\nabla z|^{2} + \lambda |z|^{2} \big) d x d t.
\end{align}
From \cref{eqDouCarlemanPf9}, we know there exists   $ \lambda_{2} \geq \max\{\mathcal{C}, \lambda_{1}\} $, such that for all $ \lambda \geq \lambda_{2} $, it holds that
\begin{align*} \notag
& \mathcal{C} \mathbb{E} \int_\delta^T  \int_G  \theta^{2} \big[f^{2} + |\nabla g|^{2} + \lambda (t+1)^{\lambda-1} |g|^{2} \big] d x d t
\\ \notag
& \geq
\mathbb{E} \int_\delta^T  \int_G  \theta^{2} \Big[
 \lambda (t+1)^{-1} |\nabla z|^{2}
 + \lambda^{2} (t+1)^{\lambda-2} z^{2}
\Big]  d x d t
\\ \notag
& \quad
- \mathcal{C}   \mathbb{E} \big[
\lambda (T+1)^{\lambda} \theta^{2}(T) |  z(T)|_{L^{2}(G)}^{2}
+    \theta^{2}(\delta)|\nabla z(\delta)|_{L^{2}(G)}^{2}
\big],
\end{align*}
which completes the proof.
\end{proof}

Now, we are in a position to prove \cref{thmDou2022}.

\begin{proof}[Proof of \cref{thmDou2022}]

Let $ v(x,t) = y(x,t) - \hat{y}(x, t) $. From \cref{sp-eq1}, we have
\begin{align*}
\left\{
\begin{aligned}
& d v \!-\! \sum_{j,k=1}^n(b^{jk} v_{x_j})_{x_k}dt=
( F(\nabla y , y) \!-\! F(\nabla \hat{y}, \hat{y}) ) dt 
+ (K(y)\!- \!K (\hat{y}))   d W(t) \!\!\!\!\!
&  \mbox{ in }Q,\\
& v=0&\mbox{ on }\Sigma .
\end{aligned}
\right.
\end{align*}
From \cref{carleman est1}, \cref{eqDouCarleman1,10.5-eq1}, for $ \lambda \geq \lambda_{0} $, we obtain that
\begin{align}\notag \label{eqDouStablePf1}
& \mathcal{C}\,\mathbb{E}\Big[
\lambda (T+1)^{\lambda -1 } \theta^{2}(T) | v (T)|_{L^{2}(G)}^{2}
+  \theta^{2}(0) |\nabla v (0)|_{L^{2}(G)}^{2}
\Big]
\\
& \geq
e^{2(t_{0}+1)^{\lambda}} \mathbb{E}  \int_{t_{0}}^{T} \int_{G} ( v^{2} + |\nabla v|^{2})  d x d t
,
\end{align}
where we use
\begin{align*}
& \int_Q \theta^2 | F(\nabla y, y)-  F(\nabla \hat y,\hat y) |^2 dxdt
+ \int_Q \theta^2  \big| \nabla \big(K(y)-K(\hat y)\big)\big|^2  dxdt\\
& \quad
+ \int_Q \theta^2  \big|  K(y)-K(\hat y) \big|^2  dxdt
\\
&\leq \mathcal{C}
{\mathbb{E}}\int_Q \theta^2 \big[  v^2 + |\nabla v|^2\big] dxdt
.
\end{align*}
It is clear that there exists $ \lambda_{1} \geq \lambda_{0}  $ such that for all $ \lambda \geq \lambda_{1} $, it holds that
\begin{align} \label{eqDouStablePf2}
\lambda (T+1)^{\lambda - 1} e^{2 (T+1)^{\lambda} - 2 (t_{0}+1)^{\lambda}} \leq e^{3 (T+1)^{\lambda}}.
\end{align}
Combining \cref{eqDouStablePf1,eqDouStablePf2}, for all $ \lambda \geq \lambda_{1} $, we get that
\begin{align} \label{eqDouStablePf3}
|v|_{L^{2}_{\mathbb{F}}(t_{0}, T; H^{1}(G))}^{2}
\leq
C_{1} e^{3(T+1)^{\lambda}} \mathbb{E} |v(T)|^{2}_{L^{2}(G)}
+ C_{2} e^{- 2(t_{0}+1)^{\lambda}} \mathbb{E} | \nabla v(0)|^{2}_{L^{2}(G)}
.
\end{align}

Letting $ s = (T+1)^{\lambda} $, we have $ (t_{0}+1)^{\lambda} = s ^{\gamma} $ for $ \gamma = \frac{\ln(t_{0}+1)}{\ln(T +1 )} $.
It follows from \cref{eqDouStablePf3} that there exists $ s_{1} > 0 $ such that for all $ s \geq s_{1} $,
\begin{align}
\label{eqDouStablePf4}
|v|_{L^{2}_{\mathbb{F}}(t_{0}, T; H^{1}(G))}^{2}
\leq
C_{1} e^{3 s } \mathbb{E} |v(T)|^{2}_{L^{2}(G)}
+ C_{2} e^{- 2 s^{\gamma}} \mathbb{E} | \nabla v(0)|^{2}_{L^{2}(G)}
.
\end{align}
Choosing $ s_{2} = \frac{1}{6} \ln (\mathbb{E}|v(T)|^{2}_{L^{2}(G)}) $,
then it holds that
\begin{align*}
e^{3s} \mathbb{E}|v(T)|^{2}_{L^{2}(G)}
\leq   |v(T)|_{L^{2}_{\mathcal{F}_{T}}(\Omega; L^{2}(G))}
.
\end{align*}
This, together  with \cref{eqDouStablePf4,thmDou2022beta}, implies that
\begin{align}
\label{eqEnergyPar0}
|v|_{L^{2}_{\mathbb{F}}(t_{0}, T; H^{1}(G))}^{2}
\leq
\beta( |v(T)|_{L^{2}_{\mathcal{F}_{T}}(\Omega; L^{2}(G))}  ).
\end{align}

For each $ t_{1} \in (0,T] $ such that  $ t_{0} < t_{1} $,
since
\begin{align}
    \label{eqEnergyPar1}
    \mathbb{E} | v(t_{1}) |^{2}_{H^{1}(G)} \leq \mathcal{C} \mathbb{E} |v(t)|^{2}_{H^{1}(G)}, \quad \quad \forall ~ t_{0} \leq t \leq t_{1},
\end{align}
integrating \cref{eqEnergyPar1} on $ [t_{0}, t_{1}] $, we have
\begin{align*}
    (t_{1}- t_{0}) \mathbb{E} | v(t_{1})|^{2}_{H^{1}(G)}
    \leq
    \mathcal{C} \mathbb{E} \int_{t_{0}}^{t_{1}} | v(t)|^{2}_{H^{1}(G)} d t
    .
\end{align*}
This, together with \cref{eqEnergyPar0}, yields
\begin{align*}
    | v(t_{1})|^{2}_{L^{2}_{\mathcal{F}_{t_{1}}}(\Omega;H^{1}(G))}
    \leq
    \mathcal{C}(t_{1}- t_{0})^{-1}
    \beta(|v(T)|_{L^{2}_{\mathcal{F}_{T}}(\Omega; L^{2}(G))})
    ,
\end{align*}
which completes the proof.
\end{proof}

\section{Solution to inverse state problem with a boundary measurement}
\label{secIllPosedCauchy}

In this section, we study \cref{probP9}.

Consider the following stochastic parabolic equation:
\begin{align} \label{eqInversPParabolicPartialBoudaryeqNew}
\left\{\begin{aligned}
    &d y -\sum_{j,k=1}^n (b^{j k}  y_{x_j})_{x_k} dt  =\left(a_1 \cdot \nabla y+a_2 y + f  \right) d t+ a_{3} y  d W(t) & & \text { in } Q, \\
& y = g_{1} & & \hspace{-1.6cm} \text { on } (0,T) \times \Gamma_{0} , \\
& \frac{\partial y}{ \partial \nu} = g_{2} & & \hspace{-1.6cm} \text { on } (0,T) \times \Gamma_{0} ,
\end{aligned}\right.
\end{align}
where  $ a_{1}, a_{2}, a_{3} $ satisfy \cref{eqAcases} and $ f \in  L^{2}_{\mathbb{F}}(0,T; L^{2}(G))$.

For the solution of \cref{eqInversPParabolicPartialBoudaryeqNew}, we have the following conditional stability.

\begin{theorem} \label{thmProb9}
For any given $\widehat{G} \subset \subset G$ and $\varepsilon \in (0, \frac{T}{2}) $, there exists a constant $\mathcal{C}>0$ such that for some generic constant $ \gamma \in (0,1) $,
letting
\begin{align}
    \notag
    \label{eqK}
    \mathcal{K}(f, g_{1}, g_{2})
    & =
        |f|_{ L^{2}_{\mathbb{F}}(0,T; L^{2}(G)) }^{2} +
        |g_{1}|_{  L^{2}_{\mathbb{F}}(0,T; H^{1}(\Gamma_{0}))}^{2} +
        |g_{1, t}|_{  L^{2}_{\mathbb{F}}(0,T; L^{2}(\Gamma_{0}))}^{2}
    \\
    & \quad
        + \left|g_{2}\right|_{L_{\mathbb{F}}^2\left(0, T ; L^2(\Gamma_{0})\right)}^{2}
        ,
\end{align}
it holds that
\begin{align} \label{eqThmProb9}
    \mathbb{E} \int_{\varepsilon}^{T-\varepsilon} \int_{\widehat{G}}\left(y^2+|\nabla y|^2\right) d x d t \leq \mathcal{C}
        |y|_{L_{\mathbb{F}}^2\left((0, T) ; H^1(G)\right)}^{2 (1- \gamma)}
        (\mathcal{K}(f, g_{1}, g_{2}))^{\gamma}
    ,
\end{align}
where $ y(\cdot) $ is the solution to  \cref{eqInversPParabolicPartialBoudaryeqNew}.
\end{theorem}


\begin{proof}
{\it Step 1}. In this step, we  introduce a cut-off function to convert the   equation \cref{system-sp1} to an equation with boundary condition.

Choose a bounded domain $ G_{2} \subset \mathbb{R}^{n} $ such that $ \partial G_{2} \cap \overline{G} = \Gamma_{0} $ and that $ \widetilde{G} = G_{2} \cup \Gamma_{0} \cup G $ enjoys a $ C^{4} $ boundary $ \partial \widetilde{G} $.
Then we obtain
\begin{align*}
G \subset \widetilde{G}, \quad
\overline{\partial G \cap \widetilde{G}} \subset \Gamma_{0}, \quad
\text{and} ~ \partial G \setminus \Gamma_{0} \subset \partial \widetilde{G}
.
\end{align*}

Let $ G_{0} \subset \! \subset \widetilde{G} \setminus G $ be an open subset. Thanks to \cref{hl1}, there exists a function $\psi\in C^4(G)$ such that
\begin{align}\label{eqThmProb9e1}
\psi > 0   \text{ in } \widetilde{G}, \quad 
 \psi = 0   \text{ on } \partial \widetilde{G}, \text{ and } 
 |\nabla \psi | > 0   \text{ in } \widetilde{G}\backslash G_0.
\end{align}
Let $ \rho = \frac{\varepsilon}{\sqrt{2}}  $. Choose $ \iota, t_{0}, N, c, \phi, \alpha $ and $\beta_{k},  Q_{k} $ for $ k=1,2,3,4 $ to be the same as those in the proof of \cref{main result}.

Next, we claim that
\begin{align} \label{eqThmProb9e2}
\partial Q_{1}  = \Sigma_{0} \cup \Sigma_{1},
\end{align}
where
$ \Sigma_{0} \subset [ 0,T] \times \Gamma_{0} $ and $ \Sigma_{1}  = \{ (t,x) \in [0,T] \times G \mid \alpha(t,x) = \beta_{1} \}  $.

In fact, for any $ (t, x) \in \partial Q_{1} $, thanks to \cref{eqQkDef},  we know $ x \in \overline{G} $ and $ \alpha(t, x) \geq \beta_{1} $. If $ x \in G $, we have $ \alpha(t,x) = \beta_{1} $. If $ x \in \partial G $, it must hold $ x \in \Gamma_{0} $.
Indeed, supposing that $ x \in \partial G \setminus \Gamma_{0} $, from $ \partial G \setminus \Gamma_{0} \subset \partial \widetilde{G} $ and \cref{eqThmProb9e1}, we get $ \psi (x) = 0 $. On the other hand, since $ \alpha(t,x) \geq \beta_{1} $, we obtain
\begin{align*}
\psi(x)-c\left(t-t_0\right)^2=-c\left(t-t_0\right)^2 \geq \frac{1}{N}|\psi|_{L^{\infty}(\widetilde{G})}-\frac{c \rho^2}{N} ,
\end{align*}
which implies
\begin{align*}
0 \leq c\left(t-t_0\right)^2 \leq \frac{1}{N}\big(c \rho^2-|\psi|_{L^{\infty}(\widetilde{G})}\big),
\end{align*}
contrary to \cref{bound psi}. Hence, \cref{eqThmProb9e2} is proved.

Let $\eta\in C_0^{\infty}(Q_2)$ be such that $ 0 \leq \eta  \leq 1 $ and $\eta = 1$  in   $Q_3$.
Let $z = \eta  y$, then
$z$ solves
\begin{equation}\label{system2}
\begin{cases}
\begin{aligned}
    &dz - \sum\limits^{n}_{j,k=1}(b^{jk}z_{x_{j}})_{x_{k}} dt = \big( \langle a_1, \nabla z \rangle + a_2 z + \tilde{f} + \eta f\big)dt + a_3 z dW(t) \!\!\!\!\!  &&\text{ in } Q_1,\\
&z=\dfrac{\partial z}{\partial\nu}=0 &&\text{ on } \Sigma_1.
\end{aligned}
\end{cases}
\end{equation}
Here $ \tilde{f} = -\sum^{n}\limits_{j,k=1} (b^{jk}_{x_{j}}\eta_{x_{k}} y + b^{jk}\eta_{x_{j}x_{k}}y + 2b^{j k}\eta_{x_{j}} y_{x_{k}}) - \langle a_1,\nabla \eta\rangle y + \eta_t y  $ and $\tilde{f}$ is supported in
$Q_2\setminus Q_3$.

\medskip

{\it Step 2}. In this step, we apply \cref{thm.fIForParablic} to the equation \cref{system2} to get the desired estimate \cref{eqThmProb9}.

Let us choose  $u=z$ in \cref{thm.fIForParablic}. Integrate \cref{eq.fIForParablic} over $ G \times (0,T) $. Noting that $ z $ is supported in $ Q_{1} $, choosing $ \theta = e^{\ell} $ and $ \ell = \lambda \alpha $, setting $ w = \theta z $,  after taking mathematical expectation in both sides, we have  \cref{eqP8eq1}.
The rest of the proof is similar to that of \cref{main result} with some minor differences, which are stated below.

It is clear that $\overline{Q_{2}} \cap \overline{\Sigma}_{1} = \emptyset$. Hence, $\eta = 0$ on $\Sigma_{1}$.
Recall that $ y = g_{1} $ on $ \Sigma_{0} $ and $ g_{1} \in L^{2}_{\mathbb{F}}(\Omega; H^{1}(0,T; L^{2}(\Gamma_{0}))) \cap L^{2}_{\mathbb{F}}(0,T; H^{1}(\Gamma_{0})) $.
Then $ z = 0 $ on $ \partial Q_{1} $, and it holds that
\begin{align*}
& \mathbb{E} \int_{Q_1} \sum_{i, j=1}^n\left(b^{i j} w_{x_{i}} d w\right)_{x_{j}} d x
= \mathbb{E} \int_{\Sigma_{0}} \sum_{j,k=1}^{n} b^{jk} w_{x_{j}} \nu^{k} d w d \Gamma
\\
& \leq
\mathcal{C} \lambda  \mu \mathbb{E} \int_{0}^{T} \int_{\Gamma_{0}} \alpha  \theta^{2} (|g_{2}|^{2} + |\nabla_{\Gamma} g_{1}|^{2} + |g_{1,t}|^{2}+ \lambda^{2} \mu^{2} \alpha^{2} |g_{1}|^{2}) d\Gamma d t
,
\end{align*}
where $ \nabla_{\Gamma} $ denotes the tangential gradient on $ \Gamma_{0} $.

Noting that $ z = \frac{\partial z}{ \partial \nu} = 0$ on $ \Sigma_{1} $, we find
\begin{align*}
& \mathbb{E} \int_{Q_{1}}  \sum_{j, k=1}^n\Big[\sum_{j^{\prime}, k^{\prime}=1}^n\big(2 b^{j k} b^{j^{\prime} k^{\prime}} \ell_{x_{j^{\prime}}} w_{x_j} w_{x_{k^{\prime}}}-b^{j k} b^{j^{\prime} k^{\prime}} \ell_{x_j} w_{x_{j^{\prime}}} w_{x_{k^{\prime}}}\big)
\\\notag
&       ~ \quad \quad \quad  ~
    \quad+\Psi b^{j k} w_{x_j} w-b^{j k}\Big(\mathcal{A} \ell_{x_j}+\frac{\Psi_{x_j}}{2}\Big) w^2\Big]_{x_k} d xd t
\\
&
\leq \mathcal{C} \lambda \mu \mathbb{E} \int_{0}^{T} \int_{\Gamma_{0}} \alpha \theta^{2} (|g_{2}|^{2} + |\nabla_{\Gamma} g_{1}|^{2} + \lambda^{2} \mu^{2} \alpha^{2} |g_{1}|^{2}) d \Gamma d t.
\end{align*}

Similar to \cref{eqP8eq15}, there exists a $ \mu_{1} > 0 $, such that for all $ \mu > \mu_{1} $, there is a $ \lambda_{1}(\mu) > 0 $, such that for all $ \lambda > \lambda_{1} $, it holds that
\begin{align} \label{eqThmProb9e3} \notag
&   {\mathbb{E}} \int_{ t_0 - \frac \iota{\sqrt N}}^{{ t_0 + \frac
\iota{\sqrt
N}}}\int_{\widehat{G}}( y^2 + |\nabla  y|^2)dxdt
\\ \notag
& \leq  {\cal C} e^{{\cal C}\lambda}
\big (
|f|_{ L^{2}_{\mathbb{F}}(0,T; L^{2}(G)) }^{2} +
|g_{1}|_{  L^{2}_{\mathbb{F}}(0,T; H^{1}(\Gamma_{0}))}^{2} +
|g_{1, t}|_{  L^{2}_{\mathbb{F}}(0,T; L^{2}(\Gamma_{0}))}^{2}
+
\left|g_{2}\right|_{L_{\mathbb{F}}^2\left(0, T ; L^2(\Gamma_{0})\right)}^{2}
\big)
\\
& \quad
 + {\cal C}  e^{-2\lambda(\beta_4 - \beta_3)} {\mathbb{E}}
\int_{Q_1}( y^2 + |\nabla  y|^2)dxdt.
\end{align}
Let $m\in{\mathbb{N}}$ such that
\begin{equation*}
\sqrt 2 \;\iota +\frac{m \iota}{\sqrt N} \leq T - \sqrt 2
\iota \leq \sqrt 2 \; \iota + \frac{(m+1)\iota}{\sqrt N}.
\end{equation*}
Recalling that $t_0$ can be any element in $[\sqrt 2 \iota, T - \sqrt 2 \iota]$, we take
\begin{equation*}
t_0 = \sqrt 2 \iota + \frac{k \iota}{\sqrt N}, \quad k =0, 1,
2, \cdots, m.
\end{equation*}
Then, it follows from \cref{eqThmProb9e3} that
\begin{align}
    \notag
    \label{eqISPpar}
&   {\mathbb{E}} \int_{\varepsilon}^{T-\varepsilon}\int_{\widehat{G} }( y^2 + |\nabla  y|^2)dxdt\\ \notag
& \leq  {\cal C} e^{{\cal C}\lambda}
\big (
|f|_{ L^{2}_{\mathbb{F}}(0,T; L^{2}(G)) }^{2} +
|g_{1}|_{  L^{2}_{\mathbb{F}}(0,T; H^{1}(\Gamma_{0}))}^{2} +
|g_{1, t}|_{  L^{2}_{\mathbb{F}}(0,T; L^{2}(\Gamma_{0}))}^{2}
+
\left|g_{2}\right|_{L_{\mathbb{F}}^2\left(0, T ; L^2(\Gamma_{0})\right)}^{2}
\big)
\\
& \quad
+ {\cal C}  e^{-2\lambda(\beta_4 - \beta_3)} {\mathbb{E}}
\int_{Q_1}( y^2 + |\nabla  y|^2)dxdt.
\end{align}
Choosing $\lambda \geq 0$ which minimize
the right-hand side of inequality \cref{eqISPpar}, from \cref{eqK}, we obtain that
\begin{equation*}
    \mathbb{E} \int_{\varepsilon}^{T-\varepsilon} \int_{\widehat{G}}\left(y^2+|\nabla y|^2\right) d x d t \leq \mathcal{C}
        |y|_{L_{\mathbb{F}}^2\left((0, T) ; H^1(G)\right)}^{2 (1- \gamma)}
        (\mathcal{K}(f, g_{1}, g_{2}))^{\gamma}
    ,
\end{equation*}
with
$ \gamma = \frac{2(\beta_{4}-\beta_{3})}{\mathcal{C}+2(\beta_{4}-\beta_{3})} $.
\end{proof}

\section{Solution to the inverse source problem with a terminal measurement}
\label{secISPTer}

In this section, we give a positive answer to \cref{probISPCS} by establishing the following result for the conditional stability of \cref{probISPCS}:

\begin{theorem}
\label{thmISPC}
Let $ g(t,x) = g_{1}(t) g_{2}(x) $, where $ g_{2} \in H^{1}(G) $ is nonzero.
For any $t_0\in (0,T]$, there exists a positive constant $\mathcal{C}=\mathcal{C}(t_0)$ such that for some generic constant $\gamma\in (0,1)$ (which is independent of $t_0$), it holds that
\begin{equation}\label{eqISPCSeq1}
| y(t_0) |_{L^2_{\mathcal{F}_{t_0}}(\Omega;L^{2}(G))}
+ | g_{1} |_{L^{2}_{\mathbb{F}}(t_{0},T)}
\leq
\mathcal{C}(g_2) |y|_{L^{2}_{\mathbb{F}}(0,T; L^{2}(G))}^{1-\gamma}
| y(T) |_{L^2_{\mathcal{F}_{T}}(\Omega;H^{1}(G))}^{\gamma},
\end{equation}
where $y(\cdot)$   is the solution to \cref{eqISPTerEq} and $\mathcal{C}(g_2)$ is a constant depending on $|\nabla g_2|_{L^{2}(G)}$.
\end{theorem}

\begin{proof}

Choose $t_1$ and $t_2$ such that $0< t_1 < t_2 < t_0$.
Set $\rho\in C^\infty(\mathbb{R})$ such that $0\leq\rho\leq 1$ and
\begin{equation}\label{eqISPCSe1}
\rho = \left\{
\begin{array}{ll}
    1, & \mbox{ if } t\geq t_2,\\
    0, &  \mbox{ if } t \leq t_1.
\end{array}
\right.
\end{equation}

Let $z=\rho y$. From \cref{eqISPTerEq}, we know that $z$ solves
\begin{equation}\label{eqISPCSe2}
\left\{
\begin{aligned}
        &dz -\! \sum\limits_{j,k=1}^n\! (b^{jk}z_{x_j})_{x_k} dt=\! (
        a_{1} \!\cdot\! \nabla z \!+ a_{2} z
         \!+\rho_t y )dt
    +( a_{3} z\!+\!\rho g) dW(t) \!\!\!\!\!\!& \mbox{ in } Q,\\
        &z=0 &\mbox{ on } \Sigma,\\
        &z(0) = 0 &\mbox{  in  }G.
\end{aligned}
\right.
\end{equation}
Applying \cref{carleman est1} with $\psi=t$ and $\delta=0$ to the
equation \cref{eqISPCSe2}, for  $\lambda\geq\lambda_1$ and $s\geq s_0(\lambda)$,
we have
\begin{align}\label{eqISPCSe3}\notag
& \lambda\mathbb{E}\int_Q \theta^2|\nabla z|^2dxdt + s\lambda^2\mathbb{E}\int_Q
\theta^2\varphi
|z|^2dxdt
+ s \lambda^{2} \mathbb{E} \int_{Q} \theta^{2} \varphi \rho^{2} g^{2}dxdt   \\\notag
&\leq \mathcal{C}\mathbb{E}\Big\{\theta^2(T)\big|\nabla  z(T)\big|^2_{L^2(G)}
+ s\lambda\varphi(T)\theta^2(T)\big|z(T)\big|^2_{L^2(G)}
\\
&\qquad\quad
+ \mathbb{E} \int_{Q} \theta^{2} |\rho_{t} y|^{2} dx d t
+ \mathbb{E} \int_{Q} \theta^{2} \rho^{2} | \nabla   g|^{2} dx d t
\Big\}
\end{align}
From the choice of $\rho$ and the fact that $ \theta(t) \leq \theta(s) $ for $ t \leq s $, we get that
\begin{align}\label{eqISPCSe4}
\mathbb{E}\int_Q \theta^2|\rho_t y|^2 dxdt
\leq \mathcal{C}
\int_{t_1}^{t_2}\int_G \theta^2 |y |^2  dxdt
\leq \mathcal{C}
\theta^2(t_2)|y  |^2_{L^2_{{\mathbb{F}}}(0,T;L^2(G))}
.
\end{align}
Since $ g = g_{1}(t) g_{2}(x) $, combining \cref{eqISPCSe4,eqISPCSe3}, for  $\lambda\geq\lambda_1$ and $s\geq s_0(\lambda)$, we get
\begin{align*}
& \lambda\mathbb{E}\int_Q \theta^2|\nabla z|^2dxdt + s\lambda^2\mathbb{E}\int_Q
\theta^2\varphi
|z|^2dxdt
+ s \lambda^{2} \mathbb{E} \int_{0}^{T} \theta^{2} \varphi \rho^{2} g_{1}^{2}dt  \int_{G} |g_{2}|^{2} d x   \\\notag
&\leq \mathcal{C}\Big\{ \theta^2(T) \mathbb{E} \big |\nabla  z(T)\big|^2_{L^2(G)}
+ s\lambda\varphi(T)\theta^2(T) \mathbb{E}\big|z(T)\big|^2_{L^2(G)}
\\
&\quad\quad
+ \theta^2(t_2)|y  |^2_{L^2_{{\mathbb{F}}}(0,T;L^2(G))}
+ \mathbb{E} \int_{0}^{T} \theta^{2} \rho^{2} g_{1}^{2} d t \int_{G} | \nabla   g_{2}|^{2} dx
\Big\}
\end{align*}
This, together with $ g_{2} \in H^{1}(G)$ is nonzero  and the fact that $\theta(t)\leq\theta(s)$  for $t\leq s$, implies that there exists a
$$\lambda_2\geq \max\bigg\{\lambda_1, \mathcal{C}\frac{| \nabla   g_{2}|_{L^{2}(G)}^{2}}{|g_{2}|_{L^{2}(G)}^{2}}\bigg\}$$
such that for all $ \lambda \geq \lambda_{2} $, we have
\begin{align}\label{eqISPCSe5} 
&
\lambda\theta^2(t_0)\mathbb{E}\int_{t_0}^T\int_G
(|\nabla y |^2
+ s \lambda \varphi y^{2} ) d x d t
+ s\lambda^2\theta^2(t_0)\mathbb{E}\int_{t_0}^T \varphi
|g_{1} |^2 dt
\\ \notag
& \leq
\mathcal{C}\mathbb{E}\Big( \theta^2(T)\big|\nabla y(T) \big|^2_{L^2(G)}+
s\lambda\varphi(T)\theta^2(T)\big|y(T)  \big|^2_{L^2(G)} \Big)
+
\mathcal{C}\theta^2(t_2)|y |^2_{L^2_{{\mathbb{F}}}(0,T;L^2(G))}
.
\end{align}

Thanks to  It\^o's formula, we obtain that
\begin{align*}
    & \mathbb{E}\int_G |y(t_0) |^2dx
    =
    \mathbb{E}\int_G |y(T) |^2dx
    - \mathbb{E} \int_{t_{0}}^{T} \int_{G} ( 2  y  d y   + |d y|^{2} ) d x.
\end{align*}
This, together with \cref{system2bu}, implies that
\begin{align}\label{eqISPCSe6} 
& \mathbb{E}\int_G |y(t_0) )|^2dx   \leq \mathbb{E}\int_G |y(T)  |^2dx +
C\mathbb{E}\int_{t_0}^T\int_G |\nabla y  |^2 dxdt
+ \mathcal{C}
\mathbb{E}\int_{t_0}^T\int_G |y |^2 dxdt.
\end{align}
Recalling that $\varphi\geq 1$, combining  the inequalities \cref{eqISPCSe5,eqISPCSe6}, for any   $\lambda\geq  \lambda_2 $ and $s\geq s_0(\lambda)$,
we have
\begin{align}\label{eqISPCSe7} \notag
&\mathbb{E}\int_G |y(t_0) |^2dx
+ \mathbb{E} \int_{t_{0}}^{T} |g_{1}|^{2} d t
\\\notag
    &\leq \mathcal{C}
\theta^2(t_2)\theta^{-2}(t_0)\big|y \big|^2_{L^2_{\mathbb{F}}(0,T;L^2(G))} \\
    &\quad +
\mathcal{C}\mathbb{E}\left( \theta^2(T)\big|\nabla y(T) \big|^2_{L^2(G)} +
s\lambda\varphi(T)\theta^2(T)\big|y(T)   \big|^2_{L^2(G)} \right).
\end{align}
Now we fix $\lambda =\lambda_2 $. From the
inequality \cref{eqISPCSe7}, we get
\begin{align}\label{eqISPCSe8} \notag
&\mathbb{E}\int_G |y(t_0) |^2dx
+ \mathbb{E} \int_{t_{0}}^{T} |g_{1}|^{2} d t \\&  \leq \mathcal{C}
\theta^2(t_2)\theta^{-2}(t_0)\big|y  \big|^2_{L^2_{{\mathbb{F}}}(0,T;L^2(G))}
+
\mathcal{C}\,\theta^2(T)\mathbb{E}\big|y(T) \big|^2_{H^1(G)}.
\end{align}
Replacing $\mathcal{C}$ by $\mathcal{C}e^{s e^{\lambda_2 T}}$, from inequality \eqref{eqISPCSe8}, for any $s \geq 0$, it holds that
\begin{align}\label{eqISPCSe9} \notag
&\mathbb{E}\int_G |y(t_0) |^2dx
+ \mathbb{E} \int_{t_{0}}^{T} |g_{1}|^{2} d t\\
& \leq \mathcal{C}e^{-2s(e^{\lambda_2 t_2}-e^{\lambda_2
t_0})}\big|y  \big|^2_{L^2_{{\mathbb{F}}}(0,T;L^2(G))}
+ \mathcal{C}e^{\mathcal{C}s}\mathbb{E}\big|y(T) \big|^2_{H^1(G)}.
\end{align}
Choosing $s \geq 0$ which minimize
the right-hand side of inequality \cref{eqISPCSe9}, we obtain that
\begin{equation*}
| y(t_0) |_{L^2_{\mathcal{F}_{t_0}}(\Omega;L^{2}(G))}
+ | g_{1} |_{L^{2}_{\mathbb{F}}(t_{0},T)}  \leq \mathcal{C}
\big|y  \big|^{1-\gamma}_{L^2_{{\mathbb{F}}}(0,T;L^2(G))}
 \big|y(T) \big|^{\gamma}_{L^{2}_{\mathcal{F}_{T}}(\Omega; H^{1}(G))},
\end{equation*}
with $ \gamma = \frac{2(e^{\lambda_2 t_0}-e^{\lambda_2 t_2})}{\mathcal{C}+2(e^{\lambda_2 t_0}-e^{\lambda_2 t_2})} $.
This completes the proof of  \cref{thmISPC}.
\end{proof}

\section{Solution to the inverse source problem with a boundary measurement}
\label{secInverseSourcePre}

In this section, we study the inverse source problem with the boundary measurements, i.e.,
\cref{prob.p7pre}, which answered by the following uniqueness result.

\begin{theorem}\label{inv th2Pre}
Let $h \in L_{\mathbb{F}}^2 (0, T ; H^1 (G^{\prime} ) )$ and $R \in C^{1,3}(\overline{Q})$ satisfy
\begin{align}\label{eq.p7RPre}
    |R(t, x)| \neq 0 \quad \text { for all }(t, x) \in[0, T] \times \overline{G}.
\end{align}
Assume that $ g= a_{3} y $, where $a_{3} \in L^{\infty}_{\mathbb{F}}(0,T; W^{2, \infty}(G))$.
If
\begin{align*}
    y \in L_{\mathbb{F}}^2 (0, T ; H^2(G) \cap H_0^1(G) )  \cap L_{\mathbb{F}}^2 (\Omega ; C ([0, T] ; H_0^1(G) ) )
\end{align*}
satisfies \cref{system2bu} and
\begin{align*}
    & \frac{\partial y}{\partial \nu}=0 \quad \text { on }(0, T) \times \partial G, ~ \text{${\mathbb{P}}$-a.s.},
\end{align*}
then
\begin{align*}
h (t, x^{\prime} )=0 & \text { for all } (t, x^{\prime} ) \in(0, T) \times G^{\prime},~ \mbox{${\mathbb{P}}$-a.s.}
\end{align*}
\end{theorem}

Using the Carleman estimate \cref{car1Lue}, we can prove \cref{inv th2Pre}.

\begin{proof}[Proof of \cref{inv th2Pre}]

For the simplicity of notations, we only deal with the case  $(b^{jk})_{1\leq j,k \leq n}$ is the identity matrix.
For general variable coefficients, thanks to Carleman estimate \cref{car1Lue} and  $(b^{jk})_{1\leq j,k\leq n} \in C^1(G;\mathbb{R}^{n\times n})$, we can also obtain \cref{eq.s45-4Pre}.

Set $ y = R z $. From  \cref{system2bu}, we have
\begin{align}\label{eq.s45-1Pre}
    \left\{\begin{aligned}
    & d z-\Delta z d t  =  \big[ 2 R^{-1} \nabla R \cdot \nabla z + R^{-1} \big(a_2 R + \Delta R - R_t + \nabla R \cdot  q_1 \big) z\big] d t \!\!\!\!\!\!\!\!\!\!\!\!\!\!\!\!\!\!\!\!  &&  \\
    & \quad \quad \quad \quad \quad  ~ + [ (a_1, \nabla z ) + h (t, x^{\prime} ) ] d t+ a_{3} z d W(t)  &&  \text {in } Q, \\
    & z=\frac{\partial z}{\partial \nu}  =0  &&\text {on } \Sigma \\
    & z(0)  =0 &&  \text {on } G .
    \end{aligned}\right.
\end{align}
Differentiate both side of \cref{eq.s45-1Pre} with respect to $ x_{1} $, and set $ u = z_{x_{1}} $.
Noting that $ \nabla z = \dfrac{\partial z}{ \partial \nu} = 0 $ on $ \Sigma $, we obtain that
\begin{align} \label{eq.s45-2Pre}
    \left\{\begin{aligned}
    & d u-\Delta u d t=\big \{  \left(a_1\right)_{x_1} \cdot \nabla z  \!+\!\left(a_1,\! \nabla u\right)\!
    +\! \big(2R^{-1} \nabla R\big)_{x_1} \cdot \nabla z  
        \\
    & \quad \quad \quad \quad \quad \quad
    + 2R^{-1} \nabla R \cdot \nabla u 
    +\big[ R^{-1} ( a_2 R + \Delta R - R_t + \nabla R \cdot  a_1 )\big]_{x_1} z   \!\!\!\!\!\!\! \\
    &\quad \quad \quad \quad \quad   \quad+R^{-1} ( a_2 R + \Delta R - R_t + \nabla R \cdot  a_1 ) u\big\} d t     +( a_{3} z)_{x_1} d W(t)
        &&  \text { in } Q, \\
    & u=0    && \text { on } \Sigma, \\
    & u(0)=0    &&
    \text { on } G .
    \end{aligned}\right.
\end{align}
Since $ u = z_{x_{1}} $, and $ z(t, 0, x') = y (t, 0, x') = 0 $ for $ (t, x') \in (0,T) \times G' $, we have
\begin{align}\label{eq.s45-2-1Pre}
z (t, x_{1}, x') = \int_{0}^{x_{1}} u (t, \eta, x') d \eta.
\end{align}
This, together with \cref{eq.s45-2Pre}, implies that $ u $ satisfies the following equation:
\begin{align} \label{eq.s45-3Pre}
\left\{\begin{aligned}
& d u-\Delta u d t=\bigg \{
 \left(a_1\right)_{x_1} \cdot  \nabla \int_{0}^{x_{1}} u (t, \eta, x') d \eta \!+ a_1 \cdot \nabla u 
\\
& \quad \quad \quad \quad \quad \quad
+ \ \big( 2 R^{-1} \nabla R \big)_{x_1} \cdot \nabla \int_{0}^{x_{1}} u (t, \eta, x') d \eta 
+ 2 R^{-1} \nabla R  \cdot \! \nabla u 
    \\
& \quad \quad \quad \quad \quad \quad
+\big[ R^{-1} ( a_2 R + \Delta R - R_t + \nabla R \cdot  a_1 )\big]_{x_1} \int_{0}^{x_{1}} u (t, \eta, x') d \eta    \!\!\!\!\!\!\! \\
&\quad \quad \quad \quad \quad   \quad+R^{-1} ( a_2 R + \Delta R - R_t + \nabla R \cdot  a_1 )  u\bigg\} d t    \\
&\quad \quad \quad~ \quad \quad  \quad+ (a_{3} )_{x_1} z  d W(t)
+ a_{3} u d W(t)
    &&  \text { in } Q, \\
& u=0    && \text { on } \Sigma, \\
& u(0)=0    &&
\text { on } G .
\end{aligned}\right.
\end{align}
For any $ \varepsilon > 0 $, we choose $ t_{1} $ and $ t_{2} $ such that $ 0 < T- \varepsilon < t_{1} < t_{2} < T$.
Let $ \chi \in C^{\infty}(\mathbb{R}) $  be a cut-off function such that $ 0 \leq \chi \leq 1 $ and that
\begin{align} \label{eq.s45-3-1Pre}
\chi= \begin{cases}1, & \mbox{ if }t \leq t_1, \\ 0, & \mbox{ if }t \geq t_2.\end{cases}
\end{align}
Set $ w = \chi u $. Then we have
\begin{align*}
\left\{\begin{aligned}
& d w-\Delta w d t=\bigg \{
    \left(a_1\right)_{x_1} \cdot  \nabla \int_{0}^{x_{1}} w (t, \eta, x') d \eta \!+ a_1 \cdot \nabla w 
   \\
   & \quad \quad \quad \quad \quad \quad
   + \ \big( 2 R^{-1} \nabla R \big)_{x_1} \cdot \nabla \int_{0}^{x_{1}} w (t, \eta, x') d \eta 
   + 2 R^{-1} \nabla R  \cdot \! \nabla w 
       \\
   & \quad \quad \quad \quad \quad \quad
   +\big[ R^{-1} ( a_2 R + \Delta R - R_t + \nabla R \cdot  a_1 )\big]_{x_1} \int_{0}^{x_{1}} w (t, \eta, x') d \eta    \!\!\!\!\!\!\! \\
   &\quad \quad \quad \quad \quad   \quad+R^{-1} ( a_2 R + \Delta R - R_t + \nabla R \cdot  a_1 )  w\bigg\} d t    \\
&\quad \quad \quad~ \quad \quad  \quad+ \chi  (a_{3} )_{x_1}   z d W(t)
+ a_{3} \chi w  d W(t)
+ \chi_{t} u d t
    &&   \text { in } Q, \\
& w=0    &&  \text { on } \Sigma, \\
& w(0)=0    &&
\text { on } G .
\end{aligned}\right.
\end{align*}

Let us choose $ \psi = - t $ and $ \delta = 0 $ in \cref{car1Lue}. Noting that $ w(0) = w(T) = 0 $, we know that, for all $ \lambda \geq \lambda_{1} $ and $ s \geq s_{0}(\lambda) $, it holds that
\begin{align}\label{eq.s45-4Pre}\notag
& \mathbb{E} \int_Q \big( \lambda |\nabla w|^2+s \lambda^{2} \varphi w^2\big) \theta^{2}  d x d t \\\notag
& \leq \mathcal{C} \Big(s \lambda \mathbb{E} \int_0^T \int_G \theta^{2} \bigg|\int_0^{x_1} w\left(t, \eta, x^{\prime}\right) d \eta\bigg|^2 d x d t \\
&\quad  + \mathbb{E} \int_0^T \int_G \theta^{2} \bigg|\nabla \int_0^{x_1} w\left(t, \eta, x^{\prime}\right) d \eta\bigg|^2 d x d t  +  \mathbb{E} \int_Q \theta^{2} \chi_t^2 u^2   d x d t\Big).
\end{align}

Noting that
\begin{align*}
\bigg |  \int_{0}^{x_{1}} w (t, \eta, x') d \eta \bigg |^{2} \leq l \int_{0}^{l} | w ( t, \eta, x') |^{2} d \eta
\end{align*}
for all $ (t, x_{1}, x') \in Q $, we can obtain that
\begin{align}\label{eq.s45-5Pre} \notag
 \int_0^T \int_G \theta^{2} \bigg|\int_0^{x_1} w\left(t, \eta, x^{\prime}\right) d \eta\bigg|^2 d x d t
& \leq
l \int_{0}^{l} d x_{1} \int_0^T \int_{G'}  \int_{0}^{l} \theta^{2} |  w\left(t, \eta, x^{\prime}\right)  |^2 d \eta d x' d t
\\
& \leq
l^{2}   \int_0^T \int_{G}   \theta^{2} |  w\left(t, x_{1}, x^{\prime}\right)  |^2 d x_{1} d x' d t.
\end{align}
By virtue of
\begin{align*}
  \nabla \int_0^{x_1} w\left(t, \eta, x^{\prime}\right) d \eta
& =\left(w\left(t, x_1, x^{\prime}\right), \int_0^{x_1} \nabla_{x^{\prime}} w\left(t, \eta, x^{\prime}\right) d \eta\right) \\
& =\left(w\left(t, x_1, x^{\prime}\right)-w\left(t, 0, x^{\prime}\right), \int_0^{x_1} \nabla_{x^{\prime}} w\left(t, \eta, x^{\prime}\right) d \eta\right)
\\
& =\int_0^{x_1} \nabla w\left(t, \eta, x^{\prime}\right) d \eta,
\end{align*}
we have
\begin{align} \label{eq.s45-6Pre} \notag
  \int_0^T \int_G \theta^{2} \bigg|\nabla \int_0^{x_1} w\left(t, \eta, x^{\prime}\right) d \eta\bigg|^2 d x d t
& =\int_0^T \int_G \theta^{2} \bigg|\int_0^{x_1} \nabla w\left(t, \eta, x^{\prime}\right) d \eta\bigg|^2 d x d t \\
& \leq l^{2} \int_0^T \int_G \theta^{2} \left|\nabla w\left(t, x_1, x^{\prime}\right)\right|^2 d x d t .
\end{align}
Combining  \cref{eq.s45-4Pre,eq.s45-5Pre,eq.s45-6Pre}, for all $ \lambda \geq \lambda_{1} $ and $ s \geq s_{0} $, we have
\begin{align*}
& \mathbb{E} \int_Q  \big(\lambda |\nabla w|^2+s \lambda^{2} \varphi w^2 \big)\theta^{2} d x d t
\\
&\leq \mathcal{C}  \mathbb{E} \int_Q \big(s \lambda w^2+|\nabla w|^2 \big) \theta^{2} d x d t   +\mathcal{C}  \mathbb{E} \int_Q \theta^{2} \chi_t^2 u^2  d x d t
.
\end{align*}
Hence, there exists a $ \lambda_{2} \geq \lambda_{1} $ such that for all $ \lambda \geq \lambda_{2} $ and $ s \geq s_{0} $, it holds that
\begin{align}
\label{eq.s45-7Pre}
& \mathbb{E} \int_Q  \big(\lambda |\nabla w|^2+s \lambda^{2} \varphi w^2 \big)\theta^{2} d x d t
\leq \mathcal{C}    \mathbb{E} \int_Q \theta^{2} \chi_t^2 u^2  d x d t.
\end{align}

It follows from \cref{eq.s45-3-1Pre} that $ \chi_{t} \neq 0 $ only holds in $ (t_{1}, t_{2}) $.
Fix $ \lambda = \lambda_{2} $.
Thanks to \cref{eq.s45-7Pre}, for $ s \geq s_{0} $, we have

\begin{align*}
\mathbb{E} \int_{0}^{T-\varepsilon} \int_G\left( |\nabla w|^2+s   w^2\right) \theta^{2} d x d t \leq \mathcal{C} \mathbb{E} \int_{t_{1}}^{t_{2}} \int_G \chi_t^2 u^2 \theta^{2} d x d t
,
\end{align*}
which implies
\begin{align*}
e^{2 s \varphi(T- \varepsilon)} \mathbb{E} \int_{0}^{T-\varepsilon} \int_G\left( |\nabla w|^2+s   w^2\right)  d x d t \leq \mathcal{C} e^{2 s \varphi(t_{1})} \mathbb{E} \int_{t_{1}}^{t_{2}} \int_G \chi_t^2 u^2   d x d t
.
\end{align*}
Hence, for $ s \geq s_{3} $, it holds that
\begin{align} \label{eq.s45-9-1Pre}
\mathbb{E} \int_{0}^{T-\varepsilon} \int_G\left( |\nabla w|^2+s   w^2\right)  d x d t \leq \mathcal{C} e^{2 s (\varphi(t_{1}) - \varphi(T- \varepsilon))} \mathbb{E} \int_{t_{1}}^{t_{2}} \int_G \chi_t^2 u^2   d x d t.
\end{align}
Note that $ \varphi(t_{1}) - \varphi(T- \varepsilon) < 0 $. By letting $ s \rightarrow + \infty $ in \cref{eq.s45-9-1Pre}, we have
\begin{align*}
w = 0 \text{ in } (0, T- \varepsilon) \times G, ~ \text{${\mathbb{P}}$-a.s.}
\end{align*}
This, together with \cref{eq.s45-2-1Pre}, implies that
\begin{align*}
z = 0 \text{ in } (0, T- \varepsilon) \times G, ~ \text{${\mathbb{P}}$-a.s.,}
\end{align*}
which means that
\begin{align}\label{eq.s45-10Pre}
y = 0 \text{ in } (0, T- \varepsilon) \times G, ~ \text{${\mathbb{P}}$-a.s.}
\end{align}
Combining \cref{eq.s45-10Pre,system2bu}, we deduce that
\begin{align*}
h(t, x')  =0, \text{ for } (t, x') \in  (0, T-\varepsilon) \times G', ~ \text{${\mathbb{P}}$-a.s.}
\end{align*}
Since $ \varepsilon > 0 $ is arbitrary, we complete the proof.
\end{proof}

\section[Solution to the inverse source problem with boundary/terminal measurement]{Solution to the inverse source problem with a boundary measurement and a terminal measurement}
\label{secInverseSource}

In this section, we study the inverse source problem with a boundary measurement  and a terminal measurement.
\cref{prob.p7} is answered by the following uniqueness result.

\begin{theorem}\label{inv th2}
Let $h \in L_{\mathbb{F}}^2 (0, T ; H^1 (G^{\prime} ) )$ and $R \in C^{1,3}(\overline{Q})$ satisfy
\begin{align}\label{eq.p7R}
|R(t, x)| \neq 0 \quad \text { for all }(t, x) \in[0, T] \times \bar{G} .
\end{align}
Assume that $g=g_1(t) g_2(x) \in L_{\mathbb{F}}^2 (0, T ; H^1(G) )$, where $g_1 \in L_{\mathbb{F}}^2 (0, T)$,  $ g_2 \in H^1(G)$ and $R=R_1(t) R_2(x)$. If
\begin{align*}
y \in L_{\mathbb{F}}^2 (0, T ; H^2(G) \cap H_0^1(G) )  \cap L_{\mathbb{F}}^2 (\Omega ; C ([0, T] ; H_0^1(G) ) )
\end{align*}
satisfies \cref{system2bu} and
\begin{align*}
& \frac{\partial y}{\partial v}=0 \quad \text { on }(0, T) \times \partial G, ~ \text{${\mathbb{P}}$-a.s.},  \\
& y(T)=0 \text { in } G, ~ \text{${\mathbb{P}}$-a.s.},  \\
&
\end{align*}
then
\begin{align*}
h (t, x^{\prime} )=0 & \text { for all } (t, x^{\prime} ) \in(0, T) \times G^{\prime},~ \mbox{${\mathbb{P}}$-a.s.}
\end{align*}
and
\begin{align*}
g(t, x)=0 & \text { for all }(t, x) \in(0, T) \times G, ~ \mbox{${\mathbb{P}}$-a.s.}
\end{align*}
\end{theorem}
%


\begin{proof}[Proof of \cref{inv th2}]

For the sake of simplicity in statement, we only prove the case where $(b^{jk})_{1\leq j,k \leq n}$ is the identity matrix.
For general variable coefficients, thanks to Carleman estimate \cref{carleman est1} and  $(b^{jk})_{1\leq j,k\leq n} \in C^1(G;\mathbb{R}^{n\times n})$, we can also obtain \cref{eq.s45-4}.

Set $ y = R z $. From  \cref{system2bu} and $ y(T) = 0 $ in $ G $, we have that
\begin{align}\label{eq.s45-1}
\left\{\begin{aligned}
& d z-\Delta z d t  =  \big[  2R^{-1} \nabla R  \cdot \nabla z + R^{-1} \big(a_2 R + \Delta R  - R_t +  \nabla R \cdot q_1 \big) z \big] d t \!\!\!\!\!\! &&  \\
& \quad \quad \quad \quad \quad  ~ + [ (a_1, \nabla z ) + h (t, x^{\prime} ) ] d t+ R^{-1} g  d W(t)  &&  \text {in } Q, \\
& z=\frac{\partial z}{\partial \nu}  =0  &&\text {on } \Sigma \\
& z(T)  =0 &&  \text {on } G .
\end{aligned}\right.
\end{align}
Differentiate both side of \cref{eq.s45-1} with respect to $ x_{1} $, and set $ u = z_{x_{1}} $.
Noting that $ \nabla z = \dfrac{\partial z}{ \partial \nu} = 0 $ on $ \Sigma $, we obtain that
\begin{align} \label{eq.s45-2}
\left\{\begin{aligned}
& d u-\Delta u d t=\big[  (a_1 )_{x_1} \cdot \nabla z  +  a_1 \cdot  \nabla u  +   ( 2R^{-1} \nabla R )_{x_1} \cdot \nabla z  
    \\
& \quad \quad \quad \quad \quad \quad
+ 2 R^{-1} \nabla R  \cdot  \nabla u 
+ \big[R^{-1} \big(a_2 R + \Delta R - R_t +  \nabla R \cdot a_1 \big) \big]_{x_1} z   \!\!\!\!\!\!\! \\
&\quad \quad \quad \quad \quad   \quad+ R^{-1} \big(a_2 R + \Delta R - R_t + \nabla R \cdot a_1 \big) u\big] d t  + (R^{-1} g  )_{x_1} d W(t) \!\!\!\!\!\!\! 
 &&  \text { in } Q, \\
& u=0    && \text { on } \Sigma, \\
& u(T)=0    &&
\text { on } G .
\end{aligned}\right.
\end{align}
Since $ u = z_{x_{1}} $, and $ z(t, 0, x') = y (t, 0, x') = 0 $ for $ (t, x') \in (0,T) \times G' $, we have
\begin{align}\label{eq.s45-2-1}
z (t, x_{1}, x') = \int_{0}^{x_{1}} u (t, \eta, x') d \eta.
\end{align}
This, together with \cref{eq.s45-2}, implies that $ u $ satisfies the following equation:
\begin{align} \label{eq.s45-3}
\left\{\begin{aligned}
& d u-\Delta u d t=\Big [
  (a_1 )_{x_1} \cdot \nabla \int_{0}^{x_{1}} u (t, \eta, x') d \eta  +\!  a_1 \cdot \nabla u 
\\
& \quad \quad \quad \quad \quad \quad
+   ( 2 R^{-1} \nabla R  )_{x_1} \cdot \nabla \int_{0}^{x_{1}} u (t, \eta, x') d \eta  
+   2  R^{-1} \nabla R \cdot  \nabla u 
 \\
& \quad \quad \quad \quad \quad \quad
+\big[R^{-1} \big(a_2 R + \Delta R - R_t +  \nabla R \cdot a_1 \big) \big]_{x_1} \int_{0}^{x_{1}} u (t, \eta, x') d \eta    \!\!\!\!\!\!\! \\
&\quad \quad \quad \quad \quad   \quad+R^{-1} \big(a_2 R + \Delta R - R_t + \nabla R \cdot a_1 \big) u\Big ] d t  + ( g R^{-1} )_{x_1} d W(t) 
\!\!\!\!&&    \text { in } Q, \\
& u=0    &&  \text { on } \Sigma, \\
& u(T)=0    && 
\text { on } G .
\end{aligned}\right.
\end{align}
For any $ \varepsilon > 0 $, we choose $ t_{1} $ and $ t_{2} $ such that $ 0 < t_{1} < t_{2} < \varepsilon $.
Let $ \chi \in C^{\infty}(\mathbb{R}) $  be a cut-off function such that $ 0 \leq \chi \leq 1 $ and that
\begin{align} \label{eq.s45-3-1}
\chi= \begin{cases}  1, & \mbox{ if }t \geq t_2, \\ 0, & \mbox{ if }t \leq t_1.\end{cases}
\end{align}
Set $ w = \chi u $. Then we have
\begin{align*}
\left\{\begin{aligned}
& d w-\Delta w d t=\Big [
  (a_1 )_{x_1} \cdot \nabla \int_{0}^{x_{1}} w (t, \eta, x') d \eta   +  a_1 \cdot  \nabla w 
\\
& \quad \quad \quad \quad \quad \quad
+   ( 2 R^{-1} \nabla R  )_{x_1} \cdot \nabla \int_{0}^{x_{1}} w (t, \eta, x') d \eta  
+  2 R^{-1} \nabla R \cdot \nabla w 
 \\
& \quad \quad \quad \quad \quad \quad
+\big[R^{-1} \big(a_2 R + \Delta R - R_t +  \nabla R \cdot a_1 \big) \big]_{x_1} \int_{0}^{x_{1}} w (t, \eta, x') d \eta    \!\!\!\!\!\!\! \\
&\quad \quad \quad \quad \quad   \quad+R^{-1} \big(a_2 R + \Delta R - R_t +  \nabla R \cdot a_1 \big) w \Big] d t    \\
&\quad \quad \quad~ \quad \quad  \quad+ \chi  ( g R^{-1} )_{x_1} d W(t)
+ \chi_{t} u d t
 &&   \text { in } Q, \\
& w=0    &&  \text { on } \Sigma, \\
& w(T)=0    &&
\text { on } G .
\end{aligned}\right.
\end{align*}

Let us choose  $ \psi = t $ and $ \delta = 0 $
in \cref{carleman est1}. Noting that $ w(0) = w(T) = 0 $, we know that, for all $ \lambda \geq \lambda_{1} $ and $ s \geq s_{0}(\lambda) $, it holds that
\begin{align}\label{eq.s45-4}\notag
& \mathbb{E} \int_Q\Big[s \lambda \varphi \chi^2\Big|\Big(\frac{g}{R}\Big)_{x_1}\Big|^2+ \lambda |\nabla w|^2+s \lambda^{2} \varphi w^2\Big] \theta^{2}  d x d t \\\notag
& \leq \mathcal{C} \varphi(T) \mathbb{E} \int_0^T \int_G \theta^{2} \bigg|\int_0^{x_1} w\left(t, \eta, x^{\prime}\right) d \eta\bigg|^2 d x d t \\\notag
&\quad  +\mathcal{C} \varphi(T) \mathbb{E} \int_0^T \int_G \theta^{2} \bigg|\nabla \int_0^{x_1} w\left(t, \eta, x^{\prime}\right) d \eta\bigg|^2 d x d t \\
&\quad  +\mathcal{C} \varphi(T) \mathbb{E} \int_Q  \theta^{2} \chi^2\bigg|\nabla\bigg(\frac{g}{R}\bigg)_{x_1}\bigg|^2  d x d t+C \varphi(T) \mathbb{E} \int_Q \theta^{2} \chi_t^2 u^2   d x d t .
\end{align}

Noting that
\begin{align*}
\bigg |  \int_{0}^{x_{1}} w (t, \eta, x') d \eta \bigg |^{2} \leq l \int_{0}^{l} | w ( t, \eta, x') |^{2} d \eta
\end{align*}
for all $ (t, x_{1}, x') \in Q $, we can obtain that
\begin{align}\label{eq.s45-5} \notag
 \int_0^T \int_G \theta^{2} \bigg|\int_0^{x_1} w\left(t, \eta, x^{\prime}\right) d \eta\bigg|^2 d x d t
& \leq
l \int_{0}^{l} d x_{1} \int_0^T \int_{G'}  \int_{0}^{l} \theta^{2} |  w\left(t, \eta, x^{\prime}\right)  |^2 d \eta d x' d t
\\
& \leq
l^{2}   \int_0^T \int_{G}   \theta^{2} |  w\left(t, x_{1}, x^{\prime}\right)  |^2 d x_{1} d x' d t.
\end{align}
By virtue of
\begin{align*}
\nabla \int_0^{x_1} w\left(t, \eta, x^{\prime}\right) d \eta
& =\left(w\left(t, x_1, x^{\prime}\right), \int_0^{x_1} \nabla_{x^{\prime}} w\left(t, \eta, x^{\prime}\right) d \eta\right) \\
& =\left(w\left(t, x_1, x^{\prime}\right)-w\left(t, 0, x^{\prime}\right), \int_0^{x_1} \nabla_{x^{\prime}} w\left(t, \eta, x^{\prime}\right) d \eta\right)
\\
& =\int_0^{x_1} \nabla w\left(t, \eta, x^{\prime}\right) d \eta,
\end{align*}
we have
\begin{align} \label{eq.s45-6} \notag
 \int_0^T \int_G \theta^{2} \bigg|\nabla \int_0^{x_1} w\left(t, \eta, x^{\prime}\right) d \eta\bigg|^2 d x d t
& =\int_0^T \int_G \theta^{2} \bigg|\int_0^{x_1} \nabla w\left(t, \eta, x^{\prime}\right) d \eta\bigg|^2 d x d t \\
& \leq l^{2} \int_0^T \int_G \theta^{2} \left|\nabla w\left(t, x_1, x^{\prime}\right)\right|^2 d x d t .
\end{align}
Combining  \cref{eq.s45-4,eq.s45-5,eq.s45-6}, for all $ \lambda \geq \lambda_{1} $ and $ s \geq s_{0} $, we have
\begin{align*}
& \mathbb{E} \int_Q\theta^{2} \Big\{s \lambda \varphi \chi^2\Big[\Big(\frac{g}{R}\Big)_{x_1}\Big]^2+ \lambda |\nabla w|^2+s \lambda^{2} \varphi w^2\Big\} d x d t
\\
&\leq \mathcal{C} \varphi(T)\Big\{ \mathbb{E}\int_Q\theta^{2}\left[w^2+|\nabla w|^2\right]  d x d t 
+  \mathbb{E} \int_Q \chi^2\theta^{2}\left|\nabla\left(\frac{g}{R}\right)_{x_1}\right|^2  d x d t
\\
& \quad \quad \quad \quad \quad 
 + \mathbb{E} \int_Q \theta^{2} \chi_t^2 u^2  d x d t\Big\}
.
\end{align*}
Hence, there exists a $ \lambda_{2} \geq \lambda_{1} $ such that for all $ \lambda \geq \lambda_{2} $ and $ s \geq s_{0} $, it holds that
\begin{align*}
& s\bigg[\mathbb{E} \int_Q\bigg|\bigg(\chi \frac{g}{R}\bigg)_{x_1}\bigg|^2 \theta^{2} d x d t-\frac{\mathcal{C}}{s} \mathbb{E} \int_Q\bigg|\nabla\bigg(\chi \frac{g}{R}\bigg)_{x_1}\bigg|^2 \theta^{2} d x d t\bigg] \\
& +\mathbb{E} \int_Q\left(\chi^2|\nabla u|^2+s \chi^2 u^2\right) \theta^{2} d x d t \leq \mathcal{C} \mathbb{E} \int_Q \chi_t^2 u^2 \theta^{2} d x d t
.
\end{align*}

Noting that $ g = g_{1}(t) g_{2}(x) $ and $ R = R_{1}(t) R_{2}(x) $, we have
\begin{align*}
& \mathbb{E}\bigg[\int_G s\bigg(\bigg|\bigg(\frac{g_2}{R_2}\bigg)_{x_1}\bigg|^2-\frac{1}{s}\bigg|\nabla\bigg(\frac{g_2}{R_2}\bigg)_{x_1}\bigg|^2\bigg) d x \int_0^T\bigg(\frac{g_1}{R_1}\bigg)^2 \theta^{2} d t\bigg]\\
& +\mathbb{E} \int_Q\left(\chi^2|\nabla u|^2+s \chi^2 u^2\right) \theta^{2} d x d t \leq \mathcal{C} \mathbb{E} \int_Q \chi_t^2 u^2 \theta^{2} d x d t.
\end{align*}
Since $ g_{1} \in L^{2}_{\mathbb{F}}(0,T)$, $ g_{2} \in H^{1}(G) $, we obtain that
\begin{align}\label{eq.s45-8} \notag
&  s\bigg[ \int_G \bigg|\bigg(\frac{g_2}{R_2}\bigg)_{x_1}\bigg|^2 dx -\frac{1}{s} \int_G\bigg|\nabla\bigg(\frac{g_2}{R_2}\bigg)_{x_1}\bigg|^2 d x \bigg] \mathbb{E} \int_0^T\bigg(\frac{g_1}{R_1}\bigg)^2 \theta^{2} d t \\
& +\mathbb{E} \int_Q\left(\chi^2|\nabla u|^2+s \chi^2 u^2\right) \theta^{2} d x d t \leq \mathcal{C} \mathbb{E} \int_Q \chi_t^2 u^2 \theta^{2} d x d t
.
\end{align}
If $  \int_G \big|\big(\frac{g_2}{R_2}\big)_{x_1}\big|^2 dx = 0$, then $ \nabla \big(\frac{g_2}{R_2}\big)_{x_1} =0 $. If $  \int_G \big|\big(\frac{g_2}{R_2}\big)_{x_1}\big|^2 dx > 0$, then there exists an $s_2>0$ such that for all $s\geq s_2$,
$$
\int_G \bigg|\bigg(\frac{g_2}{R_2}\bigg)_{x_1}\bigg|^2 dx -\frac{1}{s} \int_G\bigg|\nabla\bigg(\frac{g_2}{R_2}\bigg)_{x_1}\bigg|^2 d x>0.
$$
Consequently,   for all $ s \geq s_{3} = \max \{s_{1},  s_{2} \} $, we have that
\begin{align}\label{eq.s45-7}
    \mathbb{E} \int_Q\left(\chi^2|\nabla u|^2+s \chi^2 u^2\right) \theta^{2} d x d t \leq \mathcal{C} \mathbb{E} \int_Q \chi_t^2 u^2 \theta^{2} d x d t
    .
\end{align}

From \cref{eq.s45-3-1}, we know that $ \chi_{t} \neq 0 $ only holds in $ (t_{1}, t_{2}) $. Thanks to \cref{eq.s45-7}, for $ s \geq s_{3} $, we obtain
\begin{align*}
\mathbb{E} \int_{\varepsilon}^{T} \int_G\left( |\nabla u|^2+s   u^2\right) \theta^{2} d x d t \leq \mathcal{C} \mathbb{E} \int_{t_{1}}^{t_{2}} \int_G \chi_t^2 u^2 \theta^{2} d x d t
,
\end{align*}
which implies
\begin{align*}
e^{2 s \varphi(\varepsilon)} \mathbb{E} \int_{\varepsilon}^{T} \int_G\left( |\nabla u|^2+s   u^2\right)  d x d t \leq \mathcal{C} e^{2 s \varphi(t_{2})} \mathbb{E} \int_{t_{1}}^{t_{2}} \int_G \chi_t^2 u^2   d x d t
.
\end{align*}
Hence, for $ s \geq s_{3} $, it holds that
\begin{align} \label{eq.s45-9-1}
\mathbb{E} \int_{\varepsilon}^{T} \int_G\left( |\nabla u|^2+s   u^2\right)  d x d t \leq \mathcal{C} e^{2 s (\varphi(t_{2}) - \varphi(\varepsilon))} \mathbb{E} \int_{t_{1}}^{t_{2}} \int_G \chi_t^2 u^2   d x d t.
\end{align}
Noting that $ \varphi(t_{2}) - \varphi(\varepsilon) < 0 $, letting $ s \rightarrow + \infty $ in \cref{eq.s45-9-1}, we have
\begin{align*}
u = 0 \text{ in } (\varepsilon, T) \times G, ~ \text{${\mathbb{P}}$-a.s.}
\end{align*}
This together with \cref{eq.s45-2-1} implies that
\begin{align*}
z = 0 \text{ in } (\varepsilon, T) \times G, ~ \text{${\mathbb{P}}$-a.s.}
\end{align*}
which means that
\begin{align}\label{eq.s45-10}
y = 0 \text{ in } (\varepsilon, T) \times G, ~ \text{${\mathbb{P}}$-a.s.}
\end{align}
Combining \cref{eq.s45-10,system2bu}, we deduce that
\begin{align}\label{eq.s45-10-1}
\int_\varepsilon^t h (s, x^{\prime} ) R(s, x) d s+\int_\varepsilon^t g(s, x) d W(s)=0 \quad \text { for all } t \in(\varepsilon, T), ~ \text{${\mathbb{P}}$-a.s.}
\end{align}
Noting that $\int_\varepsilon^t h (s, x^{\prime} ) R(s, x) d s$ is absolutely continuous with respect to $t$, we get from \cref{eq.s45-10-1} that $\int_\varepsilon^t g(s, x) d W(s)$ is also absolutely continuous with respect to $t$. Consequently, its quadratic variation is zero, which concludes  that
\begin{align*}
g(t,x)= 0 , \text{ for } (t, x) \in  (\varepsilon, T) \times G, ~ \text{${\mathbb{P}}$-a.s.}
\end{align*}
This, together with \cref{eq.s45-10-1}, implies that
\begin{align*}
    \int_\varepsilon^t h (s, x^{\prime} ) R(s, x) d s, \text{ for } (t, x) \in  (\varepsilon, T) \times G, ~ \text{${\mathbb{P}}$-a.s.}
\end{align*}
Therefore,
\begin{align*}
h(t, x')  = 0 , \text{ for } (t, x) \in  (\varepsilon, T) \times G, ~ \text{${\mathbb{P}}$-a.s.}
\end{align*}
Since $ \varepsilon > 0 $ is arbitrary, we complete the proof.
\end{proof}

\section{Reconstruct the unknown state with a terminal measurement}
\label{secReconstructTer}

In this section, we focus on \cref{probRI}. In the previous
sections, we study whether the unknown state/source could be
uniquely determined through suitable measurements. In such a
context, we do not take into account the errors in the measurements.
When addressing the problem of reconstructing the unknown
state/source from the measurements, the measurement errors must be
considered. Various methods of ``regularization" have been developed
for inverse problems of deterministic PDEs to tackle this issue, all
of which, in essence, utilize additional information about the
unknown function. For instance, the method introduced by Tikhonov
aims to minimize a quadratic functional that contains higher
derivatives to reproduce the measured data. In this section, we
employ this method to investigate the reconstruction problem for the
unknown state with a terminal measurement.

Denote $ \mathcal{H}_{0} \deq L^{2}_{\mathcal{F}_{0}}(\Omega; H^{2}(G) \cap H_{0}^{1}(G)) $.
Suppose $ y^{*}_{T} \in L^{2}_{\mathcal{F}_{T}}(\Omega;L^{2}(G))$ is the exact terminal value of the system \cref{eqDou1} with the initial datum $ y^{*}_{0} \in \mathcal{H}_{0}$.
Let $ y_{T}^{\delta} \in  L^{2}_{\mathcal{F}_{T}}(\Omega;L^{2}(G)) $ be the measured data satisfying
\begin{align}\label{eqTikhonovDistance}
    |y^{*}_{T} - y_{T}^{\delta}|_{L^{2}_{\mathcal{F}_{T}}(\Omega;L^{2}(G)) } \leq \delta,
\end{align}
where $ \delta > 0 $  is the noise level of data.

Consider the following Tikhonov type functional:
\begin{align} \label{eqTikhonovFucntional} \notag
    \mathcal{J}_\tau\left(y_0\right)=
    & \mathbb{E} \int_G \Big \{y_0+\int_{0}^T\Big [\sum_{i, j=1}^n\big(b^{j k} y_{x_{j}}\big)_{x_{k}}+b_1 \cdot \nabla y+b_2 y+f\Big] d s
    \\
    &  \quad \quad ~
    +\int_{0}^T\left(b_3 y+g\right) d W(s)- y_T^\delta\Big\}^2 d x   +\tau \mathbb{E}|y_0|_{H^2(G)}^2
    ,
\end{align}
where $ y_0 \in \mathcal{H}_{0} $  and $ y $ solves \cref{eqDou1} with the initial datum $ y_{0} $.

Since   $\mathcal{J}_\tau(\cdot) $ is coercive, convex and lower semi-continuous, we have the following result.

\begin{proposition}
    For every $ \tau\in (0,1) $, there exists a unique minimizer $ \bar{y}_0 \in \mathcal{H}_{0} $ of the functional \cref{eqTikhonovFucntional}.
\end{proposition}

We have the following error estimate.

\begin{theorem}
    For any $ t_{0} \in (0,T) $,
    there exists a positive constant $ \mathcal{C} = \mathcal{C}(t_{0})$ such that for  $ \gamma = \gamma(t_{0},T) \in (0,1)$ as stated in \cref{thmDou2022}, for  all $ \delta \in (0,1) $ and $ \tau = \delta^{2} $,
    putting $ \bar{y}_{0}^{\delta} $  as the minimizer of functional  $\mathcal{J}_\tau(\cdot) $, then
    we have
    \begin{align*}
        &|y(t_{0}; \bar{y}_{0}^{\delta})  -  y(t_{0}; y^{*}_{0})|^{2}_{L_{\mathcal{F}_{t_{0}}}^2\left(\Omega ;   H ^1(G)\right)}
        \leq
        \mathcal{C}  M  \operatorname{exp}  [-  3^{-\gamma} \ln^{\gamma} (M^{-1/2} \delta^{-1}) ]
        ,
    \end{align*}
    where $ M = 1 + |y_{0}^{*}|_{\mathcal{H}_{0}}^{2}  $ and  $y(\cdot; \bar{y}_{0}^{\delta})$ (resp. $  y(\cdot; y^{*}_{0})$) is the solution to \cref{eqDou1} with respect to the initial datum $\bar{y}_{0}^{\delta}$  (resp. $   y^{*}_{0}$).
\end{theorem}

\begin{proof}
    For simplicity of notations, we denote
    \begin{align*}
        & \mathcal{A} y_{0} \deq  y_0+\int_{0}^T\Big [\sum_{i, j=1}^n\big(b^{j k} y_{x_{j}}(s;y_{0})\big)_{x_{k}}+b_1 \cdot \nabla y(s;y_{0})+b_2 y(s;y_{0})+f\Big] d s
        \\
        & \quad ~ \quad
        +\int_{0}^T\left(b_3 y(s;y_{0})+g\right) d W(s), \quad \text{${\mathbb{P}}$-a.s.}
        ,
    \end{align*}
    and
    \begin{align*}
        & \mathcal{B} y_{0} \deq  y_0+\int_{0}^T\Big [\sum_{i, j=1}^n\big(b^{j k} y_{x_{j}}(s;y_{0})\big)_{x_{k}}+b_1 \cdot \nabla y(s;y_{0})+b_2 y(s;y_{0}) \Big] d s
        \\
        & \quad ~ \quad
        +\int_{0}^T\left(b_3 y(s;y_{0}) \right) d W(s), \quad \text{${\mathbb{P}}$-a.s.}
    \end{align*}

    The G\^{a}teaux derivation of $\mathcal{J}_\tau(\cdot) $ at $y_{0}$ reads
    \begin{align*}
        \mathcal{J}'_{\tau}(y_{0}) \phi =
        2 \mathbb{E} \int_{G} \mathcal{A} y_{0} \mathcal{B} \phi d x
        - 2 \mathbb{E} \int_{G} y_{T}^{\delta} \mathcal{B} \phi d x
        + 2 \tau \mathbb{E} \langle y_{0}, \phi \rangle_{H^{2}(G)},
        \quad
        \forall \, \phi \in \mathcal{H}_{0}
        .
    \end{align*}
    Particularly, for the minimizer $ \bar{y}_{0}^{\delta} $, we obtain
    \begin{align} \label{eqTikhonovConvergencePf1}
        \mathbb{E} \int_{G} \mathcal{A} \bar{y}^{\delta}_{0} \mathcal{B} \phi d x
        +   \tau \mathbb{E} \langle \bar{y}^{\delta}_{0}, \phi \rangle_{H^{2}(G)}
        = \mathbb{E} \int_{G} y_{T}^{\delta} \mathcal{B} \phi d x ,
        \quad
        \forall \, \phi \in \mathcal{H}_{0}
        .
    \end{align}
    Recalling that $ y^{*}_{0} $ is the exact initial value, thanks to \cref{eqDou1}, we obtain
    \begin{align} \label{eqTikhonovConvergencePf2}
        \mathbb{E} \int_{G} \mathcal{A} y^{*}_{0} \mathcal{B} \phi d x
        \!+\!   \tau \mathbb{E} \langle y^{*}_{0}, \phi \rangle_{H^{2}(G)}
        \!=\! \mathbb{E} \int_{G} y^{*}_{T}  \mathcal{B} \phi d x
        \!+\!   \tau \mathbb{E} \langle y^{*}_{0}, \phi \rangle_{H^{2}(G)}
        ,
        ~
        \forall \, \phi \in \mathcal{H}_{0}
        .
    \end{align}
    From \cref{eqTikhonovConvergencePf1,eqTikhonovConvergencePf2}, we have
    \begin{align} \label{eqTikhonovConvergencePf3} \notag
        & \mathbb{E} \int_{G} \mathcal{A} (\bar{y}_{0}^{\delta} - y^{*}_{0} ) \mathcal{B} \phi d x
        +   \tau \mathbb{E} \langle (\bar{y}_{0}^{\delta} - y^{*}_{0}), \phi \rangle_{H^{2}(G)}
        \\
        & =
        \mathbb{E} \int_{G} (y_{T}^{\delta} - y^{*}_{T} ) \mathcal{B} \phi d x
        -   \tau \mathbb{E} \langle y^{*}_{0}, \phi \rangle_{H^{2}(G)}
        ,
        \quad
        \forall \, \phi \in \mathcal{H}_{0}
        .
    \end{align}

    Choose $ \phi = \bar{y}_{0}^{\delta} - y^{*}_{0}  $ in \cref{eqTikhonovConvergencePf3}.
    Noting that $ \mathcal{B} \phi = \mathcal{A} (\bar{y}_{0}^{\delta} - y^{*}_{0} ) $, using Cauchy-Schwarz inequality, we obtain that
    \begin{align} \label{eqTikhonovConvergencePf4} \notag
        & \mathbb{E} \int_{G}  ( \mathcal{A} \bar{y}_{0}^{\delta} - \mathcal{A} y^{*}_{0} )^{2}   d x
        +   \tau \mathbb{E} |\bar{y}_{0}^{\delta} - y^{*}_{0}|^{2}_{H^{2}(G)}
        \\ \notag
        & =
        \frac{1}{2} \mathbb{E} |y_{T}^{\delta} - y^{*}_{T} |_{H^{2}(G)}^{2}
        + \frac{1}{2} \mathbb{E} \int_{G}  ( \mathcal{A} \bar{y}_{0}^{\delta} - \mathcal{A} y^{*}_{0} )^{2}   d x
        +\frac{1}{2} \tau   \mathbb{E}  |y^{*}_{0}|_{H^{2}(G)}^{2}
        \\
        & \quad
        +\frac{1}{2} \tau   \mathbb{E}  |\bar{y}_{0}^{\delta} - y^{*}_{0}|_{H^{2}(G)}^{2}
        .
    \end{align}
    Combining \cref{eqTikhonovConvergencePf4,eqTikhonovDistance}, recalling that  $ \tau = \delta^{2} $, we have
    \begin{align*}
        \mathbb{E} \int_{G}  ( \mathcal{A} \bar{y}_{0}^{\delta} - \mathcal{A} y^{*}_{0} )^{2}   d x
        +   \delta^{2} \mathbb{E} |\bar{y}_{0}^{\delta} - y^{*}_{0}|^{2}_{H^{2}(G)}
        \leq
        \delta^{2} (1+ \mathbb{E}  |y^{*}_{0}|_{H^{2}(G)}^{2})
        .
    \end{align*}
    Consequently,
    \begin{align} \label{eqTikhonovConvergencePf5}
        \mathbb{E} |\bar{y}_{0}^{\delta} - y^{*}_{0}|^{2}_{H^{2}(G)} \leq 1+ \mathbb{E}  |y^{*}_{0}|_{H^{2}(G)}^{2}
        ,
    \end{align}
    and
    \begin{align} \label{eqTikhonovConvergencePf6}
        \mathbb{E} |y(T;\bar{y}_{0}^{\delta})-y^{*}_{T}|_{L^{2}(G)}^{2} \leq \mathbb{E} \int_{G}  ( \mathcal{A} \bar{y}_{0}^{\delta} - \mathcal{A} y^{*}_{0} )^{2}   d x
        \leq \delta^{2} (1+ \mathbb{E}  |y^{*}_{0}|_{H^{2}(G)}^{2}).
    \end{align}
    From \cref{eqTikhonovConvergencePf5}, \cref{eqStableM} and \cref{thmDou2022}, we get the desired result.
\end{proof}

\section{Reconstruct the  unknown state with  boundary measurements}
\label{secReconstruct}

In this section, we consider \Cref{probICPDou}.

To solve this problem, we use Tikhonov regularization method.
We first construct a Tikhonov functional and prove the uniqueness of the minimizer.
Then, we  prove the convergence rate for the optimization problem.

Define  an operator $\mathcal{P}:  L^{2}_{\mathbb{F}}(0,T;H^{2}(G))\to L^{2}_{\mathbb{F}}(0,T;L^{2}(G))$ as follows:
\begin{align*}
(\mathcal{P} u)(t,x)
\deq & u(t, x)-u(0, x)
\\
&
- \int_0^t  \sum_{j, k=1}^n\big(b^{j k}(s, x) u_{x_{j}}(s, x)\big)_{x_{k}}  d s
\\
&
- \int_0^t\left[\left\langle a_1(s, x), \nabla u(s, x)\right\rangle+a_2(s, x) u(s, x)\right] ds
\\
&
-\int_0^t a_3(s, x) u(s, x) d W(s), \quad \text{${\mathbb{P}}$-a.s.},\quad \forall u\in \mathcal{H}, \;  \mbox{a.e. }(t,x) \in Q.
\end{align*}

Let
\begin{align*}
\mathcal{H} \deq \{ u \in L^{2}_{\mathbb{F}}(0,T;H^{2}(G))
& \mid
\mathcal{P} u  \in L^{2}_{\mathbb{F}}(\Omega; H^{1}(0,T;L^{2}(G))),
\\
& \quad
u|_{\Gamma_{0}} = g_{1}, ~ \partial_{\nu} u |_{\Gamma_{0}} = g_{2} \}.
\end{align*}
If $u$ is a solution to the equation \cref{eqInversPParabolicPartialBoudaryeq}, then
$\mathcal{P} u(t)=\int_{0}^{t} f(s) d s$, which yields this $u$ belong to $\mathcal{H}$. Consequently, $\mathcal{H}\neq\emptyset$.

Denote $ \mathbf{I}f(t) = \int_{0}^{t} f(s) d s $ for $ t \in [0,T] $.
Given a function $ F \in \mathcal{H}$, we construct the Tikhonov functional as follows:
\begin{align}
\label{eqDou2023Fucntional}
\mathcal{J}_{\tau}(u) = | \mathcal{P} u - \mathbf{I}f |^{2}_{L^{2}_{\mathbb{F}}(\Omega;H^{1}(0,T;L^{2}(G)))} + \tau | u - F |^{2}_{L^{2}_{\mathbb{F}}(0,T;H^{2}(G))},
\end{align}
where $ u \in \mathcal{H} $ and $ \tau \in (0,1) $.

We have the following result.

\begin{proposition}
For $ \tau \in (0,1) $, there exists a unique minimizer $ u_{\tau} \in \mathcal{H} $ of the functional $\mathcal{J}_{\tau}(\cdot) $.
\end{proposition}

\begin{proof}
The uniqueness of the minimizer follows from the strong convexity of the functional $\mathcal{J}_{\tau}(\cdot) $. We only need to show the existence of a minimizer.

Let
\begin{align*}
\mathcal{H}_{0} \deq \{ u \in L^{2}_{\mathbb{F}}(0,T;H^{2}(G))
& \mid
\mathcal{P} u  \in L^{2}_{\mathbb{F}}(\Omega; H^{1}(0,T;L^{2}(G))),
\\
& \quad
u|_{\Gamma_{0}} =  \partial u |_{\Gamma_{0}} = 0 \}
.
\end{align*}
Fix $ \tau \in (0,1) $.
Define the inner products as follows:
\begin{align}
\label{eqDou20231}
\langle  \varphi, \psi \rangle_{\mathcal{H}_{0}} =
\langle \mathcal{P} \varphi, \mathcal{P} \psi \rangle_{L^{2}_{\mathbb{F}}(\Omega; H^{1}(0,T;L^{2}(G)))}
+
\tau \langle \varphi, \psi \rangle_{L^{2}_{\mathbb{F}}(0,T;H^{2}(G))}
\end{align}
where $ \varphi, \psi \in \mathcal{H}_{0} $.
Let $\overline{\mathcal{H}}_{0}$ be the completion of $\mathcal{H}_{0}$ with respect to the inner product $\langle \cdot, \cdot \rangle_{\mathcal{H}_{0}}$, and still denoted by $\mathcal{H}_{0}$ for the sake of simplicity.

For $ u \in \mathcal{H} $, let $ v = u - F $. Then $ v \in \mathcal{H}_{0} $.
Consider the following functional
\begin{align*}
\overline{\mathcal{J}}_{\tau}(v) = | \mathcal{P} v + \mathcal{P} F - \mathbf{I}f |^{2}_{L^{2}_{\mathbb{F}}(\Omega;H^{1}(0,T;L^{2}(G)))} + \tau | v |^{2}_{L^{2}_{\mathbb{F}}(0,T;H^{2}(G))},
\end{align*}
where $ v \in \mathcal{H}_{0} $ and $ \tau \in (0,1) $.
Clearly, $ u_{\tau} $ minimizes $\mathcal{J}_{\tau}(u) $ if and only if $ v_{\tau} = u_{\tau} - F $ is the minimizer of the functional $ \overline{\mathcal{J}}_{\tau}(v) $.

By the variational principle, the minimizer $ v_{\tau} $ of the functional $ \overline{\mathcal{J}}_{\tau}(v) $ satisfy the following equation
\begin{align}\notag
\label{eqDou20233}
\langle \mathcal{P} v_{\tau} + \mathcal{P} F - \mathbf{I}f, \mathcal{P} h   \rangle _{L^{2}_{\mathbb{F}}(\Omega;H^{1}(0,T;L^{2}(G)))} \\
+ \tau \langle v_{\tau}, h \rangle _{L^{2}_{\mathbb{F}}(0,T;H^{2}(G))}
=0
,
\quad
\forall ~ h \in \mathcal{H}_{0} .
\end{align}
Conversely, the solution $ v_{\tau} $ of equation \cref{eqDou20233} is the minimizer of the functional $ \overline{\mathcal{J}}_{\tau}(v) $.
In fact, if $ v_{\tau} $ solves \cref{eqDou20233}, then for each $ v \in \mathcal{H}_{0} $, it holds that
\begin{align*}
    & \overline{\mathcal{J}}_{\tau}(v) - \overline{\mathcal{J}}_{\tau}(v_{\tau})
    \\
    & =
    | \mathcal{P} v + \mathcal{P} F - \mathbf{I}f |^{2}_{L^{2}_{\mathbb{F}}(\Omega;H^{1}(0,T;L^{2}(G)))} + \tau | v |^{2}_{L^{2}_{\mathbb{F}}(0,T;H^{2}(G))}
    \\
    & \quad
    - | \mathcal{P} v_{\tau} + \mathcal{P} F - \mathbf{I}f |^{2}_{L^{2}_{\mathbb{F}}(\Omega;H^{1}(0,T;L^{2}(G)))}
    - \tau | v_{\tau} |^{2}_{L^{2}_{\mathbb{F}}(0,T;H^{2}(G))}
    \\
    & =
    | \mathcal{P} v_{\tau} + \mathcal{P} F - \mathbf{I}f + \mathcal{P}(v- v_{\tau})|^{2}_{L^{2}_{\mathbb{F}}(\Omega;H^{1}(0,T;L^{2}(G)))}
    \\
    & \quad
    + \tau | v_{\tau} + (v-v_{\tau}) |^{2}_{L^{2}_{\mathbb{F}}(0,T;H^{2}(G))}
    - | \mathcal{P} v_{\tau} + \mathcal{P} F - \mathbf{I}f |^{2}_{L^{2}_{\mathbb{F}}(\Omega;H^{1}(0,T;L^{2}(G)))}
    \\
    & \quad
    - \tau | v_{\tau} |^{2}_{L^{2}_{\mathbb{F}}(0,T;H^{2}(G))}
    \\
    & =
    2 \langle \mathcal{P} v_{\tau} + \mathcal{P} F - \mathbf{I}f, \mathcal{P} (v-v_{\tau})   \rangle _{L^{2}_{\mathbb{F}}(\Omega;H^{1}(0,T;L^{2}(G)))}
    \\
    & \quad
    + 2 \tau \langle v_{\tau}, v-v_{\tau} \rangle _{L^{2}_{\mathbb{F}}(0,T;H^{2}(G))}
    + | \mathcal{P} (v - v_{\tau}) |^{2}_{L^{2}_{\mathbb{F}}(\Omega;H^{1}(0,T;L^{2}(G)))}
    \\
    & \quad
    + \tau | v-v_{\tau} |^{2}_{L^{2}_{\mathbb{F}}(0,T;H^{2}(G))}
    \\
    & =
    | \mathcal{P} (v - v_{\tau}) |^{2}_{L^{2}_{\mathbb{F}}(\Omega;H^{1}(0,T;L^{2}(G)))}
    + \tau | v-v_{\tau} |^{2}_{L^{2}_{\mathbb{F}}(0,T;H^{2}(G))}
    \\
    & \geq
    0
    ,
\end{align*}
which means $ v_{\tau} $   is a minimizer of the functional $ \overline{\mathcal{J}}_{\tau}(v) $.

From \cref{eqDou20231}, equation \cref{eqDou20231} is equivalent to
\begin{align}
\label{eqDou202333}
\langle v_{\tau}, h \rangle _{\mathcal{H}_{0}} = \langle \mathbf{I}f - \mathcal{P} F , \mathcal{P} h \rangle _{L^{2}_{\mathbb{F}}(\Omega;H^{1}(0,T;L^{2}(G)))}
, \quad
\forall ~ h \in \mathcal{H}_{0}
.
\end{align}

Thanks to Cauchy-Schwarz inequality and \cref{eqDou20231}, for all $ h \in \mathcal{H}_{0} $, we have
\begin{align*}
& \big| \langle \mathbf{I}f - \mathcal{P} F , \mathcal{P} h \rangle _{L^{2}_{\mathbb{F}}(\Omega;H^{1}(0,T;L^{2}(G)))} \big|
\\
&
\leq (| \mathcal{P} F|_{L^{2}_{\mathbb{F}}(\Omega; H^{1}(0,T;L^{2}(G)))} + |f|_{L^{2}_{\mathbb{F}}(0,T;L^{2}(G))} ) |h|_{\mathcal{H}_{0}}
.
\end{align*}
By Riesz representation theorem, for all $ h \in \mathcal{H}_{0} $, there exists $ w_{\tau} \in \mathcal{H}_{0} $ such that
\begin{align*}
\langle w_{\tau}, h \rangle _{\mathcal{H}_{0}} = \langle \mathbf{I}f- \mathcal{P} F , \mathcal{P} h \rangle _{L^{2}_{\mathbb{F}}(\Omega;H^{1}(0,T;L^{2}(G)))}
.
\end{align*}
Hence $ w_{\tau} $ solves \cref{eqDou202333}, which means $ w_{\tau} $ is the minimizer of the functional $ \overline{\mathcal{J}}_{\tau}(\cdot) $.
\end{proof}

Assume that $ y^{*} $ is the true solution to the equation \cref{eqInversPParabolicPartialBoudaryeq} with the exact data
\begin{align*}
\left\{
\begin{aligned}
& f^{*} \in L^{2}_{\mathbb{F}}(0,T;L^{2}(G)),
\\
& y^{\ast }| _{\Gamma_{0}}=g_{1}^{\ast }\in
L^{2}_{\mathbb{F}}(\Omega; H^{1}(0,T; L^{2}(\Gamma_{0}))) \cap L^{2}_{\mathbb{F}}(0,T; H^{1}(\Gamma_{0})),
\\
& \partial
_{\nu}y^{\ast }|_{\Gamma_{0}}=g_{2}^{\ast
}\in L_{\mathbb{F}}^{2}(0,T;L^2(\Gamma_{0})).
\end{aligned}
\right.
\end{align*}
From \cref{thmProb9}, the solution $ y^{*} $ is unique.
Clearly, there exists $ F^{*} \in \mathcal{H} $ such that
\begin{align}
\label{eqDou20237}
| F^{*} |_{L^{2}_{\mathbb{F}}(0,T;H^{2}(G))} \leq \mathcal{C} |y^{*}|_{L^{2}_{\mathbb{F}}(0,T;H^{2}(G))}
.
\end{align}
Assume that the measurement satisfies that
\begin{align}
\label{eqDou2023Assumtion1}
|f^{*} - f|_{L^{2}_{\mathbb{F}}(0,T;L^{2}(G))} \leq \delta, ~
|g_{1}^{\ast
}-g_{1}|_{ H^1(\Gamma) }\leq \delta, ~ |g_{2}^{\ast
}-g_{2}|_{L^2_\mathbb{F}(0,T; L^2(\Gamma))}\leq \delta,
\end{align}
and
\begin{align}
\label{eqDou2023Assumtion2}
| \mathcal{P} F^{*} - \mathcal{P} F |_{L^{2}_{\mathbb{F}}(\Omega;H^{1}(0,T;L^{2}(G)))} +  | F^{*} -F  |_{L^{2}_{\mathbb{F}}(0,T;H^{2}(G))}
\leq
\delta.
\end{align}
We now establish the estimate for the error between the  minimizer $ y_{\tau} $ for the functional $\mathcal{J}_{\tau}(\cdot) $ and the exact solution $ y^{*} $.

\begin{theorem}
Assume \cref{eqDou2023Assumtion1}, \cref{eqDou2023Assumtion2}, and that $ \alpha \in (0,1] $ is a constant.
For any given $\widehat{G} \subset \subset G$ and $\varepsilon \in (0, \frac{T}{2}) $, let $ \gamma \in (0,1) $ as stated in \cref{thmProb9}.
Then there exists a constant $ \mathcal{C} > 0 $  such that  for all $ \delta \in (0,1) $,
\begin{align*}
|y_{\tau} - y^{*}|_{L^{2}_{\mathbb{F}}(\varepsilon,T-\varepsilon;H^{1}(\widehat{G}))}
\leq
\mathcal{C} (1 + |y^{*}|_{L^{2}_{\mathbb{F}}(0,T;H^{2}(G))}) \delta^{\alpha \gamma}
,
\end{align*}
where $ y_{\tau} $ is the minimizer of the functional \cref{eqDou2023Fucntional} with the regularization parameter  $ \tau = \delta^{2 \alpha} $.
\end{theorem}

\begin{proof}

Let $ v^{*} = y^{*} - F^{*} $. Then $ v^{*}  \in \mathcal{H}_{0} $ and $ \mathcal{P} v^{*} = - \mathcal{P}  F^{*} $.
Hence, for all $ h \in \mathcal{H}_{0} $, it holds that
\begin{align}
\label{eqDou20232}
\langle \mathcal{P} v^{*} + \mathcal{P} F^{*} - \mathbf{I}f, \mathcal{P} h   \rangle _{L^{2}_{\mathbb{F}}(\Omega;H^{1}(0,T;L^{2}(G)))}
=  0
.
\end{align}
Subtracting  \cref{eqDou20232} from \cref{eqDou20233} and denoting $ \tilde{v}_{\tau} = v^{*} - v_{\tau} $, $ \tilde{f} = \mathbf{I}f^{*} - \mathbf{I}f $ and $ \tilde{F} = F^{*} - F $, for all $ h \in \mathcal{H}_{0} $, we have that
\begin{align}\label{eqDou20234-2}
\notag& \langle \mathcal{P} \tilde{v}_{\tau}, \mathcal{P} h   \rangle _{L^{2}_{\mathbb{F}}(\Omega;H^{1}(0,T;L^{2}(G)))}
+ \tau \langle \tilde{v}_{\tau}, h \rangle _{L^{2}_{\mathbb{F}}(0,T;H^{2}(G))}
\\
& =
\langle  \tilde{f} - \mathcal{P} \tilde{F}, \mathcal{P} h   \rangle _{L^{2}_{\mathbb{F}}(\Omega;H^{1}(0,T;L^{2}(G)))}
+ \tau \langle v^{*}, h \rangle _{L^{2}_{\mathbb{F}}(0,T;H^{2}(G))}
.
\end{align}
By choosing $ h = \tilde{v}_{\tau} $ in \cref{eqDou20234-2}, we get that
\begin{align}
\notag
\label{eqDou20234}
& |\mathcal{P} \tilde{v}_{\tau}|_{L^{2}_{\mathbb{F}}(\Omega;H^{1}(0,T;L^{2}(G)))}^{2}
+ \tau |  \tilde{v}_{\tau} |^{2} _{L^{2}_{\mathbb{F}}(0,T;H^{2}(G))}
\\
& \leq
| \tilde{f} -  \mathcal{P} \tilde{F}|^{2}_{L^{2}_{\mathbb{F}}(\Omega;H^{1}(0,T;L^{2}(G)))}
+ \tau |v^{*}|^{2}_{L^{2}_{\mathbb{F}}(0,T;H^{2}(G))}
.
\end{align}
Since $ \tau = \delta^{2 \alpha} $ and $ \delta < 1$, it holds that $ \delta^{2} \leq \tau $.
Combining this with \cref{eqDou2023Assumtion2,eqDou20234}, we obtain that
\begin{align}
\label{eqDou20235}
|  \tilde{v}_{\tau} |^{2} _{L^{2}_{\mathbb{F}}(0,T;H^{2}(G))}
 \leq
\mathcal{C} (1 + |v^{*}|^{2}_{L^{2}_{\mathbb{F}}(0,T;H^{2}(G))} )
\end{align}
and that
\begin{align}
\label{eqDou20236}
|\mathcal{P} \tilde{v}_{\tau}|_{L^{2}_{\mathbb{F}}(\Omega;H^{1}(0,T;L^{2}(G)))}^{2}
 \leq
\mathcal{C} (1 + |v^{*}|^{2}_{L^{2}_{\mathbb{F}}(0,T;H^{2}(G))} ) \delta^{2\alpha}.
\end{align}
Let $ w_{\tau} = \tilde{v}_{\tau}\big(  1 + |v^{*}|_{L^{2}_{\mathbb{F}}(0,T;H^{2}(G))} \big)^{-1} $. From \cref{eqDou20235},
\cref{eqDou20235,eqDou20236,thmProb9}, we know that, for $ \delta \in (0, 1) $,
\begin{align*}
| w_{\tau}|_{L^{2}_{\mathbb{F}}(\varepsilon, T-\varepsilon; H^{1}(\widehat{G}))}
\leq
\mathcal{C} \delta^{\alpha \gamma}
.
\end{align*}
This, together  with \cref{eqDou20237}, implies that
\begin{align}
\label{eqDou20238}
| \tilde{v}_{\tau}|_{L^{2}_{\mathbb{F}}(\varepsilon, T-\varepsilon; H^{1}(\widehat{G}))}
\notag
& \leq
\mathcal{C} \big( 1 + |v^{*}|_{L^{2}_{\mathbb{F}}(0,T;H^{2}(G))} \big) \delta^{\alpha \gamma}
\\
& \leq
\mathcal{C} \big( 1 + |y^{*}|_{L^{2}_{\mathbb{F}}(0,T;H^{2}(G))} \big) \delta^{\alpha \gamma}.
\end{align}
Since $ \tilde{v}_{\tau} = v^{*} - v_{\tau} = y^{*} - F^{*} - y_{\tau} + F $, from \cref{eqDou2023Assumtion2}, we have
\begin{align}\notag
\label{eqDou20239}
| \tilde{v}_{\tau}|_{L^{2}_{\mathbb{F}}(\varepsilon, T-\varepsilon; H^{1}(\widehat{G}))}
& \geq
| y_{\tau} -   y^{*} |_{L^{2}_{\mathbb{F}}(\varepsilon, T-\varepsilon; H^{1}(\widehat{G}))}
- |F^{*} - F|_{L^{2}_{\mathbb{F}}(0,T;H^{2}(G))}
\\
& \geq
| y_{\tau} -   y^{*} |_{L^{2}_{\mathbb{F}}(\varepsilon, T-\varepsilon; H^{1}(\widehat{G}))}
- \delta
.
\end{align}
Noting that $ \delta^{\alpha \gamma} \geq \delta $, combining \cref{eqDou20238,eqDou20239}, we complete the proof.
\end{proof}

\section{Further comments}
\label{secFurPro}

To the best of our knowledge, \cite{Lue2012} is the first work employing Carleman estimate in studying inverse problems for stochastic parabolic equations. Subsequently, there have been numerous studies that have also addressed inverse problems for stochastic parabolic equations using Carleman estimates (e.g., \cite{Dou2022,Fu2017,Gong2021,Li2013,Li2022,Li2021,Lue2012,Lue2015b,Niu2020,Wu2020,Wu2022}).
Some of these works are presented in this chapter.
\cref{thm.fIForParablic} is first established in \cite{Tang2009}  and \cref{thm.2.1} is a direct consequence of  \cite[Theorem 5.1]{Tang2009}.
\cref{th consta} is an extension of  \cite[Theorem 1.3]{Lue2012}.
\cref{carleman est1-1} is taken from  \cite{Dou2022}.
\cref{thmISPC} is a generalization of \cite[Theorem 1.1]{Yuan2021}.
The main contents of \Cref{secIllPosedCauchy,secReconstruct} are borrowed from \cite{Dou2024},
while
the main content of \Cref{secReconstructTer} is derived from \cite{Dou2022}.

There are many open problems related to inverse problems for stochastic parabolic equations.
Some of them require new ideas and further development.

\begin{itemize}

\item \emph{Inverse coefficient problems for stochastic parabolic equations}

Carleman estimates are widely employed in studying inverse coefficient problems for deterministic parabolic equations (e.g.,\cite{Yamamoto2009}).
However, there is currently no relevant work on stochastic parabolic equations. The main difficulty lies in  that, by the classical way to utilize Carleman estimate to solve inverse coefficient problem, one needs to take derivative of the solution with respect to the time variable to reduce an inverse coefficient problem to an inverse state problem.  Such a method is not applicable to solutions of stochastic parabolic equations.

\item \emph{Inverse state and source problems for  semilinear stochastic  parabolic equations}

For a general semilinear stochastic parabolic equation \cref{sp-eq1}, can we investigate its inverse   source problem with  boundary  measurements?

\item \emph{Inverse geometric shape of the domain problems}

As we explained in \cref{Ch1-sec2}, in many practical problems, equations may involve unknown boundaries, such as a change of phase, a reaction front, or an unknown population.
Can the unknown boundaries be determined through measurements of the solutions?
This constitutes an important class of inverse problems, extensively studied in deterministic partial differential equations \cite{Bryan1998,Canuto2002,Elschner2020,Bukhgeim1999,Vessella2008}.
As fas as we know, the only result for   stochastic parabolic equations is \cite{Liao2024a}.
Compared with the deterministic counterpart, there are still many open problems in this area.

\item \emph{Inverse the boundary transfer coefficient problems}

In practical applications, the equation \cref{sp-eq1} may be with Robin boundary conditions, specifically $ \frac{\partial y}{\partial \nu} + \gamma(t,x) y = \varphi $ on $ \Sigma $.
In this case, the boundary transfer coefficient $ \gamma $ might also be unknown.
Is it possible to determine both the state and the boundary transfer coefficient using terminal measurements or internal measurements simultaneously?

\item \emph{Inverse problems with many measurements}

In this chapter, when discussing boundary measurements, we typically use Dirichlet data or Neumann data.
See Problems \ref{probP9}, \ref{prob.p7pre} and \ref{probICPDou}. If we know the Neumann data corresponding to all Dirichlet data, that is, if we know the so-called lateral Dirichlet-to-Neumann map, can we obtain better results for determining the state and source term?

\item \emph{Inverse problems with non-local measurements}

Due to inevitable noise, pointwise measurements can be highly sensitive.
In engineering applications, it is a nature way to take an (weighted) average values of various   measurement points.
For example, measurements $ h $ may take the following two non-local forms:
\begin{align*}
    \int_{D} w_{1}(x) y(t, x) d x= h(t),  \quad  t \in (0,T),
\end{align*}
or
\begin{align*}
    \int_{0}^{T} w_{2}(t) y(t, x) d t = h(x),  \quad  x \in \Gamma,
\end{align*}
where $ w_{1}, w_{2} $ are the weight function.
For these types of non-local measurements, due to  less information being given, the inverse problem becomes more challenging.

\item \emph{Efficiency algorithm for the construction of unknowns}

In \Cref{secReconstructTer,secReconstruct}, we discuss reconstructing the solution $y$ from certain  measurements using the Tikhonov regularization approach.
The reconstruction problem is reframed as determining the minimizer of the Tikhonov functional defined by  \cref{eqDou2023Fucntional,eqTikhonovFucntional}. The same approach can be extended to investigate the reconstruction of the unknown state through the internal measurement  and the unknown source by the internal/terminal measurement. We encourage interested readers to explore this.

On the other hand, many algorithms feasible for deterministic problems encounter  essential difficulties when applied to stochastic problems.
For instance, when solving stochastic problems using the conjugate gradient method, the dual system is a backward stochastic parabolic equation, which is highly challenging to  solve numerically.
Indeed, although there are some progresses (e.g., \cite{Lue2022,Prohl2021}), the problem of solving backward stochastic partial differential equations remains far from well understood.
In recent work, \cite{Dou2024} applies kernel-based learning theory to obtain numerical solutions for \cref{probICPDou}.
The article \cite{Dou2022} employs conjugate gradient methods and Picard-type algorithms for numerical solutions of \cref{probRI}.
Further  efficiency algorithm remains to be done.

\end{itemize}

\chapter{Inverse problems for stochastic hyperbolic equations}
\label{ch3}

This chapter is devoted to studying inverse problems for stochastic hyperbolic equations.

Throughout this chapter, let $(\Omega, \mathcal{F}, \mathbf{F}, \mathbb{P})$ with $\mathbf F=\{\mathcal{F}_{t}\}_{t \geq 0}$ be a complete filtered probability space on which a one-dimensional standard Brownian motion $\{W(t)\}_{t \geq 0}$ is defined and $\mathbf{F}$ is the natural filtration generated by $W(\cdot)$.
Write $ \mathbb{F} $ for the progressive $\sigma$-field with respect to $\mathbf{F}$.

Let $T > 0$, and $G \subset \mathbb{R}^{n}$ ($n \in
\mathbb{N}$) be a given bounded domain with a
$C^{2}$ boundary $\Gamma$. Let $\Gamma_0$ be a subset of  $\Gamma$.  Denote by $ \mathbf{L}(G) $ all Lebesgue measurable subsets of $G$. 
Put
\begin{align*}
Q = (0,T) \times G, \quad \Sigma = (0,T) \times \Gamma, \quad \Sigma_0 = (0,T) \times \Gamma_0.
\end{align*}
Also, unless other stated, we denote
by $\nu(x) = (\nu^1(x), \nu^2(x), \cdots, \nu^n(x))$ the unit
outward normal vector of $\Gamma$ at $x\in \Gamma$.  

Let $(b^{jk})_{1\leq j,k\leq n} \in C^1(G;{\mathbb{R}}^{n\times n})$ satisfying that $b^{jk} = b^{kj}$ for all $j,k = 1,\cdots, n$, and for some constant $s_0
> 0$,
\begin{align} \label{eq.bijGeqHyper}
\sum_{j,k=1}^nb^{jk}\xi_{j}\xi_{k} \geq s_0 |\xi|^2,
\quad  \forall\, (x,\xi)\deq (x,\xi_{1}, \cdots,
\xi_{n}) \in G \times \mathbb{R}^{n}.
\end{align}


\section{Formulation of the problems}

\subsection{Inverse state problem with a boundary measurement}

Let $F: [0,T]\times\Omega\times G\times \mathbb{R} 
\times\mathbb{R} \times\mathbb{R}^n\to\ \mathbb{R} $ and $K: [0,T]\times\Omega\times G\times \mathbb{R} \to\ {\mathbb{R}}$ be two functions such that, for each $(\eta,\varrho,\zeta)\in
\mathbb{R} \times\mathbb{R} \times\mathbb{R}^n$, $F(\cdot,\cdot,\cdot,\eta,\varrho,\zeta) : [0,T]\times\Omega \times G\to {\mathbb{R}}$ and $K(\cdot,\cdot,\cdot,\eta) : [0,T]\times\Omega \times G\to {\mathbb{R}}$
are ${\mathbb{F}}\times \mathbf{L}(G)$-measurable; and for a.e. $(t,\omega,x)\in [0,T]\times\Omega\times G$ and any $(\eta_{i},\varrho_{i},\zeta_{i})\in
\mathbb{R} \times\mathbb{R} \times\mathbb{R}^n$ ($i=1,2$),
\begin{align}\label{10.5-eq1}
\left\{
\begin{aligned}
    &|F(t,\omega,x,\eta_1,\varrho_1,\zeta_1)-F(t,\omega,x,\eta_2,\varrho_2,\zeta_2)|
    \\ 
    & \quad \quad  \le
    L(|\eta_1-\eta_2|+|\varrho_1-\varrho_2|+|\zeta_1-\zeta_2|_{\mathbb{R}^n}), \\
&|K(t,\omega,x,\eta_1)-K(t,\omega,x,\eta_2)| \leq
L|\eta_1-\eta_2|,\\
&|F(\cdot,\cdot,\cdot,0,0,0)|\in L^2_{\mathbb{F}}(0,T;L^2(G)), \quad
|K(\cdot,\cdot,\cdot,0)|\in L^2_{\mathbb{F}}(0,T;H^1_0(G))
\end{aligned}
\right.
\end{align}
for some constant $L>0$. 

Consider the following semilinear stochastic hyperbolic equation:
\begin{align}\label{ch-5-4.12-eq1}
\left\{
\begin{aligned}
    & dy_{t} - \sum_{j,k=1}^n(b^{jk}y_{x_j})_{x_k}dt=F(y, y_{t}, \nabla y) dt + K(y) dW(t)&\mbox{ in }Q,\\
  & y=0&\mbox{ on }\Sigma,\\
  & y(0)=y_0, \quad y_{t}(0) = y_{1} &\mbox{ in }G,
\end{aligned}
\right.
\end{align}
where $(y_0, y_{1}) \in L^2_{{\mathcal F}_0}(\Omega; H^{1}_{0}(G) \times L^2(G))$ are unknown random variables.
Thanks to the classical well-posedness result of SPDEs (see \cref{eqWellposed}), the equation \cref{ch-5-4.12-eq1} admits a  unique (mild) solution
\begin{align*}
y  \in {\mathbb{H}}_T \deq  L^2_{{\mathbb{F}}} (\Omega;C([0,T]; H_{0}^1(G)))\cap
L^2_{{\mathbb{F}}} (\Omega;C^{1}([0,T]; L^2(G))).
\end{align*}
Moreover,
\begin{align}
\label{eqEnergyEstimate}
\notag
&|y|_{ L^2_{{\mathbb{F}}} (\Omega;C([0,T]; H_{0}^1(G)))}
+ |y_{t}|_{L^2_{{\mathbb{F}}} (\Omega;C([0,T]; L^2(G)))}
\\ \notag
&\leq 
e^{\mathcal{C}(L^{2} + 1)}
(
    |(y_{0}, y_{1})|_{ L^2_{{\mathcal F}_0}(\Omega; H^{1}_{0}(G) \times L^2(G))}
    + |F(\cdot,\cdot,0,0,0)|_{L^2_{{\mathbb{F}}}(0,T;L^2(G))}  
\\
& \quad \quad \quad \quad \quad 
    + |K(\cdot,\cdot,0)|_{L^2_{{\mathbb{F}}}(0,T;L^2(G))}
)
.
\end{align} 
To simplify notations, we will suppress the time variable $t$ ($\in [0,T]$), the sample point $\omega$ ($\in \Omega$), and/or the space variable $x$ ($\in G$) in the functions if there is no risk of confusion.

Define a map  ${\cal M}_1:  L^2_{{\cal F}_0}(\Omega;H_0^1(G)\times  L^2(G))\to L^2_{{\mathbb{F}}}(0,T;L^2(\Gamma_0))$
as follows:
\begin{align*}
{\cal M}_1(y_0,y_1)= {\frac{\partial y}{\partial \nu}}\Big|_{\Sigma_0},\quad\forall\;(y_0,y_1)\in L^2_{{\cal F}_0}(\Omega;H_0^1(G)\times  L^2(G)),
\end{align*}
where $y$ solves the equation \eqref{ch-5-4.12-eq1}.
From the hidden regularity property of the solution (see \cref{propHidden}), we have  
$\frac{\partial y}{\partial\nu}\in L^2_{\mathbb{F}}(0,T;L^2(\Gamma_0))$, which implies that this map is well-defined.

As mentioned  in \cref{ch2}, for the given initial data $(y_0,y_1)\in
L^2_{{\cal F}_0}(\Omega;$ $H_0^1(G)\times L^2(G))$, finding the value of the
mapping ${\cal M}_1(y_0,y_1)$ is a direct problem, which involves solving the equation \cref{ch-5-4.12-eq1}. On the other hand, finding the inverse mapping of ${\cal M}_1(\cdot,\cdot)$ is an inverse
problem. Furthermore, if $ \mathcal{M}_{1}^{-1} $ is a bounded operator, one can determinate the solution $ y $ to \cref{ch-5-4.12-eq1} using the measurement $ \mathcal{M}_{1}(y_{0}, y_{1}) $ and the mapping from $ \mathcal{M}_{1}(y_{0}, y_{1}) $ to $y$ is continuous. 
Recall that this kind  of problem  is commonly referred to an inverse state problem.

More precisely, we consider the following problem:
\begin{problem}
\label{probICP}
Does there exist a constant ${\cal C} >0$ such that for any initial data $(y_0,y_1), (\hat y_0, \hat y_1)\in L^2_{{\cal F}_0}(\Omega;H_0^1(G)\times L^2(G))$,
\begin{align}
\notag
\label{ch-5-th2eq1}
& |(y_0-\hat y_0, y_1-\hat y_1)|_{L^2_{{\cal F}_0}(\Omega;H_0^1(G)\times
L^2(G))} 
\\
& \leq  {\cal C}|{\cal M}_1(y_0,y_1) -{\cal M}_1(\hat y_0,\hat
y_1)|_{L^2_{{\mathbb{F}}}(0,T;L^2(\Gamma_0))}?
\end{align}

\end{problem}

\subsection{Inverse state and source problem with a boundary measurement and a final time measurement I}\label{Ch3-ssec1.2}

Consider the following equation:
\begin{align}\label{ch-5-system1}
\left\{
    \begin{aligned}
        &dz_{t}  - \sum_{j,k=1}^n (b^{jk}z_{x_j})_{x_k}dt = \big(b_1 z_t+ b_2\cdot\nabla z  + b_3 z  + f\big)dt  
        \\
        & \quad \quad \quad \quad \quad \quad \quad \quad \quad \quad \quad  
        + (b_4 z + g)dW(t) & {\mbox { in }} Q, \\
        &z = 0 & \mbox{ on } \Sigma, \\
        &z(0) = z_0, \quad  z_{t}(0) = z_1 & \mbox{ in } G,
    \end{aligned}
\right .
\end{align}
where the initial data $(z_0, z_1) \in L^2_{{\cal F}_0}(\Omega;H_0^1(G)\times  L^2(G))$, the coefficients $b_j$
$(1 \leq j \leq 4)$ satisfy that
\begin{align} \label{ch-5-aibi}
\left\{
    \begin{aligned}
        &b_1 \in L_{{\mathbb{F}}}^{\infty}(0,T;L^{\infty}(G)), \\
&b_2 \in L_{{\mathbb{F}}}^{\infty}(0,T;L^{\infty}(G;\mathbb{R}^{n})),  \\
&b_3 \in L_{{\mathbb{F}}}^{\infty}(0,T;L^{p}(G))\;\;\hbox{for some }p\in [n,\infty] \text{ and } p> 2,\\
& b_4 \in L_{{\mathbb{F}}}^{\infty}(0,T;L^{\infty}(G)),
    \end{aligned} 
\right.
\end{align}
and the nonhomogeneous terms  $f,g \in L^2_{{\mathbb{F}}}(0,T;L^2(G))$.

Define a map  
\begin{align*}
{\cal M}_2:  L^2_{{\cal F}_0}(\Omega;H_0^1(G)\times  L^2(G)) \times L^2_{{\mathbb{F}}}(0,T;L^2(G))
\\
\to L^2_{{\mathbb{F}}}(0,T;L^2(\Gamma_0)) \times L^2_{{\cal F}_T}(\Omega;H_0^1(G))
\end{align*}
as follows:
\begin{align*}
{\cal M}_2(z_0,z_1,g)= \Big ({\frac{\partial z}{\partial \nu}}\Big|_{\Sigma_0},z(T)\Big ),
\end{align*}
where $z$ is the corresponding solution to \cref{ch-5-system1}.

This gives rise to the following inverse problem, which is discussed in \Cref{secISP}:
\begin{problem}
\label{probISP}
Does ${\cal M}_2(z_0,z_1,g)$ uniquely determine  $(z_0,z_1,g)$?
\end{problem}

Noting that the random force $\int_0^tgdW$ is assumed to cause the random vibration starting from some initial state $(z_0, z_1)$. 
The objective of the inverse source problem (\cref{probISP}) is to determine the unknown random force intensity $g$, the unknown initial displacement $z_0$ and the initial velocity $ z_1$ from the (partial) boundary observation $\left.\frac{\partial z}{\partial \nu}\right|_{\Sigma_0}$ and the measurement on the terminal displacement $z(T)$.

\subsection{Inverse state and source problem with a boundary measurement and a final time measurement  II}

In \Cref{Ch3-ssec1.2}, we consider the problem of determine the source term in the diffusion term based on  the measurement $\Big (\displaystyle   \frac{\partial z}{\partial  \nu}\Big|_{\Sigma_0},z(T)\Big )$.  It was highlighted in Remark \ref{rkStoWave} that these measurements are insufficient for determining the source term in the drift component. Then it naturally raises the question: with additional measurements or more prior information on the source terms, can we ascertain both unknown nonhomogeneous terms in the stochastic hyperbolic equation simultaneously? 
It turns out that such a simultaneous determination is indeed achievable by imposing an assumption akin to the equations for the nonhomogeneous terms.  

Denote $x=(x_1,x')\in {\mathbb{R}}^n$, where $x' = (x_2,\cdots,x_n)\in {\mathbb{R}}^{n-1}$. 
Let  $G=(0,l)\times G'$, where $l>0$ and $G'\subset {\mathbb{R}}^{n-1}$ be a bounded domain with a $C^2$ boundary. 
Consider the following stochastic hyperbolic equation:
\begin{align}\label{systemHyper}
    \left\{
    \begin{aligned} 
    & d z_{t} - \Delta z d t =  ( b_{1} z_{t} + b_{2} \cdot \nabla z + b_{3}z + h R  
    ) dt 
    + (b_{4} z + g )dW(t) &&\text{ in } Q,\\
    & z = 0 && \text{ on } \Sigma,\\
    \end{aligned}
    \right.
\end{align}
where $ b_1, b_{3}, b_{4} \in L_{{\mathbb{F}}}^{\infty}(0,T;W^{1, \infty}(G)) $, $ b_2 \in L_{{\mathbb{F}}}^{\infty}(0,T;W^{1, \infty}(G;\mathbb{R}^{n}))  $, 
$ h \in L^{2}_{\mathbb{F}}(0,T; H^{1}(G')) $, $ g \in L^{2}_{\mathbb{F}}(0,T; H^{1}(G)) $ and $ R \in C^{3}(\overline{Q}) $.
Note that $R$ is a deterministic function and $h$ does not depend on $x_1$.

We propose an inverse source problem that aims to simultaneously identify the two unknown nonhomogeneous terms using both boundary measurements and terminal measurements:
\begin{problem}
\label{probISPHyper}
For a given $R$, determine the  source function $h $ and $ g $ by means of the measurement $\Big (\dfrac{\partial z}{\partial \nu}\Big|_{\Sigma}, ~ \dfrac{\partial z_{x_1}}{\partial \nu}\Big|_{\Sigma},~ z(T),~ z_{t}(T) \Big )$.
\end{problem}
We provide the answer to this problem in \Cref{secISP2}.

\begin{remark}
    It is also quite interesting to consider \cref{probISPHyper}, when
    replacing $ \Delta z $ in \cref{systemHyper} with the general coefficient $ \sum\limits_{j,k=1}^n (b^{jk}z_{x_j})_{x_k} $.
    However, in that case, constructing a suitable Carleman weight function to ensure that the Carleman estimate in \cref{thmCarlemanHyperISP} holds remains an open problem.
\end{remark}

\subsection{Inverse state problem with the internal measurement}

All of the above inverse problems are global; that is, we are expected to obtain all the information about the unknowns.
As a cost, we need to measure enough information.
For instance, to solve \cref{probICP}, we require a boundary condition and the value of $\frac{\partial y}{\partial\nu}$ on $\Sigma_0$ for the time $T$ being large enough and $\Gamma_0$ fulfilling some geometric conditions (see \Cref{secDP} for more details).

It is natural to ask another kind of question, say, if the ability for the measurement is limited, what kind of information for the solution can ne derived? 
In what follows, we formulate one of such problems, in which the boundary condition of the equation is unknown.

Consider the following stochastic hyperbolic equation:
\begin{equation}\label{system1}
\sigma d\tilde z_{t} - \Delta \tilde z dt = \big(\tilde b_1 \tilde z_t + \tilde b_2\cdot\nabla
\tilde z + \tilde b_3 \tilde z \big)dt +  \tilde b_4 \tilde z dW(t) \qquad {\mbox { in }} Q,
\end{equation}
where, $\sigma\in C^1(\overline Q)$ is positive,
\begin{eqnarray} \label{aibi}
&\,& \tilde b_1 \in L_{{\mathbb{F}}}^{\infty}(0,T;L^{\infty}_{\rm loc}(G)),
\qquad \tilde b_2
\in L_{{\mathbb{F}}}^{\infty}(0,T;L^{\infty}_{\rm loc}(G;\mathbb{R}^{n})), \nonumber \\
&\,& \tilde b_3 \in L_{{\mathbb{F}}}^{\infty}(0,T;L^{n}_{\rm loc}(G)), \qquad
\,\, \tilde b_4 \in L_{{\mathbb{F}}}^{\infty}(0,T;L^{\infty}_{\rm loc}(G)).
\end{eqnarray}
We call $\tilde z \in  L_{{\mathbb{F}}}^2 (\Omega; C([0,T];H_{\rm loc}^1(G)))\cap
L_{{\mathbb{F}}}^2 (\Omega; C^{1}([0,T];L^2_{\rm loc}(G))) $ a solution to the equation
(\ref{system1}) if for any $t \in (0,T]$, nonempty open subset $\widehat G \subset\!\subset G$ and
$\eta \in H_0^1(\widehat G)$, it holds that
\begin{align}\label{solution to sysh}
\notag
& \int_{\widehat G} \sigma(t,x)\tilde z_t(t,x)\eta(x)dx - \int_{\widehat G}
\sigma(0,x)\tilde z_t(0,x)\eta(x)dx 
\\ \notag
& -\int_0^t \int_{\widehat G} \sigma_t(s,x)\tilde z_t(s,x)\eta(x)d x d s
\\ \notag
& = \int_0^t \int_{\widehat G}   \big[ -\nabla \tilde z \cdot\nabla\eta  + 
\big(b_1 \tilde z_t  + 
b_2\cdot\nabla \tilde z + b_3 \tilde z\big)\eta(x) \big]d x d s 
\\
& + \int_0^t \int_{\widehat G}   b_4 \tilde z \eta(x) dx dW(s), \quad
{\mathbb{P}}\mbox{-a.s. }
\end{align}

Put $S$ as a $C^2$-hypersurface in $\mathbb{R}^{n}$ with $S \subset\!\subset  G$. 
Let $x_0\in S\setminus \partial S$.
We assume that $S$ divides the open ball $B_{\rho}(x_0)\subset G$, centered at $x_0$ and with radius $\rho$, into two parts ${\cal D}_{\rho}^+$ and ${\cal D}_{\rho}^-$. 
In this case, for any $x\in S$, we denote by $\nu(x)$ the unit normal vector of $S$ at $x$ inward to ${\cal D}_{\rho}^+$.
Let $\tilde z$ be a solution to the equation (\ref{system1}), and $\varepsilon\in (0,T)$.
Consider the following local inverse problem, which is addressed in \Cref{secLIP}:
\begin{problem}
\label{probLocalInverse}
Can $\tilde z$ in ${\cal D}_{\rho}^+\times (\varepsilon,T-\varepsilon)$ be uniquely determined by the values of $\tilde z$ in ${\cal D}_{\rho}^-\times (0,T)$?
\end{problem}

\begin{remark}
\Cref{probLocalInverse} investigates whether a solution to \cref{system1} on a single side of $S$ uniquely determines the solution on the opposite side.
We remark that the number $\varepsilon$ depends on both $\sigma(\cdot)$ and $\rho$, and this number cannot be zero.
In fact, due to the finite propagation speed of solutions \cref{system1}, one needs some time to observe the solutions.
\end{remark}
\begin{remark}
Unlike the stochastic parabolic equation, due to the finite propagation speed of the solution to a stochastic hyperbolic equation, we typically cannot infer the global properties of the solution through  local observations.
\end{remark}

\subsection{Reconstruct the  unknown  state with a boundary measurement}
\label{secReconstruct-hy}

Our objective is to reconstruct the solution $ z $ to the equation \cref{ch-5-system1} with observed data from the lateral boundary. More precisely, consider the following problem:
\begin{problem}\label{Hyrestruct}
Find a function $ z $ solving \cref{ch-5-system1Cauchy} with $ \dfrac{\partial z}{ \partial \nu} = h  $ on $ \Sigma_{0} $ for a given $ h \in L^{2}_{\mathbb{F}}(0, T; L^{2}(\Gamma_{0})) $.
\end{problem}

\medskip

The rest of this paper is organized as follows. 
In \Cref{secFI}, we present some preliminaries, including a fundamental pointwise weighted identity for second order stochastic hyperbolic-like  operators, which is used to derive Carleman estimates for stochastic hyperbolic equations, a hidden regularity property of solutions to stochastic hyperbolic equations and an energy estimate for stochastic hyperbolic equations. 
Sections \ref{secDP}--\ref{secReconstructHyper} are devoted to studying Problems \ref{probICP}--\ref{Hyrestruct}, respectively.   
At last, in  \Cref{secFChyper}, we give some comments for the content in this chapter and some related open problems.     

\section{Some preliminaries}\label{sechyPre}

To obtain Carleman estimates for stochastic hyperbolic equations, this section presents a fundamental weighted identity for second order hyperbolic-like operators.
Furthermore, we establish the hidden regularity of the stochastic hyperbolic equation (\cref{propHidden}), as well as energy estimates (\cref{ch-5-energy ensi}).

\subsection{A fundamental weighted identity}
\label{secFI}

In this subsection, we prove the following result.

\begin{theorem}
\label{thmFi}
Let $\phi \in C^1((0,T)\times\mathbb{R}^n)$, $b^{jk}=b^{kj}\in
C^2(\mathbb{R}^n)$ for $j,k=1,2,\cdots,n$, 
and $\ell,\Psi \in C^2((0,T)\times\mathbb{R}^n)$. 
Let $u$ be an $H^2(G)$-valued, $\mathbf{F}$-adapted process such that $u_t$ is an $L^2(G)$-valued It\^o process. 
Set $\theta = e^\ell$ and $w=\theta u$. 
Then, for a.e.     $x\in G$ and ${\mathbb{P}}$-a.s. $\omega \in \Omega$,
\begin{align}
    \notag
    \label{hyperbolic2}  
    &   \theta \Big ( -2\phi\ell_t w_t + 2\sum_{j,k=1}^n b^{jk}\ell_{x_j} w_{x_k} + \Psi w \Big )
    \Big[ \phi du_t - \sum_{j,k=1}^n (b^{jk}u_{x_j})_{x_k} dt \Big]  
    \\ \notag
    & \quad+\sum_{j,k=1}^n \Big[ \sum_{j',k'=1}^n \Big (
    2b^{jk}b^{j'k'}\ell_{x_{j'}}w_{x_j}w_{x_{k'}} -
    b^{jk}b^{j'k'}\ell_{x_j} w_{x_{j'}}w_{x_{k'}}
    \Big )   
    \\ \notag
    &\quad  - 2 \phi b^{jk}\ell_t w_{x_j} w_t + \phi b^{jk}\ell_{x_j} w_t^2  +
    \Psi b^{jk}w_{x_j} w - \Big({\cal A}\ell_{x_j} +
    \frac{\Psi_{x_j}}{2}\Big)b^{jk}w^2 \Big]_{x_k} dt
    \\ \notag
    &\quad +d\Big\{ \phi\sum_{j,k=1}^n b^{jk}\ell_t w_{x_j} w_{x_k}-
    2\phi\sum_{j,k=1}^n b^{jk}\ell_{x_j}w_{x_k}w_t  + \phi^2\ell_t w_t^2
    - \phi\Psi w_t w 
    \\ \notag 
    & \quad \quad \quad 
    + \Big[ \phi {\cal A}\ell_t +
    \frac{1}{2}(\phi\Psi)_t\Big]w^2 \Big\}   
    \\ \notag
    &   =  \bigg\{ \Big[(\phi^2\ell_{t})_t + \sum_{j,k =1}^n (\phi
    b^{jk}\ell_{x_j})_{x_k} - \phi\Psi \Big]w_t^2   d t
    \\ \notag
    & \quad - 2\sum_{j,k=1}^n \big[(\phi
    b^{jk}\ell_{x_k})_t +
    b^{jk}(\phi\ell_{t})_{x_k}\big]w_{x_j}w_t   d t   
    +\sum_{j,k=1}^n c^{jk}w_{x_j}w_{x_k} d t + Bw^2 d t
    \\
    & \quad 
    + \Big(   -2\phi\ell_tv_t 
    + 2\sum_{j,k=1}^n b^{jk}\ell_{x_j}w_{x_k} + \Psi w
    \Big)^2\bigg\} dt + \phi^2\theta^2 \ell_t(du_t)^2,  
\end{align}
where 
\begin{align}\label{AB1} \notag
    & {\cal A}\deq\phi (\ell_t^2-\ell_{tt})-\sum_{j,k=1}^n
    (b^{jk}\ell_{x_j}\ell_{x_k}-b^{jk}_{x_k}\ell_{x_j}
    -b^{jk}\ell_{x_jx_k})-\Psi,\\ \notag
       & {\cal B}\deq{\cal A}\Psi+(\phi {\cal A}\ell_t)_t\!-\! \sum_{j,k=1}^n\!({\cal A}
    b^{jk}\ell_{x_j})_{x_k}\! +\!\dfrac 12
    \Big[\big(\phi\Psi\big)_{tt}\!-\!\sum_{j,k=1}^n\!
    \big(b^{jk}\Psi_{x_j}\big)_{x_k}\Big],\\  
      & c^{jk}\deq \big(\phi b^{jk}\ell_t\big)_t\! +\!
    \sum_{j',k'=1}^n\! \Big[2b^{jk'}(b^{j'k}\ell_{x_{j'}})_{x_{k'}}
    \!-\!\big(b^{jk}b^{j'k'}\ell_{x_{j'}}\big)_{x_{k'}}\Big] \!+\! \Psi
    b^{jk}.
\end{align}
\end{theorem}

\begin{proof}

By $ \theta = e^{\ell}, w = \theta u $, we derive that 
\begin{align}
\label{eqIDHyper1}
\theta d u_{t} 
= 
d w_{t} - 2 \ell_{t} w_{t} d t 
+ (\ell_{t}^{2} - \ell_{tt}) w d t
.
\end{align}
Let 
\begin{align}\label{eqIDHyper2}
\begin{cases}
    \begin{aligned}
        I & = -2 \phi \ell_{t} w_{t} 
        + 2 \sum_{j, k =1}^{n} b^{j k} \ell_{x_{j}} w_{x_{k}}
        + \Psi w 
        , 
        \\
    I_{2} & = 
    \phi d w_{t}
    - \sum_{j, k =1}^{n} (b^{j k} w_{x_{j}})_{x_{k}} d t  
    + \mathcal{A} w d t
    .
    \end{aligned}
\end{cases}
\end{align}
From \cref{eqIDHyper2,eqIDHyper1,eqIDParabolic1}, we obtain
\begin{align}\label{eqIDHyper2.1}
\theta I  \Big[\phi d u_{t} - \sum_{j, k =1}^{n} (b^{j k} u_{x_{j}})_{x_{k}} d t \Big] = I^{2} d t  + I I_{2}
.
\end{align}

Let us compute the terms in $I I_{2} $ one by one.  

Thanks to It\^o's formula, we obtain that 
\begin{align}\label{eqIDHyper2.2}
- 2  \phi^{2} \ell_{t} w_{t} d w_{t}
=
- d (\phi^{2} \ell_{t} w_{t}^{2}) 
+ \phi^{2} \ell_{t} (d w_{t})^{2} 
+ (\phi^{2} \ell_{t})_{t} w_{t}^{2} d t 
,
\end{align}
\begin{align}\label{eqIDHyper2.3}
&\notag 2 \sum_{j, k =1}^{n} b^{j k} \ell_{x_{j}} w_{x_{k}} \phi d w_{t}
\\& = 
d \bigg(2 \phi \sum_{j, k=1}^{n} b^{j k} \ell_{j} w_{x_{k}}  w_{t} \bigg)
- 2 \sum_{j, k=1}^{n}  ( \phi  b^{j k} \ell_{j}  )_{t} w_{t} w_{x_{k}} d t 
\\  \notag 
& \quad 
-  \sum_{j, k=1}^{n}  ( \phi  b^{j k} \ell_{j} w_{t}^{2} )_{x_{k}} d t 
+  \sum_{j, k=1}^{n}  ( \phi  b^{j k} \ell_{j}  )_{x_{k}} w_{t}^{2}d t 
,
\end{align}
and that
\begin{align}\label{eqIDHyper2.4}
\Psi w \phi d w_{t} 
= 
d(\phi\Psi w w_{t})
- \phi\Psi w_{t}^{2} d t 
- \frac{1}{2} d \big[(\phi\Psi)_{t} w^{2}\big]
+ \frac{1}{2} (\phi\Psi)_{t t} w^{2} d t.
\end{align}
By direct computations, we get that
\begin{align}\label{eqIDHyper2.5}
\notag& 2 \sum_{j, k=1}^{n} \phi \ell_{t} w_{t} (b^{j k} w_{x_{j}})_{x_{k}} d t 
\\ \notag
& = 
 \sum_{j, k=1}^{n} \big[ 2 (\phi \ell_{t} w_{t} b^{j k} w_{t} w_{x_{j}})_{x_{k}}  
- 2  b^{j k} (\phi \ell_{t} )_{x_{k}} w_{t} w_{x_{j}}  
-   ( b^{j k} \phi \ell_{t} w_{x_{k} w_{x_{j}}})_{t}  
\\  
& \quad \quad \quad \quad 
+   ( b^{j k} \phi \ell_{t})_{t} w_{x_{k} w_{x_{j}}} 
\big] d t 
, 
\end{align}
\begin{align}\label{eqIDHyper2.6}
\notag& 2 \sum_{j, k=1}^{n} \sum_{j', k'=1}^{n} b^{j k} \ell_{x_{j}} w_{x_{k}} (- b^{j' k'} w_{x_{j'}})_{x_{k'}} d t 
\\\notag
& = 
 \sum_{j, k=1}^{n} \sum_{j', k'=1}^{n} \big[
    - 2( b^{j k} b^{j' k'} \ell_{x_{j}} w_{x_{k}}  w_{x_{j'}})_{x_{k'}}  
+ 2  ( b^{j k} \ell_{x_{j}})_{x_{k'}} b^{j' k'}  w_{x_{k}}  w_{x_{j'}}  
\\
& \quad \quad \quad \quad \quad \quad 
+  ( b^{j k} b^{j' k'} \ell_{x_{j}} w_{x_{k'}}  w_{x_{j'}})_{x_{k}}  
-  ( b^{j k} b^{j' k'} \ell_{x_{j}})_{x_{k}} w_{x_{k'}}  w_{x_{j'}}  
\big] d t 
,
\end{align}
and 
\begin{align}\label{eqIDHyper2.7}
&   \sum_{j, k=1}^{n} \Psi w  (- b^{j k} w_{x_{j}})_{x_{k}} d t 
\\ \notag
& = 
 \sum_{j, k=1}^{n}   \bigg[
     ( b^{j k}  \Psi w w_{x_{j}} )_{x_{k}}  
    +     b^{j k}      w_{x_{k}}  w_{x_{j}}  
    + \frac{1}{2} ( b^{j k}  \Psi_{x_{k}} w^{2} )_{x_{j}}  
    - \frac{1}{2} ( b^{j k}  \Psi_{x_{k}} )_{x_{j}}  w^{2} 
\bigg] d t 
.
\end{align}
Similarly, from \cref{eqIDHyper2} and It\^o's formula, we have
\begin{align}\label{eqIDHyper2.8}
\notag& I \mathcal{A} w d t
\\
& = 
\sum_{j, k=1}^{n} (b^{j k } \ell_{x_{j}} \mathcal{A} w^{2})_{x_{k}} d t 
- d (\phi \ell_{t} \mathcal{A} w^{2}) 
+ (\phi \mathcal{A} \ell_{t})_{t} w^{2} d t 
\\ \notag
& \quad 
- \sum_{j, k=1}^{n} (b^{j k} \ell_{x_{j}} \mathcal{A})_{x_{k}} w^{2} d t 
+ \Psi \mathcal{A} w^{2} d t   
. 
\end{align}
Combining \cref{eqIDHyper2.1}--\cref{eqIDHyper2.8}, we get \cref{hyperbolic2}.
\end{proof}

\subsection{Hidden regularity property of solutions to stochastic hyperbolic equations}

In this subsection, we derive a hidden regularity property of solutions to the equation \cref{ch-5-4.12-eq1}.

\begin{lemma}\label{ch-5-hidden v} \cite[Lemma 3.1]{Lions1988}
There exists a vector field
$\xi=(\xi_1,\cdots,\xi_n)\in
C^1(\mathbb{R}^n;\mathbb{R}^n)$ such
that $\xi=\nu$ on $\Gamma$.
\end{lemma}
\begin{proof} Since the boundary $\Gamma$ is of class $C^2$, for any fixed  
	point $x_0\in \Gamma$, there exists an open neighborhood $V$ of $x_0$ 
	in $\mathbb{R}^n$ and a function $\varphi: V \to \mathbb{R}$ of class $C^2$
	such that the gradient $\nabla \varphi(x)$ is non-zero for all $x$ in $V$ and $\varphi(x) = 0$ for $x$ in $V\cap\Gamma$. If necessary, we can consider $-\varphi$ instead of $\varphi$ to ensure that $\nu(x_0)\cdot\nabla\varphi(x_0) > 0$. By choosing $V$ sufficiently small, we can also assume that $V\cap\Gamma$ is a connected set. The function $\psi: V \to \mathbb{R}^n$ defined by unit vector $\eta = \nabla\varphi/|\nabla\varphi|$ is of class $C^1$
	and equals the normal vector $\nu$ on $V\cap\Gamma$.
	
	Given that $\Gamma$ is compact, it can be covered by a finite number of neighborhoods $V_1,\cdots, V_m$
	of this type. Let $\eta_1,\cdots,\eta_m$
	denote the corresponding functions. We then have the following conditions:
	\begin{align*}
		\Gamma &\subset \bigcup_{j=1}^m V_j,\\
		\eta_j &= \nu \text{ on } V_j\cap\Gamma, \text{ for } j=1,\cdots,m.
	\end{align*}
	
	We select a compact subset $V_0\subset\subset G$ such that:
	\begin{align*}
		\overline G &\subset \bigcup_{j=0}^m V_j
	\end{align*}
	and define $\eta_0 : V_0 \to \mathbb{R}^n$
	by setting $\eta_0(x) = 0$ for $x\in V_0$. 
	Let $\alpha_0,\cdots,\alpha_m$
	be a $C^2$
	partition of unity corresponding to the covering $V_0, \cdots, V_m$
	of $G$. Then, the vector field $\zeta = \sum\limits_{j=0}^m\alpha_j\eta_j$
	serves as the desired vector field under these conditions.
\end{proof}

\begin{proposition}
\label{propHidden}
For any solution to the equation \eqref{ch-5-4.12-eq1}, we have
\begin{align}\label{ch-5-hidden ine}
    \notag
    &  \Big|\frac{\partial y}{\partial \nu}\Big
    |_{L^2_{{\mathbb{F}}}(0,T;L^2(\Gamma_0))} 
    \\ \notag
    & \leq e^{{\cal C}(L^2 + 1)}
    \Big(|(y_0,y_1)|_{L^2_{{\cal F}_0}(\Omega; H_0^1(G)\times L^2(G))}+ |F(\cdot,\cdot,0,0,0)|_{L^2_{{\mathbb{F}}}(0,T;L^2(G))} 
    \\
    &  \quad \quad \quad \quad \quad 
    + |K(\cdot,\cdot,0)|_{L^2_{{\mathbb{F}}}(0,T;L^2(G))}\Big).
\end{align}
\end{proposition}
\begin{proof}
For any $ \rho = (\rho^{1}, \cdots, \rho^{n}) \in C^{1}(\mathbb{R}^{1+n}; \mathbb{R}^{n})$, by a direct computation, we obtain 
\begin{align}
    \notag
    \label{eqHid1}
    & \sum_{j,k=1}^{n} \Bigl(
        2 \rho\cdot \nabla y b^{jk} y_{x_{j}} 
        - \sum_{i=1}^{n} \rho^{k} b^{i j} y_{x_{i}} y_{x_{j}}
    \Bigr)_{x_{k}}
    \\
    & =
    2 \sum_{j,k=1}^{n} \Bigl(
         \rho\cdot \nabla y (b^{jk} y_{x_{j}})_{x_{k}} 
        + \sum_{i=1}^{n} \rho^{i}_{x_{k}} b^{j k} y_{x_{i}} y_{x_{j}}
    \Bigr) 
    - \sum_{i, j, k=1}^{n} (\rho^{k} b^{i j})_{x_{k}} y_{x_{i}} y_{x_{j}}
    .
\end{align}
Thanks to It\^o's formula, it holds that 
\begin{align}
\label{eqHid2}\notag
2 \rho \cdot \nabla y d y_{t} 
& = 
2 d (y_{t} \rho \cdot \nabla y)
- 2 y_{t} \rho_{t} \cdot \nabla y d t 
- 2 \rho y_{t} \nabla y_{t}
\\
& = 
2 d (y_{t} \rho \cdot \nabla y)
- 2 y_{t} \rho_{t} \cdot \nabla y d t 
- \sum_{j=1}^{n} (\rho^{j} y_{t}^{2})_{x_{j}} d t 
+ \sum_{j=1}^{n} \rho^{j}_{x_{j}} y_{t}^{2} d  t 
.
\end{align}
From \cref{eqHid1,eqHid2}, we get  
\begin{align}
\notag
\label{eqHid3}
& \sum_{k=1}^{n} \Bigl(
    2 \rho\cdot \nabla y \sum_{j=1}^{n} b^{jk} y_{x_{j}} 
    + \rho^{k} y_{t}^{2}
    - \sum_{i, j=1}^{n} \rho^{k} b^{i j} y_{x_{i}} y_{x_{j}}
\Bigr)_{x_{k}} d t
\\ \notag
& =
2 \rho\cdot \nabla y \Big[ 
    \sum_{j,k=1}^{n} (b^{j k} y_{x_{j}})_{x_{k}} d t 
    - d y_{t}
\Big]
+ 2 d (y_{t} \rho \cdot \nabla y)
- 2 y_{t} \rho_{t} \cdot \nabla y d t 
\\
& \quad 
+  \sum_{j=1}^{n} \rho^{j}_{x_{j}} y_{t}^{2} d t 
+  \sum_{i, j, k=1}^{n} \bigl(
    2 b^{j k} \rho^{i}_{x_{k}} y_{x_{j}} y_{x_{i}} d t 
    - (b^{i j} \rho^{k})_{x_{k}} y_{x_{j}} y_{x_{i}} 
\bigr)d t 
.
\end{align} 
Since $ \Gamma \in C^{2} $, by \cref{ch-5-hidden v}, there exists a vector field $ \zeta \in C^{1}(\mathbb{R}^{n}; \mathbb{R}^{n}) $ such that $ \zeta = \nu  $ on $ \Gamma $. 
Letting $ \rho = \zeta $, integrating \cref{eqHid3} in $ Q $, and taking expectation on $ \Omega $, recalling \cref{ch-5-4.12-eq1}, we deduce that 
\begin{align}
\notag
\label{eqHid4}
& \mathbb{E} \int_{\Sigma} \sum_{k=1}^{n} \Bigl(
    2 \rho\cdot \nabla y \sum_{j=1}^{n} b^{jk} y_{x_{j}} 
    + \rho^{k} y_{t}^{2}
    - \sum_{i, j=1}^{n} \rho^{k} b^{i j} y_{x_{i}} y_{x_{j}}
\Bigr)_{x_{k}} \nu^{k} d \Gamma d t 
\\ \notag
& =
2 \mathbb{E} \int_{G}  y_{t}(T) \rho(T) \cdot \nabla y(T)  dx
- 2 \mathbb{E} \int_{G}  y_{1} \rho(0) \cdot \nabla y_{0}  dx
\\ \notag
& \quad 
+ \mathbb{E} \int_{Q} \Big[ 
    - 2 \rho\cdot \nabla y  F(y, y_{t}, \nabla y)
    - 2 y_{t} \rho_{t} \cdot \nabla y 
    + \sum_{j=1}^{n} \rho^{j}_{x_{j}} y_{t}^{2}  
\\
& \quad \quad \quad \quad \quad 
+  \sum_{i, j, k=1}^{n} \bigl(
    2 b^{j k} \rho^{i}_{x_{k}} y_{x_{j}} y_{x_{i}}  
    - (b^{i j} \rho^{k})_{x_{k}} y_{x_{j}} y_{x_{i}} 
    \bigr)  
\Big] d x d t 
.
\end{align}
Noting that $ y = 0 $ and $ \rho = \nu $ on $ \Gamma $, we have 
\begin{align}
\notag
\label{eqHid5}
& \mathbb{E} \int_{\Sigma} \sum_{k=1}^{n} \Bigl(
    2 \rho\cdot \nabla y \sum_{j=1}^{n} b^{jk} y_{x_{j}} 
    + \rho^{k} y_{t}^{2}
    - \sum_{i, j=1}^{n} \rho^{k} b^{i j} y_{x_{i}} y_{x_{j}}
\Bigr)_{x_{k}}  \nu^{k} d \Gamma d t 
\\
& = 
\mathbb{E} \int_{\Sigma}  \sum_{j, k=1}^{n}   b^{j k} \nu^{j} \nu^{k}  \bigg| \frac{\partial y}{\partial \nu} \bigg|^{2} d \Gamma d t 
\geq 
s_{0} \mathbb{E} \int_{\Sigma}    \bigg| \frac{\partial y}{\partial \nu} \bigg|^{2} d \Gamma d t 
.
\end{align}
Combining \cref{eqHid4,eqHid5,eqEnergyEstimate}, the desired estimate \cref{ch-5-hidden ine} holds.
\end{proof}

Let 
\begin{align}\label{9.29-eq1}
\begin{cases}\displaystyle
r_1 \mathop{\buildrel\Delta\over=} |b_1|_{L_{{\mathbb{F}}}^{\infty}(0,T;L^{\infty}(G))} +
|b_2|_{L_{{\mathbb{F}}}^{\infty}(0,T;L^{\infty}(G;\mathbb{R}^{n}))} +
|b_4|_{L_{{\mathbb{F}}}^{\infty}(0,T;L^{\infty}(G))},\\
r_2 \mathop{\buildrel\Delta\over=} |b_3|_{L_{{\mathbb{F}}}^{\infty}(0,T;L^{p}(G))}.
\end{cases}
\end{align}
From \cref{propHidden}, we immediately have the following result.
\begin{corollary}\label{ch-5-hidden r-1}
For any solution to the equation \cref{ch-5-system1}, it holds
that
\begin{align}\label{ch-5-hidden ine-1}
\notag
\Big|\frac{\partial z}{\partial \nu}\Big
|_{L^2_{{\mathbb{F}}}(0,T;L^2(\Gamma_0))} 
\leq e^{{\cal C}(r_1^2 + r_2^2 +1)}
&  \Big(|(z_0,z_1)|_{L^2_{{\cal F}_0}(\Omega; H_0^1(G)\times L^2(G))}
\\
&+ |f|_{L^2_{{\mathbb{F}}}(0,T;L^2(G))}+
|g|_{L^2_{{\mathbb{F}}}(0,T;L^2(G))}\Big ).
\end{align}
\end{corollary}

\subsection{An energy estimate for stochastic hyperbolic equations}

In this subsection, we derive an energy estimate for \cref{ch-5-system1}.

\begin{proposition}\label{ch-5-energy ensi} 
For any $z$ solving the equation \cref{ch-5-system1} and any $0\leq s,t\leq T$, it holds that
\begin{align}\label{ch-5-en esti}
    \notag
    & {\mathbb{E}}\int_G\big( |z_t(t,x)|^2 + |\nabla
    z(t,x)|^2 \big) dx 
    \\ \notag
    & \leq \displaystyle e^{{\cal C}\big(r_1^2 + r_2^{\frac{1}{2-n/p}}+1\big)T}
    {\mathbb{E}}\int_G \big(|z_t(s,x)|^2 + |\nabla z(s,x)|^2 \big) dx
    \\
    & \quad 
    + {\cal C}{\mathbb{E}}\int_0^T\int_G \big(f(\tau,x)^2 + g(\tau,x)^2\big)
    dxd\tau.
\end{align}
\end{proposition}

\begin{proof}
    \emph{Step 1.} 
Without loss of generality, we consider the case $ t \leq s $.
Put
\begin{align*}
{\cal E}(t) = {\mathbb{E}}\int_G\big( |z_t(t,x)|^2 + |\nabla
z(t,x)|^2 +  r_2^{\frac{2}{2-n/p}} |z(t,x)|^2\big)
dx.
\end{align*}
Noting that $ z = 0 $ on $ \Sigma $, from Poincar\'e's inequality, we obtain 
\begin{equation}\label{en eq0}
 {\cal E}(t) \leq \mathcal{C} \Big (r_2^{\frac{2}{2-n/p}}
+1\Big ){\mathbb{E}}\int_G\big( |z_t(t)|^2 + |\nabla z(t)|^2
\big) dx.
\end{equation}
Thanks to It\^o's formula, we get 
\begin{align*}
d\big(z^2_t + r_{2}^{\frac{2}{2-n/p}} z^{2}\big) = 2 z_t d z_t + (dz_t)^2 + 2 r_{2}^{\frac{2}{2-n/p}} z z_{t}.
\end{align*}
This, together with \cref{ch-5-system1}, implies 
\begin{align}\label{en eq1}
\notag
& {\mathbb{E}}\int_G  \Big(|z_t(s,x)|^2 + r_2^{\frac{2}{2-n/p}}|z(s,x)|^2 
+ \sum_{j,k=1}^{n}b^{j k}(x) z_{x_{j}}(s,x) z_{ x_{j} }(s,x)  
\Big) dx 
\\ \notag
& \quad 
-  {\mathbb{E}}\int_G \Big(|z_t(t,x)|^2 + r_2^{\frac{2}{2-n/p}}|z(t,x)|^2
+ \sum_{j,k=1}^{n}b^{j k}(x) z_{x_{j}}(t,x) z_{ x_{j} }(t,x) 
\Big) dx  
\\ \notag
&  =  
 {\mathbb{E}}\int_t^s\int_G z_t(\tau,x)\big(b_1(\tau,x) z_t(\tau,x) 
+  b_2(\tau,x)\cdot\nabla z(\tau,x)  
\\ \notag
&   \quad \quad \quad \quad \quad \quad \quad \quad \quad 
+ b_3(\tau,x) z(\tau,x) + f(\tau,x) \big)dxd\tau 
\\ \notag
& \quad 
+ {\mathbb{E}}\int_t^s\int_G \big(b_4(\tau,x) z(\tau,x)
+ g(\tau,x) \big)^2 dxd\tau 
\\
&  \quad 
+ 2r_2^{\frac{2}{2-n/p}}{\mathbb{E}}\int_t^s\int_G z_t(\tau,x)z(\tau,x)dxd\tau.
\end{align}
From \cref{en eq1,9.29-eq1}, we see that 
\begin{align}
    \label{eqEnergyEstimateE1}
\notag
& {\mathbb{E}}\int_G  \Big(|z_t(s,x)|^2 + r_2^{\frac{2}{2-n/p}}|z(s,x)|^2 
+ \sum_{j,k=1}^{n}b^{j k}(x) z_{x_{j}}(s,x) z_{ x_{j} }(s,x)  
\Big) dx 
\\ \notag
& \quad 
-  {\mathbb{E}}\int_G \Big(|z_t(t,x)|^2 + r_2^{\frac{2}{2-n/p}}|z(t,x)|^2
+ \sum_{j,k=1}^{n}b^{j k}(x) z_{x_{j}}(t,x) z_{ x_{j} }(t,x) 
\Big) dx  
\\ \notag
&  \leq 
\mathcal{C} (r_{1}^{2}+1) {\mathbb{E}} \int_{t}^{s} \int_G\Big(|z_t(\tau,x)|^2 + |\nabla
z(\tau,x)|^2 +  r_2^{\frac{2}{2-n/p}} |z(\tau,x)|^2\Big) d x d \tau  
\\ \notag
& \quad 
+ 2{\mathbb{E}}\!\int_t^s\!\int_G\big( f(\tau,x)^2\! + \! g(\tau,x)^2 \big)dxd\tau 
+   \mathbb{E}\int_t^s\!\int_G\! b_3(\tau,x)
z(\tau,x)z_t(\tau,x)dxd\tau 
\\
& \quad 
+  2r_2^{\frac{2}{2-n/p}}{\mathbb{E}}\int_t^s\int_G z_t(\tau,x)z(\tau,x)dxd\tau
.
\end{align}

\emph{Step 2.} In this step, we estimate the last two terms in \cref{eqEnergyEstimateE1}.

Let  $p_1 = \frac{2p}{n-2}$ and $p_2 = \frac{2p}{p-n}$. It is easy to check that
\begin{align*}
\frac{1}{p} + \frac{1}{p_1} + \frac{1}{p_2} +
\frac{1}{2} = 1,
\end{align*}
which, together with H\"{o}lder's inequality, implies that 
\begin{align}\label{en eq2.1}
\notag
& \Big|{\mathbb{E}}\int_{G}b_3(\tau,x) z(\tau,x)z_t(\tau,x)dx\Big|
\\ \notag
& 
 \leq{\mathbb{E}}\int_{G}|b_3(\tau,x)|
|z(\tau,x)|^{\frac{n}{p}}|z(\tau,x)|^{1-\frac{n}{p}}|z_t(\tau,x)|dx
\\ \notag
& 
 \leq r_2{\mathbb{E}}
\Big (\big||z(\tau,\cdot)|^{\frac{n}{p}}\big|_{L^{p_1}(G)}
\big||z(\tau,\cdot)|^{1-\frac{n}{p}}\big|_{L^{p_2}(G)}\big|
z_t(\tau,\cdot) \big|_{L^2(G)} \Big )
\\ \notag
& 
= r_2 {\mathbb{E}} \Big ( \big|
z(\tau,\cdot)\big|^{\frac{n}{p}}_{L^{\frac{2 n}{n-2}}(G)}
 \big|
z(\tau,\cdot)\big|^{1-\frac{n}{p}}_{L^{2}(G)}\big|z_t(\tau,\cdot)
\big|_{L^2(G)} \Big )
\\
 &= r_2^{\frac{1}{2-n/p}}{\mathbb{E}} \Big ( \big|
z(\tau,\cdot)\big|^{\frac{n}{p}}_{L^{\frac{2 n}{n-2}}(G)}
r_2^{\frac{1-n/p}{2-n/p}}\big|
z(\tau,\cdot)\big|^{1-\frac{n}{p}}_{L^{2}(G)}\big|z_t(\tau,\cdot)
\big|_{L^2(G)} \Big ).
\end{align}
Noting that 
\begin{align*}
\frac{1}{2(n/p)^{-1}} +
\frac{1}{2(1-n/p)^{-1}} + \frac{1}{2} = 1,
\end{align*}
from Sobolev's embedding theorem, we find that
\begin{align}
\label{eqEnergyH1}
\notag
\big| z(\tau,\cdot)\big|^{\frac{n}{p}}_{L^{\frac{2 n}{n-2}}(G)}
& \leq 
\mathcal{C} |\nabla z (\tau, \cdot)|^{\frac{n}{p}}_{L^{2}(G)}
\\
& \leq
\mathcal{C} \Big[\int_G\Big ( |z_t(\tau,x)|^2 + |\nabla z(\tau,x)|^2
+ r_2^{\frac{2}{2-n/p}} |z(\tau,x)|^2\Big ) dx
\Big]^{\frac{n}{2p}}.
\end{align}
Clearly, it holds that 
\begin{align}
\notag
\label{eqEnergyH2}
 & r_2^{\frac{1-n/p}{2-n/p}}\big|
z(\tau,\cdot)\big|^{1-\frac{n}{p}}_{L^{2}(G)} 
\\ \notag
& \leq 
\Big(\int_{G} r_2^{\frac{2}{2-n/p}} |z(\tau, x)|^{2} d x  \Big)^{\frac{1}{2}-\frac{n}{2p}}
\\
& \leq
\Big[\int_G\Big ( |z_t(\tau,x)|^2 + |\nabla z(\tau,x)|^2 +
r_2^{\frac{2}{2-n/p}} |z(\tau,x)|^2\Big ) dx
\Big]^{\frac{1}{2}-\frac{n}{2p}} 
\end{align}
and that
\begin{align}
\label{eqEnergyH3}
\big|z_t(\tau,\cdot) \big|_{L^2(G)} \leq
\Big[\int_G\Big ( |z_t(\tau,x)|^2 + |\nabla z(\tau,x)|^2 +
r_2^{\frac{2}{2-n/p}} |z(\tau,x)|^2\Big ) dx
\Big]^{\frac{1}{2}}.
\end{align}
From \cref{en eq2.1,eqEnergyH2,eqEnergyH3,eqEnergyH1}, we obtain that
\begin{align}
\label{en eq2.2}
\Big|{\mathbb{E}}\int_{G}b_3(\tau,x)
z(\tau,x)z_t(\tau,x)dx\Big| \leq \mathcal{C}
r_2^{\frac{1}{2-n/p}}{\cal E}(\tau).
\end{align}
By Cauchy-Schwartz's inequality, we have 
\begin{align}\notag
\label{en eq2.3}
r_2^{\frac{2}{2-n/p}}{\mathbb{E}}\int_G
z(\tau,x)z_t(\tau,x)dx 
&\leq
\frac{1}{2}r_2^{\frac{1}{2-n/p}}{\mathbb{E}}\int_G\Big (
r_2^{\frac{2}{2-n/p}}z^2(\tau,x) +
z_t^2(\tau,x)\Big )dx 
\\ 
&
\leq
\frac{1}{2}r_2^{\frac{1}{2-n/p}}{\cal E}(\tau).
\end{align}

\emph{Step 3.}
Combining \cref{en eq2.2,en eq2.3,eqEnergyEstimateE1}, we conclude that 
\begin{align}\label{en eq3}
\notag
 {\cal E}(t)   \leq \mathcal{C} \Big[ &{\cal E}(s) +  \Big (r_1^2 +
r_2^{\frac{1}{2-n/p}}+1\Big ) \int_t^s {\cal E}(\tau)d\tau
\\
& +
\mathbb{E}\int_t^s\int_G\big( f(\tau,x)^2 + g(\tau,x)^2
\big)dxd\tau\Big].
\end{align}
From \cref{en eq3} and backward Gronwall's inequality, we find that 
\begin{align}
\label{en eq4}
{\cal E}(t) \leq e^{\mathcal{C}\Big (r_1^2 +
r_2^{\frac{1}{2-n/p}}+1\Big )(s-t)}{\cal E}(s) +
\mathcal{C}{\mathbb{E}}\int_t^s\int_G\big( f(\tau,x)^2 + g(\tau,x)^2
\big)dxd\tau.
\end{align}
Combining \cref{en eq4,en eq0}, we get 
\begin{align*}
& {\mathbb{E}}\int_G  \big(|z_t(t,x)|^2 + |\nabla z(t,x)|^2
\big) dx \\
& \leq \mathcal{C} e^{\mathcal{C}\Big (r_1^2 +
r_2^{\frac{1}{2-n/p}}+1\Big )(s-t)} {\mathbb{E}}\int_G
\big(|z_t(s,x)|^2 + |\nabla z(s,x)|^2
\big) dx
\\
& \quad +
\mathcal{C}{\mathbb{E}}\int_t^s\int_G\big( f(\tau,x)^2 + g(\tau,x)^2
\big)dxd\tau,
\end{align*}
which implies \cref{ch-5-en esti} immediately.
\end{proof}

\section{Solution to inverse state problem with the bo\-undary measurement}
\label{secDP}

In this section, we address \cref{probICP}. Prior to unveiling the solution to \cref{probICP}, certain assumptions regarding $(b^{jk})_{1\leq j, k \leq n}$, $T$, and $\Gamma_{0}$ are required.

\begin{condition}\label{condition of d}
There exists a positive function $\psi(\cdot) \in C^2(\overline{G})$
satisfying the following:

{\rm (1)} For some constant $\mu_0 > 0$ and any $ (x,\xi) \in     \overline{G} \times \mathbb{R}^n $, it holds that
\begin{align*}
    \sum_{j,k=1}^n\Big\{ \sum_{j',k'=1}^n\Big[
2b^{jk'}(b^{j'k}\psi_{x_{j'}})_{x_{k'}} -
b^{jk}_{x_{k'}}b^{j'k'}\psi_{x_{j'}} \Big] \Big\}\xi_{j}\xi_{k} \geq
\mu_0 \sum_{j,k=1}^nb^{jk}\xi_{j}\xi_{k} .
\end{align*}

{\rm (2)} There is no critical point of $\psi(\cdot)$ in
$\overline{G}$, i.e.,
\begin{align*}
    \min_{x\in \overline{G} }|\nabla \psi(x)| > 0.
\end{align*}
\end{condition}

\begin{remark}
\cref{condition of d} can be viewed as a form of pseudo-convex condition
(e.g., \cite[Section 28.3]{Hormander}). 
Further exploration on this subject can be found in \cite{Liu2013}.
It is worth noting that if  $(b^{jk})_{1\leq j,k\leq n}$ is the identity matrix, then $\psi(x)=|x-x_0|^2$ satisfies
\cref{condition of d}, where $x_0$ is any point belongs to $\mathbb{R}^n\setminus\overline G$. 
If $c\in C^1(\overline G)$  satisfies $c(x)\neq 0$ for any $x\in\overline G$ and
\begin{align*}
    \big(x-x_0\big)\cdot\nabla c^{-2}(x) \geq \mu_1,\quad\forall\;
x\in\overline G,
\end{align*}
for some  $\mu_1>0$, then the same function $\psi(x)=|x-x_0|^2$ also meets the condition \cref{condition of d} when $(b^{jk})_{1\leq j,k\leq n}=c^2I$
(e.g., \cite[Theorem 1.10.2]{Beilina2012}).  
\end{remark}

One can verify  that if $\psi(\cdot)$ satisfies \cref{condition of d}, then for any constants $a \geq 1$ and $b \in \mathbb{R}$, the function $\tilde{\psi} = a\psi + b$ also satisfies this condition with $\mu_0$ replaced by $a\mu_0$.
Therefore, we can choose $\psi$, $\mu_0$, $c_0>0$, $c_1>0$ and $T$  such that the following
condition holds:
\begin{condition}
\label{ch-5-condition2}
\begin{align*}
\left\{
    \begin{aligned}
       & \frac{1}{4}\sum_{j,k=1}^nb^{jk}(x)\psi_{x_j}(x)\psi_{x_k}(x) \geq
        R^2_1\deq\max_{x\in\overline G}\psi(x)\geq R_0^2 \deq
        \min_{x\in\overline G}\psi(x),\quad \forall x\in \overline G,\\ 
        &   T> T_0\deq2 R_1, \quad \Big (\frac{2R_1}{T}\Big )^2<c_1<\frac{2R_1}{T},\quad \mu_0 - 4c_1 -c_0 > 0.
    \end{aligned}
\right.
\end{align*}
\end{condition}

Put 
\begin{align}\label{ch-5-def gamma0}
	\Gamma_0 \deq \Big\{ x\in \Gamma \;\Big|\;
	\sum_{j,k=1}^nb^{jk}(x)\psi_{x_j}(x)\nu^{k}(x)
	> 0 \Big\}.
\end{align}

We have the following positive answer to  \cref{probICP}.
\begin{theorem}\label{th2} 
Assume \cref{condition of d} and \cref{ch-5-condition2} are satisfied.
Then there exists a constant ${\cal C} >0$ such that for any initial data $(y_0,y_1),  (\hat
y_0, \hat y_1)\in L^2_{{\cal F}_0}(\Omega;H_0^1(G)\times L^2(G))$, the
inequality \cref{ch-5-th2eq1} holds for $\Gamma_0$ given by  \cref{ch-5-def gamma0}.
\end{theorem}

\begin{remark} \cref{th2} shows that the state $y(t)$ of \cref{ch-5-4.12-eq1} (for $t\in [0,T]$) can be uniquely determined by the measurement $\displaystyle{\frac{\partial y}{\partial \nu}} \Big|_{\Sigma_0}$, ${\mathbb{P}}$-a.s., and continuously depends on it. 
Therefore, it provides an affirmative answer to the inverse state problem of the system \eqref{ch-5-4.12-eq1}.
\end{remark}

In what follows, for $\lambda \in \mathbb{R}$ and $r\in\mathbb{N}$, we  use $O(\lambda^r)$ to denote a function of order $\lambda^r$ for large $\lambda$.

First, we derive a boundary observability estimate for the equation
\cref{ch-5-system1}.
\begin{theorem}
\label{ch-5-observability}
Assume  \cref{condition of d} and \cref{ch-5-condition2} are
satisfied. For any solution $z$ to the equation \cref{ch-5-system1}, it holds that 
\begin{align} \label{ch-5-obser esti2}
\notag
& |(z_0,z_1)|_{L^2_{{\cal F}_0}(\Omega; H_0^1(G)\times L^2(G))}\\
 & \leq
e^{{\cal C}\big(r_1^2 + r_2^{\frac{1}{ 3/2 - n/p}}+1\big)}
\Big(\Big|\frac{\partial z}{\partial \nu}\Big
|_{L^2_{{\mathbb{F}}}(0,T;L^2(\Gamma_0))} + |f|_{L^2_{{\mathbb{F}}}(0,T;L^2(G))}  + |g|_{L^2_{{\mathbb{F}}}(0,T;L^2(G))}\Big).
\end{align}
\end{theorem}
\begin{proof}
The proof is divided into four steps.

\emph{Step 1.} In this step, we introduce a cut-off function and reduce the equation \cref{ch-5-system1} to an equation with both initial and final data equal zero.
 
Put
\begin{align*}
\ell(t,x)=\lambda\Big[\psi(x)-c_1\Big (t-\frac{T}{2}\Big )^{2}\Big],\quad (t,x)\in Q.
\end{align*}
From \cref{ch-5-condition2}, there exists $\varepsilon_1\in (0,1/2)$ such that for all 
$ (t,x) \in \big[ \big (0, \frac{T}{2} - \varepsilon_1 T\big )\bigcup \big (\frac{T}{2}+\varepsilon_1 T,T\big )  \big]\times G $, it holds that 
\begin{align}\label{ch-5-e1}
\ell(t,x) 
\leq \lambda [R_{1}^{2} - c_{1} (\varepsilon_{1} T)^{2}]
< 0.
\end{align}
Thanks to \cref{ch-5-condition2} again, noting that
\begin{align*}
\ell\Big ( \frac{T}{2},x\Big ) = \lambda\psi(x)\geq \lambda R_0^2, \qquad \forall \; x\in G,
\end{align*}
there exists $\varepsilon_{0}\in (0, \varepsilon_1)$ such that  
\begin{align}
    \label{eqHyperE1}
\ell(t,x)\geq \frac{\lambda R_0^2}{2},\quad \quad  \forall \; (t,x)\in
\Big (\frac{T}{2}-\varepsilon_0 T,\frac{T}{2}+\varepsilon_0 T\Big ) \times G.
\end{align}

Choose a cut-off function $ \chi \in C_{0}^{\infty}[0,T] $ satisfying 
\begin{align}
\label{eqChi}
\chi(t)=1, \quad \quad  \forall \; t \in  \Big (\frac{T}{2}-\varepsilon_1 T,\frac{T}{2}+\varepsilon_1 T\Big ).
\end{align}

Put $u=\chi z$.
Then, $u$ solves the following equation:
\begin{align}{\label{ch-5-system2}}
\left\{
\begin{aligned}
& du_{t} - \sum_{j,k=1}^n(b^{jk}u_{x_j})_{x_k}dt = \big(b_1 u_t +
b_2\cdot\nabla u + b_3 u + \tilde{f} \big)dt 
\\
& \quad \quad \quad \quad \quad \quad \quad \quad \quad \quad \quad 
+ \big(b_4 u + \chi g\big)dW(t)
& {\mbox { in }} Q,
\\
& 
u = 0 & \mbox{ on } \Sigma, 
\\
& 
u(0) = u(T) = 0,\; u_{t}(0) = u_t(T) = 0 & \mbox{ in } G,
\end{aligned}
\right.
\end{align}
where $\tilde{f} =\chi f+\chi_{tt}z + 2\chi_t z_t - b_1\chi_t z$.

\medskip

\emph{Step 2.} In this step, we give the weight function of the Carleman estimate and compute the order of the terms in the equality \cref{hyperbolic2} with respect to $\lambda$. 

Thanks to \cref{thmFi}, choosing 
\begin{align}
 \label{1.27-eq2}
\phi\equiv 1,\quad p^{jk} = b^{jk}, \quad\Psi = \ell_{tt} +
\sum_{j,k=1}^n\big(b^{jk}\ell_{x_j}\big)_{x_k} - c_0\lambda,
\end{align}
we will estimate the resulting terms in \cref{hyperbolic2} one by one.

For $j,k=1,2,\cdots,n$, it is straightforward  to verify  that
\begin{equation}\label{hyper 1.1 h2}
\ell_t= -  \lambda c_{1}(2t-T), 
\quad 
\ell_{tt} = - 2 \lambda c_{1},
\quad
\ell_{x_j}= \lambda \psi_{x_{j}},
\quad 
\ell_{t x_j}=0.
\end{equation}
Combining \cref{1.27-eq2,hyper 1.1 h2}, we obtain 
\begin{align}
\label{ch-5-coeffvt}
\Big[\ell_{tt} + \sum_{j,k=1}^n(b^{jk}\ell_{x_j})_{x_k}
-\Psi\Big]w_{t}^2 = c_0\lambda  w_{t}^2,
\end{align}
and 
\begin{align}
\label{ch-5-bcoeffvtvi}
\sum_{j,k=1}^n\big[(b^{jk}\ell_{x_j})_t +
b^{jk}\ell_{tx_j}\big]w_{x_k} w_t = 0.
\end{align}
By \cref{condition of d,hyper 1.1 h2,1.27-eq2}, it holds that 
\begin{align}\label{vivj}\notag
& \sum_{j,k=1}^n c^{jk}w_{x_j}w_{x_k} 
\\ \notag
& = 
\sum_{j,k=1}^n \Big\{ (b^{jk}\ell_t)_t \!+\! \sum_{j',k'=1}^n
\big[ 2b^{jk'}(b^{j'k}\ell_{x_{j'}})_{x_{k'}} \!-\!
(b^{jk}b^{j'k'}\ell_{x_{j'}})_{x_{k'}} \big] \!+\!
\Psi b^{jk} \Big\}w_{x_j} w_{x_k} 
\\ \notag
& = 
\sum_{j,k=1}^n \Big\{  2 b^{jk}\ell_{tt} + \sum_{j',k'=1}^n
\big[ 2b^{jk'}(b^{j'k}\ell_{x_{j'}})_{x_{k'}}-
 b^{jk}_{x_{k'}}b^{j'k'}\ell_{x_{j'}} \big]
- c_{0} \lambda  b^{jk} \Big\}w_{x_j} w_{x_k} 
\\
& \geq 
\lambda (\mu_0 -4c_1 - c_0)\sum_{j,k=1}^nb^{jk}v_{x_j}v_{x_k}
.
\end{align}

Now we compute the coefficient of $w^2$. 
Combining \cref{AB1,hyper 1.1 h2,1.27-eq2}, we have 
\begin{align}
\label{eqCarHyper1}
\notag
\mathcal{A}
& =
\ell_{t}^{2} - 2 \ell_{tt} - \sum_{j,k=1}^{n} b^{j k} \ell_{x_{j}} \ell_{x_{k}} + c_{0} \lambda
\\
& =
\lambda^2  \Big[ c_1^2(2t - T)^2 - \sum_{j,k=1}^nb^{jk}\psi_{x_j} \psi_{x_k} \Big] + O(\lambda).
\end{align}
From \cref{1.27-eq2,eqCarHyper1}, we see that 
\begin{align*}
& ( \mathcal{A} \ell_{t})_{t} 
= 
- 6 c_{1}^{3} \lambda^{3} (2 t - T)^{2}
+2 c_{1} \lambda^{3} \sum_{j, k=1}^{n} b^{j k} \psi_{x_{j}} \psi_{x_{k}}
,
\end{align*}
and
\begin{align*}
- \sum_{j, k=1}^{n} ( \mathcal{A} b^{j k} \ell_{x_{j}})_{x_{k}} 
&= 
\lambda^{3} \sum_{j, k=1}^{n} \Big\{
- c_{1}^{2}  (b^{j k} \psi_{x_{j}})_{x_{k}} 
+ \sum_{j', k'=1}^{n}   (b^{j k} \psi_{x_{j}} \psi_{x_{k}})_{x_{k'}} b ^{j' k'} \psi_{x_{j'}}
\\
& \quad \quad \quad \quad \quad 
+\sum_{j', k'=1}^{n}   b^{j k} \psi_{x_{j}} \psi_{x_{k}} (b ^{j' k'} \psi_{x_{j'}})_{x_{k'}}
\Big\}
.
\end{align*}
This, together with \cref{eqCarHyper1,1.27-eq2,hyper 1.1 h2,AB1}, we conclude that 
\begin{align}
\notag
\label{CH-5-B1}
{\cal B}  &  =  (4c_1+c_0)\lambda^3 \sum_{j,k=1}^nb^{jk}\psi_{x_j}\psi_{x_k} +
\lambda^3\sum_{j,k=1}^n\sum_{j',k'=1}^nb^{jk}\psi_{x_j}
\big(b^{j'k'}\psi_{x_{j'}}\psi_{x_{k'}}\big)_{x_k} 
\\
& \quad   - (8c_1^3 + c_0c_1^2)\lambda^3(2t-T)^2 + O(\lambda^2).
\end{align}

From \cref{condition of d}, we have 
\begin{align}
\label{eqCarHyper2}
\notag
& \sum_{j,k=1}^n\sum_{j',k'=1}^nb^{jk}\psi_{x_j} \big(b^{j'k'}\psi_{x_{j'}}\psi_{x_{k'}}\big)_{x_k}
\\\notag
& = 
\sum_{j,k=1}^n\sum_{j',k'=1}^n ( 
    b^{jk} b^{j'k'}_{x_{k}} \psi_{x_j} \psi_{x_{j'}}\psi_{x_{k'}} 
    + 2 b^{jk} b^{j'k'} \psi_{x_j} \psi_{x_{k} x_{j'}}\psi_{x_{k'}} 
)
\\\notag
& = 
\sum_{j,k=1}^n\sum_{j',k'=1}^n ( 
    b^{jk} b^{j'k'}_{x_{k}} \psi_{x_j} \psi_{x_{j'}}\psi_{x_{k'}} 
    -  2  b^{j'k'} b^{jk}_{x_{j'}} \psi_{x_j} \psi_{x_{k}}\psi_{x_{k'}} 
\\\notag
& \quad \quad \quad \quad \quad \quad 
    +  2  b^{j'k'} (b^{jk}  \psi_{x_{k}} )_{x_{j'}}\psi_{x_j}\psi_{x_{k'}} 
)
\\\notag
& = 
\sum_{j,k=1}^n\Big\{ \sum_{j',k'=1}^n\Big[
2b^{jk'}(b^{j'k}\psi_{x_{j'}})_{x_{k'}} -
b^{jk}_{x_{k'}}b^{j'k'}\psi_{x_{j'}} \Big] \Big\}\psi_{x_{j}}\psi_{x_{k}} 
\\
& \geq
\mu_0 \sum_{j,k=1}^nb^{jk}\psi_{x_{j}}\psi_{x_{k}} .
\end{align}
From \cref{CH-5-B1,eqCarHyper2,ch-5-condition2}, we deduce that 
\begin{align*}
{\cal B}  \geq 2\big(4c_1+c_0\big)\lambda^3 \big(4R_1^2 - c_1^2T^2\big) +
O(\lambda^2).
\end{align*}
Since $ 4 R_{1}^{2} - c_{1}^{2} T^{2} > 0 $, there exists  $\lambda_1 > 0$ so that, for any $\lambda > \lambda_{1}$, 
\begin{align}
\label{Ch-5-B ine}
{\cal B}w ^2 \geq 8c_1  \big(4R_1^2 - c_1^2T^2\big) \lambda^3 w^2.
\end{align}

\emph{Step 3.} In this step, we apply \cref{thmFi} to the equation \cref{ch-5-system2}.

Integrating (\ref{hyperbolic2}) in $Q$, taking mathematical expectation and by \cref{ch-5-coeffvt,ch-5-bcoeffvtvi,vivj,Ch-5-B ine}, and noting 
\begin{align}\notag 
	\label{eqCarHyper3}
	\mathbb{E} \int_{0}^{T} \int_{G} d &\Big[ \sum_{j,k=1}^n b^{jk}\ell_t w_{x_j} w_{x_k}-
	2 \sum_{j,k=1}^n b^{jk}\ell_{x_j}w_{x_k}w_t  +  \ell_t w_t^2
	- \Psi w_t w 
	\\ 
	&  
	+ \Big(  {\cal A}\ell_t +
	\frac{1}{2}  \Psi _t\Big)w^2 \Big]  d x 
	= 0,
\end{align}
where we use $ w(0) = w(T) = 0 $, it holds that 
\begin{align}\notag
\label{bhyperbolic31}
&  \mathbb{E} \int_Q \theta\Big\{ \Big( -2\ell_t w_t + 2
\sum_{j,k=1}^n b^{jk}\ell_{x_j}w_{x_k} +
\Psi w \Big) \Big[du_t- \sum_{j,k=1}^n\!\big(b^{jk}u_{x_j}\big)_{x_k}dt \Big]  \Big\}dx 
\\ \notag
&\quad + \lambda
\mathbb{E}\!\int_{\Sigma}\sum_{j,k=1}^n\!\sum_{j',k'=1}^n\!\big(
2b^{jk}b^{j'k'}\psi_{x_{j'}}w_{x_j} w_{x_{k'}}\!\! -\!
b^{jk}b^{j'k'}\psi_{x_j} w_{x_{j'}}w_{x_{k'}} \big)\nu_k d \Gamma d t  
\\\notag
& \quad 
-  \mathbb{E} \int_Q  \theta \ell_t |du_t|^2 d x
\\ \notag
& \geq {\cal C} \mathbb{E}\int_Q \theta^2\Big[\big( \lambda u_t^2 + \lambda
|\nabla u|^2 \big) 
+   \lambda^3 u^2 \Big]dxdt 
\\
& \quad 
+ \mathbb{E}\int_Q\Big(-2\ell_t w_t +
2\sum_{j,k=1}^nb^{jk}\ell_{x_j}w_{x_k} + \Psi w\Big)^2
dxdt. 
\end{align}
Noting that $ u = 0 $ on $ \Sigma $, ${\mathbb{P}}$-a.s., from \cref{ch-5-def gamma0}, we deduce that 
\begin{align} \notag
\label{bhyperbolic32}
& \mathbb{E}\int_{\Sigma}\sum_{j,k=1}^n\sum_{j',k'=1}^n\Big(
2b^{jk}b^{j'k'}\psi_{x_{j'}}w_{x_j} w_{x_{k'}} -
b^{jk}b^{j'k'}\psi_{x_j} w_{x_{j'}}w_{x_{k'}}
\Big)\nu^k d \Gamma d t   
\\ \notag
&  =  \mathbb{E}\int_{\Sigma}\Big( \sum_{j,k=1}^nb^{jk}\nu^j \nu^k
\Big)\Big( \sum_{j',k'=1}^nb^{j'k'}\psi_{x_{j'}}\nu^{k'}
\Big)\theta^2\Big|\frac{\partial u}{\partial
\nu}\Big|^2d \Gamma d t  
\\
&  \leq \mathbb{E}\int_{\Sigma_0}\Big( \sum_{j,k=1}^nb^{jk}\nu^j
\nu^k \Big)\Big( \sum_{j',k'=1}^nb^{j'k'}\psi_{x_{j'}}\nu^{k'}
\Big)\theta^2\Big|\frac{\partial u}{\partial \nu}\Big|^2d \Gamma d t .
\end{align}

Thanks to \cref{ch-5-system2}, it holds that 
\begin{align}
\notag
\label{bhyperbolic4}
& \mathbb{E} \int_Q \theta\Big\{ \Big( -2\ell_t w_t + 2
\sum_{j,k=1}^n b^{jk}\ell_{x_j}w_{x_k} +
\Psi w \Big) \Big[du_t- \sum_{j,k=1}^n\!\big(b^{jk}u_{x_j}\big)_{x_k}dt \Big]  \Big\}dx  
\\ \notag
& 
\leq {\cal C}\bigg[\mathbb{E}\int_Q \theta^2\big( |\tilde{f}|^2 + 
|g|^2\big)dxdt 
+ |b_1|^2_{L^{\infty}_{{\mathbb{F}}}(0,T;L^{\infty}(G))}  \mathbb{E}\int_Q \theta^2 u_t^2 dxdt 
\\ 
& \qquad\;\,  
+|b_2|^2_{L^{\infty}_{{\mathbb{F}}}(0,T;L^{\infty}(G,\mathbb{R}^n))}
\mathbb{E}\int_Q\theta^2 |\nabla u|^2dxdt    
+ \mathbb{E}\int_Q \theta^2  b_3^2 u^2 dxdt \nonumber
\\
& \quad \quad \;\,
+ |b_4|_{L^{\infty}_{{\mathbb{F}}}(0,T;L^{\infty}(G))}^2 \mathbb{E} \int_Q
\theta^2 u^2 dxdt
\bigg]\nonumber
\\
&  \qquad\;\, + \mathbb{E} \int_Q \Big ( - 2\ell_t w_t + \sum_{j,k=1}^n
b^{jk}\ell_{x_j} w_{x_k} +  \Psi w \Big )^2dxdt,
\end{align}
and 
\begin{align}
\label{eqCarHyper8}
-  \mathbb{E} \int_Q  \theta \ell_t |du_t|^2 d x
\leq 
\lambda |b_4|_{L^{\infty}_{{\mathbb{F}}}(0,T;L^{\infty}(G))}^2 \mathbb{E} \int_Q     \theta^2 u^2 dxdt
+ \lambda \mathbb{E} \int_{Q} \theta^{2} |g|^{2} d x d t 
.
\end{align}

\medskip

{\it Step 4}. In this step, we handle the terms in \cref{bhyperbolic4,{eqCarHyper8}}.

Choosing $ s = 2p / (p-2) $, we have 
\begin{align*}
\frac{2}{p} + \frac{2}{s} = 1,
\end{align*}
which, combined with H\"older's inequality and \cref{9.29-eq1}, implies that 
\begin{align}
\label{eqCarHyper4}
|b_3\theta u|^2_{L^2_{\mathbb{F}}(0,T;L^2(G))} \leq r_2^{2} |\theta u|^2_{L^2_{\mathbb{F}}(0,T;L^{s}(G))} 
.
\end{align}
Since $ \frac{1}{s} = \frac{1}{2} - \frac{n}{p} $, by Sobolev's inequality, we get 
\begin{align}
\label{eqCarHyper5}
|\theta u|^2_{L^2_{\mathbb{F}}(0,T;L^{s}(G))}  \leq |\theta u|^2_{L^2_{\mathbb{F}}(0,T;H^{n/p}(G))}
.
\end{align}
For any $\rho\in L^2_{{\cal F}_T}(\Omega;H^1({\mathbb{R}}^n))$, by H\"{o}lder's
inequality again, we obtain 
\begin{align}\notag
\label{eqCarHyper6}
|\rho|^2_{L^2_{{\cal F}_T}(\Omega;H^{n/p}({\mathbb{R}}^n))}  
&  = {\mathbb{E}}\int_{{\mathbb{R}}^n}(1+|\xi|^2)^{n/p}|\hat \rho(\xi)|^{2n/p}|\hat \rho(\xi)|^{2(1-n/p)}d\xi 
\\
& 
\leq |\rho|^{2n/p}_{L^2_{{\cal F}_T}(\Omega;H^1({\mathbb{R}}^n))}
|\rho|^{2(1-n/p)}_{L^2_{{\cal F}_T}(\Omega;L^2({\mathbb{R}}^n))}.
\end{align}
Hence,  it holds that 
\begin{align}
\label{eqCarHyper6-1}
|\widetilde \rho|^2_{L^2_{{\cal F}_T}(\Omega;H^{n/p}(G))} \leq
{\cal C}|\widetilde
\rho|^{2n/p}_{L^2_{{\cal F}_T}(\Omega;H_0^1(G))}|\widetilde
\rho|^{2(1-n/p)}_{L^2_{{\cal F}_T}(\Omega;L^2(G))}, \; \forall\;\widetilde
\rho \in L^2_{{\cal F}_T}(\Omega;H_0^1(G)),
\end{align}
and 
\begin{align}
\label{eqCarHyper7}
|\overline \rho|^2_{L^2_{\mathbb{F}}(0,T;H^{n/p}(G))} \! \leq {\cal C}|\overline
\rho|^{2n/p}_{L^2_{\mathbb{F}}(0,T;H_0^1(G))}|\overline
\rho|^{2(1-n/p)}_{L^2_{\mathbb{F}}(0,T;L^2(G))},\;\forall\;\overline
\rho \in L^2_{\mathbb{F}}(0,T;H_0^1(G)).
\end{align}
Combining \cref{eqCarHyper4,eqCarHyper5,eqCarHyper6,eqCarHyper6-1,eqCarHyper7}, thanks to Cauchy-Schwarz inequality, for all $ \varepsilon > 0 $, we deduce that 
\begin{align}
\label{eqCarHyper9}
\notag
& |b_3\theta u|^2_{L^2_{\mathbb{F}}(0,T;L^2(G))}\nonumber
\\ \notag
& \leq {\cal C} r_{2}^{2}| \theta u|^{2n/p}_{L^2_{\mathbb{F}}(0,T;H_0^1(G))}| \theta
u|^{2(1-n/p)}_{L^2_{\mathbb{F}}(0,T;L^2(G))}
\\
&  \leq \varepsilon\lambda| \theta u|^{2}_{L^2_{\mathbb{F}}(0,T;H_0^1(G))} +
{\cal C}(\varepsilon)r_2^{ 2p/(p-n)}\lambda^{-n/(p-n)}| \theta
u|^{2}_{L^2_{\mathbb{F}}(0,T;L^2(G))}.
\end{align}

Choosing
\begin{align*}
\lambda_2 = \max\Big\{\lambda_1,{\cal C}\Big (r_1^2 +
r_2^{\frac{1}{3/2-n/p}}+1\Big )\Big\},
\end{align*}
from \cref{bhyperbolic31,bhyperbolic32,bhyperbolic4,eqCarHyper8,eqCarHyper9}, for all $ \lambda \geq \lambda_{2} $, it holds that 
\begin{align}
\label{eqCarHyper10}
\notag
&{\cal C} \lambda \mathbb{E}\int_{\Sigma_0}\theta^2\Big(
\sum_{j,k=1}^nb^{jk}\nu^{j} \nu^{k} \Big)\Big(
\sum_{j',k'=1}^nb^{j'k'}\psi_{x_{j'}}\nu^{k'} \Big)\Big|\frac{\partial
u}{\partial \nu}\Big|^2d\Gamma d t
\\ \notag
& \quad + {\cal C}\mathbb{E}\int_Q \theta^2( |\tilde{f}|^2 + \lambda |g|^2)dxdt
\\
& \geq \mathbb{E}\int_Q \theta^2 \Big (  \lambda u_t^2 + \lambda |\nabla u|^2
+ \lambda^3 u^2  \Big )dxdt.
\end{align}
Noting that \cref{eqChi} and $ u = \chi z $, by \cref{eqCarHyper10}, we get 
\begin{align}\notag
\label{eqCarHyper11}
&{\cal C} \lambda \mathbb{E}\int_{\Sigma_0}\theta^2\Big(
\sum_{j,k=1}^nb^{jk}\nu^{j} \nu^{k} \Big)\Big(
\sum_{j',k'=1}^nb^{j'k'}\psi_{x_{j'}}\nu^{k'} \Big)\Big|\frac{\partial
z}{\partial \nu}\Big|^2d\Gamma d t
\\ \notag
& \quad 
+ \mathcal{C} \mathbb{E} \int_{(0, \frac{T}{2} - \varepsilon_{1}T) \cup (\frac{T}{2}+ \varepsilon_{1}T, T)} \int_{G} \theta^2 \big(z_t^2 +  |\nabla z|^2 + z^2\big)dxdt 
\\ \notag
& \quad + {\cal C}\mathbb{E}\int_Q \theta^2( |f|^2 + \lambda |g|^2)dxdt
\\
& \geq \mathbb{E}\int_{\frac{T}{2} - \varepsilon_{0}T}^{\frac{T}{2} + \varepsilon_{0}T} \int_{G} \theta^2 \Big (  \lambda z_t^2 + \lambda |\nabla z|^2
+ \lambda^3 z^2  \Big )dxdt.
\end{align}
From \cref{ch-5-en esti,ch-5-e1,eqHyperE1}, we obtain that 
\begin{align}\notag
    \label{eqHyperE2}
    & \mathbb{E} \int_{(0, \frac{T}{2} - \varepsilon_{1}T) \cup (\frac{T}{2}+ \varepsilon_{1}T, T)} \int_{G} \theta^2 \big(z_t^2 +  |\nabla z|^2 + z^2\big)dxdt
    \\\notag
    &  \leq 
    \mathcal{C} e^{\lambda [R_{1}^{2} - c_{1} (\varepsilon_{1} T)^{2}]} \mathbb{E} \int_{(0, \frac{T}{2} - \varepsilon_{1}T) \cup (\frac{T}{2}+ \varepsilon_{1}T, T)} \int_{G}   \big(z_t^2 +  |\nabla z|^2  \big)dxdt
    \\\notag
    & \leq
     \mathcal{C} e^{{\cal C}\big(r_1^2 + r_2^{\frac{1}{2-n/p}}+1\big)T} e^{\lambda [R_{1}^{2} - c_{1} (\varepsilon_{1} T)^{2}]} \mathbb{E}   \int_{G}   \big(z_1^2 +  |\nabla z_{0}|^2  \big) d x 
    \\
    & \quad 
    + \mathcal{C} e^{R^{2}_{1} \lambda} \mathbb{E} \int_{Q} (|f|^{2} + |g|^{2}) d x d t 
    ,
\end{align}
and that 
\begin{align}\notag
    \label{eqHyperE3}
    & \mathbb{E}\int_{\frac{T}{2} - \varepsilon_{0}T}^{\frac{T}{2} + \varepsilon_{0}T} \int_{G} \theta^2 \Big (  \lambda z_t^2 + \lambda |\nabla z|^2
    + \lambda^3 z^2  \Big )dxdt 
    \\\notag
    & \geq
    \mathcal{C} e^{\frac{R_{0}^{2}}{2} \lambda} \mathbb{E} \int_{\frac{T}{2} - \varepsilon_{0}T}^{\frac{T}{2} + \varepsilon_{0}T} \int_{G} \Big (    z_t^2 +   |\nabla z|^2  \Big )dxdt 
    \\\notag
    & \geq 
    \mathcal{C} e^{\frac{R_{0}^{2}}{2} \lambda}  e^{-{\cal C}\big(r_1^2 + r_2^{\frac{1}{2-n/p}}+1\big)T}  \mathbb{E}   \int_{G}   \big(z_1^2 +  |\nabla z_{0}|^2  \big) d x 
    \\
    & \quad 
    - \mathcal{C}  e^{\frac{R_{0}^{2}}{2} \lambda}  \mathbb{E} \int_{Q} (|f|^{2} + |g|^{2}) d x d t 
    .
\end{align}
Combining \cref{eqHyperE3,eqCarHyper11,eqHyperE2}, noting that $  R_{1}^{2} - c_{1} (\varepsilon_{1} T)^{2} < 0 $, there exists a
\begin{align*}
\lambda_3={\cal C}\Big (r_1^2 + r_2^{\frac{1}{3/2-n/p}}+1\Big )\geq \lambda_2
\end{align*}
such that for all $\lambda\geq \lambda_3$, it holds that
\begin{align}\label{bhyperbolic7}
\notag
& {\cal C} \lambda \mathbb{E} \int_{\Sigma_0} \theta^2\Big( \sum_{j,k=1}^n
b^{jk}\nu_{j} \nu_{k} \Big)\Big( \sum_{j',k'=1}^n
b^{j'k'}\psi_{x_{j'}}\nu_{k'} \Big)\Big|\frac{\partial z}{\partial \nu}\Big|^2
d \Gamma d t 
\\ \notag
& \quad 
+ {\cal C}e^{ R^{2}_{1}  \lambda} \mathbb{E} \int_Q   \big(|f|^2 + |g|^2\big) dxdt 
\\
& \geq \mathbb{E} \int_G  \big( z_1^2 +
|\nabla z_0|^2  \big)dx.
\end{align}
Taking $\lambda=\lambda_3$, we deduce that 
\begin{align}\label{bhyperbolic8}
\notag
& {\cal C}e^{ R^{2}_{1} \lambda_3}  \Big[\mathbb{E}\int_{\Sigma_0} \Big|\frac{\partial z}{\partial
\nu}\Big|^2d \Gamma d t 
+  \mathbb{E} \int_Q  \big(|f|^2 + |g|^2\big) dxdt \Big]  
\\
& \geq
\mathbb{E} \int_G  \big( z_1^2 + |\nabla
z_0|^2 \big)dx,
\end{align}
which gives the desired inequality \eqref{ch-5-obser esti2}.
\end{proof}

We are in a position to prove \cref{th2}.

\begin{proof}[Proof of \cref{th2}]
Let $ \hat{y} $ solve \cref{ch-5-4.12-eq1} with the initial datum $ (\hat{y}_{0}, \hat{y}_{1}) $.
Put $ z =  y - \hat{y} $.
Then $ z $ fulfills \cref{ch-5-system1} with 
\begin{align*}
\left\{
    \begin{aligned}
        & b_1=\int_0^1\partial_\varrho F(\hat y,\hat y_t+s(y_t-\hat
        y_t),\nabla \hat y)ds, 
        \\
        & b_2=\int_0^1\partial_\zeta F(\hat y,
        \hat y_t,\nabla  \hat y+s(\nabla  y-\nabla \hat y))ds, 
        \\
        &   b_3=\int_0^1\partial_\eta F(\hat y+s(y-\hat y)  \hat
        y_t ,\nabla \hat y)ds, 
        \\
        &   b_4=\int_0^1\partial_\eta K(\hat y+s(y-\hat
        y))ds,\quad f=0,  \quad  g=0,\\
        &  z_0=y_0-\hat y_0, \quad z_1=y_1-\hat y_1.
    \end{aligned}
\right.
\end{align*}
The proof is completed directly by \cref{ch-5-obser esti2}.
\end{proof}

\section[Solution to inverse source problem. I]{Solution to inverse source problem with the boundary measurement and  the final time  measurement. I}
\label{secISP}

This section is devoted to  \cref{probISP}. Since the random source   $g(\cdot)$ appears in the right hand side of
\cref{ch-5-obser esti2}, we cannot  use \cref{ch-5-observability} directly to solve to solve \cref{probISP}. To address this, we must develop a new Carleman estimate for \cref{ch-5-system1} in which the source term $g$ and the initial data are bounded by the measurement. To this end, for $ \psi $ satisfying \cref{condition of d}, we assume the following condition:
\begin{condition}
\label{condition of c0c1c2T}
\begin{align*}
      \left\{
    \begin{aligned}
        &\mu_0 - 4c_1 -c_0 > 0,\\ 
     &   \frac{\mu_0}{(8c_1 +
    c_0)}\sum_{j,k=1}^nb^{jk}\psi_{x_j}\psi_{x_k}
    > 4c_1^2T^2
    > \sum_{j,k=1}^nb^{jk}\psi_{x_j}\psi_{x_k}.
    \end{aligned}
\right.
\end{align*}
\end{condition}

Given that $ \Gamma_{0} $ is defined in \cref{ch-5-def gamma0}, the following is a positive answer to \cref{probISP}.

\begin{theorem}
\label{eqOber}
Let   \cref{condition of d,condition of c0c1c2T} 
hold. If the solution $z\in {\mathbb{H}}_T$ to  \cref{ch-5-system1} 
satisfies that $z(T)=0$ in $G$, ${\mathbb{P}}$-$\hbox{\rm a.s.}$, then
\begin{align*}
  &  |(z_0, z_1)|_{L^2_{{\cal F}_0}(\Omega; H^1_0(G) \times
L^2(G))}+|\sqrt{T-t}g|_{ L^2_{{\mathbb{F}}}(0,T;L^2(G))}
\\
& \le
{\cal C}e^{\big(r_1^2 + r_2^{\frac{1}{ 3/2 - n/p}}+1\big)} \bigg( \left|\frac{\partial z}{\partial \nu}\right|_{L^2_{{\mathbb{F}}}(0,T;L^2(\Gamma_0))} + |f|_{L^{2}_{\mathbb{F}}(0,T; L^{2}(G))} \bigg),
\end{align*}
where $ r_{1} $ and $ r_{2} $ are defined in \cref{9.29-eq1}.
\end{theorem}

Similar to the proof in \cref{ch-5-observability}, \cref{eqOber} can be directly derived by the following global Carleman estimate for \cref{ch-5-system1}.

\begin{theorem}\label{ch-5-2-uniqueness} 
Assume   \cref{condition of d,condition of c0c1c2T} are satisfied.
Then, there exists a constant $\tilde\lambda>0$ such that for any
$\lambda\geq\tilde \lambda$ and each solution $z\in {\mathbb{H}}_T$ to the equation
\cref{ch-5-system1}  satisfying $z(T)=0$ in $G$, ${\mathbb{P}}$-a.s., it
holds that
\begin{align}\notag
    \label{eqObser}
    &   \mathbb{E}\int_G\theta^2(0) \big( \lambda|z_1|^2 + \lambda |\nabla z_0|^2 + \lambda^3
|z_0|^2\big)dx + \lambda\mathbb{E}\int_Q (T-t) \theta^2 g^2dxdt
\\
&  \leq {\cal C}\Big (
\lambda \mathbb{E} \int_{\Sigma_0}\theta^2\Big| \frac{\partial z}{\partial\nu}
\Big|^2 d\Gamma dt  +  \mathbb{E}\int_Q  \theta^2 f^2dxdt\Big ).
\end{align}
\end{theorem}

\begin{proof} The proof is divided into three steps.

\emph{Step 1.} 
In this step, we give the weight function of the Carleman estimate and compute the order of the terms in the equality \cref{hyperbolic2} with respect to $\lambda$. 

We choose
\begin{align*}
\theta = e^\ell, \qquad \ell  = \lambda\big[\psi(x) - c_1 ( t-T )^2\big],
\end{align*}
where $\lambda>0$ is a parameter, $\psi(\cdot)$ fulfills \cref{condition of d,condition of c0c1c2T}, and $c_1$ is the constant in \cref{condition of c0c1c2T}.

Similarly to the proof of  \cref{ch-5-observability},   choosing 
\begin{align}
 \label{1.27-eq2-1}
\phi\equiv 1,\quad u=z, \quad  p^{jk} = b^{jk}, \quad\Psi = \ell_{tt} +
\sum_{j,k=1}^n\big(b^{jk}\ell_{x_j}\big)_{x_k} - c_0\lambda,
\end{align}
we will estimate the resulting terms in \cref{hyperbolic2} one by one.
The only difference is the way to deal with the corresponding terms at $t=0$ and $t=T$. 

For $j,k=1,2,\cdots,n$, it is easy  to show  that
\begin{equation}\label{hyper 1.1 h2-1}
\ell_t= - 2 \lambda c_{1}( t-T), 
\quad 
\ell_{tt} = - 2 \lambda c_{1},
\quad
\ell_{x_j}= \lambda \psi_{x_{j}},
\quad 
\ell_{t x_j}=0.
\end{equation}
Hence, \cref{ch-5-coeffvt,ch-5-bcoeffvtvi,vivj} still holds.
Next, we compute the coefficient of $ w^{2} $.
Combining \cref{AB1,1.27-eq2-1,hyper 1.1 h2-1}, we obtain 
\begin{align}
\label{eqHyperA}
\mathcal{A}
=
\lambda^2  \Big[ 4 c_1^2( t - T)^2 - \sum_{j,k=1}^nb^{jk}\psi_{x_j} \psi_{x_k} \Big] + O(\lambda).
\end{align}
From \cref{1.27-eq2-1,eqHyperA}, we have 
\begin{align*}
& ( \mathcal{A} \ell_{t})_{t} 
= 
- 24 c_{1}^{3} \lambda^{3} (  t - T )^{2}
+2 c_{1} \lambda^{3} \sum_{j, k=1}^{n} b^{j k} \psi_{x_{j}} \psi_{x_{k}}
\end{align*}
and 
\begin{align*} 
- \sum_{j, k=1}^{n} ( \mathcal{A} b^{j k} \ell_{x_{j}})_{x_{k}} 
&
= 
\lambda^{3} \sum_{j, k=1}^{n} \Big\{
- c_{1}^{2}  (b^{j k} \psi_{x_{j}})_{x_{k}} 
+ \sum_{j', k'=1}^{n}   (b^{j k} \psi_{x_{j}} \psi_{x_{k}})_{x_{k'}} b ^{j' k'} \psi_{x_{j'}}
\\
& \quad \quad \quad \quad \quad 
+\sum_{j', k'=1}^{n}   b^{j k} \psi_{x_{j}} \psi_{x_{k}} (b ^{j' k'} \psi_{x_{j'}})_{x_{k'}}
\Big\}
.
\end{align*}
Combining it with  \cref{AB1,1.27-eq2-1,hyper 1.1 h2-1,eqHyperA}, we obtain 
\begin{align}\notag
    \label{eqHyperBe1}
    {\cal B}  &  =  (4c_1+c_0)\lambda^3 \sum_{j,k=1}^nb^{jk}\psi_{x_j}\psi_{x_k} +
\lambda^3\sum_{j,k=1}^n\sum_{j',k'=1}^nb^{jk}\psi_{x_j}
\big(b^{j'k'}\psi_{x_{j'}}\psi_{x_{k'}}\big)_{x_k} 
\\
& \quad   - 4 c_1^2 (8c_1  + c_0)\lambda^3( t-T)^2 + O(\lambda^2).
\end{align}
Thanks to \cref{condition of d}, similar to \cref{eqCarHyper2}, we can get 
\begin{align}
    \label{eqHyperBe2}
    \sum_{j,k=1}^n\sum_{j',k'=1}^nb^{jk}\psi_{x_j} \big(b^{j'k'}\psi_{x_{j'}}\psi_{x_{k'}}\big)_{x_k}
    \geq
    \mu_0 \sum_{j,k=1}^nb^{jk}\psi_{x_{j}}\psi_{x_{k}} .
\end{align}
From \cref{eqHyperBe1,eqHyperBe2,condition of c0c1c2T},  there exists  $\lambda_1 > 0$ such that, for any $\lambda > \lambda_{1}$, 
\begin{align}
\label{Ch-5-B ine1}
{\cal B}w ^2 \geq \mathcal{C} \lambda^3 w^2.
\end{align}

\medskip

\emph{Step 2.} In this step, we deal with the terms corresponding to 
$t=0$ and $t=T$.

For $ t =0 $, from Cauchy-Schwarz inequality, we obtain 
\begin{align}
\notag
\label{eqCarHyperI1}
&  \sum_{j,k=1}^n b^{jk}\ell _tw_{x_j}w_{x_k}- 2 \sum_{j,k=1}^n
b^{jk}\ell _{x_j}w_{x_k}w_t + \ell _tw_t^2 -\Psi w_t w  +
\Big({\cal A}\ell _t+ \frac{1}{2}\Psi_t \Big)w^2
\\ \notag
&   =    
2c_1 T \lambda \sum_{j,k=1}^n b^{jk}w_{x_j}w_{x_k} - 2\lambda
\sum_{j,k=1}^n b^{jk}\psi_{x_j}w_{x_k}w_t  
\\ \notag
& \quad  - \lambda\Big[ -2c_1+
\sum_{j,k=1}^n (b^{jk}\psi_{x_j})_{x_k} - c_0\Big]w_t w+ 2c_1 T \lambda
w_t^2
\\ \notag
&  \quad  + \Big[ 2c_1 T\Big(4c_1^2T^2 -
\sum_{j,k=1}^nb^{jk}\psi_{x_j}\psi_{x_k}\Big)\lambda^3 + O(\lambda^2) \Big]w^2
\\ \notag
& \geq   2c_1 T \lambda \sum_{j,k=1}^n b^{jk}w_{x_j}w_{x_k} - \lambda \Big(
\sum_{j,k=1}^n b^{jk}\psi_{x_j}\psi_{x_k} \Big)^{\frac{1}{2}}
\sum_{j,k=1}^n b^{jk}w_{x_j}w_{x_k}
\\ \notag
& \quad 
- \lambda \Big( \sum_{j,k=1}^n
b^{jk}\psi_{x_j}\psi_{x_k}\Big)^{\frac{1}{2}}w_t^2
+ 2c_1 T \lambda w_t^2 - w_t^2  
\\
&  \quad  
+ \Big[2c_1T\Big(4c_1^2T^2 -
\sum_{j,k=1}^nb^{jk}\psi_{x_j}\psi_{x_k}\Big)\lambda^3 +
O(\lambda^2)\Big]w^2
.
\end{align}
Thanks to \cref{condition of c0c1c2T}, we have 
\begin{align}
\label{eqCarHyperI2}
4c_1^2T^2 - \sum_{j,k=1}^nb^{jk}\psi_{x_j}\psi_{x_k}
> 0.
\end{align}
When $ t = 0 $, from \cref{eqCarHyperI1,eqCarHyperI2}, we know that there exists $ \lambda_2 > 0$ such that for any $ \lambda \geq  \lambda_2$, it holds that
\begin{align}\notag
\label{eqCarHyperI3}
&\sum_{j,k=1}^nb^{jk}\ell _tw_{x_j}w_{x_k} - 2
\sum_{j,k=1}^nb^{jk}\ell _{x_j}w_{x_k}w_t + \ell _tw_t^2-\Psi w_t w
+ \Big( {\cal A}\ell_t + \frac{1}{2}\Psi_t \Big)w^2
\\
&\geq {\cal C}\Big[\lambda\big(w_t^2 + |\nabla w|^2\big) +
\lambda^3 w^2 \Big].
\end{align}

Noting that $ \ell(T) = \ell_{t} (T) = 0 $ and $ w(T)  = |\nabla w(T)|  = 0 $, for $ t= T $, we get that 
\begin{align}
\label{eqCarHyperI4}
& \sum_{j,k=1}^nb^{jk}\ell _tw_{x_j}w_{x_k} \!-\! 2
\sum_{j,k=1}^nb^{jk}\ell _{x_j}w_{x_k}w_t \!+\! \ell _tw_t^2 \!-\!\Psi w_t w
\!+ \!\Big(
{\cal A}\ell _t + \frac{1}{2}\Psi_t \Big)w^2
= 0.
\end{align}

\emph{Step 3.} In this step, we apply \cref{thmFi} to the equation \cref{ch-5-system1}.

Integrating \cref{hyperbolic2} in $ Q $, taking mathematical expectation, and noting \cref{bhyperbolic32,ch-5-coeffvt,ch-5-bcoeffvtvi,vivj,Ch-5-B ine1,eqCarHyperI3,eqCarHyperI4}, we have 
\begin{align*}
& \mathbb{E}\int_Q \theta\Big\{\Big( -2\ell _t v_t +
2\sum_{j,k=1}^n\!b^{jk}\ell _{x_j}v_{x_k} + \Psi v \Big) \Big[ dz_t
- \sum_{j,k=1}^n (b^{jk}z_{x_j})_{x_k}dt\Big]\Big\}dx
\\
&\quad + \lambda \mathbb{E}\int_{\Sigma_{0}}\Big(
\sum_{j,k=1}^nb^{jk}\nu_{j} \nu_{k} \Big)\Big(
\sum_{j',k'=1}^nb^{j'k'}\psi_{x_{j'}}\nu_{k'} \Big)\Big|\frac{\partial
v}{\partial
\nu}\Big|^2d \Gamma d t  
\\
&\geq  {\cal C} \mathbb{E} \int_Q  \theta^2\Big( \lambda  z_t^2 +
\lambda |\nabla z|^2 + \lambda^3 z^2 \Big)  d x d t 
\\
& \quad 
+ {\cal C} \mathbb{E} \int_Q  \Big( - 2\ell _tv_t + 2
\sum_{j,k=1}^n b^{jk}\ell _{x_j}v_{x_k} + \Psi
v\Big)^2 dxdt  
\\
&\quad   
+ {\cal C}\mathbb{E} \int_G \theta^2(0)\Big[\lambda(|\nabla z_0|^2 +
|z_1|^2) + \lambda^3|z_0|^2 \Big]dx  
\\
& \quad 
+ {\cal C}\lambda\mathbb{E} \int_Q (T-t) \theta^2
|b_4z + g|^2 dxdt 
,
\end{align*}
for $ \lambda \geq \max\{\lambda_{1}, \lambda_{2}\} $.

Since $ |b_4 z + g|^2 \geq \frac{1}{2} |g|^2 - 2|b_4z|^2, $
we deduce that 
\begin{align}\label{ch-5-2-hyperbolic3.1} 
& \lambda\mathbb{E}\int_Q (T-t) \theta^2 |b_4z + g|^2 dxdt  
  \geq
\frac{1}{2}\lambda\mathbb{E}\int_Q (T-t) \theta^2  g^2   dxdt - 2\lambda
T\mathbb{E}\int_Q \theta^2 b_4^2 z^2 dxdt.
\end{align}
Similar to the proof of \cref{eqCarHyper10}, taking $ \tilde{\lambda} =   \max \{\lambda_1, \lambda_{2},{\cal C}  (r_1^2 + r_2^{\frac{1}{3/2-n/p}}+1  ) \} $, for any $ \lambda \geq \tilde{\lambda} $, we deduce \cref{eqObser}. 
\end{proof}
\begin{remark}
In contrast to the inverse source problem (\cref{prob.p7}) for the stochastic parabolic equation in \cref{ch2}, here we can determine not only the source term $ g $ but also simultaneously determine the initial datum of the system. Further, we do not need assumptions on $g$ as the one in \cref{inv th2}.
\end{remark}
\begin{remark}\label{rkStoWave}
Noting that we used boundary and terminal measurements to uniquely determine $g$ in \cref{probISP}.
An interesting question arises: Can we determine $ f $ that is in the drift term and $ (z_{0}, z_{1}) $ through the measurement $\Big ( \dfrac{\partial
	z}{\partial \nu}\Big|_{\Sigma_0}\!\!,z(T)\Big )$?
To be more specific, we consider the following stochastic hyperbolic equation:
\begin{align}{\label{ch-5-2-sy1}}
	\left\{
	\begin{aligned}
		&dz_{t} - \!\sum_{j,k=1}^n(b^{jk}z_{{x_j}})_{{x_k}}dt = \left(b_1
		z_t + b_2\cdot\nabla z  +\! b_3 z  + f\right)dt  +  b_4z
		dW(t) \!\!& {\mbox { in }} Q,
		\\
		&   z = 0 & \mbox{ on } \Sigma, \\
		&   z(0) = z_0,\ z_{t}(0) = z_1 & \mbox{ in } G,
	\end{aligned}
	\right.
\end{align}
in which $z_0$, $z_1$ and $f$ are unknowns. 
As pointed in \cite[Remark 2.7]{Lue2015a}, the answer is no even for  deterministic
wave equations. To see this,  let us choose a nonzero  $y\in C_0^\infty(Q)$.
Put $f=y_{tt} - \Delta y$. Then, it is easy to see that $y$ solves the
following wave equation:
$$
\left\{
\begin{array}{ll}\displaystyle
	y_{tt} -\Delta y = f & \mbox{ in }Q,\\
	\noalign{\smallskip} \displaystyle y=0  &\mbox{ on } \Sigma,\\
	\noalign{\smallskip} \displaystyle y(0)=0,\ y_t(0)=0 &\mbox{ in }G.
\end{array}
\right.
$$
We have $y(T)=0$ in $G$ and $\frac{\partial y}{\partial \nu}=0$ on
$\Sigma$. However, it is clear that $f$ does not vanish in $Q$. This counterexample illustrates that the formulation of inverse source problems for SPDEs can differ significantly from their deterministic counterparts. Furthermore, it indicates that \cref{probISP} is inherently stochastic, distinct from any deterministic scenario, and is meaningful only within the stochastic framework.
\end{remark}

\section[Solution to inverse source problem. II]{Solution to inverse source problem with the boundary measurement and  the final time  me\-asurement.  II}
\label{secISP2}

In this section, we consider the inverse source problem of simultaneously determining the two unknown source terms for \cref{systemHyper}.

\begin{theorem}
\label{thmISPHyper}
Let $ h \in L^{2}_{\mathbb{F}}(0,T; H^{1}(G')) $, $ g \in L^{2}_{\mathbb{F}}(0,T; H^{1}(G)) $ and $ R \in C^{3}(\overline{Q}) $ satisfy 
\begin{align}
    \label{eqISPHyperR}
    |R(t, x)| \neq 0 \quad \text { for all }(t, x) \in  [0,T] \times \overline{G}.
\end{align}
If $ z \in \mathbb{H}_{T} $ satisfies \cref{systemHyper} and 
\begin{align*}
    & \frac{\partial z}{\partial \nu} = \frac{\partial z_{x_{1}}}{\partial \nu} = 0 \quad \text { on }(0, T) \times \partial G, ~ \text{${\mathbb{P}}$-a.s.},  \\
    & z(T) = z_{t}(T) =0  \text { in } G, ~ \text{${\mathbb{P}}$-a.s.},  \\
    &
\end{align*}
then 
\begin{align*}
    h (t, x^{\prime} )=0 & \text { for all } (t, x^{\prime} ) \in(0, T) \times G^{\prime},~ \text{${\mathbb{P}}$-a.s.}  \\
    g(t, x)=0 & \text { for all }(t, x) \in(0, T) \times G, ~ \text{${\mathbb{P}}$-a.s.} 
\end{align*}
\end{theorem}

To prove \cref{thmISPHyper}, we need a new global Carleman estimate for the special case of \cref{ch-5-system1}, where $ (b^{jk})_{1\leq j,k \leq n} $ is the identity matrix and $ G = (0,l) \times G' $.

For any $ x_{0}' \in \mathbb{R}^{n-1} \backslash \overline{G'} $, let 
\begin{align}
\label{eqVarphiH}
\theta = e ^{\ell}, \quad \ell = \lambda \varphi, \quad  \varphi (t,x) = |x ' - x_{0}'|^{2} - \beta |x_{1}|^{2} + |t + m|^{2}
,
\end{align}
where constants $ \beta $ and $ m $ satisfy
\begin{align}
\label{eqbmH}
\beta \in \Big(0, \frac{4}{5}\Big) \text{ and } m > \max_{x' \in \overline{G}'} |x' - x_{0}'| + \beta l 
.
\end{align}
We have the following global Carleman estimate:
\begin{theorem}
\label{thmCarlemanHyperISP}
Let $ (b^{jk})_{1\leq j,k \leq n} $ be the identity matrix and $ \varphi, \theta $ be given in \cref{eqVarphiH}. 
Then there exist constants $ \mathcal{C} > 0 $ and $ \lambda_{0} > 0 $ such that for all $ \lambda \geq \lambda_{0} $, for any solution $ z $ to the equation \cref{ch-5-system1}, it holds that 
\begin{align}
    \label{eqCarHyperISP}
    \notag
    &\lambda \mathbb{E}  \int_Q \theta^{2}  (g^2+z_t^2+|\nabla z|^2+\lambda^2 z^2 )  d x d t 
        \\ \notag
        &  \quad 
        +\lambda \mathbb{E} \int_G \theta^{2}(0)  (z_t^2(0)+|\nabla z(0)|^2+\lambda^2 z^2(0) )  d x 
        \\
        & \leq  \mathcal{C} \mathbb{E}\Big(\lambda \int_{\Sigma} \theta^{2} \Big|\frac{\partial z}{\partial \nu}\Big|^2   d \Gamma d t +\int_Q \theta^{2} f^{2} d x d t\Big).
\end{align}
\end{theorem}
\begin{proof} The proof is divided into two steps.
	
{\it Step 1}. 	
In this step, we  compute the order of the terms in the equality \cref{hyperbolic2} with respect to $\lambda$. 	
	
Choosing 
\begin{align}
 \label{eqISPH1}
\phi\equiv 1, \quad  u = z, \quad p^{jk} = \delta_{jk}, \quad\Psi = \ell_{tt} +
\Delta \ell - c_0\lambda,
\end{align}
in \cref{thmFi}, we will estimate the terms in \cref{hyperbolic2} one by one.
For $j =1, \cdots,n$, it is straightforward  to verify  that
\begin{equation}\label{eqISPH2}
\ell_t= 2 \lambda (t + m ), 
\quad 
\ell_{tt} =  2 \lambda ,
\quad 
\ell_{x_{1}} = - 2 \beta  \lambda x_{1},
\quad 
\ell_{t x_j}=0,
\end{equation}
and for $ k = 2, \cdots, n $,
\begin{align}\label{eqISPH3}
\ell_{x_k}=2 \lambda (x_{k} - x_{0,k}').
\end{align}
Combining \cref{eqISPH1,eqISPH2,eqISPH3}, we obtain 
\begin{align}
\label{eqISPH4}
(\ell_{tt} + \Delta \ell     -\Psi\Big)w_{t}^2 = c_0 \lambda  w_{t}^2,
\end{align}
and 
\begin{align}
\label{eqISPH5}
2 \sum_{j =1}^n \ell_{tx_j} w_{x_j} w_t = 0.
\end{align}
By \cref{AB1,eqISPH1,eqISPH2,eqISPH3}, for $ c_{1} \leq \min\{ 8-c_{0}, 4(1-\beta) - c_{0} \} $, it holds that 
\begin{align}\label{eqISPH6}\notag
\sum_{j,k=1}^n c^{jk}w_{x_j}w_{x_k} 
& = 
\sum_{j =1 }^n (\ell_{tt} + 2 \ell_{x_{j}x_{j}} - \Delta \ell + \Psi ) w_{x_j}^{2}  
\\ \notag
& =  
\sum_{j =2 }^n  (8 - c_{0}) \lambda w_{x_{j}}^{2} 
+ [4 (1- \beta) - c_{0}] \lambda w_{x_{1}}^{2}
\\
& \geq 
c_{1} |\nabla w|^{2}
.
\end{align}
Now we compute the coefficient of $ w^{2} $.
Thanks to \cref{AB1,eqISPH1,eqISPH2,eqISPH3}, we have 
\begin{align}
\label{eqISPH7}
\notag
\mathcal{A} 
& = \ell_{t}^{2} -\ell_{tt} - |\nabla \ell|^{2} + \Delta \ell - \Psi
\\
& =
4 \lambda^{2} [(t+m)^{2} - |x' - x'_{0} |^{2} - \beta^{2}x_{1}^{2}]  + (c_{0} - 4) \lambda
.
\end{align}
From \cref{eqISPH2,eqISPH3,eqISPH4,eqISPH7}, we obtain 
\begin{align}
\label{eqISPH8}
(\mathcal{A} \ell_{t} )_{t}
=
8 \lambda^{3} [ 3 (t+m)^{2} - |x' - x'_{0} |^{2} - \beta^{2}x_{1}^{2}]  + O(\lambda^{2}),
\end{align}
and 
\begin{align}
\label{eqISPH9}
\notag
\sum_{j=1}^{n}(\mathcal{A} \ell_{x_{j}} )_{x_{j}}
& =
- 8 (n-1-\beta)  \lambda^{3} [   (t+m)^{2} - |x' - x'_{0} |^{2} - \beta^{2}x_{1}^{2}]  
\\
& \quad 
+ 16 \lambda^{3} |x' - x'_{0}|^{2}
-16 \beta^{3} \lambda^{3}   x_{1}^{2}
+ O(\lambda^{2})
.
\end{align}
Combining \cref{eqISPH9,eqISPH8,eqISPH7,AB1,eqISPH1,eqbmH}, for $ c_{2} \leq m^{2} - \beta^{2} l^{2} $, we conclude that 
\begin{align*}
\mathcal{B} 
& = 
4 \lambda^{3} [ (8- c_{0})  (t+m)^{2}  + c_{0} |x' - x'_{0} |^{2} 
+ \beta^{2} (-4-4\beta + c_{0}) x_{1}^{2}]  
+ O(\lambda^{2})
\\
& \geq 
28 \lambda^{3} ( m^{2} - \beta^{2} \ell^{2}) + O(\lambda^{2})
\\
& \geq 
c_{2} \lambda^{3} + O(\lambda^{2})
.
\end{align*}
Hence, there exists $ \lambda_{1} >0 $ so that for any $ \lambda > \lambda_{1} $, 
\begin{align}
\label{eqISPH10}
\mathcal{B} w^{2} \geq  c_{2} \lambda^{3} w^{2}.
\end{align}

\smallskip

\emph{Step 2.} In this step, we apply \cref{thmFi} to the equation \cref{ch-5-system1}.

\smallskip

Since  $ z(T) = z_{t}(T) = 0 $, we have $ w(T) = w_{t}(T) = 0 $. Also, noting that $ z = 0  $ on $ \Sigma $, ${\mathbb{P}}$-a.s., integrating (\ref{hyperbolic2}) in $Q$, taking mathematical expectation and by \cref{eqISPH4,eqISPH5,eqISPH6,eqISPH10},   for $ \lambda \geq \lambda_{1} $, we obtain
\begin{align}
\label{eqISPH11}
\notag
&  \mathbb{E} \int_Q \theta \big[ ( -2\ell_t w_t + 2 \nabla \ell \nabla w  + \Psi w  )  (dz_t-  \Delta z  dt  ) \big]dx 
+   \lambda \mathbb{E} \int_{\Sigma}  \theta^{2} \bigg|\frac{\partial z}{\partial\nu}\bigg|^{2} d \Gamma d t   
\\\notag
& \quad 
+ \mathbb{E} \int_{G} \big(\ell_{t} |\nabla w|^{2} - 2 \nabla \ell \cdot \nabla w w_{t} + \ell_{t} w_{t}^{2} - \Psi w_{t} w + \mathcal{A} \ell_{t} w^{2}\big) \big |_{t=0} d x 
\\  \notag
& \geq 
{\cal C} \lambda \mathbb{E}\int_Q \theta^2  (  z_t^2 +  |\nabla z|^2  
+   \lambda^2 z^2  )dxdt 
+
\mathbb{E} \int_Q  \theta \ell_t |d z_t|^2 d x
\\ 
& \quad 
+ \mathbb{E}\int_Q (-2\ell_t w_t + 2 \nabla \ell \cdot \nabla w   + \Psi w )^2
dxdt. 
\end{align}

At $ t = 0 $, from \cref{eqISPH7,eqISPH1,eqVarphiH,eqbmH}, there exists $ \lambda_{2} >0  $ such that for all $ \lambda \geq \lambda_{2} $, we have 
\begin{align}
\label{eqISPH12}
\notag
& \ell_{t} |\nabla w|^{2} - 2 \nabla \ell \cdot \nabla w w_{t} + \ell_{t} w_{t}^{2} - \Psi w_{t} w + \mathcal{A} \ell_{t} w^{2}
\\ \notag
& \geq
2 \lambda m (|\nabla w |^{2} + w_{t}^{2}) 
- 2 \lambda (|x' - x'_{0}| + \beta l ) |\nabla w| |w_{t}| 
- \mathcal{C} \lambda | w_{t}| |w| 
\\ \notag
& \quad 
+ 8  \lambda^{3} m (m^{2} - |x'- x'_{0}|^{2}- \beta^{2} l^{2}) w^{2}
+ O(\lambda^{2}) w^{2}
\\ \notag
& \geq
\mathcal{C} \lambda (w_{t}^{2} + |\nabla w|^{2} +  \lambda^{2} w^{2})
\\
&\geq
\mathcal{C} \lambda \theta^{2} (z_{t}^{2} + |\nabla z|^{2} +  \lambda^{2} z^{2})
.
\end{align}
It follows from \cref{eqVarphiH} that
\begin{align}
\label{eqISPH13}
\mathbb{E} \int_Q  \theta \ell_t |d z_t|^2 d x
\geq
\mathcal{C} \mathbb{E} \int_{Q} \theta^{2} g^{2} d x d t  
- 2 |b_{4}|_{_{L_{{\mathbb{F}}}^{\infty}(0,T;L^{\infty}(G))}}  \lambda m \mathbb{E} \int_{Q} \theta^{2} z^{2} d x d t  
.
\end{align}
By \cref{ch-5-system1}, we get that 
\begin{align}
\notag
\label{eqISPH14}
&  \mathbb{E} \int_Q \theta \big[ ( -2\ell_t w_t + 2 \nabla \ell \nabla w  + \Psi w  )  (dz_t-  \Delta z  dt  )  \big]dx 
\\ \notag
& 
\leq {\cal C}\Big[\mathbb{E}\int_Q \theta^2\big( |f|^2 + 
|g|^2\big)dxdt 
+ |b_1|^2_{L^{\infty}_{{\mathbb{F}}}(0,T;L^{\infty}(G))}  \mathbb{E}\int_Q \theta^2 z _t^2 dxdt 
\\ 
& \qquad \;  
+|b_2|^2_{L^{\infty}_{{\mathbb{F}}}(0,T;L^{\infty}(G,\mathbb{R}^n))}
\mathbb{E}\int_Q\theta^2 |\nabla z |^2dxdt    
+ \mathbb{E}\int_Q \theta^2  b_3^2 z ^2 dxdt \nonumber
\\
& \quad \quad  \;
+ |b_4|_{L^{\infty}_{{\mathbb{F}}}(0,T;L^{\infty}(G))}^2 \mathbb{E} \int_Q
\theta^2 z ^2 dxdt
 \Big ]\nonumber
\\
&  \qquad \; + \mathbb{E} \int_Q   \big( - 2\ell_t w_t + 2 \nabla \ell \cdot \nabla w +  \Psi w \big)^2dxdt
.
\end{align}
Combining \cref{eqISPH11,eqISPH12,eqISPH13,eqISPH14}, similar to \cref{eqCarHyper10}, we know there exists 
\begin{align*}
\lambda_{0} \geq \max\{\lambda_{1}, \lambda_{2}\} 
\end{align*}
such that for all $ \lambda \geq \lambda_{0} $,   \cref{eqCarHyperISP} holds. 
\end{proof}

\begin{proof}[Proof of \cref{thmISPHyper}]
Set  $ z = R u  $. 
From \cref{systemHyper}, we have 
\begin{align}
    \label{eqISPT1}
    \left\{\begin{aligned} 
        & d u_{t}-\Delta u d t  =       [(b_{1} -2  R^{-1} R_{t} ) u_{t}   + (2 R^{-1} \nabla R + b_{2}) \cdot \nabla u 
        &&  \\
        & \quad \quad \quad \quad \quad  \quad ~ 
        + R^{-1}(- R_{tt} + \Delta R + b_{1} R_{t}   + b_{2} \cdot \nabla R + b_{3} R) u]d t   
        &&  \\
        & \quad \quad \quad \quad \quad  \quad ~ 
        +  f (t, x^{\prime} )   d t+ (R^{-1} g+b_{4} u) d W(t)  &&  \text {in } Q, \\
        & u=\frac{\partial u}{\partial \nu}  =0  &&\text {on } \Sigma. \\
        \end{aligned}\right.
\end{align}
Differentiate both side of \cref{eqISPT1} with respect to $ x_{1} $, and set $ w = u_{x_{1}} $. 
Noting that $ \nabla u  = \dfrac{\partial u }{ \partial \nu} \nu = 0 $ on $ \Sigma $, we obtain that 
\begin{align}
    \label{eqISPT2}
    \left\{\begin{aligned} 
    & d w_{t}-\Delta w d t  =       
    [(b_{1} -2  R^{-1} R_{t} ) w_{t}   + (2 R^{-1} \nabla R + b_{2}) \cdot \nabla w 
    &&  \\
    & \quad \quad \quad \quad \quad  \quad \quad  
    + R^{-1}(- R_{tt} + \Delta R + b_{1} R_{t} w + b_{2} \cdot \nabla R + b_{3} R) w]d t   
    &&  \\
    & \quad \quad \quad \quad \quad  \quad ~ 
    + \{(b_{1} -2  R^{-1} R_{t} )_{x_{1}} u_{t}   + (2 R^{-1} \nabla R + b_{2})_{x_{1}} \cdot \nabla u 
    &&  \\
    & \quad \quad \quad \quad \quad  \quad \quad \quad  
    + [R^{-1}(- R_{tt} + \Delta R + b_{1} R_{t}   + b_{2} \cdot \nabla R + b_{3} R)]_{x_{1}} u\}d t   
    &&  \\
    & \quad \quad \quad \quad \quad  \quad ~ 
    + ( (R^{-1} g)_{x_{1}}+b_{4} w + b_{4,x_{1}} u ) d W(t)  &&  \text {in } Q, \\
    & w  =0  &&\text {on } \Sigma. \\
    \end{aligned}\right.
\end{align}
Thanks to \cref{thmCarlemanHyperISP}, there exists $ \lambda_{1} > \lambda_{0} $ such that for all $ \lambda \geq \lambda_{1} $, it holds that 
\begin{align}
    \label{eqISPT3} \notag
    &\lambda \mathbb{E}  \int_Q \theta^{2}  (|(R^{-1}g)_{x_{1}} + b_{4, x_{1}} u| ^2+w _t^2+|\nabla w |^2+\lambda^2 w ^2 )  d x d t 
    \\
    & \leq  \mathcal{C} \mathbb{E}\bigg\{\lambda \int_{\Sigma} \theta^{2} \bigg|\frac{\partial w }{\partial \nu}\bigg|^2   d \Gamma d t +\int_Q \theta^{2} \tilde{f}^{2} d x d t\bigg\},
\end{align}
where 
\begin{align}
    \label{eqISPT4} \notag
    \tilde{f} 
    & =
    (b_{1} -2  R^{-1} R_{t} )_{x_{1}} u_{t}   + (2 R^{-1} \nabla R + b_{2})_{x_{1}} \cdot \nabla u 
     \\
    &  \quad 
    + [R^{-1}(- R_{tt} + \Delta R + b_{1} R_{t}   + b_{2} \cdot \nabla R + b_{3} R)]_{x_{1}} u
    .
\end{align}
Since $ w  = u_{x_{1}} $, and $ u(t, 0, x') = z (t, 0, x') = 0 $ for $ (t, x') \in (0,T) \times G' $, we have 
\begin{align}
    \label{eqISPT4-1}
    u (t, x_{1}, x') = \int_{0}^{x_{1}} w (t, \eta, x') d \eta.
\end{align}
Similar to \cref{eq.s45-5,eq.s45-6}, from \cref{eqISPT4,eqISPT4-1} we obtain 
\begin{align}
    \label{eqISPT5}
    \int_Q \theta^{2} \tilde{f}^{2} d x d t
    \leq 
    \mathbb{E}  \int_Q \theta^{2}  ( w _t^2+|\nabla w |^2+ w ^2 )  d x d t 
    ,
\end{align}
and 
\begin{align} \notag
    \label{eqISPT5-1}
    & \lambda \mathbb{E}  \int_Q \theta^{2}  (|(R^{-1}g)_{x_{1}} + b_{4, x_{1}} u| ^2 d x d t 
    \\
    & 
    \geq 
    \frac{1}{2} \lambda \mathbb{E}  \int_Q \theta^{2}  (|(R^{-1}g)_{x_{1}}  | ^2 d x d t 
    - \mathcal{C} \lambda \mathbb{E}  \int_Q \theta^{2}   w ^2    d x d t 
    .
\end{align}
From \cref{eqISPT3,eqISPT5,eqISPT5-1}, we know there exists $ \lambda_{2} \geq \lambda_{1} $ such that for all $ \lambda \geq \lambda_{2} $, it holds that
\begin{align*}
 \lambda \mathbb{E}  \int_Q \theta^{2}  (|(R^{-1}g)_{x_{1}}| ^2+w _t^2+|\nabla w |^2+\lambda^2 w ^2 )  d x d t 
 \leq  \mathcal{C} \mathbb{E} \lambda \int_{\Sigma} \theta^{2} \bigg|\frac{\partial w }{\partial \nu}\bigg|^2   d \Gamma d t  
 .
\end{align*}
Since $ \dfrac{\partial w}{\partial \nu} = \dfrac{\partial u_{x_{1}}}{\partial \nu} = 0  $ on $ \Sigma $, for $ \lambda \geq \lambda_{2} $, we have 
\begin{align*}
\lambda \mathbb{E}  \int_Q \theta^{2}  \big(|(R^{-1}g)_{x_{1}}| ^2+w _t^2+|\nabla w |^2+\lambda^2 w ^2 \big)  d x d t 
\leq  0.
\end{align*}
This yields that 
\begin{align*}
w = 0 \text{ in } (0, T) \times G, ~ \text{${\mathbb{P}}$-a.s.,} 
\end{align*}
which, together with \cref{eqISPT4-1}, implies that
\begin{align*}
u = 0 \text{ in } (0, T) \times G, ~ \text{${\mathbb{P}}$-a.s.} 
\end{align*}
Consequently,
\begin{align}\label{eqISPT4-2}
z = 0 \text{ in } (0, T) \times G, ~ \text{${\mathbb{P}}$-a.s.} 
\end{align}
From \cref{systemHyper} and \cref{eqISPT4-2}, for any $ t \in (0,T) $, we have 
\begin{align}\label{eqISPT4-3}
\int_{0}^{t} h(s, x') R(s, x) ds + \int_{0}^{t} g(s,x) d W(s) = 0.
\end{align}
Noting that $\int_{0}^{t} h(s, x') R(s, x) ds$ is absolutely continuous with respect to $t$, we get from \cref{eqISPT4-3} that $\int_{0}^{t} g(s,x) d W(s)$ is also absolutely continuous with respect to $t$. Consequently, its quadratic variation is zero, which concludes  that 
\begin{align*}
 g  = 0 \text{ in } (0, T) \times G, ~ \text{${\mathbb{P}}$-a.s.} 
\end{align*}
This, together with \cref{eqISPT4-3}, implies that  
\begin{align*}
\int_{0}^{t} h(s, x') R(s, x) =  0,\qquad \forall t \in (0,T). 
\end{align*}
Therefore,
\begin{align*}
h(s, x') = 0 \text{ in } (0, T) \times G, ~ \text{${\mathbb{P}}$-a.s.} 
\end{align*}
\end{proof}

\section{Solution to inverse state problem with the internal measurement}
\label{secLIP}

In this section, we study \cref{probLocalInverse}. To this end, we need the following definition.

\begin{definition}\label{def1}
Assume $x_0\in S$ and $ K$ is a nonnegative number. 
Given the hypersurface $S$, it is said to  satisfy the outer
paraboloid condition with $K$ at $x_0$ if there exists a
neighborhood ${\cal V}$ of $x_0$ and a paraboloid ${\cal P}$ tangential to $S$
at $x_0$ such that ${\cal P}\cap{\cal V}\subset {\cal D}^{-}_\rho$ with ${\cal P}$ being congruent to $\displaystyle x_1 = - K\sum_{j=2}^n x_j^2$.
\end{definition}

The following theorem gives a positive answer to \cref{probLocalInverse}.

\begin{theorem}\label{thmInveHyperDeter}
Let $ x_0 \in S\setminus\partial S$ satisfy $\frac{\partial  \sigma(T/2,x_0)}{\partial\nu}<0$, and  $S$ satisfy the outer  paraboloid condition with
\begin{equation}\label{th1-eq1}
K <\frac{-\frac{ \partial
\sigma}{\partial\nu}(T/2,x_0)}{4(|\sigma|_{L^\infty(B_\rho(T/2,x_0))}+1)}.
\end{equation}
If $z \in   L_{{\mathbb{F}}}^2 (\Omega; C([0,T];H_{\rm loc}^1(G)))\cap L_{{\mathbb{F}}}^2 (\Omega; C^{1}([0,T];L^2_{\rm loc}(G))) $ is a solution to the equation     \cref{system1} satisfying that
\begin{equation}\label{th1-eq2}
z=\frac{\partial z}{\partial\nu} = 0 \quad\mbox{ on } (0,T)\times
S,\;\;{\mathbb{P}}\mbox{-}\hbox{\rm a.s.} ,
\end{equation}
then for the neighborhood ${\cal V}$ given in  \cref{def1} and some $\varepsilon\in (0,T/2)$,
it holds that
\begin{equation}\label{th1-eq3}
z = 0 \quad\mbox{ in } ({\cal V}\cap {\cal D}_{\rho}^+)\times
(\varepsilon,T-\varepsilon),\;\;{\mathbb{P}}\mbox{-}\hbox{\rm a.s.}
\end{equation}
\end{theorem}

\begin{remark}
Noting that in  \cref{thmInveHyperDeter}, we assume that $\frac{\partial  \sigma(T/2,x_0)}{\partial\nu}<0$, which can be regarded as  a kind of pseudoconvex conditions (e.g.,
\cite[Chapter XXVII]{Hormander}).  
This assumption is reasonable . 
In fact, even in the deterministic setting, without such an assumption, it may lead to the involved (local) uniqueness failing (e.g., \cite{Alinhac1983}).
\end{remark}

We first give the weight function which will be used to establish
a   Carleman estimate for the equation \cref{system1}.

Without loss of generality, we assume that $ x_{0} = 0 \in S \setminus \partial S$ and $\nu(0)=(1, 0, \cdots,0)$.
Since $S$ is $C^2$, for some $r >0$ and $\gamma(\cdot)\in C^2(\mathbb{R}^{n-1})$, we can parameterize $S$ in a neighborhood of the origin by
\begin{equation}\label{car eq1}
\big\{(x_1,x_2,\cdots,x_n)\in \mathbb{R}^{n}\;\big|\;x_1=\gamma(x_2,\cdots,x_n),\;|x_2|^2+\cdots+|x_n|^2 <r\big\}.
\end{equation}

We put
\begin{equation}\label{car eq3}
\begin{cases}
\displaystyle B_r\Big (\frac T2,0\Big )=\left\{(t,x)\in{\mathbb{R}}^{n+1}\;\Big|\; |x|^2
+ \Big (t-\frac T2\Big )^2 < r^2\right\},\\
B_r(0)=\big\{ x\in{\mathbb{R}}^n\;\big|\;  |x|<r \big\}.
\end{cases}
\end{equation}
Then, it follows from \cref{th1-eq1} that
\begin{align}\label{car eq2}
\left\{
\begin{aligned}
    & -\alpha_0  \deq \frac{\partial
\sigma}{\partial \nu}\Big (\frac T2,0 \Big )<0,\\ 
& K <\frac{\alpha_0}{4(|\sigma|_{L^\infty(B_r(T/2,0))}+1)}.
\end{aligned}
\right.
\end{align}
Thanks to \cref{def1,car eq1}, we get 
\begin{align}
\label{eqParaHyper}
 - K \sum_{j=2}^n x_j^2 <\gamma(x_2,\cdots,x_n), \quad  \text{ if }\sum_{j=2}^n x_j^2 < r.
\end{align}

Denote
\begin{align*}
{\cal D}^{-}_r=\big\{x\in B_r(0)\;\big|\;x_1<\gamma(x_2,\cdots,x_n)\big\}
,\quad {\cal D}^+_r = B_r(0)\setminus \overline{D^-_r}.
\end{align*}
Since $\sigma\in C^1(\overline Q)$, for any $ \alpha \in (0, \alpha_{0}) $, we can find a sufficiently small $ \delta_{0} > 0 $ such that $ \delta_{0} < \min\{ 1, r^{2}\} $ and 
\begin{align}
\label{eqSnH}
\frac{\partial \sigma(t,x)}{\partial \nu}<-\alpha,  \quad  \text{ if } |x|^2 + \Big (t-\frac T2\Big )^2 \leq \delta_0.
\end{align}

Put 
\begin{align}
\label{car eq4}
M_0=|\sigma|_{L^\infty(B_r(T/2,0))},\quad M_1=\max\left\{ |\sigma|_{C^1(B_r(0,0))},1 \right\}.
\end{align}
From \cref{car eq2}, we can choose $ \alpha >0$ and $ h > 0 $ such that 
\begin{align}
\label{car eq6}
K<\frac{1}{2h}<\frac{\alpha}{4(M_0+1)}.
\end{align}

For $ \tau \in (0,1) $, denote the weight function
\begin{align}
\label{weight1}
\varphi(t,x) = h x_1 + \frac{1}{2}\sum_{j=2}^n x_j^2 +
\frac{1}{2}\Big (t-\frac T2\Big )^2 + \frac{1}{2}\tau.
\end{align}
For any positive number $ \mu $ with $  \tau \geq \mu >  \frac{\tau}{2} $, put 
\begin{align}
\label{eqDefQm}
Q_\mu = \Big\{ (t,x)\in{\mathbb{R}}^{n+1}\;\Big|\; x_1>\gamma
(x_2,x_3,\cdots,x_n),\; \sum_{j=2}^N x_j^2<\delta_0, \varphi(t,x)<\mu
\Big\}.
\end{align}
We claim that $ Q_{\mu} \neq \emptyset $.

In fact, from \cref{eqParaHyper,eqDefQm}, we have 
\begin{align*}
h x_{1} > h \gamma(x_2,\cdots,x_n) > - h K \sum_{j=2}^n x_j^2,
\end{align*}
which, together with  \cref{weight1,car eq6}, implies that 
\begin{align*}
\varphi (t,x ) > \biggl(\frac{1}{2} - h K \biggr) \sum_{j=2}^n x_j^2 + \frac{\tau}{2} \geq \frac{\tau}{2} 
.
\end{align*}
This, together with \cref{eqDefQm}, yields $Q_\mu \neq \emptyset$.

Next, we show that  there exists $ \tau_{1} > 0 $  such that for all $ \tau \leq \tau_{1} $ and $ \mu \leq \tau $, we have 
\begin{align}
\label{eqtxQm}
|x|^{2}  + \bigg (t-\frac T2\bigg )^2 \leq \delta_0, \quad  \forall \; (t,x ) \in Q_{\mu}.
\end{align}

Since $ \varphi < \mu \leq  \tau $, from \cref{weight1}, we have 
\begin{align}
\label{eqTau1}
x_{1} < - \frac{1}{2 h} \sum_{j=2}^n x_j^2 -     \frac{1}{2 h}\Big (t-\frac T2\Big )^2 + \frac{1}{2 h}\tau \leq \frac{\tau}{ 2 h}
.
\end{align}
By \cref{eqTau1} and
\begin{align*}
- K \sum_{j=2}^n x_j^2  < \gamma(x_2,\cdots,x_n) < x_{1}
,
\end{align*}
we get 
\begin{align}
\label{eqxj}
\sum_{j=2}^n x_j^2 \leq \frac{\tau}{1- 2 K h},
\end{align}
which implies 
\begin{align}
\label{eqTau2}
- x_{1} \leq \frac{K \tau}{ 1 - 2 K \tau}.
\end{align}
From \cref{eqTau2,eqTau1}, we deduce that 
\begin{align}
\label{eqTau3}
|x_{1} | \leq \max \bigg \{  \frac{K }{ 1 - 2 K \tau}, \frac{1}{2 h} \bigg \}\tau
.
\end{align}
Combining \cref{eqTau3,eqxj}, noting that $ \varphi < \tau $, we obtain that 
\begin{align*}
- \frac{K h \tau}{ 1 - 2 K \tau} + \frac{1}{2} \sum_{j=2}^n x_j^2 
+ \frac{1}{2} \biggl( t - \frac{T}{2} \biggr)^{2} + \frac{\tau}{2} \leq \varphi < \tau,
\end{align*}
which implies
\begin{align}
\label{eqT}
\bigg (t-\frac T2\bigg )^2 \leq \frac{\tau}{1- 2 K h}.
\end{align}
From \cref{eqTau3,eqT}, for $ (t,x) \in Q_{\mu} $, we have 
\begin{align}
\label{eqMuTao}
|x|^{2}  + \bigg (t-\frac T2\bigg )^2 
\leq  \mu_0(\tau) \deq   \Big|\max\left\{\frac K {1- 2K h}, \frac 1{2h}\right\}\Big|^2 \tau ^2 + \frac{2\tau} {1 -2K h} 
.
\end{align}
Choosing  $ \tau_{1} > 0 $ such that $ \mu_0(\tau_{1}) \leq \delta_{0} $, for all $ \tau \leq \tau_{1} $, we obtain \cref{eqtxQm}.

\begin{theorem}
\label{thmCarlemanFinal}
Let $u$ be an $H^2_{\rm loc}(G)$-valued $\mathbf{F}$-adapted process such that $u_t$ is an $L^2_{\rm loc}(G)$-valued It\^o process. 
Set $ \ell = s \varphi^{-1} $ with $ \varphi $ given by \cref{weight1}, $\theta = e^\ell$ and $w=\theta u$. 
Then there exists $ \tau_{0} >0 $ such that for all $0< \tau \leq \tau_{0} $ and $u$ supported  in $Q_{\tau}$, there exist positive constants ${\cal C} $, $ \lambda_{0} $ and   $s_0$, such that for all $ \lambda \geq \lambda_{0} $ and  $s \geq s_0$, it holds that 
\begin{align}\label{4.15-eq20}
&  {\mathbb{E}}\int_{Q_\tau}\theta \big( -2\sigma\ell_t w_t +
2\nabla\ell\cdot\nabla w \big)
\big( \sigma du_t - \Delta u dt \big)dx\nonumber  \\
& \geq {\cal C} {\mathbb{E}}\int_{Q_\tau} \Big[s\lambda^2
\varphi^{-\lambda-2}(|\nabla w|^2+w_t^2)+
s^3\lambda^4 \varphi^{-3\lambda-4}w^2 \Big]dxdt \\
& \quad + {\mathbb{E}}\int_{Q_\tau}\big( -2\sigma \ell_tw_t +
2\nabla\ell\cdot \nabla w \big)^2 dxdt + {\cal C}
{\mathbb{E}}\int_{Q_\tau}\sigma^2\theta^2\ell_t(du_t)^2dx. \nonumber
\end{align}
\end{theorem}

\begin{proof}
Choose $ (b^{jk})_{1\leq j,k\leq n} $ be the identity matrix,  $ \Psi = 0 $ and $ \phi = \sigma $ in \cref{thmFi}. Let us first estimate terms in \cref{hyperbolic2} one by one.

For $j=1,2,\cdots,n$, noting that  $ \ell = s \varphi^{-1} $ and \cref{weight1}, it is straightforward  to verify  that
\begin{equation}\label{hyperL}
\left\{
\begin{aligned}
    & \ell_t= -s \lambda \varphi^{-\lambda-1} \Big(t-\frac{T}{2}\Big), 
    \\ 
    &\ell_{tt}  =s \lambda(\lambda+1)  \varphi^{-\lambda-2} \Big(t-\frac{T}{2}\Big) ^2 -s \lambda \varphi^{-\lambda-1},
    \\
    &\ell_{x_j}= -s \lambda \varphi^{-\lambda-1} \varphi_{x_{j}},
    \\
    &\ell_{x_j x_{j}}= s \lambda(\lambda+1)  \varphi^{-\lambda-2}  \varphi_{x_{j}}^{2} -s \lambda \varphi^{-\lambda-1} \varphi_{x_{j}x_{j}},
    \\ 
    &\ell_{t x_j}= s \lambda(\lambda+1)  \varphi^{-\lambda-2}  \varphi_{x_{j}}  \Big(t-\frac{T}{2}\Big).
\end{aligned}
\right.
\end{equation}

From \cref{hyperL,eqSnH}, we deduce that 
\begin{align}
\label{eqCarHypIII1}\notag
& [ (\sigma^2 \ell_t )_t+\nabla \cdot(\sigma \nabla \ell) ] w_t^2
\\ \notag
& =   (2 \sigma \sigma_t \ell_t+\sigma^2 \ell_{t t}+\nabla \sigma \cdot \nabla \ell+\sigma \Delta \ell ) w_t^2 
\\ \notag
& =  -s \lambda \varphi^{-\lambda-1}\Big[2 \sigma \sigma_t\Big(t-\frac{T}{2}\Big)+\sigma^2+(\nabla \sigma \cdot \nabla \varphi+\sigma \Delta \varphi)\Big] w_t^2 
\\ \notag
& \quad  +s \lambda(\lambda+1) \varphi^{-\lambda-2}\Big[\sigma^2\Big(t-\frac{T}{2}\Big)^2+\sigma|\nabla \varphi|^2\Big] w_t^2
\\
& \geq 
s \lambda^{2} h^2 \sigma \varphi^{-\lambda-2} w_t^2  
+ O(s \lambda \varphi^{-\lambda - 1}) w_{t}^{2}
.
\end{align}
Combining \cref{hyperL,eqSnH,eqMuTao}, we have 
\begin{align}
\label{eqCarHypIII2}\notag
& -2\big[(\sigma \nabla \ell)_t+\nabla(\sigma \ell_t)\big] \cdot \nabla w w_t
\\ \notag
& = 
2 s \lambda \varphi^{-\lambda -1 } \Big[ \sigma_{t} \nabla \varphi + \nabla \sigma \Big( t - \frac{T}{2} \Big) \Big] \cdot \nabla w w _{t} 
\\ \notag
& \quad 
- 4  s \lambda (\lambda+1) \varphi^{-\lambda -2 } \sigma \Big( t - \frac{T}{2} \Big) \nabla \varphi \cdot \nabla w w _{t} 
\\ \notag
& \geq 
- 2 s \lambda^{2} \varphi^{-\lambda -2 } \big( h M_{1} \sqrt{\mu_{0}(\tau)} + M_{1} \sqrt{\mu_{0}(\tau)}\big) \big(|\nabla w|^{2} + |w_{t}|^{2}\big)
\\
& \quad 
+ O(s \lambda \varphi^{-\lambda -1}) (|\nabla w|^{2} + |w_{t}|^{2})
.
\end{align}
Thanks to \cref{eq.identyDetail,hyperL,weight1}, we obtain that
\begin{align}
\label{eqCarHypIII3}\notag
& \big[(\sigma \ell_t)_t+\Delta \ell\big]|\nabla w|^2 
\\ \notag
& =  
-s \lambda \varphi^{-\lambda-1}\Big[\sigma_t\Big(t-\frac{T}{2}\Big)+\sigma+\Delta \varphi\Big]|\nabla w|^2 
\\ \notag
& \quad 
+s \lambda(\lambda+1) \varphi^{-\lambda-2}\Big[\sigma\Big(t-\frac{T}{2}\Big)^2+|\nabla \varphi|^2\Big]|\nabla w|^2
\\
& \geq 
s \lambda^{2} \varphi^{-\lambda - 2 } h^{2} |\nabla w|^{2}
+ O(s \lambda \varphi^{-\lambda -1}) |\nabla w|^{2}.
\end{align}
Hence, from \cref{eqCarHypIII3,eqCarHypIII2,eqCarHypIII1}, by letting $ \sigma_{0} = \min \{\sigma, 1\} $, we know there exist $ \tau_{2} > 0 $ with $ \tau_{2} < \tau_{1} $  and $ \lambda_{1} >  0 $ such that for all $ \tau \leq \tau_{2} $ and $ \lambda \geq \lambda_{1} $,  it holds that
\begin{align}
\label{eqCarHypIII4}\notag
& \big[ (\sigma^2 \ell_t )_t+\nabla \cdot(\sigma \nabla \ell) \big] w_t^2
-2\big[(\sigma \nabla \ell)_t+\nabla(\sigma \ell_t)\big] \cdot \nabla w w_t
+ \big[(\sigma \ell_t)_t+\Delta \ell\big]|\nabla w|^2 
\\ \notag
& \geq 
s \lambda^{2} \varphi^{- \lambda - 2} \big(h^{2} \sigma_{0} - 2 h   M_{1} \sqrt{\mu_{0}(\tau)} - M_{1} \sqrt{\mu_{0}(\tau)}\big) \big(|\nabla w|^{2} + |w_{t}|^{2}\big)
\\ \notag
& \quad 
+ O\big(s \lambda \varphi^{-\lambda -1}\big) (|\nabla w|^{2} + |w_{t}|^{2})
\\
& \geq 
\mathcal{C} s \lambda^{2} \varphi^{- \lambda - 2} (|\nabla w|^{2} + |w_{t}|^{2})
.
\end{align}

From \cref{AB1,hyperL,weight1}, it is easy to show that 
\begin{align*}
\mathcal{A} 
& = 
\sigma (\ell_t^2-\ell_{t t} )-(|\nabla \ell|^2-\Delta \ell )
\\
& \geq 
s^{2} \lambda^{2} \varphi^{- 2 \lambda - 2} \Big[ \sigma \biggl( t - \frac{T}{2} \biggr)^{2} - |\nabla \varphi|^{2} \Big] 
+ O(s \lambda^{2} \varphi^{- \lambda - 2})
.
\end{align*}
This, together with \cref{hyperL,weight1,eqMuTao}, yields that there exist positive constants $ \tau_{3}, \lambda_{2}, s_{1}  $ with $ \tau_{3} < \tau_{1} $, such that for all $ \tau \leq \tau_{3} $, $ \lambda \geq \lambda_{2} $ and $ s \geq s_{1} $, it holds that 
\begin{align}
\label{eqCarHypIII5} \notag
\mathcal{B} 
& =
(\sigma \mathcal{A} \ell_t )_t-\nabla \cdot(\mathcal{A} \nabla \ell) 
\\ \notag
& =
\sigma_t \mathcal{A} \ell_t+\sigma \mathcal{A}_t \ell_t+\sigma \mathcal{A} \ell_{t t}-\nabla \mathcal{A} \cdot \nabla \ell-\mathcal{A} \Delta \ell
\\ \notag
& \geq 
3 s^{3} \lambda^{4} \varphi^{-3 \lambda - 4} \Big[ \sigma \Big( t - \frac{T}{2} \Big)^{2} - |\nabla \varphi|^{2} \Big]^{2}
+ O(s^{3} \lambda^{3} \varphi^{- 3 \lambda - 4}) 
+ O(s^{2} \lambda^{4} \varphi^{- 2 \lambda - 3}) 
\\ \notag
& \geq 
3 s^{3} \lambda^{4} \varphi^{-3 \lambda - 4}  \big[h^{2} + (1-M_{0}) \mu_{0}(\tau)\big]^{2}
+ O\big(s^{3} \lambda^{3} \varphi^{- 3 \lambda - 4}\big) 
+ O\big(s^{2} \lambda^{4} \varphi^{- 2 \lambda - 3}\big) 
\\ 
& \geq 
\mathcal{C} s^{3} \lambda^{4} \varphi^{-3 \lambda - 4}
.
\end{align}

Integrating \cref{hyperbolic2} over $ Q_{\tau} $, taking mathematical expectation, and noting that $u$ is supported  in $Q_{\tau}$, from \cref{eqCarHypIII4,eqCarHypIII5}, by choosing $ \tau_{0} = \min\{ \tau_{1}, \tau_{2}, \tau_{3}\} $, $ \lambda_{0} = \max \{ \lambda_{1}, \lambda_{2} \} $ and $ s_{0} = s_{1} $, we obtain \cref{4.15-eq20}.
\end{proof}

With the help of the Carleman estimate, we are in a position to prove \cref{thmInveHyperDeter}.

\begin{proof}[Proof of \cref{thmInveHyperDeter}]

Without loss of generality, we make the same assumption in \cref{car eq1}. 
By \cref{eqDefQm}, the boundary $ \Gamma_{\mu} $ of $ Q_{\mu} $ can be decomposed as the following three parts:
\begin{align*}
\left\{
    \begin{aligned}
        \Gamma_\mu^1=\Big\{(x, t) \in \mathbb{R}^{n+1} \mid x_1=\gamma\left(x_2, x_3, \ldots, x_n\right), \sum_{j=2}^n x_j^2<\delta_0, \varphi(x, t)<\mu\Big\}, \\
    \Gamma_\mu^2=\Big\{(x, t) \in \mathbb{R}^{n+1} \mid x_1>\gamma\left(x_2, x_3, \ldots, x_n\right), \sum_{j=2}^n x_j^2<\delta_0, \varphi(x, t)=\mu\Big\}, \\
    \Gamma_\mu^3=\Big\{(x, t) \in \mathbb{R}^{n+1} \mid x_1>\gamma\left(x_2, x_3, \ldots, x_n\right), \sum_{j=2}^n x_j^2=\delta_0, \varphi(x, t)<\mu\Big\}
    .
    \end{aligned}
\right.
\end{align*}

We claim that $ \Gamma_\mu^3 = \emptyset $. 
In fact, thanks to \cref{eqxj} and $ \sum\limits_{j=2}^n x_j^2=\delta_0 $, we obtain 
\begin{align*}
\delta_{0} \leq  \frac{\tau}{1- 2 K h},
\end{align*}
which contradicts $ \mu_{0}(\tau) \leq \delta_{0} $.
Consequently, $ \Gamma_{\mu} =  \Gamma_{\mu}^{1} \cup \Gamma_{\mu}^{2}    $.
Put $ t_0=\sqrt{\frac{\tau}{1-2 K h}}  $. 
From \cref{eqT}, we deduce that 
\begin{align}
\label{eqBoudndary}
\left\{
    \begin{aligned}
        & \Gamma_\mu^1 \subset\left\{x \mid x_1=\gamma\left(x_2, x_3, \ldots, x_n\right)\right\} \times\left\{t \mid | t-T / 2 \mid \leq t_0\right\}, \\
        & \Gamma_\mu^2 \subset\{x \mid \varphi(x, t)=\mu\}, \quad \mu \in(0, \tau] ,
    \end{aligned}
\right.
\end{align}
and $ \Gamma_{\mu}^{1} \subset \Gamma_{\tau}^{1} $ for $ \mu \leq \tau $.

Let $ Q_{1} = Q_{\tau} $, and 
for fixing $ \tilde{\tau} \in \big(0, \frac{\tilde{\tau}}{8}\big) $, let 
\begin{align}
\label{eqdefQk}
Q_{k+1}=\{(t, x) \in Q_{1} \mid \varphi(x, t)<\tau-k \widetilde{\tau}, k=1,2,3\} .
\end{align}
Then we have $ Q_{4} \subset Q_{3} \subset Q_{2} \subset Q_{1} $.

Choose a cut-off function $ \chi \in C_{0}^{\infty}(Q_{2};[0,1]) $ satisfying $ \chi = 1 $ in $ Q_{3} $. 
Put $ u = \chi z $. By \cref{th1-eq2,eqBoudndary}, $ u $ solves the following equation:
\begin{align}
\label{eqDerH1}
\left\{
    \begin{aligned}
        & \sigma d u_t-\Delta u d t= (b_1 u_t+b_2 \cdot \nabla u+b_3 u +f ) d t+b_4 u d W(t) & \text { in } Q_\tau 
        \\ 
        & u=0, \quad \frac{\partial u}{\partial \nu}=0 & \text { on } \Gamma_\tau 
        ,
    \end{aligned}    
\right.
\end{align}
where $ f=\sigma \chi_{t t} z+2 a \chi_t z_t-2 \nabla \chi \cdot \nabla z-z \Delta \chi-b_1 \chi_t z-b_2 \cdot z \nabla \chi $.
Then  \cref{4.15-eq20} holds.
From \cref{eqDerH1}, we have 
\begin{align}
\label{eqDerH2} \notag
& {\mathbb{E}}\int_{Q_\tau}\theta \big( -2\sigma\ell_t w_t +
2\nabla\ell\cdot\nabla w \big)
\big( \sigma du_t - \Delta u dt \big)dx
\\ \notag
& 
\leq {\mathbb{E}}\int_{Q_\tau}\big( -2\sigma \ell_tw_t +
2\nabla\ell\cdot \nabla w \big)^2 dxdt
+ {\mathbb{E}}\int_{Q_\tau} \theta^{2}|f|^{2} dx d t
\\
& \quad 
+ \mathcal{C} \mathbb{E} \int_{Q_{\tau}} \theta^{2} (|u_{t}|^{2} + |\nabla u|^{2} + |u|^{2} ) d x d t 
.
\end{align}
Combining \cref{eqDerH1,hyperL}, we get that 
\begin{align}
\label{eqDerH3}  
{\mathbb{E}}\int_{Q_\tau}\sigma^2\theta^2\ell_t(du_t)^2dx
\leq
\mathcal{C} {\mathbb{E}}\int_{Q_\tau}  \theta^2 |u|^{2} dx  d t
.
\end{align}
By \cref{4.15-eq20,eqDerH2,eqDerH3}, we deduce that 
\begin{align}
\label{eqDerH4} \notag  
& {\mathbb{E}}\int_{Q_\tau} \Big[s\lambda^2 \varphi^{-\lambda-2}(|\nabla w|^2+w_t^2)+ s^3\lambda^4 \varphi^{-3\lambda-4}w^2 \Big]dxdt
\\ \notag
&
\leq \mathcal{C} \mathbb{E} \int_{Q_{\tau}} \theta^{2} (|u_{t}|^{2} + |\nabla u|^{2} + |u|^{2} ) d x d t 
+{\mathbb{E}}\int_{Q_\tau} \theta^{2}|f|^{2} dx d t
\\  
&
\leq \mathcal{C} \mathbb{E} \int_{Q_{\tau}} \theta^{2} (|u_{t}|^{2} + |\nabla u|^{2} + |u|^{2} ) d x d t 
+\mathcal{C} {\mathbb{E}}\int_{Q_{2} \setminus \overline{Q_{3}}} \theta^{2} (|z_{t}|^{2} + |\nabla z|^{2} + |z|^{2} ) dx d t
.
\end{align}
From \cref{hyperL}, we have 
\begin{align}
\label{eqDerH5}   
|\nabla w|^2+w_t^2 \geq \mathcal{C} \theta^2\left(s^2 \lambda^2 \varphi^{-2 \lambda-2} u^2+|\nabla u|^2+u_t^2\right) .
\end{align}
Combining \cref{eqDerH5,eqDerH4}, and noting that $ u =z  $ in $ Q_{3} $, we know there exist $ \lambda_{1} > \lambda_{0} $ and $ s_{1} > s_{0} $ such that for all $ \lambda \geq \lambda_{1} $ and $ s \geq s_{1} $, it holds that 
\begin{align*}
& {\mathbb{E}}\int_{Q_3} \Big[s\lambda^2 \varphi^{-\lambda-2}(|\nabla z|^2+z_t^2)+ s^3\lambda^4 \varphi^{-3\lambda-4}z^2 \Big]dxdt
\\
&\leq
\mathcal{C} {\mathbb{E}}\int_{Q_{2} \setminus \overline{Q_{3}}} \theta^{2} (|z_{t}|^{2} + |\nabla z|^{2} + |z|^{2} ) dx d t
\end{align*}
If follows from \cref{eqdefQk} that
\begin{align*}
\left\{
    \begin{aligned}
        & \theta(t, x) >e^{s(\tau-3 \widetilde{\tau})^{-\lambda}},  && \forall \; (t,x) \in  Q_4  
        \\
        & e^{s(\tau-\widetilde{\tau})^{-\lambda}}<\theta(t, x)<e^{s(\tau-2 \widetilde{\tau})^{-\lambda}},  && \forall \; (t,x) \in  Q_2 \setminus \overline{Q_{3}}  
        .
    \end{aligned}
\right.
\end{align*}
Hence, we have 
\begin{align}\label{eqDerH5-1}  
\notag
& {\mathbb{E}}\int_{Q_4}  ( |\nabla z|^2+z_t^2  + z^2 )dxdt
\\
&\leq
\mathcal{C} e^{2[s(\tau-2 \widetilde{\tau})^{-\lambda}-s(\tau-3 \widetilde{\tau})^{-\lambda}]} {\mathbb{E}}\int_{ Q_{\tau}} \theta^{2} (|z_{t}|^{2} + |\nabla z|^{2} + |z|^{2} ) dx d t
.
\end{align}
Letting $ s \rightarrow \infty $ in \cref{eqDerH5-1}, we find $ z = 0 $ in $ Q_{4} $, which completes the proof.
\end{proof}

\section{Reconstruct  the unknown  state with a boundary measurement}\label{secReconstructHyper}

In this section, we study  \cref{Hyrestruct} for $g=0$ in \cref{ch-5-system1}.
Recall that the reconstruction problem means to reconstruct the solution $z$ from some observation or measurement.

Similar to Section \ref{secReconstructTer}, to solve  \cref{Hyrestruct}, we use Tikhonov regularization strategy. To begin with, we give the following immediate corollary of \cref{ch-5-observability,ch-5-energy ensi}.

Consider the following stochastic hyperbolic equation:
\begin{align}\label{ch-5-system1Cauchy}
	\left\{
	\begin{aligned}
		&dz_{t}  - \sum_{j,k=1}^n (b^{jk}z_{x_j})_{x_k}dt = \big(b_1 z_t+ b_2\cdot\nabla z  + b_3 z  + f\big)dt  
		+  b_{4} z dW(t) \!\!\!&  {\mbox { in }} Q, \\
		&z = 0 & \mbox{ on } \Sigma,
	\end{aligned}
	\right .
\end{align}
where coefficients $ b_{j}  ~(1\leq j \leq 4) $ satisfy \cref{ch-5-aibi}.

\begin{proposition}
	\label{ch-5-observability-1}
	Assume  \cref{condition of d} and \cref{ch-5-condition2} are
	satisfied. For any solution $z$ to the equation \cref{ch-5-system1Cauchy}, it holds that 
	\begin{align} \label{ch-5-obser esti2-1}
		\notag
		& |(z, z_{t})|_{L^2_{\mathbb{F}}(0,T; H^{1}(G)) \times L^2_{\mathbb{F}}(0,T; L^{2}(G))}\\
		& \leq
		e^{{\cal C}\big(r_1^2 + r_2^{\frac{1}{ 3/2 - n/p}}+1\big)}
		\Big(\Big|\frac{\partial z}{\partial \nu}\Big
		|_{L^2_{{\mathbb{F}}}(0,T;L^2(\Gamma_0))} + |f|_{L^2_{{\mathbb{F}}}(0,T;L^2(G))}  \Big),
	\end{align}
	where $ \Gamma_{0} \subset \Gamma $ is defined in \cref{ch-5-def gamma0}. 
\end{proposition}

Denote $ \mathcal{H} = L^{2}_{\mathbb{F}}(0,T; H^{2}(G)) \cap L^{2}_{\mathbb{F}}(\Omega; H^{1} (0,T; L^{2}(G))) $. Define an operator $\mathcal{P}:\mathcal{H}\to L^{2}_{\mathbb{F}}(0,T; L^{2}(G))$ as follows: 
\begin{align*}
(\mathcal{P} u)(t,x) 
\deq & u_{t}(t, x)-u_{t}(0, x) 
\\
& 
- \int_0^t  \sum_{j, k=1}^n\big(b^{j k}(s, x) u_{x_{j}}(s, x)\big)_{x_{k}}  d s
\\
& 
- \int_0^t\left[
    b_{1}(s, x) u_{t}(s, x)
    + \left\langle b_2(s, x), \nabla u(s, x)\right\rangle
    + b_3(s, x) u(s, x)\right] ds 
\\
&   
-\int_0^t b_4(s, x) u(s, x) d W(s), \quad \text{${\mathbb{P}}$-a.s.,}\quad \forall\,u \in \mathcal{H},\;(t,x) \in Q.
\end{align*}
Put 
\begin{align*}
\mathcal{U} \deq \{ u \in \mathcal{H}
& \mid 
\mathcal{P} u  \in L^{2}_{\mathbb{F}}(\Omega; H^{1}(0,T;L^{2}(G))), ~
u|_{\Sigma} = 0, ~ \partial_{\nu} u |_{\Sigma_{0}} = h \}
.
\end{align*}
If $u$ is a solution to the equation \cref{ch-5-system1Cauchy}, then 
$\mathcal{P} u(t)=\int_{0}^{t} f(s) d s$, which yields that $u$ belong to $\mathcal{U}$. Consequently, $\mathcal{U}\neq\emptyset$.

Denote $ \mathbf{I}f (t) = \int_{0}^{t}  f(s) d s $ for $ t \in [0,T] $.
Fixing a function $ F \in \mathcal{U} $, we construct the Tikhonov functional as follows:
\begin{align}
\label{eqDou2024Fucntional}
\mathcal{J}_{\gamma}(u) = | \mathcal{P} u - \mathbf{I}f |^{2}_{L^{2}_{\mathbb{F}}(\Omega;H^{1}(0,T;L^{2}(G)))} + \gamma | u - F |^{2}_{\mathcal{H}},
\end{align} 
where $ u \in \mathcal{U} $ and $ \gamma \in (0,1) $.

We have the following result regarding the existence of the minimizer.

\begin{proposition}
For $ \gamma \in (0,1) $, there exists a unique minimizer $ u_{\gamma} \in \mathcal{U} $ of the functional $ \mathcal{J}_{\gamma}(u) $. 
\end{proposition}
\begin{proof}
Let  
\begin{align*}
\mathcal{U}_{0} \deq \{ u \in \mathcal{H} 
& \mid 
\mathcal{P} u  \in L^{2}_{\mathbb{F}}(\Omega; H^{1}(0,T;L^{2}(G))), 
 u|_{\Gamma} = 0, ~ \partial u |_{\Gamma_{0}} = 0 \}
.
\end{align*}
Fix $ \gamma \in (0,1) $.
We define the inner product as follows:
\begin{align}
\label{eqDou20231-1}
\langle  \varphi, \psi \rangle_{\mathcal{U}_{0}} = 
\langle \mathcal{P} \varphi, \mathcal{P} \psi \rangle_{L^{2}_{\mathbb{F}}(\Omega; H^{1}(0,T;L^{2}(G)))} 
+ 
\gamma \langle \varphi, \psi \rangle_{\mathcal{H}},
\end{align}
where $ \varphi, \psi \in \mathcal{U}_{0} $. 
Let $\overline{\mathcal{U}}_{0}$ be the completion of $\mathcal{U}_{0}$ with respect to the inner product $\langle \cdot, \cdot \rangle_{\mathcal{U}_{0}}$, and still denoted by $\mathcal{U}_{0}$ for the sake of simplicity.

For $ u \in \mathcal{U} $, let $ v = u - F $. Then $ v \in \mathcal{U}_{0} $. 
Consider the following functional
\begin{align*}
\overline{\mathcal{J}}_{\gamma}(v) = | \mathcal{P} v + \mathcal{P} F - \mathbf{I}f |^{2}_{L^{2}_{\mathbb{F}}(\Omega;H^{1}(0,T;L^{2}(G)))} + \gamma | v |^{2}_{\mathcal{H}},
\end{align*} 
where $ v \in \mathcal{U}_{0} $ and $ \gamma \in (0,1) $.
Clearly, $ u_{\gamma} $ minimizes $\mathcal{J}_{\gamma}(u) $ if and only if $ v_{\gamma} = u_{\gamma} - F $ is the minimizer of the functional $ \overline{\mathcal{J}}_{\gamma}(v) $.

By the variational principle, the minimizer $ v_{\gamma} $ of the functional $ \overline{\mathcal{J}}_{\gamma}(v) $ satisfy the following equation
\begin{align}
\label{eqDou20233-1} 
\langle \mathcal{P} v_{\gamma} + \mathcal{P} F - \mathbf{I}f, \mathcal{P} \rho   \rangle _{L^{2}_{\mathbb{F}}(\Omega;H^{1}(0,T;L^{2}(G)))}
+ \gamma \langle v_{\gamma}, \rho \rangle _{\mathcal{H}}
=0
, \quad 
\forall ~ \rho \in \mathcal{U}_{0}.
\end{align}
Conversely, the solution $ v_{\gamma} $ of equation \cref{eqDou20233-1} is the minimizer of the functional $ \overline{\mathcal{J}}_{\gamma}(v) $.
In fact, if $ v_{\gamma} $ solves \cref{eqDou20233-1}, then for each $ v \in \mathcal{U}_{0} $, it holds that 
\begin{align*}
    & \overline{\mathcal{J}}_{\gamma}(v) - \overline{\mathcal{J}}_{\gamma}(v_{\gamma}) 
    \\
    & =
    | \mathcal{P} v + \mathcal{P} F - \mathbf{I}f |^{2}_{L^{2}_{\mathbb{F}}(\Omega;H^{1}(0,T;L^{2}(G)))} + \gamma | v |^{2}_{\mathcal{H}}
    \\
    & \quad 
    - | \mathcal{P} v_{\gamma} + \mathcal{P} F - \mathbf{I}f |^{2}_{L^{2}_{\mathbb{F}}(\Omega;H^{1}(0,T;L^{2}(G)))} 
    - \gamma | v_{\gamma} |^{2}_{\mathcal{H}}
    \\
    & = 
    | \mathcal{P} v_{\gamma} + \mathcal{P} F - \mathbf{I}f + \mathcal{P}(v- v_{\gamma})|^{2}_{L^{2}_{\mathbb{F}}(\Omega;H^{1}(0,T;L^{2}(G)))} 
    \\
    & \quad 
    + \gamma | v_{\gamma} + (v-v_{\gamma}) |^{2}_{\mathcal{H}}
    - | \mathcal{P} v_{\gamma} + \mathcal{P} F - \mathbf{I}f |^{2}_{L^{2}_{\mathbb{F}}(\Omega;H^{1}(0,T;L^{2}(G)))} 
    - \gamma | v_{\gamma} |^{2}_{\mathcal{H}}
    \\
    & =
    2 \langle \mathcal{P} v_{\gamma} + \mathcal{P} F - \mathbf{I}f, \mathcal{P} (v-v_{\gamma})   \rangle _{L^{2}_{\mathbb{F}}(\Omega;H^{1}(0,T;L^{2}(G)))}
    \\
    & \quad 
    + 2 \gamma \langle v_{\gamma}, v-v_{\gamma} \rangle _{\mathcal{H}}
    + | \mathcal{P} (v - v_{\gamma}) |^{2}_{L^{2}_{\mathbb{F}}(\Omega;H^{1}(0,T;L^{2}(G)))} 
    + \gamma | v-v_{\gamma} |^{2}_{\mathcal{H}}
    \\
    & =
    | \mathcal{P} (v - v_{\gamma}) |^{2}_{L^{2}_{\mathbb{F}}(\Omega;H^{1}(0,T;L^{2}(G)))} 
    + \gamma | v-v_{\gamma} |^{2}_{\mathcal{H}}
    \\
    & \geq
    0
    ,
\end{align*}
which means $ v_{\gamma} $   is a minimizer of the functional $ \overline{\mathcal{J}}_{\gamma}(v) $.

From \cref{eqDou20231-1}, equation \cref{eqDou20233-1} is equivalent to  
\begin{align}
\label{eqDou202333-1}
\langle v_{\gamma}, \rho \rangle _{\mathcal{U}_{0}} = \langle \mathbf{I}f - \mathcal{P} F , \mathcal{P} \rho \rangle _{L^{2}_{\mathbb{F}}(\Omega;H^{1}(0,T;L^{2}(G)))} ,
\quad 
\forall ~ \rho \in \mathcal{U}_{0}
.
\end{align}

Thanks to Cauchy-Schwarz inequality and \cref{eqDou20231-1}, for all $ \rho \in \mathcal{U}_{0} $, we have 
\begin{align*}
& \big| \langle \mathbf{I}f - \mathcal{P} F , \mathcal{P} \rho \rangle _{L^{2}_{\mathbb{F}}(\Omega;H^{1}(0,T;L^{2}(G)))} \big| 
\\
&
\leq (| \mathcal{P} F|_{L^{2}_{\mathbb{F}}(\Omega; H^{1}(0,T;L^{2}(G)))} + |f|_{L^{2}_{\mathbb{F}}(0,T;L^{2}(G))} ) |\rho|_{\mathcal{U}_{0}} 
.
\end{align*}
Utilizing Riesz representation theorem, for all $ \rho \in \mathcal{U}_{0} $, there exists $ w_{\gamma} \in \mathcal{U}_{0} $ such that 
\begin{align*}
\langle w_{\gamma}, \rho \rangle _{\mathcal{U}_{0}} = \langle \mathbf{I}f- \mathcal{P} F , \mathcal{P} \rho \rangle _{L^{2}_{\mathbb{F}}(\Omega;H^{1}(0,T;L^{2}(G)))}
.
\end{align*}
Hence $ w_{\gamma} $ solves \cref{eqDou202333-1},
which means  $ w_{\gamma} $ is the minimizer of the functional $ \overline{\mathcal{J}}_{\gamma}(\cdot) $.
\end{proof}

Assume that $ z^{*} $ is the true solution to the equation \cref{ch-5-system1Cauchy} with the exact data 
\begin{align*}
f^{*} \in L^{2}_{\mathbb{F}}(0,T;L^{2}(G)), \quad 
z^{\ast }| _{\Gamma}=0, \quad   \partial
_{\nu}z^{\ast }|_{\Gamma_{0}}=h^{\ast
}\in L_{\mathbb{F}}^{2}(0,T;L^2(\Gamma_{0})).
\end{align*}
From \cref{ch-5-observability-1}, the solution $ z^{*} $ is unique.
Clearly, there exists $ F^{*} \in \mathcal{U} $ such that 
\begin{align}
\label{eqDou20237-1}
| F^{*} |_{\mathcal{H}} \leq \mathcal{C} |z^{*}|_{\mathcal{H}}  
.
\end{align}
We assume that 
\begin{align}
\label{eqDou2023Assumtion1-1}
|f^{*} - f|_{L^{2}_{\mathbb{F}}(0,T;L^{2}(G))} \leq \delta,  \quad  |h^{\ast }-h|_{L^2_\mathbb{F}(0,T; L^2(\Gamma))}\leq \delta,
\end{align}
and 
\begin{align}
\label{eqDou2023Assumtion2-1}
| \mathcal{P} F^{*} - \mathcal{P} F |_{L^{2}_{\mathbb{F}}(\Omega;H^{1}(0,T;L^{2}(G)))} +  | F^{*} -F  |_{\mathcal{H}} 
\leq 
\delta.
\end{align}

We now establish the estimate for the error between the  minimizer $ z_{\gamma} $ for the functional $\mathcal{J}_{\gamma}(\cdot) $ and the exact solution $ z^{*} $.

\begin{theorem}
    \label{thmConverge}
Assume \cref{eqDou2023Assumtion1-1}, \cref{eqDou2023Assumtion2-1}, and \cref{condition of d,ch-5-condition2} hold. 
Then there exists a constant $ \mathcal{C} > 0 $ such that for all $ \delta > 0 $, 
\begin{align*}
    |z_{\gamma} - z^{*}|_{\mathcal{H}}
    \leq 
    \mathcal{C} (1 + |z^{*}|_{\mathcal{H}}) \delta
\end{align*}
where $ z_{\gamma} $ is the minimizer of the functional \cref{eqDou2024Fucntional} with the regularization parameter  $ \gamma = \delta^{2} $.
\end{theorem}

\begin{proof}

Let $ v^{*} = z^{*} - F^{*} $. Then $ z^{*}  \in \mathcal{U}_{0} $ and $ \mathcal{P} v^{*} = - \mathcal{P}  F^{*} $. 
Hence, for all $ \rho \in \mathcal{U}_{0} $, it holds that
\begin{align}
\label{eqDou20232-1}
\langle \mathcal{P} v^{*} + \mathcal{P} F^{*} - \mathbf{I}f, \mathcal{P} \rho   \rangle _{L^{2}_{\mathbb{F}}(\Omega;H^{1}(0,T;L^{2}(G)))}
=  0
.
\end{align}
Subtracting  \cref{eqDou20232-1} from \cref{eqDou20233-1} and denoting $ \tilde{v}_{\gamma} = v^{*} - v_{\gamma} $, $ \tilde{f} = \mathbf{I}f^{*} - \mathbf{I}f $ and $ \widetilde{F} = F^{*} - F $, for all $ \rho \in \mathcal{U}_{0} $, we obtain 
\begin{align}\label{eqDou20232-2} \notag
& \langle \mathcal{P} \tilde{v}_{\gamma}, \mathcal{P} \rho   \rangle _{L^{2}_{\mathbb{F}}(\Omega;H^{1}(0,T;L^{2}(G)))}
+ \gamma \langle \tilde{v}_{\gamma}, \rho \rangle _{\mathcal{H}}
\\
& =  
\langle  \tilde{f} - \mathcal{P} \widetilde{F}, \mathcal{P} \rho   \rangle _{L^{2}_{\mathbb{F}}(\Omega;H^{1}(0,T;L^{2}(G)))}
+ \gamma \langle v^{*}, \rho \rangle _{\mathcal{H}}.
\end{align}
Choose $ \rho = \tilde{v}_{\gamma} $ in \cref{eqDou20232-2}. Then we have 
\begin{align}
\label{eqDou20234-1}
& |\mathcal{P} \tilde{v}_{\gamma}|_{L^{2}_{\mathbb{F}}(\Omega;H^{1}(0,T;L^{2}(G)))}^{2}
+ \gamma |  \tilde{v}_{\gamma} |^{2} _{\mathcal{H}}
 \leq 
| \tilde{f} -  \mathcal{P} \tilde{F}|^{2}_{L^{2}_{\mathbb{F}}(\Omega;H^{1}(0,T;L^{2}(G)))}
+ \gamma |v^{*}|^{2}_{\mathcal{H}}
.
\end{align}
Since $ \gamma = \delta^{2 } $, from \cref{eqDou2023Assumtion2-1,eqDou20234-1,eqDou20234-1}, we obtain
\begin{align}
\label{eqDou20235-1}
|  \tilde{v}_{\gamma} |^{2} _{\mathcal{H}} 
 \leq 
\mathcal{C} (1 + |v^{*}|^{2}_{\mathcal{H}} ),
\end{align}
and
\begin{align}
\label{eqDou20236-1}
	|\mathcal{P} \tilde{v}_{\gamma}|_{L^{2}_{\mathbb{F}}(\Omega;H^{1}(0,T;L^{2}(G)))}^{2} 
	  \leq 
	\mathcal{C} (1 + |v^{*}|^{2}_{\mathcal{H}} ) \delta^{2 }
	.
\end{align}
Let $ w_{\gamma} = \tilde{v}_{\gamma}\big(  1 + |v^{*}|_{\mathcal{H}} \big)^{-1} $. 
It follows from \cref{eqDou20235-1}, \cref{eqDou20236-1,ch-5-observability-1}  that,
\begin{align*}
| w_{\gamma}|_{\mathcal{H}} 
\leq 
\mathcal{C} \delta 
.
\end{align*}
This, together  with \cref{eqDou20237-1}, implies that
\begin{align}
\label{eqDou20238-1}
| \tilde{v}_{\gamma}|_{L^{2}_{\mathbb{F}}(\varepsilon, T-\varepsilon; H^{1}(\widehat{G}))}  
& \leq 
\mathcal{C} \big( 1 + |v^{*}|_{\mathcal{H}} \big) \delta
\leq 
\mathcal{C} \big( 1 + |z^{*}|_{\mathcal{H}} \big) \delta.
\end{align}
Since $ \tilde{v}_{\gamma} = v^{*} - v_{\gamma} = z^{*} - F^{*} - z_{\gamma} + F $, from \cref{eqDou2023Assumtion2-1}, we have 
\begin{align}\notag
\label{eqDou20239-1}
| \tilde{v}_{\gamma}|_{\mathcal{H}} 
& \geq 
| z_{\gamma} -   z^{*} |_{\mathcal{H}} 
- |F^{*} - F|_{\mathcal{H}} 
\\
& \geq 
| z_{\gamma} -   z^{*} |_{\mathcal{H}} 
- \delta
.
\end{align}
Combining \cref{eqDou20238-1,eqDou20239-1}, we complete the proof.
\end{proof}

\section{Further comments}
\label{secFChyper}

To the best of our knowledge, \cite{Zhang2008} is the pioneering work that utilized Carleman estimates in investigating inverse problems for stochastic hyperbolic equations. Subsequently, multiple studies have delved into the realm of inverse problems for stochastic hyperbolic equations using Carleman estimates,  such as \cite{Lue2013a,Lue2015a,Lue2020a,Wu2022,Yuan2015}.
Several of these works are introduced in this chapter. \cref{thmFi} is first established in \cite{Lue2020a}, while \cref{propHidden} make its debut in \cite{Zhang2008}. The energy estimate in \cref{ch-5-energy ensi} is established in \cite{Lue2013a}. \cref{th2} is proved in \cite{Lue2013a}, and the primary content of Section \ref{secISP} is taken from \cite{Lue2015a}. \cref{thmISPHyper} is taken from \cite{Yuan2015}, and the core material of Section \ref{secReconstructHyper} is referenced from \cite{Dou2024a}.

In this chapter, we focus on the inverse state problems and inverse source problems for stochastic hyperbolic equations. 
However, there are still many open problems related to inverse problems for that equation. 

\begin{itemize}

\item \emph{Inverse problems for  stochastic hyperbolic equations with less restrictions}

To prove inequality \cref{ch-5-obser esti2}, we assume  \Cref{condition of d,ch-5-condition2} hold.
However, for deterministic hyperbolic equations, the time $T$ does not need to be that large.
It is an interest question to weaken the condition on time $T$ for \cref{ch-5-observability} holds.

\item \emph{Inverse problems for  stochastic hyperbolic equations with general coefficients}

Replacing $ \Delta \tilde{z} $ in \cref{system1,systemHyper} with $ \sum\limits_{j, k=1}^n\left(b^{j k} \tilde{z}_{x_j}\right){x_k} $, it is also quite interesting to consider Problems \ref{probISPHyper} and \ref{probLocalInverse}.
In this case, the main challenge lies in constructing an appropriate weight function for Carleman estimates to ensure the validity of \cref{thmCarlemanHyperISP,thmInveHyperDeter}.

\item \emph{Inverse state and source problems for  semilinear stochastic  hyperbolic equations}

For a general semilinear stochastic hyperbolic equation \cref{ch-5-4.12-eq1}, can we investigate its inverse   source problem with both boundary and final time measurements, for instance, \cref{probISP}?

\item \emph{Reconstruct problem for the unknown state and source term}

For the equation \cref{systemHyper}, we obtain an estimate of the initial values $(z_{0}, z_{1})$ and  the diffusion term $g$ in \cref{eqOber}.
Analogous to \cref{secReconstructHyper}, can we reconstruct the unknown state and source problem with a boundary measurement and a final time measurement?

\item \emph{Efficiency algorithm for the construction of unknowns}

In \Cref{secReconstructHyper}, we discuss reconstructing the solution from boundary measurement using the Tikhonov regularization approach. 
When applying gradient methods to solve these optimization problems, challenges arise in  solving backward stochastic partial differential equations numerically, similar to stochastic parabolic equations (see \Cref{secFurPro}).
However, it is highly challenging to  solve backward stochastic partial differential equations numerically. 
Therefore, novel numerical algorithms must be developed for stochastic problems.

\item \emph{The strategy for selecting the prior information $F$}

In the reconstruct problem (\Cref{secReconstructHyper}), our objective is to minimize the Tikhonov functional \cref{eqDou2024Fucntional}, which requires selecting a function $F$.
How does the selection of different prior information $F$ affect    the efficacy of numerically solving the approximate solution?
Can a suitable method be proposed to find a better $F$?

\item \emph{The strategy for quantitatively selecting the regularization parameter}
 
In the Tikhonov functional \cref{eqDou2024Fucntional}, we introduce a regularization parameter $ \gamma $. 
In general, the regularization parameter  $ \gamma $   has a certain impact on the numerical error. 
In \cref{thmConverge}, we set a prior selection of the regularization parameter as $ \delta^{2} $. 
Can a method be provided for selecting the regularization parameter based on the input data?

\item \emph{The optimal measurement method}

In practical problems, measurements are typically made through the deployment of costly sensors. 
Numerous instances in engineering applications emphasize the significant impact of sensor placement on the achievable accuracy of inverse problems. 
Therefore, determining where to locate the observation becomes an important problem. 
Some interesting findings already exist in PDEs \cite{Ucinski2005}. 
However, as far as we know, there has been no published work on this problem for  SPDEs.

\item \emph{What can we benefit from uncertainties?}

In \cref{rkStoWave}, we mention that \cref{probISP} is essentially stochastic.
It can be observed that the uncertainty in SPDEs benefits certain inverse problems.
What benefits can we derive from uncertainties in solving inverse problems for SPDEs remains a largely unresolved puzzle.
We believe that studying this problem can lead to new insights into uncertainty.

\end{itemize}

\appendix

\chapter{Stochastic calculus}
\label{chApp}

In this chapter,  we recall the mathematical background material for SPDEs needed in this book. All of them are classical and can be found in standard books. We present them here for the readers' convenience. 
Additional details can be found in \cite[Chapter 2]{Lue2021a}.

\section{Measure and integration}

Let $ \Omega $ be a nonempty set. 

\begin{definition}
    A collection $ \mathcal{F} $ of subsets of $ \Omega $ is called a $ \sigma $-field if 
    \begin{enumerate}[(i)]
        \item  $ \Omega \in \mathcal{F} $,
        \item  $ E \in \mathcal{F} $ implies $ \Omega \setminus E  \in \mathcal{F} $
        \item $ \bigcup\limits_{i=1}^{\infty} E_{i} \in \mathcal{F} $ whenever each $ E_{i} \in \mathcal{F}$.
    \end{enumerate}
\end{definition}

We call $ (\Omega, \mathcal{F}) $ a measurable space if $ \mathcal{F} $ is a $ \sigma $-field on $ \Omega $.
Any element $ E \in \mathcal{F} $ is called a measurable set on $ (\Omega, \mathcal{F}) $.

In what follows, we fix a measurable space $ (\Omega, \mathcal{F}) $.

\begin{definition}
    A measure on $ (\Omega, \mathcal{F}) $ is a function $ \mu : \mathcal{F} \rightarrow [0, +\infty] $ such that 
    \begin{enumerate}[(i)]
        \item $ \mu(\emptyset) = 0 $,
        \item if $ \{E_{i}\}_{i=1}^{\infty} $ is a sequence of disjoint sets in $ \mathcal{F} $, i.e., $ E_{i} \cap E_{j}  = \emptyset$ for all $ i,j \in \mathbb{N}$ with $ i \neq j $, then $ \mu(\bigcup\limits_{i=1}^{\infty} E_{i}) = \sum\limits_{i=1}^{n} \mu(E_{i}) $.
    \end{enumerate}
\end{definition}
If $ \mu $ is a measure on $ (\Omega, \mathcal{F}) $, the triple $ (\Omega, \mathcal{F}, \mu) $ is called a measure space.

Next, we fix a measure space $ (\Omega, \mathcal{F}, \mu) $. 
A set $ E \in \mathcal{F} $ with $ \mu(E) = 0 $ is called a $ \mu $-null set.
If a proposition about points $ x \in \Omega $ is true except for $ x $ in some $ \mu $-null set, it is usual to say that the proposition is true $ \mu $-a.e. (or simply a.e., if there is no ambiguity).

\begin{definition}
    The measure space $ (\Omega, \mathcal{F}, \mu) $ is called complete (and $ \mu $ is said to be complete on $ \mathcal{F} $), if $ \mathcal{F} $ includes all subsets of null sets, i.e., 
    \begin{align}
        \label{eqNullset}
        \mathcal{N} = \{ G \subset \Omega \mid  G \subset E \text{ for some $ \mu $-null set } G \} \subset \mathcal{F}.
    \end{align}
\end{definition}

Completeness can always be achieved by enlarging the domain of $ \mu $. 

\begin{definition}
    Let $ \mathcal{X} $ be a topology space. The smallest $ \sigma $-field containing all open sets of $ \mathcal{X} $ is called the Borel $ \sigma $-field of $ \mathcal{X} $ and denoted by $ \mathcal{B}(\mathcal{X}) $. 
    Any set $ E \in \mathcal{B}(\mathcal{X}) $ is called a Borel set.
\end{definition}
\begin{definition}
    Let $ (\Omega, \mathcal{F}) $ and  $ (\widetilde{\Omega}, \widetilde{\mathcal{F}}) $ be two measurable spaces. A mapping  $ f: \Omega \rightarrow \widetilde{\Omega} $ is called $ \mathcal{F}/\widetilde{\mathcal{F}} $-measurable or simply $ \mathcal{F} $-measurable or even $ measurable $ if $ f^{-1}(\widetilde{\mathcal{F}}) \subset \mathcal{F} $. 
    In particular, if $ (\widetilde{\Omega}, \widetilde{\mathcal{F}}) = (\mathcal{X},\mathcal{B}(\mathcal{X}) )  $, then $ f $ is said to be an ($ \mathcal{X} $-valued) $ \mathcal{F} $-measurable (or measurable) function.
\end{definition}


For a measurable map $  f: (\Omega, \mathcal{F})\rightarrow (\widetilde{\Omega}, \widetilde{\mathcal{F}}) $, denote by $f^{-1}(\widetilde{\mathcal{F}})$ the $\sigma$-field generated by the set $\{f^{-1}(F)|F\in \widetilde{\mathcal{F}}\}$.  
One can easily find that $ f^{-1}(\widetilde{\mathcal{F}}) $,  is a sub $ \sigma $-field of $ \mathcal{F} $ which is called the $ \sigma $-field generated by $ f $, and denoted by $ \sigma(f) $.
In addition, for a family of measurable maps $ \{f_{\lambda}\}_{\lambda \in \Lambda} $, we denote by $ \sigma(f_{\lambda};\lambda \in \Lambda) $ the $ \sigma $-field generated by $ \bigcup_{\lambda \in \Lambda} \sigma(f_{\lambda}) $.

At last, we shall present the definition of measurable and integrable for Banach space-valued functions.
We will fix a Banach space $ \mathcal{X} $.

\begin{definition}
    Let $ f: \Omega \rightarrow \mathcal{X} $ be an $ \mathcal{X} $-valued function.
    \begin{enumerate}[(i)]
        \item The function $ f(\cdot ) $ is called an $ \mathcal{F} $-simple  function (or simple function when $ \mathcal{F} $ is clear) if 
        \begin{align}
            \label{eqSimplefuction}
            f (\cdot) = \sum_{j=1}^{k} \chi_{E_{j}}(\cdot) u_{j},
        \end{align}
        for some $ k \in \mathbb{N} $, $ u_{j} \in \mathcal{X} $, and $ \{E_{j}\}_{j=1}^{k} \subset \mathcal{F} $.
        Here, $ \chi_{E}(\cdot) $ is the characteristic function of $ E $.

        \item The function $ f(\cdot ) $ is said to be strongly $ \mathcal{F} $-measurable with respect to $ \mu $ (or simplify strongly measurable) if there exist simple functions $ \{ f_{k} \}_{k = 1}^{\infty} $ converging to $ f $ in $ \mathcal{X} $, $ \mu $-a.e.
    \end{enumerate}
\end{definition}

\begin{theorem}
    If the Banach space $ \mathcal{X} $ is separable and the measure space $ (\Omega, \mathcal{F}, \mu) $ is $ \sigma $-finite, then, a function  $ f: \Omega \rightarrow \mathcal{X} $ is strongly measurable if and only if $ f $ is measurable.
\end{theorem}

In what follows, we fix a $ \sigma $-finite measure space $ (\Omega, \mathcal{F}, \mu) $.

\begin{definition}
    If $ f(\cdot) $ is a simple function in the form \cref{eqSimplefuction}. 
    We call $ f(\cdot) $ Bochner integrable if $ \mu(E_{j}) < \infty $ for any $ j = 1, \cdots, k $.
    For any $ E \in \mathcal{F} $, we define the Bochner integral of $ f(\cdot) $ over $ E $ as 
    \begin{align*}
        \int_E f(s) d \mu=\sum_{j=1}^k \mu\left(E \cap E_j\right) u_{j} .
    \end{align*}
\end{definition}

For general strongly measurable functions, we have the following definition.

\begin{definition}
    We say the strongly measurable function $ f(\cdot) : \Omega \rightarrow \mathcal{X} $ is Bochner integrable with respect to $ \mu $ if there exists a sequence of Bochner integrable simple  functions $ \{ f_{i}(\cdot) \}_{i=1}^{\infty}    $ converging to $ f(\cdot) $ in $ X $, $ \mu $-a.e., such that 
    \begin{align*}
        \lim _{i, j \rightarrow \infty} \int_{\Omega}\left|f_i(s)-f_j(s)\right|_{\mathcal{X}} d \mu=0
        .
    \end{align*}
    For  $ E \in \mathcal{F} $, the Bochner integral of $ f(\cdot) $ over $ E $ is defined by 
    \begin{align*}
        \int_E f(s) d \mu=\lim _{i \rightarrow \infty} \int_{\Omega} \chi_E(s) f_i(s) d \mu(s) \quad \text { in } \mathcal{X} \text {. }
    \end{align*}
\end{definition}

The following theorem reveals the relationship between the Bochner integral and Lebesgue integral.

\begin{theorem}
    A strongly measurable function $ f(\cdot) : \Omega \rightarrow \mathcal{X} $ is Bochner integrable if and only if the scalar function $ |f(\cdot)|_{\mathcal{X}}: \Omega \rightarrow \mathbb{R} $ is integrable.
    In this case, for any $ E \in \mathcal{F} $,
    \begin{align*}
        \left|\int_E f(s) d \mu\right|_{\mathcal{X}} \leq \int_E|f(s)|_{ \mathcal{X}} d \mu
        .
    \end{align*}
\end{theorem}

\section{Probability, random variables and expectation}

A probability space is a measure space $ (\Omega, \mathcal{F}, \mathbb{P}) $ with $ \mathbb{P}(\Omega) = 1 $, where $ \mathbb{P} $ is called a probability measure. 
Set $ \Omega $ is called a sample space; $ \omega \in \Omega $ is called a sample point; any $ A \in \mathcal{F} $ is called an  event.

In the following, we fix a probability space $ (\Omega, \mathcal{F}, \mathbb{P}) $ and a Banach space $ \mathcal{X} $. 

\begin{definition}
    The function $ f: (\Omega, \mathcal{F}) \rightarrow (\mathcal{X}, \mathcal{B}(\mathcal{X})) $ is called an ($ \mathcal{X} - valued $) random variable if $ f $ is strongly measurable. 
\end{definition}

\begin{definition}
    Let $ f $ be a random variable. If $ f $ is Bochner integrable with respect to the probability measure $ \mathbb{P} $, then we denote 
    \begin{align*}
        \mathbb{E} f = \int_{\Omega} f d \mathbb{P}
    \end{align*}
    and call it the mean or mathematical expectation of $ f $.
\end{definition}

Next, we recall the notion of independence, which is a  fundamental  feature of probability theory.

\begin{definition}
    Let  $ (\Omega, \mathcal{F}, \mathbb{P}) $  be a probability space. 
    \begin{enumerate}[(i)]
        \item Let $ A, B \in \mathcal{F} $. We say that  $ A, B $ are independent w.r.t $ \mathbb{P} $ if $ \mathbb{P} (A \cap  B) = \mathbb{P} (A) \mathbb{P}(B) $.
        \item Let $ \mathcal{J}_{1} $ and $ \mathcal{J}_{2} $ be two subfamilies of $ \mathcal{F} $. We say that $ \mathcal{J}_{1} $ and $ \mathcal{J}_{2} $ are independent if for all $ (A,B) \in \mathcal{J}_{1} \times \mathcal{J}_{2} $, it holds that $ \mathbb{P}(A \cap B) = \mathbb{P}(A) \mathbb{P}(B) $.
        \item Let $f, g:(\Omega, \mathcal{F}) \rightarrow(\mathcal{X}, \mathcal{B}(\mathcal{X}))$ be two random variables. We say that $f$ and $g$  are independent if $\sigma(f)$ and $\sigma(g) $ are independent.
        \item Let $f: (\Omega, \mathcal{F}) \rightarrow(\mathcal{X}, \mathcal{B}(\mathcal{X}))$ be a random variable and $ \mathcal{J} \subset \mathcal{F} $. We say that $f$ and $ \mathcal{J} $  are independent if $\sigma(f)$ and $ \mathcal{J} $ are independent.
    \end{enumerate}
\end{definition}

\begin{definition}
    Let $ X = (X_{1}, \cdots, X_{n}) $ be a $ \mathbb{R}^{n} $-valued random variable. The function 
    \begin{align*}
        F(x) \equiv F\left(x_1, \cdots, x_n\right) \deq \mathbb{P}\left\{X_1 \leq x_1, \cdots, X_n \leq x_n\right\}
    \end{align*}
    is called the distribution function of $ X $. 
    Furthermore, if there exists a nonnegative function $ f(\cdot) $ such that 
    \begin{align*}
        F(x) \equiv F\left(x_1, \cdots, x_n\right)=\int_{-\infty}^{x_1} \cdots \int_{-\infty}^{x_n} f\left(\xi_1, \cdots, \xi_n\right) d \xi_1 \cdots d \xi_n,
    \end{align*}
    then the function $ f(\cdot) $ is called the density of $ X $.
\end{definition}

If random variable $ X $ has the following density: 
\begin{align*}
    f(x)=(2 \pi)^{-\frac{n}{2}}|\operatorname{det} Q|^{-1 / 2} \exp \left\{-\frac{1}{2 \mu}(x-\lambda)^{\top} Q^{-1}(x-\lambda)\right\} \text {, }
\end{align*}
where $ \lambda \in \mathbb{R}^{n} $ and $ Q $ is a positive definite $ n $-dimensional matrix, then $ X $ is called a  normally distributed random variable (or $ X $  has a normal distribution) and denoted by $ X \sim \mathcal{N}(\lambda, Q) $.  
Especially, we call $ X $ a standard normally distributed random variable if $ \lambda = 0 $ and $ Q $ is the $ n $-dimensional identity matrix.

\section{Stochastic process}

Let $T>0$. A stochastic process is a family of $ \mathcal{X} $-valued random variables $ \{X(t)\}_{t \in [0,T]} $ on $ ( \Omega, \mathcal{F}, \mathbb{P}) $.
If there is no ambiguity, we shall use $ \{X(t)\}_{t \in [0,T]}  $, $ X(\cdot) $ or $ X $ to denote a stochastic process.

For every $ \omega \in \Omega $ the function $ t \mapsto X(t, \omega) $ is called a sample path (or path) of $ X $.
The stochastic process $ X(\cdot) $ is said to be continuous is there exists a $ \mathbb{P} $-null set $ N \in \mathcal{F} $ such that for any $ \omega \in \Omega \setminus N$, the path $ X(\cdot, \omega) $  is continuous in $ \mathcal{X} $.

\begin{definition}
    A family of sub-$ \sigma $-field $ \{\mathcal{F}_{t}\}_{t \in [0,T]} $ is called a filtration if $ \mathcal{F}_{s} \subset \mathcal{F}_{t} $ for any $ s, t \in [0,T]$ and $ s \leq t $.
    We simply denote it by $ \mathbf{F} $.
\end{definition}

For any $ t \in [0,T) $, we put 
\begin{align*}
    \mathcal{F}_{t+} \deq \bigcap_{s \in(t, T]} \mathcal{F}_s, \quad \mathcal{F}_{t-} \deq \bigcup_{s \in[0, t)} \mathcal{F}_s .
\end{align*}
If $ \mathcal{F}_{t+} = \mathcal{F}_{t} $ (resp. $ \mathcal{F}_{t-} = \mathcal{F}_{t} $), then we call $ \mathbf{F} $ is right (resp. left) continuous. 

\begin{definition}
    We call $ (\Omega, \mathcal{F}, \mathbf{F}, \mathbb{P}) $ a filtered probability space, and it satisfies the usual condition if $ (\Omega, \mathcal{F}, \mathbb{P}) $ is complete, $ \mathcal{F}_{0} $ contains all $ \mathbb{P} $-null sets in $ \mathcal{F} $, and $ \mathbf{F} $ is right continuous.
\end{definition}

In the sequel, we always assume that the filtered probability $ (\Omega, \mathcal{F}, \mathbf{F}, \mathbb{P}) $ satisfies the usual condition.

\begin{definition}
    Let $ X(\cdot) $ be an $ \mathcal{X} $-valued stochastic process.
    \begin{enumerate}[(i)]
        \item  We say that $ X(\cdot) $ is measurable if the map $ (t , \omega) \mapsto X(t, \omega) $ is strongly $ (\mathcal{B}([0,T] )\times  \mathcal{F}) / \mathcal{B}(\mathcal{X}) $-measurable; 
        \item We say that $ X(\cdot) $ is $ \mathbf{F} $-adapted if it its measurable and for each $ t \in [0,T] $, the map $ \omega \mapsto X(t, \omega) $ is strongly $ \mathcal{F}_{t}/ \mathcal{B}(\mathcal{X}) $-measurable;
        \item We say that $ X(\cdot) $ is $ \mathbf{F} $-progressively measurable if for each $ t \in [0,T] $, the map $ (s, \omega) \mapsto X(s, \omega) $ from $ [0, t] \times \Omega $ is strongly $ ( \mathcal{B}([0,t]) \times \mathcal{F}_{t})/ \mathcal{B}(\mathcal{X}) $ -measurable.
    \end{enumerate}
\end{definition}

\begin{definition}
    We call a set $ A \in [0,T] \times \Omega $ is progressively measurable with respect to $ \mathbf{F} $ if the stochastic process $ \chi_{A}(\cdot) $ is progressively measurable, where $ \chi_{A} $ is the   indicator function  of $ A $.
\end{definition}

We denote the class of all progressively measurable sets by $ \mathbb{F} $, which is a $ \sigma $-field. 
It is clear that a process $ \varphi: [0,T]\times \Omega \rightarrow \mathcal{X} $ is $ \mathbf{F} $-progressively measurable if and only if it is strongly $ \mathbb{F} $-measurable.

Clearly, if $ X(\cdot) $ is $ \mathbf{F} $-progressively measurable, it must be $ \mathbf{F} $-adapted. 
Conversely, one can prove that, for any $ \mathbf{F} $-adapted process $ X(\cdot) $, there exists an $ \mathbf{F} $-progressively measurable process $ \overline{X}(\cdot) $, such that 
\begin{align*}
    \mathbb{P}(\{X(t)=\bar{X}(t)\})=1, \quad \forall \; t \in [0,T].
\end{align*}
Hence, by saying that a process $ X(\cdot) $ is $ \mathbf{F} $-adapted, we mean that it is $ \mathbf{F} $-progressively measurable.

\begin{definition}
    \label{defBrowian}
    The  standard  Brownian motion is a continuous, $ \mathbb{R}  $-valued, $ \mathbf{F} $-adapted stochastic process $ \{W(t)\}_{t \geq 0} $ such that
    \begin{enumerate}[(i)]
        \item $ W(0) = 0 $, $ \mathbb{P} $-a.s.; and 
        \item For any $ s, t \in [0, \infty) $ with $  s < t $, the random variable $ W(t) - W(s) $ is independent of $ \mathcal{F}_{s} $ and $ W(t) - W(s)  \sim \mathcal{N}(0, t-s ) $.
    \end{enumerate}
\end{definition}

We fix a standard Brownian motion $ W(\cdot) $.
For each $ t \in [0,T] $, put $ \mathcal{F}_{t}^{W} = \sigma(W(s); s \in [0,t]) \subset \mathcal{F}_{t} $. 
Let $ \{ \hat{\mathcal{F}}_{t}^{W} \}_{t \in [0,T]} $ be the argumentation of $ \{ \mathcal{F}_{t}^{W} \}_{t \in [0,T]} $ by adding all $ \mathbb{P} $-null sets.
Then, $ \{ \hat{\mathcal{F}}_{t}^{W} \}_{t \in [0,T]} $ is continuous and $ W(\cdot) $ is still a Brownian motion on the filtered probability space $ (\Omega, \mathcal{F}, \{ \hat{\mathcal{F}}_{t}^{W} \}_{t \in [0,T]} , \mathbb{P}) $.
We call $ \mathbf{F} $  the natural filtration generated by $ W(\cdot) $, if $ \mathbf{F} $ is  $ \{ \hat{\mathcal{F}}_{t}^{W} \}_{t \in [0,T]} $  generated as described above.

The notations to be given below will be used in the rest of the book.

For any $ p, q \in [1, \infty) $, denote

\begin{align*}
    L_{\mathbb{F}}^p\left(\Omega ; L^q(0, T ; \mathcal{X})\right) \deq 
    \Big\{\varphi:(0, T) \times \Omega \rightarrow \mathcal{X} 
    & \mid \varphi(\cdot)  \text{ is  $\mathbf{F}$-adapted and } 
    \\
    & \quad 
    \mathbb{E}\Big(\int_0^T|\varphi(t)|_{\mathcal{X}}^q d t\Big)^{\frac{p}{q}}<\infty \Big \}
    ,
\end{align*}
and
\begin{align*}
    L_{\mathbb{F}}^q\left(0, T ; L^p(\Omega ; \mathcal{X})\right) 
    \deq \Big \{\varphi:(0, T) \times \Omega \rightarrow \mathcal{X} \mid & \varphi(\cdot) \text { is $ \mathbf{F} $-adapted and } \\
    &  \int_0^T\left(\mathbb{E}|\varphi(t)|_{\mathcal{X}}^p\right)^{\frac{q}{p}} d t<\infty \Big \}
    .
\end{align*}
Similarly, we may also define (for $1 \leq p, q<\infty$ )
\begin{align*} 
        L_{\mathbb{F}}^{\infty}\left(0, T ; L^p(\Omega ; \mathcal{X})\right), && L_{\mathbb{F}}^q\left(0, T ; L^{\infty}(\Omega ; \mathcal{X})\right), && L_{\mathbb{F}}^{\infty}\left(0, T ; L^{\infty}(\Omega ; \mathcal{X})\right). 
\end{align*}
All these spaces are Banach spaces with the canonical norms. 
In the sequel, we will simply denote $
L_{\mathbb{F}}^p\left(0, T ; L^p(\Omega ; \mathcal{X})\right)$ by $L_{\mathbb{F}}^p(0, T ; \mathcal{X})$; and denote $ L_{\mathbb{F}}^p(0, T ; \mathbb{R}) $ by $L_{\mathbb{F}}^p(0, T)$.

For any $ p \in [1, \infty) $, write 
\begin{align*}
    L_{\mathbb{F}}^p(\Omega ; C([0, T] ; \mathcal{X})) \deq 
    \Big \{\varphi:&[0, T] \times \Omega \rightarrow \mathcal{X} 
     \mid \varphi(\cdot) \text{ is continuous, }   
    \\
    & \quad 
    \text{ $ \mathbf{F} $-adapted and }  \mathbb{E}\big(|\varphi(\cdot)|_{C([0, T] ; \mathcal{X})}^p\big)<\infty \Big\}
\end{align*}
It is clear that 
$L_{\mathbb{F}}^p(\Omega ; C([0, T] ; \mathcal{X}))$ 
is Banach spaces equipped with norms
\begin{align*}
|\varphi(\cdot)|_{L_{\mathbb{F}}^p(\Omega ; C([0, T] ; \mathcal{X}))}=\big[\mathbb{E}|\varphi(\cdot)|_{C([0, T] ; \mathcal{X})}^p\big]^{1 / p}
.
\end{align*}

\section{It\^{o}'s integral}

In this section, we denote $ H $ as a Hilbert space.
Let $ W(\cdot) $ be a standard Brownian motion on the filtered probability space $ (\Omega, \mathcal{F}, \mathbf{F}, \mathbb{P}) $.
We shall define the It\^o integral $ \int_{0}^{T} X(t) d W(t) $ of an $ H $-valued, $ \mathbf{F} $-adapted stochastic process $ X(\cdot) $ with respect to $ W(t) $.

Let us start with the step process $ f \in L^{2}_{\mathbb{F}}(0,T,H) $, which has the form 
\begin{align}
    \label{eqStepPro}
    f(t, \omega)=\sum_{j=0}^n f_j(\omega) \chi_{\left[t_j, t_{j+1}\right)}(t), \quad(t, \omega) \in[0, T] \times \Omega
    ,
\end{align}
where $ n \in \mathbb{N}$, $ 0 =t_0<t_1<\cdots<t_{n+1}=T$, $ f_j $ is $ \mathcal{F}_{t_{j}} $-measurable with 
\begin{align*}
    \sup \left\{\left|f_j(\omega)\right|_H \mid j \in\{0, \cdots, n\}, ~ \omega \in \Omega\right\}<\infty.
\end{align*}
The set of step processes will be denoted by $ \mathcal{L}_{0} $ and is dense in $ L^{2}_{\mathbb{F}}(0,T,H) $.

The stochastic integral of a step process $ f \in \mathcal{L}_{0} $ of the form \cref{eqStepPro} is defined by 
\begin{align*}
    I(f)(t, \omega)=\sum_{j=0}^n f_j(\omega)\left[W\left(t \wedge t_{j+1}, \omega\right)-W\left(t \wedge t_j, \omega\right)\right]
    .
\end{align*}
Clearly, $ I(f) \in L^{2}_{\mathcal{F}_{t}}(\Omega; H) $ and the following  It\^o isometry holds:
\begin{align*}
    |I(f)|_{L_{\mathcal{F}_t}^2(\Omega ; H)}=|f|_{L_{\mathbb{F}}^2(0, t ; H)}
    .
\end{align*}

Since $ \mathcal{L}_{0} $ is dense in $ L^{2}_{\mathbb{F}}(0, T; H) $, for $ f \in  L^{2}_{\mathbb{F}}(0, T; H)  $, there exists a sequence of $ \{ f_{k} \}_{k=1}^{\infty} \subset \mathcal{L}_{0}$ such that 
\begin{align*}
    \lim _{k \rightarrow \infty}\left|f_k-f\right|_{L_{\mathbb{F}}^2(0, T ; H)}=0.
\end{align*}
From It\^o isometry, one can find that $ \{ I(f_{k}) \}_{k=1}^{\infty} $ is a Cauchy sequence in $ L^{2}_{\mathcal{F}_{t}}(\Omega; H) $. 
Hence, it converges to a unique element in $ L^{2}_{\mathcal{F}_{t}}(\Omega; H) $, which is independent of the choice of $ \{ f_{k} \}_{k=1}^{\infty} $.
We call this element the It\^o integral of $ f $ on $ [0,t] $ and denote it by $ \int_{0}^{t} f d W $.

For $ 0 \leq s < t \leq T $, we define the It\^o integral of $ f $ on $ [s,t] $ by 
\begin{align*}
    \int_{s}^{t} f d W = \int_{s}^{t} f(\tau) d W(\tau) \deq \int_{0}^{t} f d W - \int_{0}^{s} f d W
    .
\end{align*}


One can prove the following properties of It\^o integral.

\begin{theorem}
    Let $ f, g \in L^{2}_{\mathbb{F}}(0, T; H) $ and $ a, b \in L^{2}_{\mathcal{F}_{s}} (\Omega) $, $ 0 \leq s  < t \leq T $. Then 
    \begin{enumerate}[(i)]
        \item $ \displaystyle \int_s^t f d W \in L_{\mathbb{F}}^2(\Omega ; C([s, t] ; H)) $;
        \item $ \displaystyle \int_s^t(a f+b g) d W=a \int_s^t f d W+b \int_s^t g d W, \text { a.s. } $;
        \item $ \displaystyle \mathbb{E}\Big(\int_s^t f d W\Big)=0 $;
        \item $ \displaystyle \mathbb{E}\Big(\Big\langle\int_s^t f d W, \int_s^t g d W\Big\rangle_H\Big)=\mathbb{E}\Big(\int_s^t\langle f(\tau, \cdot), g(\tau, \cdot)\rangle_H d \tau\Big) $.
    \end{enumerate}
\end{theorem}

\begin{definition}
    We call an $ H $-valued, $ \mathbf{F} $-adapted, continuous stochastic process $ X(\cdot) $ an It\^o process, if there exist two $ H $-valued process $ \phi(\cdot) \in L^{1}_{\mathbb{F}}(0,T; H) $ and $ \Phi(\cdot) \in L^{2}_{\mathbb{F}}(0,T; H) $ such that 
    \begin{align}
        \label{eqItoProcess}
        X(t)=X(0)+\int_0^t \phi(s) d s+\int_0^t \Phi(s) d W(s), \quad \text { a.s., } \quad \forall \; t \in[0, T].
    \end{align}
\end{definition}

The following is a fundamental result called It\^o formula.

\begin{theorem}
    Let $ X(\cdot) $ be an It\^o process given by \cref{eqItoProcess} and $ F : [0,T] \times H \rightarrow  \mathbb{R} $ be a function such that its partial derivatives $ F_{t}, F_{x}, F_{xx} $ are uniformly continuous on any bounded subset of $ [0,T]\times H $. It holds that 
    \begin{align}
        \notag
        \label{eqItoFor}
        & F(t, X(t))-F(0, X(0))
         \\\notag
        & =\int_0^t F_x(s, X(s)) \Phi(s) d W(s)+\int_0^t \Big[F_t(s, X(s))+F_x(s, X(s)) \phi(s) 
        \\
        &  \quad+\frac{1}{2}\left\langle F_{x x}(s, X(s)) \Phi(s), \Phi(s)\right\rangle_H \Big] d s, \quad \text { a.s. }, \quad \forall t \in[0, T]
        .
    \end{align}
\end{theorem}

Usually, we write \cref{eqItoFor} in the following differential form:
\begin{align*}
    d F(t, X(t))= & F_x(t, X(t)) \Phi(t) d W(t)+F_t(t, X(t)) d t+F_x(t, X(t)) \phi(t) d t \\
    & +\frac{1}{2}\left\langle F_{x x}(t, X(t)) \Phi(t), \Phi(t)\right\rangle_H d t
    .
\end{align*}

When dealing with stochastic partial differential equations, the following  form of It\^o's formula for weak solutions will be useful.

Let $ \mathcal{V} $ be a Hilbert space such that the embedding $ \mathcal{V} \subset H $ is continuous and dense. 
Let $ \mathcal{V}^{*} $ be the dual space of $ \mathcal{V} $ with respect to the pivot space $ H $.
Hence, $ \mathcal{V} \subset H = H^{*} \subset \mathcal{V}^{*} $, continuously and densely.

\begin{theorem}
    Suppose that $X_0 \in L_{\mathcal{F}_0}^2(\Omega ; H)$,  $ \phi(\cdot) \in L_{\mathbb{F}}^2\left(0, T ; \mathcal{V}^*\right)$, and $\Phi(\cdot) \in$ $L_{\mathbb{F}}^2 ( 0, T ; H) $.
    Assume
    \begin{align*}
    X(t)=X_0+\int_0^t \phi(s) d s+\int_0^t \Phi(s) d W(s), \quad t \in[0, T]
    .
    \end{align*}
    If $X(\cdot) \in L_{\mathbb{F}}^2(0, T ; \mathcal{V})$, then $X(\cdot) \in C([0, T] ; H)$, a.s., and for any $t \in[0, T]$,
    \begin{align}
        \notag
        \label{eqItoWeakFor}
        |X(t)|_H^2= & \left|X_0\right|_H^2+2 \int_0^t\langle\phi(s), X(s)\rangle_{\mathcal{V}^*, \mathcal{V}} d s \\
        & +2 \int_0^t\langle\Phi(s), X(s)\rangle_H d W(s)+\int_0^t|\Phi(s)|_H^2 d s, \quad \text { a.s. }
    \end{align}
\end{theorem}

As before, we typically express \cref{eqItoWeakFor} in the following differential form:
\begin{align*}
    d|X(t)|_H^2
    =2\langle\phi(t), X(t)\rangle_{\mathcal{V}^*, \mathcal{V}} d t
    +2\langle\Phi(t), X(t)\rangle_H d W(t)
    +|\Phi(t)|_H^2 d t
\end{align*}
and sometimes denote $|\Phi(t)|^{2}_{H} dt$ as $|dX|^{2}_{H}$ for simplicity. 

\section{Stochastic evolution equations}

Let  $ H $ be a separable Hilbert space, and $ A $ be an unbounded linear operator (with domain $ D(A) $) on $ H $.
Assume that $ A $ is the infinitesimal generator of a $C_{0} $-semigroup $ \{S(t)\}_{t \geq 0} $.

We consider the following stochastic evolution equation:
\begin{align}
    \label{eqSEE}
    \left\{ 
        \begin{aligned}
            &d X(t)=(A X(t)+F(t, X(t))) d t+\widetilde{F}(t, X(t)) d W(t) \quad \text { in }(0, T], \\
            &X(0)=X_0,
        \end{aligned}    
    \right.
\end{align}
where $ X_{0}  \in L^{2}_{\mathcal{F}_{0}}(\Omega, H) $ and $ F(\cdot, \cdot), $ $ \widetilde{F}(\cdot, \cdot)  $ are two functions from $ [0,T] \times \Omega \times H $ to $ H $.

\begin{definition}
    We call an $ H $-valued, $ \mathbf{F} $-adapted, continuous stochastic process $ X(\cdot) $ a strong solution to \cref{eqSEE} if 
    \begin{enumerate}[(i)]
        \item $X(t) \in D(A)$ for a.e. $(t, \omega) \in[0, T] \times \Omega$ and $A X(\cdot) \in L^1(0, T ; H)$ a.s.; 
        \item $F(\cdot, X(\cdot)) \in L^1(0, T ; H)$ a.s., $\widetilde{F}(\cdot, X(\cdot)) \in L_{\mathbb{F}}^{2}(0, T ; H)$;  
        \item For all $t \in[0, T]$,
        \begin{align*}
        X(t)=X_0 \!+\! \int_0^t(A X(s)+F(s, X(s))) d s+\int_0^t \widetilde{F}(s, X(s)) d W(s), \text{ a.s. }
        \end{align*}
    \end{enumerate}
\end{definition}

In general, the conditions ensuring the  existence of strong solutions to \cref{eqSEE} are highly restrictive.
Hence, we introduce the following definitions of the weak and the mild solution to \cref{eqSEE}, respectively.

\begin{definition}
    We call an $ H $-valued, $ \mathbf{F} $-adapted, continuous stochastic process $ X(\cdot) $ a weak solution to \cref{eqSEE}     if $F(\cdot, X(\cdot)) \in L^1(0, T ; H)$ a.s., $\widetilde{F}(\cdot, X(\cdot)) \in L_{\mathbb{F}}^{2}(0, T ; H)$, and for any $t \in[0, T]$ and $\xi \in D\left(A^*\right)$,
    \begin{align*}
        \langle X(t), \xi\rangle_H= & \left\langle X_0, \xi\right\rangle_H+\int_0^t\left(\left\langle X(s), A^* \xi\right\rangle_H+\langle F(s, X(s)), \xi\rangle_H\right) d s \\
        & +\int_0^t\langle\widetilde{F}(s, X(s)), \xi\rangle_H d W(s), \quad \text { a.s. }
    \end{align*}
\end{definition}

\begin{definition}
    We call an $ H $-valued, $ \mathbf{F} $-adapted, continuous stochastic process $ X(\cdot) $ a mild solution to \cref{eqSEE} if   $F(\cdot, X(\cdot)) \in L^1(0, T ; H)$ a.s., $\widetilde{F}(\cdot, X(\cdot)) \in L_{\mathbb{F}}^{2}(0, T; H)$, and for any $t \in[0, T]$,
    \begin{align*}
    X(t)\!=\!S(t) X_0\!+\!\int_0^t\! S(t\!-\!s) F(s, X(s)) d s\!+\!\int_0^t\! S(t\!-\!s) \widetilde{F}(s, X(s)) d W(s),  \;\,  \text { a.s. }
    \end{align*}
\end{definition}

\begin{proposition}\label{propWeM}
     Any weak solution to \cref{eqSEE} is also a mild solution to the same equation, and conversely.
\end{proposition}
From \cref{propWeM}, we do not distinguish the mild and the weak solution to \cref{eqSEE}.

It is obvious that a strong solution is also a weak and mild solution to the same equation.
Conversely, the following is a sufficient condition for a mild solution to be a strong solution to \cref{eqSEE}.

\begin{proposition}
     A mild solution $X(\cdot)$  is a strong solution  to \cref{eqSEE} if the following three conditions hold for all $x \in H$ and $0 \leq s \leq t \leq T$, a.s.,  
    \begin{enumerate}[(i)]
        \item  $X_0 \in D(A), S(t-s) F(s, x) \in D(A)$ and $S(t-s) \widetilde{F}(s, x) \in D(A)$;
        \item  $|A S(t-s) F(s, x)|_H \leq \alpha(t-s)|x|_H$ for some real-valued stochastic process $\alpha(\cdot) \in L_{\mathbb{F}}^{1}(0, T)$ a.s.; 
        \item  $|A S(t-s) \widetilde{F}(s, x)|_H \leq \beta(t-s)|x|_H$ for some real-valued stochastic process $\beta(\cdot) \in L_{\mathbb{F}}^{2}(0, T)$ a.s.
    \end{enumerate}
\end{proposition}

In what follows, we assume that  both $F(\cdot, x)$ and $\widetilde{F}(\cdot, x)$ are $\mathbf{F}$ adapted for each $x \in H$, and there exist two nonnegative (real-valued) functions $L_1(\cdot), ~ L_2(\cdot) \in L^2(0, T)$ such that for a.e. $t \in[0, T]$ and all $y, z \in H$,
\begin{align*}
\left\{
    \begin{aligned}
        & |F(t, y)-F(t, z)|_H \leq L_1(t)|y-z|_H, \quad \text { a.s., } 
        \\
        &
|\widetilde{F}(t, y)-\widetilde{F}(t, z)|_H \leq L_2(t)|y-z|_H, \quad \text { a.s., } 
\\
&
F(\cdot, 0) \in L_{\mathbb{F}}^2 ( 0, T ; H) , \quad \widetilde{F}(\cdot, 0) \in L_{\mathbb{F}}^2 ( 0, T ; H) 
.
    \end{aligned}
\right.
\end{align*}





The following is the well-posedness of mild solutions for a class of stochastic evolution equations.

\begin{theorem}
    \label{eqWellposed}
    Assume that $A$ is a self-adjoint, negative definite   linear  operator on $H$. Then, the equation \cref{eqSEE} admits a unique mild solution $X(\cdot) \in L_{\mathbb{F}}^2(\Omega ; C([0, T] ; H)) \cap L_{\mathbb{F}}^2 (0, T ; D((-A)^{\frac{1}{2}} ) $. Moreover,
\begin{align*}
    & |X(\cdot)|_{L_{\mathbb{F}}^2(\Omega ; C([0, T] ; H))}+|X(\cdot)|_{L_{\mathbb{F}}^2 (0, T ; D((-A)^{\frac{1}{2}} )} \\
    & \leq \mathcal{C}\big(\left|X_0\right|_{L_{\mathcal{F}_0}^2(\Omega ; H)}+|F(\cdot, 0)|_{L_{\mathbb{F}}^2 ( 0, T ; H)}+|\widetilde{F}(\cdot, 0)|_{L_{\mathbb{F}}^2 ( 0, T ; H)}\big)
    .
\end{align*}
\end{theorem}

Generally speaking, the stochastic convolution
$\int_0^t S(t-s)\widetilde F(s,X(s))dW(s)$ is no longer
a martingale. Hence, one cannot apply It\^{o}'s
formula directly to mild solutions to
\cref{eqSEE}. For example, when
establishing the pointwise identities (for
Carleman estimates) for stochastic partial
differentia equations of second order, we need
the function to be second differentiable in the
sense of weak derivative w.r.t the space
variables. Nevertheless, this sort of problems
can be solved by the following strategy:
\begin{enumerate}
	\item Introduce suitable approximating equations with
	strong solutions such that the limit of these
	strong solutions is the mild/weak solution to
	the original equation;
	\item  Obtain the desired
	properties for the above strong solutions;
	\item Utilize the density argument  to establish the desired
	properties for the mild/weak solutions.
\end{enumerate}
There are many methods to implement the above
three steps in the setting of deterministic
partial differential equations. Roughly
speaking, any of these methods, which do not
destroy the adaptedness of solutions, can be
applied to stochastic partial differential
equations.  Here we  present one.  To this end, we
introduce the following approximating equations of
\cref{eqSEE}:
\begin{equation}\label{c1-system3}
	\left\{\!\!
	\begin{array}{ll}\displaystyle
		dX_\lambda(t)\! =\! \big(AX_\lambda(t) + R(\lambda)
		F(t,X_\lambda(t))\big)dt\\ \displaystyle \qquad\qquad +
		R(\lambda)\widetilde F(t,X_\lambda(t))dW(t) &\mbox{ in } (0,T],\\
	 \displaystyle X_\lambda(0)=R(\lambda)X_0.
	\end{array}
	\right.
\end{equation}
Here, $\lambda$ belongs to \index{$\rho(A)$} $\rho(A)$, the resolvent
set of $A$, and $R(\lambda)\deq \lambda(\lambda I-A)^{-1}$.

We have the following result.

\begin{theorem}\label{ch-2-app1}
Assume that	$A$ generates a contraction semigroup,
	$F(\cdot, 0)\in L^2_{\mathbb{F}}(\Omega;L^1(0,T;H))$. Then, for
	each $X_0\in L^2_{\mathcal{F}_0}(\Omega;H)$ and
	$\lambda\in\rho(A)$, the equation \eqref{c1-system3}
	admits a unique strong solution  
	$X_\lambda(\cdot)\in L^2_{\mathbb{F}}(\Omega;C([0,T];D(A)))$.
	Moreover, as $\lambda\to\infty$, the solution
	$X_{\lambda}(\cdot)$ converges to $X(\cdot)$ in 
	$L^2_{\mathbb{F}}(\Omega;C([0,T];$ $H))$, where $X(\cdot)$
	solves \cref{eqSEE} in the sense of the
	mild solution.
\end{theorem}

\end{document}